\documentclass{amsart}%Indique la nature de l'écrit pour que les règles de typographie correspondent
\usepackage[english,french]{babel}%Les différentes langues utilisées, la principale étant la dernière
\usepackage[utf8]{inputenc}%Où l'on spécifit que l'encodage est universel que l'on peut mettre des accents dans le .tex
\usepackage[TS1,T1]{fontenc}%Où l'on demande qu'il fasse des jolis accents
\usepackage{lmodern}%Parce que quand c'est moderne, c'est toujours bien ;)
\usepackage{amsmath,amssymb,amsthm}%Pour écrire des jolies mathématiques
\usepackage{color}%Pour mettre de la couleur
\usepackage[a4paper]{geometry}%Pour contrôler la dimension des marges
\usepackage{hyperref}%Pour mettre des liens hypertextes internes ou externes dans le PDF
\usepackage{textcomp}% Pour avoir accès à tout plein de symboles non mathématiques
\usepackage{url}

\newtheorem{theo}{Théorème}[section]  
\newtheorem{prop}[theo]{Proposition}
\newtheorem{lemme}[theo]{Lemme}
\newtheorem*{theo*}{Théorème}

\newtheorem{corollaire}[theo]{Corollaire}
\theoremstyle{remark}
\newtheorem*{remarque}{Remarque}

\numberwithin{equation}{section}

\DeclareMathOperator{\Cl}{Cl}
\newcommand{\comp}{\mathcal{K}}
\newcommand{\diff}{\mathop{}\mathopen{}\mathrm{d}}% l'intégrand des intégrales
\newcommand{\e}{\mathrm{e}}% l'exponentielle
\newcommand{\N}{\mathbf{N}}
\newcommand{\Z}{\mathbf{Z}}
\newcommand{\Q}{\mathbf{Q}}
\newcommand{\R}{\mathbf{R}}
\newcommand{\C}{\mathbf{C}}
\newcommand{\K}{\mathbf{K}}
\newcommand{\J}{\mathfrak{j}}
\newcommand{\I}{\mathfrak{i}}
\newcommand{\Qid}{\mathfrak{q}}
\newcommand{\Kid}{\mathfrak{K}}
\newcommand{\ens}{\mathcal{E}}
\renewcommand{\L	}{\mathbf{L}}
\newcommand{\Rid}{\mathfrak{r}}
\newcommand{\Sid}{\mathfrak{s}}
\DeclareMathOperator{\Aire}{Vol}
\def\prim {\mathcal{P}}
\newcommand{\OK}{\mathcal{O}_{\mathbf{K}}}
\newcommand{\OL}{\mathcal{O}_{\mathbf{L}}}
\newcommand{\JK}{\mathcal{J}(\mathbf{K})}
\newcommand{\GK}{G(\mathbf{K})}
\newcommand{\X}{x}
\newcommand{\Y}{y}
\newcommand{\z}{z}
\newcommand{\HK}{H(\mathbf{K})}
\newcommand{\PK}{P(\mathbf{K})}
\newcommand{\hK}{h(\mathbf{K})}
\newcommand{\rK}{r(\mathbf{K})}
\newcommand{\IK}{\mathcal{I }( \mathbf{K})}
\newcommand{\degpol}{g}
\usepackage{hyperref}
\def\m {\boldsymbol m}
\def\n {\boldsymbol n}

\def\v {\boldsymbol v}
\newcommand{\nbin}{D_k}

\makeatletter
\newcommand{\pushright}[1]{\ifmeasuring@#1\else\omit\hfill$\displaystyle#1$\fi\ignorespaces}
\newcommand{\pushleft}[1]{\ifmeasuring@#1\else\omit$\displaystyle#1$\hfill\fi\ignorespaces}
\makeatother

\usepackage{geometry}
\begin{document}\author{Armand Lachand}

\address{%A. Lachand :
Institut \'Elie Cartan, Universit\'e de Lorraine 
1, B.P. 70239, \goodbreak
54506 Vand\oe uvre-l\`es-Nancy Cedex, France}

\email{armand.lachand@univ-lorraine.fr }

\title{Fonctions arithmétiques et formes binaires irréductibles de degré $3$}

\date{\today}
\subjclass[2010]{Primary: 11N32; Secondary: 
11N25, 11N36, 11N37. }
\selectlanguage{english}
\begin{abstract}Let $F(X_1,X_2)\in\Z[X_1,X_2] $ be an irreducible binary form of degree $3$ and $h$ an  arithmetic function. 
We give some estimates for 
 the average order $\sum_{\substack{|n_1|\leq\X,|n_2|\leq\X}}h(F(n_1,n_2))$ when $h$ satisfy certain conditions.  
As an application, we provide some asymptotic formula  for the number  of $\Y$-friable values of $F(n_1,n_2)$ when the variables $n_1,n_2$ lies in the square $[1,\X]^2$ and uniformly in the region $\exp\left(\frac{\log \X}{(\log\log \X)^{1/2-\varepsilon}}\right)\leq\Y\leq \X$. This improves the result of 
\cite{BBDT12} in the case of cubic form. \end{abstract}

\maketitle
\tableofcontents\selectlanguage{french}

\section{Introduction}

Soient $F(X_1,X_2)\in\Z[X_1,X_2]$ une forme binaire irréductible et primitive de degré $3$ et $h$ une fonction arithmétique. 
L'objet de cet article est l'obtention de formules asymptotiques de l'ordre moyen
\begin{align*}
 S(\X;h,F):=\sum_{1\leq n_1,n_2\leq \X}h(F(n_1,n_2))
\end{align*}
pour une certaine classe de fonctions $h$. 

L'intérêt de considérer de telles sommes est multiple. D'un côté, la quantité $S(\X;h,F)$ apparaît dans de nombreux problèmes arithmétiques (conjecture de Manin dans \cite{BT12} et \cite{BT13}, crible algébrique (NFS) exposé dans le chapitre 6 de \cite{CP05}, etc.). D'autre part, il s'agit d'un angle d'attaque efficace pour  étudier l'indépendance entre la propriété arithmétique relative à $h$ et la représentation d'un entier par la forme binaire $F$. Enfin, les valeurs de  $F$ forment un ensemble lacunaire d'entiers, dans la mesure où
\begin{displaymath}
\frac{\log\#\left\{n_1,n_2:|F(n_1,n_2)|\leq\X\right\}}{\log\{|n|\leq \X\}}\rightarrow \frac{2}{3}\qquad(\X\rightarrow+\infty)
\end{displaymath}
et la persistance d'un phénomène arithmétique dans un ensemble de densité nulle est un sujet important de la théorie des nombres à l'origine de nombreuses conjectures.

Greaves~\cite{Gr70}  initie l'étude de telles sommes en considérant  $h=\tau$ la fonction de compte des diviseurs. À partir de la convolution $\tau=1*1$, il montre  que, pour tout $\varepsilon>0$ et toute forme cubique, primitive et  irréductible $F$% de discriminant non nul
, il existe des constantes réelles $c_0(F)>0$ et $c_1(F)$ telles que l'on ait
\begin{align*}
S(\X;\tau,F)=c_0(F)  \X^2\log \X+c_1(F)  \X^2+O_{\varepsilon,F}\left(  \X^{27/14+
\varepsilon}\right),
\end{align*} formule asymptotique précisée par Daniel \cite{Da99} qui améliore sensiblement le terme d'erreur précédent en le remplaçant par $  \X^{15/8+\varepsilon}$.

Plus récemment, La Bretèche et Browning~\cite{BB06} obtiennent une borne supérieure de $S(\X;h,F)$ uniforme pour la classe $\mathcal{M}(A,B)$ des fonctions  $h$ positives et sous-multiplicatives satisfaisant  les inégalités
$h(p^k)\leq A^k$ pour tout premier $p$ et $k\geq1$ et $h(n)\leq B(\varepsilon)n^{\varepsilon}$ pour tout entier $n$ et tout réel $\varepsilon>0$. Si l'on omet la dépendance en $F$ rendue explicite dans leurs travaux, le corollaire 1 de \cite{BB06} assure que, pour  $F$ une forme binaire irréductible de degré $\degpol\geq2$, $h\in\mathcal{M}(A,B)$ et $\X\geq2$, on a
\begin{equation}\label{majoration Breteche Browning}
S(\X;h,F)\ll_{A,B,F}\  \X^2\prod_{3<p\leq \X}\left(1+\frac{\gamma_F(p)}{p^2}(h(p)-1)\right)+\X(\log \X)^{A^{\degpol}-1}
\end{equation}où
\begin{align*}
 \gamma_F(d):=\#\{1\leq n_1,n_2\leq d:d|F(n_1,n_2)\}.
\end{align*}
On pourra également consulter \cite{BT12} pour une généralisation de ce résultat.

Lorsque $h$ est l'indicatrice $1_{\mathcal{E}}$ d'un ensemble d'entiers $\mathcal{E}$, l'estimation précédente fournit une borne supérieure non triviale lorsque $\mathcal{E}$ est un ensemble criblé. La question de la détermination d'une minoration de $S(\X;1_{\mathcal{E}},F)$ est en général difficile et dépend fortement de l'ensemble $\mathcal{E}$ en question. Dans cette direction, Greaves considère l'ensemble $\mathcal{E}_k$ des entiers qui ne sont divisibles par aucune puissance $k$-ième de nombres premiers. Il montre ainsi dans \cite{Gr92} que, pour tout $k\geq2$ et toute forme  cubique primitive et irréductible $F$, on a la formule asymptotique
\begin{align*}
S(\X;1_{\mathcal{E}_k},F)
=\prod_p\left(1-\frac{\gamma_F(p^k)}{p^{2k}}\right)  \X^2+O_{F,k}\left(\frac{  \X^2}{\log \X}\right)
\end{align*}
et obtient également des formules asymptotiques similaires  pour des formes $F$ de degré $\degpol\geq4$ (voir aussi le travail de Browning~\cite{Br11}).

On peut  également
considérer le choix $\mathcal{E}:=S(\Y)$ de l'ensemble des entiers $\Y$-friables, c'est-à-dire
\begin{align*}
S(\Y):=\left\{n:P^+(n)\leq \Y \right\}
\end{align*}
où $P^+(n)$ désigne le plus grand facteur premier de l'entier $n$. Dans toute la suite,  le cardinal  $S(\X;1_{S(\Y)},F)$ est noté $\Psi_F(\X,\Y)$.
Son importance  dans l'algorithme de factorisation du crible algébrique (NFS) est essentielle et se trouve détaillée dans \cite{CP05} ou \cite{Gr08}. Un raisonnement heuristique permet en particulier de conjecturer que, pour une forme irréductible de degré $\degpol$, on a la formule asymptotique
\begin{equation}\label{conjecture rho friable}
\Psi_F(\X,\Y)\sim   \X^2\rho(\degpol u)
\end{equation}
où $u:=\frac{\log \X}{\log \Y }$ est borné et $\rho$ désigne la fonction de Dickman. La démonstration d'un tel résultat dans la cas des formes cubiques constitue la motivation initiale de cet article.

Dans leur étude de $\Psi_F(\X,\Y)$, Balog, Blomer, Dartyge et Tenenbaum~[\cite{BBDT12}, théorème 1] ont  montré entre autre   que, pour toute forme binaire $F$ de degré $g$, il existe une constante $\alpha_F\in [0,g[$ telle que, pour tout $\alpha>\alpha_F$, on ait
\begin{equation}\label{théorème BBDT}
\Psi_F(\X,  \X^{\alpha})\gg_{F,\alpha}   \X^2.
\end{equation}
En particulier, dans le cas où $F$ est une forme cubique irréductible,  le théorème 2 de \cite{BBDT12} permet de choisir $\alpha_F=e^{-1/2}$.

Tous ces travaux ont en commun 
l'emploi récurrent de sommes de Type I, c'est-à-dire des 
estimations du cardinal \begin{align*}
 N(\X,d)=\#\left\{1\leq n_1,n_2\leq\X:d|F(n_1,n_2)\right\}
\end{align*}
en moyenne sur $d\leq \X^{2-\varepsilon}$ où $\varepsilon>0$. 
En découpant le carré $[1,\X]^2\cap\Z^2$ en classes de congruences modulo $d$, on est en droit d'espérer approcher $
 N(\X,d)$
par la quantité $\frac{\gamma_F(d)}{d^2}\  \X^2$. La pertinence d'une telle approximation est notamment étudiée dans \cite{Gr71} où est établie la majoration de  sommes, dites  de Type I,  suivante 
 valide pour tout $\varepsilon>0$, 
\begin{align*}
\sum_{d\leq D}
\left|r_d(\X)
\right|
\ll_{\varepsilon} D^{\varepsilon}(\X+D)
\end{align*} 
où 
\begin{equation}
\label{définition rdX}r_d(\X):=N(\X,d)-\frac{\gamma_F(d)}{d^2}\  \X^2,
\end{equation}
estimation par la suite précisée dans le lemme 3.3 de \cite{Da99} sous la forme
\begin{equation}\label{Type I Daniel}
\sum_{d\leq D}
\left|r_d(\X)
\right|
\ll \X\sqrt{D}(\log 2D)^{7203}+D(\log 2D)^{8}.
\end{equation}

Lorsque $h(m)$ présente un caractère oscillant  au regard du nombre de facteurs premiers de $m$ comptés avec multiplicité, noté $\Omega(m)$, il n'est en général pas possible de déduire des résultats satisfaisants à partir des seules estimations de sommes de Type I. Ce principe, connu sous le nom de phénomène de parité, a été mis en lumière  dans \cite{Se52} et \cite{Bo76} -- on trouvera un survol dans le chapitre 16 de \cite{FI10}. Il stipule essentiellement que les techniques de crible, fondées  par construction sur des estimations de sommes de Type I, ne permettent pas d'isoler les valeurs de $F$ ayant un nombre de facteurs premiers dont la parité soit fixée. 
En principe, il est possible de contourner le phénomène de parité si l'on dispose  d'une estimation de sommes de Type II, c'est-à-dire d'une majoration non triviale de sommes de la forme
\begin{align*}
 \sum_{U<r\leq 2U}\sum_{V<s\leq2V}a_rb_s\#\{1\leq n_1,n_2\leq \X:rs=F(n_1,n_2)\}
\end{align*}
pour des choix convenables de $U$ et $V$, où $b_s$ satisfait une hypothèse de type Siegel-Walfisz, c.-à-d.   que le comportement de $b_s$ est suffisamment régulier dans les progressions arithmétiques de petits modules.

 Heath-Brown~\cite{HB01} établit le premier un tel résultat pour la forme $X_1^3+2X_2^3$ à l'aide, entre autres, d'une inégalité de grand crible et de comptage de points sur une hypersurface. Il en déduit l'existence d'une infinité de nombres premiers de la forme 
 $n_1^3+2n_2^3$ en établissant la formule asymptotique 
 \begin{align*}
\#\left\{\X<n_1,n_2\leq \X\left(1+(\log \X)^{-c_0}\right): n_1^3+2n_2^3\text{ est premier }\right\}\underset{\X\rightarrow +\infty}{\sim}c_1 
\frac{\X^2}
{(\log \X)^{c_0+1}}
\end{align*} 
pour une constante $c_0,c_1>0$.
Une généralisation de ce travail aux valeurs des polynômes de la forme $F(a_1+X_1q,a_2+X_2q)$ où $F\in\Z[X_1,X_2]$ est une forme cubique irréductible et $(a_1,a_2,q)=1$ a été effectuée par Heath-Brown et Moroz~\cite{HM02,HM04}.  

 Les estimations de Type II obtenues dans \cite{HM02} ont par la suite  été utilisées par Helfgott pour établir la validité de la conjecture de Chowla%\cite{Ch65}
, bien que les identités combinatoires mises en œuvre soient différentes de celles intervenant dans les travaux de Heath-Brown et Moroz. Dans le théorème principal de \cite{He}, il montre  en particulier la formule suivante, valide pour toute forme  irréductible $F\in\Z[X_1,X_2]$ de degré $3$, $\varepsilon>0$ et uniforme en  $\X\geq3$, \begin{align}
\sum_{1\leq n_1,n_2\leq \X}\mu(F(n_1,n_2))\ll_{F,\varepsilon}\frac{(\log_2\X)^4(\log_3X)^{\varepsilon}}{\log \X}  \X^2=o(  \X^2)\label{theo helfgott irreductible}
\end{align}
où $\log_k\X$ désigne le $k$-ième logarithme itéré.\\

Soient $q\geq1$ et $1\leq a_1,a_2\leq q$ des entiers tels que $ (a_1,a_2,q)=1$. Étant donnée une fonction arithmétique $h$, une forme binaire $F$ et un  compact $\comp\subset[0,1]^2$, on définit
\begin{align*}
S^{(1)}(\comp.\X;h,F;a_1,a_2,q):=
\sum_{\substack{(n_1,n_2)\in \comp.\X\\ (a_1+n_1q,a_2+n_2q)=1}}h(F(a_1+n_1q,a_2+n_2q))
\end{align*}
où $\comp.\X:=\left\{(y_1X,y_2X):(y_1,y_2)\in\comp\right\}$. On cherche  une estimation asymptotique de cet ordre moyen, avec une uniformité en $q\leq(\log\X)^A$ pour tout $A>0$ fixé (c'est-à-dire du type Siegel-Walfisz).  
À l'aide d'un argument standard de convolution, on pourra en général déduire des informations relatives à $S(\X;h,F)$ à partir  de $S^{(1)}(\comp.\X;h,F;a_1,a_2,q)$. 

La méthode développée  dans cet article et détaillée dans le paragraphe \ref{description de la méthode}
permet d'obtenir des résultats asymptotiques concernant $S^{(1)}(\comp.\X;h,F;a_1,a_2,q)$
 pour une large classe de fonctions $\mathcal{F}$.  Le lecteur est renvoyé au début du paragraphe \ref{description de la méthode} pour une définition précise de l'ensemble $\mathcal{F}$. Généralisant les travaux de Heath-Brown et Moroz, le résultat central de cet article, à savoir le corollaire \ref{corollaire des applications}, ramène l'estimation de $S^{(1)}(\comp.\X;h,F;a_1,a_2,q)$ à l'étude asymptotique de l'ordre moyen
 \begin{equation}\label{ordre moyen général}
\sum_{\substack{\J\text{ idéal de }\OK\\\X<N(\J)\leq
\X\left(1+(\log\X)^{-c_0(F)}\right)}}h(N(\J))\sigma_q(\J) \end{equation}
où $c_0(F)>0$, $\K$ est un corps de nombres cubique dont l'anneau des entiers est noté $\OK$ et $\sigma_q$ est une fonction définie au paragraphe \ref{Premières propriétés} dont la série  de Dirichlet associée $
\sum_{\J}\frac{\sigma_q(\J)}{N(\J)^s}
$ possède un comportement analytique semblable à celui de la fonction zêta de Dedekind $\zeta_{\K}(s)$ du corps $\K$.

L'ensemble $\mathcal{F}$ contient en particulier toute fonction  $h$  qui s'écrit \begin{equation}\label{première décomposition h}
 h(m)=\left\{\begin{array}{ll}
 \widetilde{h}(m)&\text{ si }\Y_1\leq P(m)\leq \Y_2,\\
 0&\text{ sinon},
 \end{array}\right.
 \end{equation}où $\widetilde{h}$ est une fonction multiplicative telle que $\widetilde{h}(p)=z$ pour tout premier, $\Y_1,\Y_2\geq1$, $P$ peut désigner le noyau, le plus petit ou plus grand facteur premier de $m$ et $\widetilde{h}:\N\rightarrow\C$. Les deux exemples d'application que nous décrivons à présent concernent des fonctions $h$ de cette forme.

Comme première illustration, on considère dans le paragraphe \ref{Application 1} l'ensemble $\mathcal{M}(z)\subset\mathcal{F}$  des fonctions $h$ multiplicatives, à valeurs dans le disque unité fermé, et satisfaisant $h(p)=z$ pour tout premier $p$. Les fonctions $z^{\omega(n)}$, $z^{\Omega(n)}$, la fonction de Möbius $\mu$ et la fonction de Liouville $\lambda$ sont notamment des éléments de cet ensemble. La méthode développée dans cet article, combinée à l'estimation de l'ordre moyen (\ref{ordre moyen général}) dans le paragraphe \ref{section ordre moyen}, permettra d'établir un ordre moyen de $h$ sur les valeurs de $F$.

\begin{theo}\label{ordre moyen fonction multiplicative}
Soient $A\geq0$, $\varepsilon>0$, $z$ un complexe non nul tel que $|z|\leq 1$, $F\in\Z[X_1,X_2]$ une forme binaire primitive et irréductible de degré $3$, $\comp\subset[0,1]^2$
un compact dont le bord est paramétré par un lacet  de classe 
 $\mathcal{C}^1$ par morceaux et $h\in\mathcal{M}(z)$. 
On a, uniformément en $\X\geq3$, $q\leq(\log \X)^A$ et $0\leq a_1,a_2\leq q$ des entiers tels que $(a_1,a_2,q)=1$,
\begin{equation}
 S^{(1)}(\comp.\X;h,F;a_1,a_2,q)=\Aire(\comp)  \X^2\sigma_q(F)\left(3\log \X\right)^{z-1}
\frac{\sigma_q(F,h)}{\Gamma(z)}+ O\left(\frac{  \X^2}{(\log_2\X)^{1-\varepsilon}}\right)
\end{equation}
où $\Aire(K)$ désigne l'aire de $\mathcal{K}$, $\sigma_q(F)$ et $\sigma_q(F,h)$ sont définies respectivement par (\ref{définition sigmaqF})  et (\ref{définition sigma qh})  avec la convention $1/\Gamma(-1)=1/\Gamma(0)=0$.
\end{theo}
\begin{remarque}Lorsque $h\in\mathcal{M}(-1)$, le théorème  \ref{ordre moyen fonction multiplicative} assure que l'on a la formule 
\begin{equation*}
 S^{(1)}(\comp.\X;h,F;a_1,a_2,q)\ll\frac{  \X^2}{(\log_2\X)^{1-\varepsilon}},
\end{equation*}
ce qui permet de retrouver la formule (\ref{theo helfgott irreductible}) due à Helfgott~\cite{He}  avec un terme d'erreur plus faible en effectuant le choix  $h(n)=\mu(n)$. 
\end{remarque}

 Dans la seconde application, nous étudions le cardinal
\begin{multline*}
\Psi^{(1)}_F(\comp.\X,Y;a_1,a_2,q)=\#\left\{\vphantom{P^+}(n_1,n_2)\in\comp.\X:\gcd(a_1+n_1q,a_2+n_2q)=1\right.\\
\left.\text{ et }P^+(F(a_1+n_1q,a_2+n_2q))\leq \Y \right\},
\end{multline*} 
quantité qui intervient directement dans l'algorithme de factorisation du crible algébrique (voir le chapitre 6.2 de \cite{CP05}). 

\begin{theo}\label{théorème principal friable}
Soient $\varepsilon>0$, $A\geq0$ et  $F\in\Z[X_1,X_2]$ une forme binaire primitive et irréductible de degré $3$ et $\comp\subset[0,1]^2$
un compact dont le bord est paramétré par un lacet de classe 
 $\mathcal{C}^1$ par morceaux. 
Dans le domaine \begin{equation}\label{domaine cubique irréductible}
\X\geq3,\quad \exp\left(\frac{\log \X}{(\log_2 \X)^{1-\varepsilon}}\right)\leq \Y 
\end{equation} et uniformément pour $q\leq(\log \X)^A$  et   $0\leq a_1,a_2\leq q$ des entiers tels que $(a_1,a_2,q)=1$, on a
\begin{align}\label{formule Psi1 irr}
 \Psi^{(1)}_F(\X,\Y ;a_1,a_2,q)
  =\frac{  \X^2}{\zeta_q(2)}\rho(3u)+O\left(\frac{\X^2}{(\log_2\X)^{1-\varepsilon}}\right) .
\end{align} 
%où $u:=\frac{\log \X}{\log \Y }$ et $\rho$ désigne la fonction de Dickman.

De plus, si $(\X,\Y)$ sont dans le domaine
\begin{equation}\label{deuxième domaine irréductible}
\X\geq3,\qquad \exp\left(\frac{\log \X}{(\log_2 \X)^{1/2-\varepsilon}}\right)\leq \Y \leq \X^{1/2-\varepsilon},\end{equation} alors le terme d'erreur $(\log_2\X)^{-(1-\varepsilon)}$ dans (\ref{formule Psi1 irr}) peut être remplacé par $\X^2\exp\left(-(\log_2 \X)^{1/2-\varepsilon}\right) $.
\end{theo}

Compte tenu de l'estimation 
\begin{displaymath}\rho(u)=\exp\left(-u\log u(1+o(1))\right)\qquad\left(u\rightarrow +\infty\right),
\end{displaymath}
on observe que le théorème \ref{théorème principal friable} ne donne un équivalent de $\Psi^{(1)}_F(\X,\Y ;a_1,a_2,q)$ que dans le domaine  $\exp\left(\frac{\log \X}{(\log_2 \X)^{1/2-\varepsilon}}\right)\leq \Y$.

En sommant sur les différentes classes de congruence modulo $q$ et en utilisant une convolution pour enlever la condition de coprimalité, on peut en déduire le corollaire suivant, qui améliore les résultats de \cite{BBDT12} dans le cas des formes cubiques irréductibles, en établissant la validité de (\ref{conjecture rho friable}) et en augmentant la région dans laquelle on dispose d'une estimation de $\Psi_F(\X,\Y)$, passant de $u<\sqrt{e}$ à une région non bornée en $u$.

\begin{corollaire}
Soient $\varepsilon>0$ et  $F\in\Z[X_1,X_2]$ une forme binaire irréductible de degré $3$.
On a, uniformément dans le domaine (\ref{domaine cubique irréductible}), 
\begin{align}\label{formule Psi irr}
 \Psi_F(\X,\Y)
  =  \X^2\rho(3u)\left(1+O\left(\frac{1}{(\log_2\X)^{1-\varepsilon}}\right)\right) .
\end{align}

De plus, si $(\X,\Y)$ sont dans le domaine \ref{deuxième domaine irréductible}, alors le terme d'erreur $(\log_2\X)^{-(1-\varepsilon)}$ dans (\ref{formule Psi irr}) peut être remplacé par $\eta^2\X^2\exp\left(-(\log_2 \X)^{1/2-\varepsilon}\right) $.
\end{corollaire}
Au premier abord, il peut paraître surprenant  d'obtenir un résultat plus précis pour les grandes valeurs de $u$, à savoir $u>2$. Ceci est dû notamment à la méthode employée :  une partie de la contribution des termes impliqués dans l'identité combinatoire introduite dans le paragraphe \ref{description de la méthode} disparaît lorsque  $u$ devient suffisamment grand.

Dans  la suite de cet article, 
on fixe $A>0$ et l'on suppose $q\leq(\log X)^A$, hypothèse qui sera toujours implicite dans les arguments et les énoncés à venir. On emploie les notations en usage dans les travaux de Heath-Brown et Moroz, 
à savoir que la lettre $c$ désignera une constante positive et, pour tout paramètre $B$,
la lettre $c(B)$ désignera une fonction de $B$, l'une et l'autre pouvant varier à
chaque occurrence. Il convient de souligner que, dans \cite{HB01}, \cite{HM02}, \cite{HM04}, \cite{He} et dans le présent travail, les constantes  ne sont pas effectives, en raison de l'incidence du zéro de Siegel dans les arguments utilisés pour établir les estimations de sommes de Type II.\\

L'article présente l'organisation suivante. Au chapitre $2$, nous réduisons le problème à l'étude d'une fonction arithmétique sur les idéaux d'un corps cubique $\K$. Au chapitre $3$, nous étudions la distribution multiplicative de certains ensembles d'idéaux de $\OK$. Le chapitre $4$ contient la définition et les premières propriétés de la fonction de densité $\sigma_q$ qui intervient dans 
Un argument combinatoire permet au chapitre $5$ d'énoncer le résultat principal de cet article, sous l'hypothèse où la fonction $h$ est bornée. Les chapitres $6$ et $7$ sont essentiellement consacrés à l'estimation des termes d'erreur qui apparaissent dans la discussion du chapitre $5$. Enfin, on illustre la méthode dans le chapitre $8$ en l'utilisant pour démontrer les théorèmes \ref{ordre moyen fonction multiplicative} et \ref{théorème principal friable}.

L'auteur tient à exprimer ici ses sincères remerciements à sa directrice de thèse, Cécile Dartyge, pour son accompagnement sans faille  et  sa présence attentive lors de la réalisation de ce travail.

\section{Formes cubiques et normes d'idéaux dans un corps cubique}\label{Algèbre}

En général, un problème de nature multiplicative portant sur les valeurs d'une forme cubique  trouve un cadre de traitement plus flexible
 une fois transcrit dans le langage des corps de nombres.  
 On peut justifier le passage  entre forme cubique irréductible et corps de nombres à l'aide du 
 lemme 2.1 de \cite{HM02}. 
 
\begin{lemme}[\cite{HM02}, lemme 2.1]
\label{lien forme cubique corps cubique}
Soit $F\in \Z[X_1,X_2]$ une forme binaire cubique, primitive et irréductible. 
Il existe un corps de nombres réel $\K$ de degré $3$, 
d'anneau d'entiers $\OK$ et des entiers $\omega_1,\omega_2\in\OK$ 
satisfaisant les conditions suivantes :
\begin{itemize}
  \item le sous-module 
  $\Lambda(\omega_1,\omega_2):=\left\{n_1\omega_1+n_2\omega_2:(n_1,n_2)\in\Z^2\right\}$ 
est 
de rang $2$, 
\item $ \K=\Q
(\theta_0)$ où $\theta_0:=\frac{\omega_2}{\omega_1}$,
\item si $\mathfrak{d}$ désigne l'idéal de $\OK$ engendré par $\Lambda(\omega_1,\omega_2)$, 
$N_{ \K/\Q}$  la norme de $\K$ sur $\Q$ et $N$  la norme des idéaux de $\OK$, alors 
\begin{equation}\label{def f}
F(X_1,X_2)=\frac{N_{ \K/\Q}(X_1\omega_1+X_2\omega_2)}{N(\mathfrak{d})}.
\end{equation}\end{itemize}
Réciproquement, l'équation (\ref{def f}) définit une forme binaire cubique, primitive et irréductible. 
\end{lemme}

Dans toute la suite, on considère  un corps de nombres réel $\K$ de degré $3$. 
On note $\OK$ son anneau d'entiers, $\JK$ le monoïde des idéaux entiers non nuls de $\OK$, $\GK$ 
le groupe des idéaux fractionnaires, $\PK$ le sous-groupe des idéaux principaux de $\GK$ 
et $\HK=\GK/\PK$ le groupe de classes. On introduit également 
 $\hK$ le nombre de classes de $\K$, $\lambda_{\K}$ le résidu de la fonction zêta de Dedekind $\zeta_{\K}(s)$ en $s=1$, 
  $\sigma_1  $, $\sigma_2$ et $\sigma_3$ les trois plongements de $\K$ et $\rK\in\{1,3\}$
  le nombre de plongements réels.
Si $\rK=1$ (resp. $\rK=3$), on désigne par $\varepsilon_1$ l'unité fondamentale 
satisfaisant $\varepsilon_1>1$ 
(resp. on se donne
deux unités fondamentales multiplicativement indépendantes $\varepsilon_1$ et $\varepsilon_2$ 
telles que $N(\varepsilon_1)=N(\varepsilon_2)=1$).
Dans toute la suite $\mathfrak{p}$ 
désignera un idéal premier au dessus de $p$ 
  et 
$\I$, $\J$, $\Qid$, $\Rid$ et $\Sid$ des idéaux entiers.

En général, l'anneau des entiers $\OK$ n'est  pas principal,
d'où l'absence de bijection canonique
 entre $\JK$ et $\OK/\OK^*$. 
 En reprenant les arguments du paragraphe 6 de \cite{HM02} (p 275--277), on peut néanmoins faire apparaître une correspondance entre idéaux et entiers algébriques qui sera utilisée dans le paragraphe \ref{paragraphe Type II}  pour l'étude des sommes de Type II. 
Le groupe de classes $\HK$ étant abélien et d'ordre fini $\hK$, il existe  $t\geq 1$ 
et $h_1,\ldots, h_t\geq 1$ tels que 
\begin{align*}h_1\cdots h_t=\hK\quad\text{ et }\quad
\HK\simeq\Z/(h_1\Z)\times\cdots\times \Z/(h_t\Z).
\end{align*}
Par suite, il existe des idéaux entiers  $\mathfrak{a}_1, \ldots, \mathfrak{a}_t$ 
tels que tout idéal fractionnaire non nul $\J$  se décompose sous la forme
$
\J=(\alpha)\mathfrak{a}_1^{l_1}\cdots\mathfrak{a}_t^{l_t}$ où $\alpha\in \K^*$   et  $0\leq l_j<h_j$ pour $j\in\{1,\ldots,t\}$. On fixe  $\alpha_\J\in\OK$ tel 
que $\mathfrak{a}_j^{h_j}=(\alpha_j)$ et $\gamma_j$ 
une solution de l'équation algébrique $\gamma_j^{h_j}=\alpha_j$ ce qui permet de construire
 l'extension
$ \L:= \K(\gamma_1,\ldots,\gamma_t)$. 
Pour tout plongement $\sigma $ de $\K$ dans $\C$, on fixe un prolongement  à 
$\L$, noté encore $\sigma$

En désignant par $\IK$ le sous-groupe multiplicatif de $\L$ engendré par $\K^*$ et 
$\{\gamma_1,\ldots,\gamma_t\}$, on peut remarquer que l'on a un isomorphisme canonique 
 $\IK/\OK^*\simeq \GK$. 
 Pour tout $\gamma\in\IK$, on note $(\gamma)$ 
 l'idéal de $\GK$ qui lui est associé par cet isomorphisme et on fixe $\delta$ 
 un générateur de $\mathfrak{d}$. On observe que  $N((\gamma))=\left|\prod_{i=1}^3\sigma_i(\gamma)\right|$ pour tout $\gamma\in\IK$.
 
 La décomposition de $\GK$ en classes d'idéaux, intrinsèque à la définition de $\HK$, 
 induit canoniquement 
 une partition de $\IK$ en $\hK$ classes d'équivalence, 
 l'ensemble de ces classes étant naturellement muni d'une structure
 de groupe abélien, compatible avec le produit du corps $\L$. 
 Pour tout  $\gamma\in\IK$, il existe une base entière $(u_1,u_2,u_3)$ de sa classe d'équivalence $\Cl (\gamma)$ c'est-à-dire telle que 
 \begin{align*}
\Cl (\gamma)\cup\{0\}=\left\{a_1u_1+a_2u_2+a_3u_3, a_i\in\Q\right\}.
\end{align*}
\begin{align*}
\left(\Cl (\gamma)\cup\{0\}\right)\cap\OL=\left\{a_1u_1+a_2u_2+a_3u_3, a_i\in\Z\right\}.
\end{align*}
On note $(u_1^*,u_2^*,u_3^*)$ sa base duale, définie comme l'unique base de $\Cl (\gamma^{-1})$ satisfaisant 
\begin{equation}\label{base duale}
\text{Tr}_{\K/\Q}(u_iu_j^*)=\left\{\begin{array}{ll}
1&\text{si }i=j,\\
0&\text{sinon}.
\end{array}\right.
\end{equation}
Cette base duale trouvera son utilité au paragraphe \ref{paragraphe Type II} où elle facilitera la détection des congruences dans $\OK$.

\subsection{Fonctions arithmétiques sur les idéaux}

Conséquence de la théorie des anneaux de Dedekind, l'existence et l'unicité de la décomposition d'un idéal 
en idéaux premiers 
permet d'étendre aux éléments de $\JK$
des notions arithmétiques initialement définies sur les anneaux factoriels (voir la note 5 du chapitre 9 
de \cite{Na04}).
Aussi les notions de valuation $\mathfrak{p}$-adique, notée $v_{\mathfrak{p}}$, et
de plus grand diviseur commun, déjà utilisée dans la définition de $\mathfrak{d}$,
sont-elles 
sans équivoque. 
Dans cette direction, on dira qu'une fonction $g:\JK\rightarrow \C$ est multiplicative 
si $g(\J_1\J_2)=g(\J_1)g(\J_2)$ dès que $(\J_1,\J_2)=\OK$. 
Des exemples de fonctions multiplicatives sur les idéaux sont donnés par la fonction de Möbius $\mu_{\K}$ 
et  de la fonction diviseur $\tau_{\K}$ sur $\JK$, définies  
comme les fonctions multiplicatives satisfaisant,
pour  un idéal premier $\mathfrak{p}$ et $k\geq 1$, la formule
\begin{align*}
\mu_{\K}(\mathfrak{p}^{k}):=\left\{\begin{array}{ll}
               -1&\text{ si }k=1,\\
0&\text{sinon},
              \end{array}\right.
\qquad \tau_{\K}(\mathfrak{p}^k):=k+1.
\end{align*}
On introduit également les fonctions $\Omega_{\K}$ et $\omega_{\K}$ qui comptent le nombre de facteurs premiers avec ou sans leur ordre de multiplicité, définies par
\begin{displaymath}
\Omega(\J):=\sum_{\mathfrak{p}|\J}v_{\mathfrak{p}}(\J)\quad\text{et}\quad\omega(\J):=\#\left\{\mathfrak{p}|\J\right\}.
\end{displaymath}
Étant donné une fonction arithmétique $g$ définie sur $\Z$, on  lui associe naturellement une fonction définie 
sur $\JK$ notée encore $g$ et définie par $g(\J):=g(N(\J))$. Réciproquement, on associe à une fonction $g$ définie sur $\JK$  
 la fonction  $g^{\Z}$
 définie par
\begin{align*}
g^{\Z}(n):=\sum_{N(\J)=n}g(\J).
\end{align*}

Dans la preuve du lemme 4.2 de \cite{HB01}, Heath-Brown énonce les inégalités suivantes, conséquences de
la décomposition des idéaux en idéaux premiers, 
\begin{equation}\label{tau nombre de norme}\tau_{\K}(\J)\leq\tau(N(\J))^3\quad\text{ et }\quad
\#\left\{\J\in\JK:N(\J)=n\right\}\leq\tau(n)^2.
\end{equation}
où $\tau:=\tau_{\Q}$ désigne la fonction diviseur standard.

Le lemme 6.1 de \cite{HM02}, reformulé ci-dessous, s'inspire de la preuve de Landreau de l'iné\-galité 
de Van der Corput relative à la somme des diviseurs (voir \cite{La89}) et fournit une première estimation 
de l'ordre moyen de $\tau_{\K}$ sur les idéaux de la forme 
$(n_1\omega_1+n_2\omega_2)$.
\begin{lemme}[\cite{HM02},lemme 6.1]\label{somme des diviseurs}
Soient $\X_1\geq \X_2\geq 1$ et
$\alpha,\beta$ deux éléments de $\IK\cap \OL$ tels que
$\Cl (\alpha)=\Cl (\beta)$. 
Supposons
 qu'il existe $r\geq 1$ tel que $|\sigma(\alpha)|,|\sigma(\beta)|\leq \X_1^r$
 pour tout plongement $\sigma$ de $\K$ dans $\C$. 
 Alors, pour tout entier $B\geq 0$, il existe une constante $c(B,r)\geq 0$ telle que 
\begin{equation}\label{formule diviseur}
\sum_{\substack{|n_1|\leq \X_1,|n_2|\leq \X_2\\n_2\neq 0}}
\left|\tau_{\K}((n_1\alpha+n_2\beta))^B\right|\ll \tau_{\K}((\alpha)+(\beta))^B\X_1\X_2(\log 2\X_1)^{c(B,r)}.
\end{equation}
\end{lemme}

\subsection{Notations et définition des idéaux admissibles} 
Comme observé dans le paragraphe 2 de \cite{HM02}, certains nombres premiers présentent 
un comportement particulier et nécessiteront 
 un traitement spécifique. 
On dira qu'un nombre premier est $q$\textbf{-singulier}
 s'il est ramifié dans $\OK$ ou s'il divise  
 $q$, $N(\omega_1\omega_2)$ ou  l'indice $\left[\OK,\Z[\theta]\right]$
 avec  $\theta:=\theta_0N(\omega_1\mathfrak{d}^{-1})$, 
 et $q$\textbf{-régulier} sinon. On note $P(q)=\max\left(\{1\}\cup\{p:p\text{ est }q\text{-singulier}\}\right)$.
Par extension, un idéal premier $\mathfrak{p}$ est dit $q$-singulier (resp. $q$-régulier) 
s'il est situé au-dessus d'un nombre premier 
$q$-singulier (resp. $q$-régulier). 
   Un entier ou un idéal dont tous les diviseurs premiers sont $q$-singuliers (resp. $q$-réguliers) 
   est également dit $q$-singulier (resp. $q$-régulier).
L'unicité de la décomposition en idéaux premiers dans les anneaux de Dedekind permet de définir,   pour un idéal $\J$ et un entier $m$, les parties $q$-singulières  et $q$-régulières
 par les formules
\begin{align*}
\J_{q\text{-r}}:=\prod_{\mathfrak{p}\text{ } q\text{-régulier}}\mathfrak{p}^{v_{\mathfrak{p}}(\J)},\qquad \J_{q\text{-s}}:=\prod_{\mathfrak{p}\text{ }q\text{-singulier}}\mathfrak{p}^{v_{\mathfrak{p}}(\J)},\qquad
\\m_{q\text{-r}}:=\prod_{p\text{ }q\text{-régulier}}p^{v_{p}(m)}\quad\text{ et }\quad m_{q\text{-s}}:=\prod_{m\text{ }q\text{-singulier}}m^{v_p(m)}.
\end{align*}
où $v_p$ désigne la valuation $p$-adique rationnelle.
D'une manière analogue, on définit la composante $\Y$-friable et la composante $\Y$-criblée 
 par
\begin{align*}
\J^{-}(\Y):=\prod_{p\leq\Y }\prod_{\mathfrak{p}|p}\mathfrak{p}^{v_{\mathfrak{p}}(\J)},\qquad
\J^+(\Y):=\prod_{p>\Y }\prod_{\mathfrak{p}|p}\mathfrak{p}^{v_{\mathfrak{p}}(\J)},\\ m^{-}(\Y):=\prod_{p\leq \Y  }p^{v_{p}(m)}\quad \text{ et } \quad m^+(\Y):=\prod_{p> \Y}p^{v_{p}(m)}.
\end{align*}
On introduit enfin la \textbf{composante $p$-adique} d'un idéal $\J$, notée $\J_p$, comme l'unique  diviseur de $\J$ satisfaisant
 $N(\J_p)=p^{v_p(N(\J))}$, de sorte que $\J=\prod_p \J_p$.

Le comportement des idéaux premiers $1$-réguliers qui interviennent dans la décomposition des idéaux $(n_1\omega_1+n_2\omega_2)$ 
fait l'objet du lemme 2.2 de \cite{HM02}. 
\begin{lemme}[\cite{HM02}, lemme 2.2]\label{admissible}
Soient $n_1$ et $n_2$ deux entiers premiers entre eux, $p$ un nombre premier $1$-régulier,
$\mathfrak{p}_1$ et $\mathfrak{p}_2$ 
deux idéaux premiers au dessus de $p$ tels que  $\mathfrak{p}_1$ et $\mathfrak{p}_2$ divisent 
$(n_1\omega_1+n_2\omega_2)\mathfrak{d}^{-1}$. Alors $\mathfrak{p}_1=\mathfrak{p}_2$ et $N(\mathfrak{p}_1)=p$.
\end{lemme}

Au regard du lemme \ref{admissible}, on dira qu'un idéal $\J$ est  \textbf{admissible} si sa partie $1$-régulière est de la forme
\begin{align*} 
\J_{1\text{-r}}=\mathfrak{p}_1^{k_1}\cdots\mathfrak{p}_n^{k_n}
\end{align*}
où $N(\mathfrak{p}_i)%=p_i
$ est premier et $N(\mathfrak{p}_i)\neq N(\mathfrak{p}_j)$ dès que  $1\leq i< j\leq n$.\\

\subsection{Transformation des sommes $S^{(1)}(\comp.\X;h,F;a_1,a_2,q)$ en sommes sur des idéaux admissibles}
Soit $\comp\subset[0,1]^2$ un compact dont le bord est paramétré par un lacet de classe $C^1$ par morceaux\footnote{Comme mentionné dans l'introduction de \cite{HB01}, l'argument détaillé dans cet article est valide plus généralement pour un compact Jordan-mesurable $\comp$ de mesure non nulle. %A.1.b de \cite{CF89}
}. En recouvrant $\comp.\X$ en carrés de la forme
\begin{align*}\mathcal{C}(N_1,N_2):=\left\{(n_1,n_2):\X\eta N_i\leq n_i< \X\eta (N_i+1)\text{ pour }
i=1,2\right\}
\end{align*}
où $\eta:=(\log \X)^{-c_0}$ pour une constante $c_0>0$ que l'on déterminera \textit{a posteriori} et
 $0\leq N_i\leq\eta^{-1}$ pour $i=1,2$, on observe que les éléments de $\comp$ mis à l'écart par cette approche 
 contribuent de manière négligeable puisque l'on a
 \begin{equation}\label{estimation pathologiques}
\#\{0\leq N_1,N_2\leq \eta^{-1}:\emptyset\subsetneq \mathcal{C}(N_1,N_2)\cap \comp.\X\subsetneq \comp.\X\}\ll\eta^{-1}.
\end{equation}

On définit 
\begin{equation}\label{définitionvraie cN1N2}c(N_1,N_2):=\left|F(N_1\eta,N_2\eta)\right|.\end{equation}
Pour s'assurer que $c(N_1,N_2)$ soit suffisamment grand,   on introduit 
l'ensemble
\begin{align*}
 \mathcal{N}(\eta):=\left\{1\leq N_1,N_2\leq \eta^{-1}: 
 \mathcal{C}(N_1,N_2)\subset \comp.\X,(N_1\eta,N_2\eta)\not\in  \bigcup_{j=0}^2S(\theta_j)
\right\}
\end{align*}
où $\theta_1,$ $\theta_2$ et $\theta_3$ désignent les racines de $F(X_1,1)$ et 
\begin{align*}
 S(\theta):=\left\{0\leq y_1,y_2\leq 1:|y_1-y_2\theta|\leq \eta^{\frac{1}{4}}\right\}.
\end{align*}
On a alors, uniformément en $(N_1,N_2)\in \mathcal{N}(\eta)$, 
\begin{equation}\label{pas petit}
  c(N_1,N_2)
\gg
\prod_{j=0}^2\left|N_1\eta-N_2\eta \theta_j\right|> \eta^{\frac{3}{4}}
\end{equation}
et
\begin{equation}\label{contribution éléments hors mathcalN eta}
 \#\left\{0\leq N_1,N_2\leq \eta^{-1}:(N_1\eta,N_2\eta)\in  \bigcup_{j=0}^2S(\theta_j)
\right\}\ll \eta^{-\frac{7}{4}}.
\end{equation}

On introduit l'ensemble d'idéaux
\begin{multline}\label{définition A}
\mathcal{A}(a_1,a_2,q;N_1,N_2):=\left\{
((a_1+n_1q)\omega_1+(a_2+n_2q)\omega_2)\mathfrak{d}^{-1}:(n_1,n_2)\in \mathcal{C}(N_1,N_2)\right.\\\left.\text{ et }
 (a_1+n_1q,a_2+n_2q)=1\vphantom{\mathfrak{d}^{-1}}\right\}
\end{multline}
que l'on notera plus simplement $\mathcal{A}$ dans la suite, s'il n'y a pas d’ambiguïté.
En s'inspirant du lemme 2.3 de \cite{HM02}, on peut montrer que les éléments de $\mathcal{A}$ sont tous distincts et que 
$c(N_1,N_2)q^3  \X^3$
est la valeur moyenne de $|F|$ sur $\mathcal{C}(N_1,N_2)$.
\begin{lemme}\label{non associes}
On a uniformément  en $\X\geq2$, $( N_1,N_2)\in\mathcal{N}(\eta)$ et $(n_1,n_2)\in \mathcal{C}(N_1,N_2)$,
\begin{equation}\label{définition cN1N2}
|F(a_1+n_1q,a_2+n_2q)|=c(N_1,N_2)q^3  \X^3\left(1+O\left(\eta^{\frac{1}{4}} \right)\right).
\end{equation}

De plus, il existe $\X_0\geq1$ tel que la relation \begin{equation}\label{correspondance bijective}
\begin{array}{ccc}\{(n_1,n_2)\in\mathcal{C}(N_1,N_2):
(a_1+n_1q,a_2+n_2q)=1\}&\longrightarrow&\mathcal{A}\\(n_1,n_2)&\mapsto&((a_1+n_1q)\omega_1+(a_2+n_2q)\omega_2)\mathfrak{d}^{-1}\end{array}
                                       \end{equation}
induise une bijection dès que $\X\geq \X_0$. 
\end{lemme}
\begin{proof}
La relation (\ref{définition cN1N2}) est une conséquence de la formule de Taylor et de la définition 
(\ref{définitionvraie cN1N2}). 

Pour établir  le caractère bijectif de la correspondance (\ref{correspondance bijective}), on suppose  
que $(n_1,n_2)$ et $(m_1,m_2)\in\mathcal{C}(N_1,N_2)$ engendrent le même idéal et l'on considère  
l'entier algébrique\begin{align}\nonumber
\alpha:=\frac{(a_1+n_1q)\omega_1+(a_2+n_2q)\omega_2}{(a_1+m_1q)\omega_1+(a_2+m_2q)\omega_2},
                                                    \end{align}
inversible dans $\OK$. 
En raisonnant comme dans la preuve du lemme 2.3 de \cite{HM02}, on observe que $\text{Tr}_{\K/\Q}(\alpha)=\text{Tr}_{\K/\Q}(\alpha^{-1})=3$ dès que $\X$ est assez grand ce qui implique que $\alpha=1$.
\end{proof}

Compte tenu des observations précédentes, on peut écrire
\begin{multline}
 S^{(1)}(\comp.\X;h,F;a_1,a_2,q)=  \sum_{(N_1,N_2)\in \mathcal{N}(\eta)}S(\mathcal{A}(a_1,a_2,q;N_1,N_2);h)
\\+\sum_{\substack{0\leq N_1,N_2\leq\eta^{-1}\\\mathcal{C}(N_1,N_2)\cap\comp.\X\neq\emptyset\\
(N_1,N_2)\notin \mathcal{N}(\eta)
}}\sum_{\substack{(n_1,n_2)\in\mathcal{C}(N_1,N_2)}}h(F(a_1+n_1q,a_2+n_2q))\label{décomposition Jordaniesque}
\end{multline}
où
\begin{equation}\label{définition SAh}
S(\mathcal{A};h):=\sum_{\J\in\mathcal{A}}h(\J).
\end{equation}
Au vu de (\ref{estimation pathologiques}) et 
(\ref{contribution éléments hors mathcalN eta}), on s'attend à ce que la contribution de la seconde somme de (\ref{décomposition Jordaniesque}) soit négligeable
(voir le corollaire 1 de \cite{BB06}). Les termes restants fourniront le terme principal et concentreront notre attention dans la suite.
Comme souligné dans l'introduction, on verra  au paragraphe \ref{description de la méthode}
que la relation (\ref{définition cN1N2}) nous amène à  comparer $S(\mathcal{A}(a_1,a_2,q;N_1,N_2);h)$ à la quantité plus régulière  $S(\mathcal{B}(q;N_1,N_2);h\sigma_q)$
où  $\sigma_q$ est une fonction de densité qui sera introduite au paragraphe \ref{Fonction de densité sigmaq},
\begin{align}\label{définition B}
\mathcal{B}=\mathcal{B}(q;N_1,N_2):=\{\J\in\JK:c(N_1,N_2)q^3  \X^3<N(\J)\leq c(N_1,N_2)q^3  \X^3(1+\eta)\}.
\end{align} 
et\begin{align}\label{définition SBh}
S(\mathcal{B};h\sigma_q)
:=\sum_{\J\in\mathcal{B}}h(N(\J))\sigma_q(\J).
\end{align}

\section{Sommes de Type I et niveau de distribution}\label{paragraphe Type I}

Dans cette partie, on étudie la distribution multiplicative des éléments de $\mathcal{A}$ et $\mathcal{B}$ définis par (\ref{définition A}) et (\ref{définition B}).

\subsection{Sommes d'exponentielles}

Étant donné des entiers $g_1$ et $g_2$ et un idéal admissible $\I$, on considère la somme d'exponentielles
\begin{align*}
S\left(g_1,g_2;\I\right)=S\left(a_1,a_2,q;g_1,g_2;\I\right)
:=\sum_{\substack{1\leq n_1,n_2\leq N(\I)\\\I|((a_1+n_1q)\omega_1+(a_2+n_2q)\omega_2)\mathfrak{d}^{-1}}}
\e\left(\frac{g_1n_1+g_2n_2}{N(\I)}\right)
\end{align*}
où $\e(t):=\exp(2i\pi t)$.
L'estimation d'une telle quantité constitue un ingrédient de base dans de nombreux
travaux relatifs aux formes cubiques (voir par exemple le paragraphe 2.4 de \cite{Gr71} ou le paragraphe
 5 de \cite{HB01}) et intervient fréquemment dans la majoration des sommes de Type~I. 

Une étude de $S(g_1,g_2;\I)$ s'initie en observant la relation de multiplicativité 
\begin{align}\label{relation multiplicativité S(g1g2I)}
S(g_1,g_2;\I_1\I_2)=S(g_1,g_2;\I_1)S(g_1,g_2;\I_2)
 \end{align}
 valable dès que $(N(\I_1),N(\I_2))=1$. Conséquence directe du théorème des restes chinois (voir le
 lemme 2.1 de \cite{HM04}), l'identité (\ref{relation multiplicativité S(g1g2I)}) permet essentiellement de ramener l'étude de $S\left(g_1,g_2;\I\right)$ à la démonstration du lemme suivant.
 \begin{lemme}\label{Somme exponentielle, cubique irréductible, raison régulière}
Soient $\mathfrak{p}$ un idéal premier non ramifié tel que $N(\mathfrak{p})=p$ avec $p\nmid qN(\omega_1\omega_2)$ et $k\geq1$.  On a
\begin{align*}
S(g_1,g_2;\mathfrak{p}^{k})=\left\{\begin{array}{ll}
p^{k}\e\left(-\frac{(a_1g_1+a_2g_2)q^{-1}}{p^k}\right)&\text{ si }\mathfrak{p}^{k}|(g_2\omega_1-g_1\omega_2),\\
0&\text{ sinon},\end{array}\right.
\end{align*}où $q^{-1}$ est la solution modulo $p^k$ du système $qq^{-1}\equiv1\pmod{ p^k}$.
\end{lemme}

En vue d'établir ce résultat, nous étudions dans un premier temps la somme d'exponentielles
\begin{align*}
S^{(1)}\left(g_1,g_2;\mathfrak{p}^{k}\right):=\sum_{\substack{1\leq n_1,n_2\leq N(\mathfrak{p}^{k})\\\left(n_1,n_2,p\right)=1\\
\mathfrak{p}^{k}|(n_1\omega_1+n_2\omega_2)}}\e\left(\frac{g_1n_1+g_2n_2}{N(\mathfrak{p}^{k})}\right).
\end{align*} 
\begin{lemme}\label{Somme exponentielle, cubique irréductible, arguments premiers entre eux}
Soient $\mathfrak{p}$ un idéal premier  non ramifié tel que $N(\mathfrak{p})=p$, $p\nmid N(\omega_1\omega_2)$
et $k\geq 1$. 

Si $(g_1,g_2,p)=1$, alors on a
\begin{equation}S^{(1)}\left(g_1,g_2;\mathfrak{p}^{k}\right)=\left\{\begin{array}{ll}
\left(p-1\right)p^{k-1} &\text{ si }\mathfrak{p}^{k}|(g_2\omega_1-g_1\omega_2),\\
-p^{k-1} &\text{ si }
\mathfrak{p}^{k-1}\parallel (g_2\omega_1-g_1\omega_2),\\
0&\text{ sinon}.
\end{array}\right.\label{S p quelconque}\end{equation}

	De plus, si $(g_1,g_2,p^{k})=p^{k_0}$ avec $k_0<k$, alors on a 
	\begin{equation}S^{(1)}(g_1,g_2;\mathfrak{p}^k)=p^{k_0}S^{(1)}\left(\frac{g_1}{p^{k_0}},
	\frac{g_2}{p^{k_0}},\mathfrak{p}^{k-k_0}\right).
                                                                     \label{g pas coprime}\end{equation}
\end{lemme}

\begin{proof}
La preuve s'articule essentiellement autour de l'argument développé dans le paragraphe 6 de \cite{Gr70}.
On remarque tout d'abord que
\begin{align*}
S^{(1)}\left(g_1,g_2;\mathfrak{p}^{k}\right)=\sum_{1\leq m\leq p^{k}}N_1\left(\mathfrak{p}^{k},m\right) \e\left(\frac{m}{p^{k}}\right)
\end{align*}
où $N_1(\mathfrak{p}^k,m)$ désigne le nombre de couples $(n_1,n_2)$ modulo $p^k$ tels que \begin{align*}
(n_1,n_2,p)=1,\quad
\mathfrak{p}^{k}|(n_1\omega_1+n_2\omega_2)\quad\text{ et }\quad g_1n_1+g_2n_2\equiv m\pmod{p^{k}}%\right\}
.\end{align*}
Pour tout couple  $(m,m')$ satisfaisant $(m,p^{k})=(m',p^{k})$, il existe un entier $\lambda$ premier à $p$ tel que 
$m'\equiv\lambda m\pmod{p^{k}}$.
Ainsi, la relation 
\begin{align*}
(n_1\pmod{p^k},n_2\pmod{p^k})\rightarrow(\lambda n_1\pmod{p^k},\lambda n_2\pmod{p^k})
 \end{align*}
 définit une bijection de  $N_1\left(\mathfrak{p}^{k},m\right)$ dans $N_1\left(\mathfrak{p}^{k},m'\right)$.  
Il s'ensuit que
\begin{align}
 S^{(1)}\left(g_1,g_2;\mathfrak{p}^{k}\right)&=\sum_{0\leq j\leq k}N_1\left(\mathfrak{p}^{k},p^{j}\right) 
\sum_{\substack{1\leq \lambda\leq p^{k }\\(\lambda,p^{k})=p^{j}}}\e\left(\frac{\lambda}{p^{k-j}}\right)= N_1\left(\mathfrak{p}^{k},p^{k}\right)-N_1\left(\mathfrak{p}^{k},p^{k-1}\right).\label{N-N}
\end{align}
Les rôles de $g_1$ et $g_2$ étant symétriques, l'hypothèse $(g_1,g_2,p)=1$ permet de supposer sans perte de généralité que $(g_1,p)=1$. On peut donc écrire
\begin{align*}
N_1\left(\mathfrak{p}^{k},p^{k}\right)
&=\#
\left\{1\leq n_1,n_2\leq p^{k},(n_2,p)=1,\mathfrak{p}^{k}|(n_1\omega_1+n_2\omega_2),n_1\equiv-g_2n_2g_1^{-1}\pmod{ p^{k}}\right\}\end{align*}
d'où l'on déduit que \begin{equation}\label{N p alpha}
N_1\left(\mathfrak{p}^{k},p^{k}\right)=\left\{\begin{array}{ll}(p-1)p^{k-1}&\text{ si }\mathfrak{p}^{k}|(g_2\omega_1-g_1\omega_2),\\
                        0&\text{ sinon}.\end{array}\right.\end{equation}
De manière similaire, les conditions $p\nmid N(\omega_1)$ et $\mathfrak{p}$ non ramifié  entraînent que
\begin{align}
N_1\left(\mathfrak{p}^{k },p^{k-1}\right)
=&\nonumber\#
\left\{1\leq n_1,n_2\leq p^{k}:(n_2,p)=1,\mathfrak{p}^{k}|\left(p^{k-1}\omega_1-n_2(g_2\omega_1-g_1\omega_2)\right),\right.\\
&\pushright{ \left.n_1
\equiv\left(p^{k-1}-g_2n_2\right)g_1^{-1}\pmod{ p^{k}}\right\}}\nonumber\\
=&\left\{\begin{array}{ll}p^{k-1}&\text{ si }\mathfrak{p}^{k-1}\parallel (g_2\omega_1-g_1\omega_2),\\
                        0&\text{ sinon}.\end{array}\right.\label{N p alpha-1}
           \end{align}
La relation (\ref{S p quelconque}) est une simple conséquence de  (\ref{N-N}), (\ref{N p alpha}) et (\ref{N p alpha-1}).

Supposons à présent que $(g_1,g_2,p^k)=p^{k_0}$. 
Sous l'hypothèse $k_0<k$, on a
\begin{align}
S^{(1)}\left(g_1,g_2;\mathfrak{p}^{k}\right)
&=\label{vers g pas coprime1}\sum_{\substack{1\leq n_1',n_2'<p^{k-k_0}\\(n_1',n_2',p)=1}}
\e\left(\frac{\frac{g_1}{p^{k_0}}n_1'+\frac{g_2}{p^{k_0}}n_2'}{p^{k-k_0}}\right)
N_2\left(n_1',n_2',\mathfrak{p}^{k},p^{k_0}\right)
\end{align}
où $N_2\left(n_1',n_2',\mathfrak{p}^{k},p^{k_0}\right)$ désigne le nombre de couples $(n_1,n_2)$ modulo $p^k$ tels que 
\begin{align*}
(n_1,n_2)\equiv (n_1',n_2')\pmod{p^{k-k_0}}
\quad\text{ et }\quad\mathfrak{p}^{k}|(n_1\omega_1+n_2\omega_2)
.\end{align*}
Puisque $\mathfrak{p}$ n'est pas ramifié, on a 
\begin{align}
\nonumber N_2\left(n_1',n_2',\mathfrak{p}^{k},p^{k_0}\right)
&=\#\left\{0\leq n^*_1,n^*_2< p^{k_0}, \mathfrak{p}^{k}|\left((n_1'+n_1^*p^{k-k_0})\omega_1+(n_2'+n_2^*p^{k-k_0})\omega_2\right)\right\}\\
&=\#\nonumber\left\{0\leq n^*_1,n^*_2< p^{k_0}, \mathfrak{p}^{k}|\left((n_1'\omega_1+n_2'\omega_2)+(n_1^*\omega_1+n_2^*
\omega_2)p^{k-k_0}\right)\right\}\\
&=\label{vers g pas coprime}\left\{\begin{array}{ll}
p^{k_0}&\text{ si }\mathfrak{p}^{k-k_0}|(u'_1\omega_1+u'_2\omega_2)\\
0&\text{ sinon}.\end{array}\right.
\end{align}
La relation (\ref{g pas coprime}) est donc une conséquence directe de (\ref{vers g pas coprime1}) et (\ref{vers g pas coprime}).
\end{proof}
\begin{proof}[Preuve du lemme \ref{Somme exponentielle, cubique irréductible, raison régulière}]
Supposons 
que $\left(g_1,g_2,p^{k}\right)=p^{k_0}$ et
$\mathfrak{p}^{k_1}\parallel (g_2\omega_1-g_1\omega_2)$, de sorte que  $k_0\leq k_1$.  
Au vu des hypothèses sur $\mathfrak{p}$, 
 il vient
\begin{align}
S(0,0,1;g_1,g_2;\mathfrak{p}^k)
\nonumber&=
\sum_{0\leq j\leq k }\sum_{\substack{1\leq n_1,n_2\leq p^{ k }\\
\left(n_1,n_2,p^{ k }\right)=p^{j}\\ 
\mathfrak{p}^{ k }|(n_1\omega_1+n_2\omega_2)}}
\e\left(\frac{g_1n_1+g_2n_2}{p^{ k }}\right)\\
&=1+\sum_{1\leq j \leq k}S^{(1)}\left(g_1,g_2;\mathfrak{p}^{j}\right).\label{S p vraiment quelconques}\end{align}

Si $1\leq j \leq k_0$, on observe que
\begin{align}
S^{(1)}\left(g_1,g_2;\mathfrak{p}^{j}\right)&\nonumber=\#\{1\leq n_1,n_2\leq p^{j }, \left(n_1,n_2,p\right)=1, 
\mathfrak{p}^{ j}|(n_1\omega_1+n_2\omega_2)\}\\
&=(p-1)p^{ j -1}.\label{beta+alpha}
\end{align}
Dans le cas où $k _0<j\leq k$, on applique le lemme 
\ref{Somme exponentielle, cubique irréductible, arguments premiers entre eux} pour obtenir que
\begin{align}
S^{(1)}\left(g_1,g_2;\mathfrak{p}^{j  }\right)&\nonumber=p^{k_0}S^{(1)}\left(\frac{g_1}{p^{ k_0}},\frac{g_2}{p^{k_0}};\mathfrak{p}^{ j-k_0}\right)\\
&=\left\{\begin{array}{ll}
\left(p-1\right)p^{ j -1}&\text{ si } j  \leq k_1,\\
-p^{ k_1}&\text{ si } j  = k_1+1,\\
0&\text{ sinon}.
\end{array}\right.\label{beta-alpha}\end{align}

La combinaison des formules (\ref{S p vraiment quelconques}), (\ref{beta+alpha}) et (\ref{beta-alpha}) entraîne alors la relation
\begin{equation}\label{S001}
S(0,0,1;g_1,g_2;\mathfrak{p}^k)
=\left\{\begin{array}{ll}
p^{ k}&\text{ si } k \leq k_1,\\
0&\text{ sinon}.
\end{array}\right.
\end{equation}
Puisque $(p,q)=1$, le lemme \ref{Somme exponentielle, cubique irréductible, raison régulière} 
est une conséquence de la formule (\ref{S001}) et  de la relation
\begin{align*}
&S(a_1,a_2,q;g_1,g_2;\mathfrak{p}^k)
=\e\left(-\frac{(a_1g_1+a_2g_2)q^{-1}}{p^k}\right)S(0,0,1;g_1q^{-1},g_2q^{-1};\mathfrak{p}^k).
\end{align*}
\end{proof}

\subsection{Niveau de distribution de $\mathcal{A}$}

Dans ce paragraphe, on cherche des estimations, en moyenne sur $\I$, du cardinal de 
\begin{align*}
\mathcal{A}_\I:=\left\{\J\in\mathcal{A}:\I|\J\right\}.
\end{align*}  
De tels résultats sont obtenus, en moyenne sur des idéaux $\I$ sans facteur carré, dans le lemme 2.2  de \cite{HM04}. 
Nous généralisons ici leur démonstration à des idéaux quelconques en y introduisant 
l'estimation de sommes d'exponentielles du lemme \ref{Somme exponentielle, cubique irréductible, raison régulière}.

Pour faciliter le traitement de la partie $q$-singulière des idéaux, on introduit, pour tout réel $\z\geq1$, l'ensemble d'idéaux
\begin{align*}
\mathcal{M}(\z):=\left\{\J\in\JK:N(\J_{q\text{-s}})\leq \z \text{ et }\J\text{ admissible}\right\}.
\end{align*}

On omet dans un premier temps la condition $(a_1+n_1q,a_2+n_2q)=1$ en étudiant,   pour  $\widetilde{\eta}\geq1$ et $\X_1,\X_2>0$, le cardinal de l'ensemble \begin{multline}
\mathcal{S}(\widetilde{\eta},\X_1,\X_2;a_1,a_2,q;\I)
:=
\left\{(n_1,n_2):\vphantom{\mathfrak{d}^{-1}}\widetilde{\eta} \X_i\leq n_i<\widetilde{\eta} (\X_i+1)\text{ pour }i=1,2\right.\\\left.\text{ et }  \I|((a_1+n_1q)\omega_1+(a_2+n_2q)\omega_2)\mathfrak{d}^{-1}
\right\},\label{définition S vers mathcalA}
\end{multline}
noté encore $\mathcal{S}(\I)$. 
En effectuant une partition de  $\mathcal{S}(\I)$ selon les  classes de congruences modulo $N(I)$, on peut écrire
\begin{align}
\# \mathcal{S}(\I)
=&\nonumber\frac{1}{N(\I)^2}\sum_{\substack{1\leq g_1,g_2\leq N(\I)}}S\left(g_1,g_2;\I\right)
\sum_{\substack{\widetilde{\eta}\X_1\leq n_1< \widetilde{\eta} (\X_1+1)\\\widetilde{\eta} \X_2\leq n_2< \widetilde{\eta} (\X_2+1)
}}\e\left(\frac{-g_1n_1-g_2n_2}{N(\I)}\right)\\
=\nonumber&\frac{S(0,0;\I)
}{N(\I)^2}\left(\widetilde{\eta}^2+O(\widetilde{\eta})\right)
\\&+O\left(\sum_{\substack{(g_1,g_2)\neq (0,0)\mod N(\I)\\|g_1|,|g_2|\leq \frac{N(\I)}{2}}}
\frac{|S(g_1,g_2;\I)|}{N(\I)^2}
\min\left(\widetilde{\eta},\frac{N(\I)}{|g_1|}\right)\min\left(\widetilde{\eta},\frac{N(\I)}{|g_2|}\right)\right)\label{S modulo}.
\end{align}
Compte tenu de la définition des idéaux admissibles, 
on peut observer que, pour tout idéal admissible $\I$ et tout premier  $1$-régulier $p$, on a 
\begin{align}\label{S pour les 1 réguliers}
 S(0,0;\I_p)\in\left\{0,\min\left(p^{v_p(q)}N(\I_p),N(\I_p)^2\right)\right\}
\end{align}
tandis que $S(0,0;\I_p)=0$ n'est possible que si $p|q$.
De plus, en utilisant les bornes supérieures de [\cite{Da99}, lemme 3.1], à savoir, uniformément en $p$ premier et $k\geq1$, \begin{equation}\label{majoration de gammaf}
\gamma_F(p^k)\ll p^{4k/3}\qquad\text{et}\qquad \gamma_F(p)=p\nu(p)+O(1),
\end{equation} 
on a, uniformément en $p$ premier et $N(\I_p)>p^{v_p(q)}$, 
\begin{align}
S(0,0;\I_p)&\leq\#\left\{1\leq n_1,n_2\leq N(\I_p):
N(\I_p)|F\left(a_1+n_1q,a_2+n_2q\right)\right\}\nonumber\\
&\nonumber\leq 
p^{2v_p(q)}\gamma_F(N(\I_p))\\
&\ll N(\I_p)^{4/3}p^{2v_p(q)}.\nonumber
\end{align}
Cette estimation, combinée à la la majoration triviale $S(0,0;\I_p)\leq N(\I_p)^2$ lorsque $N(\I_p)\leq p^{v_p(q)}$ entraîne finalement la borne supérieure, uniforme en $p$ premier,
\begin{equation}
S(0,0;\I_p)\ll N(\I_p)^{4/3}\min(p^{2v_p(q)},N(\I_p)^{2/3}).
\label{inégalité S singulier}
\end{equation}

Le terme d'erreur relatif aux fréquences non nulles de (\ref{S modulo}) est estimé dans le lemme suivant.
\begin{lemme}\label{pas premiers}
Soient $B\geq0$ et $\varepsilon>0$. Il existe $c(B)\geq 0$ tel que, uniformément en $\X_1,\X_2\geq0$, $\widetilde{\eta}\geq1$, 
 $\z\geq 1$ et $D\geq 1$, 
\begin{align*}
\Sigma_1:=\sum_{\substack{ \I\in \mathcal{M}(\z)\\D<N(\I)\leq 2D}}\tau_{\K}(\I)^B\left| \#\mathcal{S}(\I)
-\frac{S(0,0;\I)\widetilde{\eta}^2}{N(\I)^2}
\right|\ll (\widetilde{\eta} +D)\z^{1+\varepsilon}(\log 2D)^{c(B)}
\end{align*}\end{lemme}

\begin{proof}
Au vu de (\ref{S modulo}), on peut écrire
$
\Sigma_1\ll \Sigma_{11}+\Sigma_{12}
$
où
\begin{align*}
\Sigma_{11}&:=\sum_{\substack{ \I\in \mathcal{M}(\z)\\D<N(\I)\leq 2D}}\tau_{\K}(\I)^B
\frac{S(0,0;\I)}{N(\I)^2}\widetilde{\eta} 
\end{align*}
et
\begin{align*}
\Sigma_{12}:=\sum_{\substack{ \I
\in \mathcal{M}(\z)\\D<N(\I)\leq 2D}}\tau_{\K}(\I)^B
\sum_{\substack{(g_1,g_2)\neq (0,0)\mod N(\I)\\ |g_1|,|g_2|\leq \frac{N(\I)}{2}}}
\frac{|S(g_1,g_2;\I)|}{N(\I)^2}
\min\left(\widetilde{\eta} ,\frac{N(\I)}{|g_1|}\right)\min\left(\widetilde{\eta} ,\frac{N(\I)}{|g_2|}\right).
\end{align*}
Sous l'hypothèse $\I\in \mathcal{M}(\z)$, %
 les relations (\ref{S pour les 1 réguliers}), (\ref{inégalité S singulier}) et (\ref{tau nombre de norme}) entraînent que
\begin{align}
\nonumber\Sigma_{11}&\leq \widetilde{\eta} \z\sum_{\substack{ D<N(\I)\leq 2D}}\frac{\tau_{\K}(\I)^B}{N(\I)}\nonumber\\
&\ll \widetilde{\eta} \z(\log 2D)^{c(B)}\label{logZ}.
\end{align}
En suivant l'argument développé dans la preuve du lemme 3.1 de \cite{HM02}, on écrit la décomposition 
$\Sigma_{12}:=\Sigma_{13}+\Sigma_{14}$ où la sommation $\Sigma_{13}$ (resp. $\Sigma_{14}$) porte sur 
les phases $(g_1,g_2)$ satisfaisant $g_1g_2=0$ (resp. $g_1g_2\neq0$).    
En utilisant successivement le fait que $(N(\I_{q\text{-r}}),qN(\omega_1\omega_2))=1$, 
le lemme \ref{Somme exponentielle, cubique irréductible, raison régulière}, l'inégalité  
$\tau_{\K}(\I)^B\ll_{B,\varepsilon} N(\I)^{\varepsilon}$ et (\ref{tau nombre de norme}), il vient
\begin{align}
\nonumber\Sigma_{13}&
\ll \widetilde{\eta} \z\sum_{0<g\leq D}\frac{1}{g}\sum_{\substack{\I\in \mathcal{M}(\z)\\D<N(\I)\leq 2D\\ 
\I_{q\text{-r}}|g}}\tau_{\K}(\I)^{B}
\\
&\ll \widetilde{\eta} \z^{1+\varepsilon}\sum_{0<g\leq D}\frac{\tau(g)^{c(B)}}{g}\nonumber\\
&\ll \widetilde{\eta} \z^{1+\varepsilon}(\log 2D)^{c(B)}\label{Sigma13}.
\end{align}
On remarque de même que 
\begin{align}
\nonumber\Sigma_{14}
&\ll D\z\sum_{1\leq |g_1|,|g_2|\leq D}\frac{1}{|g_1g_2|}
\sum_{\substack{D<N(\I)\leq 2D\\ \I\in \mathcal{M}(\z)\\\I_{q\text{-r}}|g_2\omega_1-g_1\omega_2}}
\tau_{\K}(\I)^B\\
&\nonumber\ll 
D\z^{1+\varepsilon}\sum_{1\leq |g_1|,|g_2|\leq D}\frac{\tau_{\K}
((g_2\omega_1-g_1\omega_2))^{c(B)}}{|g_1g_2|}.\end{align}
En découpant le domaine des variables $g_1$ et $g_2$ de manière dyadique et en faisant appel au 
lemme \ref{somme des diviseurs}, il suit que
\begin{align}
\Sigma_{14}
&\ll D\z^{1+\varepsilon}(\log 2D)^{c(B)}\label{Sigma14}.\end{align} Le résultat annoncé est alors une conséquence des estimations (\ref{logZ}), (\ref{Sigma13}) et (\ref{Sigma14}).
\end{proof}

Il est possible de relier $\mathcal{S}(\I)$ au cardinal de $\mathcal{A}_{\I}$ par un argument de convolution. En effet, on remarque que  la formule d'inversion de Möbius permet  d'écrire, 
sous la condition $(a_1,a_2,q)=1$, la relation
\begin{multline}
 \#\mathcal{A}_{\I}=\sum_{(q,d)=1}\mu(d)
 \#\left\{(n_1,n_2)\in\mathcal{C}(N_1,N_2):\I|((a_1+n_1q)\omega_1+(a_2+n_2q)\omega_2)\mathfrak{d}^{-1}
\right. \\\left.\text{ et }d|(a_1+n_1q,a_2+n_2q)\vphantom{\mathfrak{d}^{-1}}\right\}.
\label{convolution AI}\end{multline}
Pour tout $d$ premier à $q$ et 
$i=1,2$,
on définit
l'entier $b_i(d)\in\{1,\ldots,d\}$ comme l'unique solution de la congruence 
$a_i-b_i(d)q\equiv0\pmod d$. Ceci permet d'introduire la correspondance $
(n_1,n_2)\mapsto\left(\frac{n_1+b_1(d)}{d},\frac{n_2+b_2(d)}{d}\right)$ entre
 
\begin{align*}
\left\{(n_1,n_2)\in\mathcal{C}(N_1,N_2):\I|((a_1+n_1q)\omega_1+(a_2+n_2q)\omega_2)\mathfrak{d}^{-1}
 \text{ et }d|(a_1+n_1q,a_2+n_2q)\right\}
 \end{align*}
 et 
 \begin{align*}\mathcal{S}\left(\frac{\eta\X}{d},N_1+\frac{b_1(d)}{\eta\X},N_2+\frac{b_2(d)}{\eta\X};a_1(d),a_2(d),
q;\frac{\I}{(\I,d)}\right)\end{align*}
où $\mathcal{S}$ est défini par (\ref{définition S vers mathcalA}) et $a_i(d):=\frac{a_i-b_i(d)q}{d}$. 
Au vu du lemme \ref{pas premiers} et de la convolution (\ref{convolution AI}), on peut ainsi espérer approcher le cardinal de $\mathcal{A}_{\I}$ par
 \begin{align*} 
\sum_{d\geq1}\mu(d)\frac{\eta^2  \X^2}{d^2N
\left(\frac{\I}{(\I,d)}\right)^2} S_d(\I)
                             \end{align*}  
où $S_d(\I)$ désigne le nombre de couples $(n_1,n_2)$ modulo 
$N\left(\frac{\I}{(\I,d)}\right)$ tels que
\begin{align*}\frac{\I}{(\I,d)}\left|
\left(\left(a_1(d)+n_1q\right)
\omega_1+\left(a_2(d)+n_2q\right)\omega_2\right)\mathfrak{d}^{-1}\right.\end{align*}
 si $(d,q)=1$ et $S_d(\I)=0$ sinon. 

Si $S_1(\I)\neq0$, le théorème des restes chinois implique que
$d\mapsto\frac{S_d(\I)}{S_1(\I)}$ est multiplicative. Ceci permet de faire apparaître un produit eulérien que l'on décompose en $4$ parties selon que $p|q$ ou non, $p|N(\I)$ ou non. On obtient alors 
\begin{align}
\sum_{d\geq1}\mu(d)\frac{S_d(\I)}{d^2N\left(\frac{\I}{(\I,d)}\right)^2}&=
\frac{S_1(\I)}{N(\I)^2}\prod_{p
}\left(1-\frac{S_p(\I)N((\I,p))^2}{S_1(\I)p^2}\right)
 =\frac{\alpha_q(\I)}{\zeta_q(2)N(\I)}\label{définition alpha}
\end{align}
où
\begin{align*}
 \zeta_q(s):=\zeta(s)\prod_{p|q}\left(1-\frac{1}{p^s}\right)
\end{align*}
et
$\alpha_q$ est la fonction à support sur les idéaux admissibles définie, pour tout idéal admissible $\I$, par
\begin{equation}\label{définition rho2}
\alpha_q(\I):=
\prod_{\substack{p|N(\I)\\p\nmid q
}}
\left(1-\frac{1}{p^2}\right)^{-1}\left(\frac{S_1(\I_p)}{N(\I_p)}
-\frac{S_p(\I_p)N((\I_p,p)^2)}{p^2N(\I_p)}\right)
\prod_{\substack{p|(N(\I),q)}}\frac{S_1(\I_p)}{N(\I_p)}.\end{equation}
La formule (\ref{définition alpha}) demeure vraie si $S_1(\I)=0$ puisque l'on a alors $S_d(\I)=0$ pour tout $d$.
On observe que la relation (\ref{S pour les 1 réguliers}) et le fait que $N((\I_p,p)=p$ pour tout idéal $\I$ admissible et tout  $p$ premier $1$-régulier
entraînent que l'on a 
\begin{equation}\label{écriture rho2 non singulier}
\alpha_q(\I_p)=\left\{\begin{array}{ll}(1+p^{-1})^{-1}&\text{si }p\nmid q,\\
                          0\text{ ou }\min\left(p^{v_p(q)},N(\I_p)\right)&\text{si }p|q.
                                         \end{array}
\right.
\end{equation}De même, les majorations  (\ref{inégalité S singulier}) et  $S_p(\I_p)\leq S_1\left(\frac{\I_p}{(\I_p,p)}\right)$ permettent  d'établir  la borne supérieure suivante,
uniforme pour $p$ premier et $\I$ admissible,
 \begin{align}
\nonumber|\alpha_q(\I_p)|&\leq\left(1-\frac{1}{p^2}\right)^{-1}\max
\left(\frac{S_1(\I_p)}{N(\I_p)},\frac{S_p(\I_p)N((\I_p,p))^2}{p^2N(\I_p)}\right)\\&\ll 
N(\I_p)^{1/3}\min(p^{2v_p(q)},N(\I_p)^{2/3}).
\label{inégalité alpha singulier}\end{align} 

On déduit finalement
de ce qui précède que 
\begin{equation}\label{définition rAI}
\#\mathcal{A}_{\I}=\frac{\eta^2  \X^2}{\zeta_q(2)}\frac{\alpha_q(\I)}{N(\I)}
+r(\mathcal{A},\I)\end{equation}
où
\begin{multline}
|r(\mathcal{A},\I)|
\leq\sum_{(d,q)=1}\mu^2(d)\left|
\#\mathcal{S}\vphantom{-\frac{\eta^2  \X^2 S_d\left(\I\right)}{d^2N\left(\frac{\I}{(\I,d)}\right)^2}}\left(\frac{\eta \X }{d},N_1+\frac{b_1(d)}{\eta\X},N_2+\frac{b_2(d)}{\eta\X};a_1(d),a_2(d)
,q;\frac{\I}{(\I,d)}\right)\right.\\-
\left.\frac{\eta^2  \X^2 S_d\left(\I\right)}{d^2N\left(\frac{\I}{(\I,d)}\right)^2}\right|.
\qquad\label{equation AI}
\end{multline}

 Le lemme suivant montre que, sous la condition $\I\in \mathcal{M}(\z)$, on a effectivement le niveau de distribution attendu.

 \begin{lemme}\label{AI admissible} Soient $B\geq0$ et $\varepsilon>0$. 
 Il existe $c(B)\geq0$  tel que, uniformément en  $\X\geq 2$%, $  \X^{-\frac{1}{2}}<\eta\leq1$
, $ (N_1,N_2)\in\mathcal{N}(\eta)%\leq\eta^{-1}-1
$, 
$\z\geq 1$ et $1\leq D\leq   \X^2$,
\begin{align*}
\Sigma_2:=\sum_{\substack{D<N(\I)\leq2D\\\I\in \mathcal{M}(\z)}}\tau_{\K}(\I)^B|r(\mathcal{A},\I)|&\ll
\left(  \X^{3/2}+\X D^{1/2}\right)\z^{1+\varepsilon}%+q^2
(\log \X)^{c(B)}.
\end{align*}
\end{lemme}
\begin{proof}Soit $1\leq \Delta\leq   \X^{1/2}$ un paramètre qui sera explicité en toute fin de  preuve. 
La formule (\ref{equation AI}) permet d'écrire $\Sigma_2\ll\Sigma_{21}+\Sigma_{22}+\Sigma_{23}$ où
\begin{multline*}
\Sigma_{21}:=\sum_{\substack{\I\in \mathcal{M}(\z)\\D<N(\I)\leq2D\\d\leq\Delta,(d,q)=1}}
\tau_{\K}(\I)^B\left|
\#\mathcal{S}\left(\frac{\eta \X }{d},N_1+\frac{b_1(d)}{\eta\X},N_2+\frac{b_2(d)}{\eta\X};a_1(d),a_2(d)
,q;\frac{\I}{(\I,d)}\right)\vphantom{\frac{\eta^2  \X^2S_d\left(\I\right)}{d^2
N\left(\frac{\I}{(\I,d)}\right)^2}}\right.\\\left.-\frac{\eta^2  \X^2S_d\left(\I\right)}{d^2
N\left(\frac{\I}{(\I,d)}\right)^2}\right|
,\end{multline*}
\begin{align*}
\Sigma_{22}:=\sum_{\substack{\I\in \mathcal{M}(\z)\\D<N(\I)\leq2D\\d>\Delta,(d,q)=1}}
\tau_{\K}(\I)^B
\#\mathcal{S}\left(\frac{\eta \X }{d},N_1+\frac{b_1(d)}{\eta\X},N_2+\frac{b_2(d)}{\eta\X};a_1(d),a_2(d)
,q;\frac{\I}{(\I,d)}\right)\end{align*}
et
\begin{align*}
\Sigma_{23}:=\sum_{\substack{\I\in \mathcal{M}(\z)\\D<N(\I)\leq2D\\d>\Delta}}\mu^2(d)
\tau_{\K}(\I)^B\left|\frac{\eta^2  \X^2S_d\left(\I\right)}{d^2N\left(\frac{\I}{(\I,d)}\right)^2}\right|.
\end{align*}

En écrivant $\I=\mathfrak{a}\mathfrak{b}$ avec $\mathfrak{a}=(\I,d)$, on observe que les facteurs $\mathfrak{a}$ et $\mathfrak{b}$ demeurent dans $\mathcal{M}(\z)$.
L'hypothèse sur $\Delta$ entraîne que $\frac{\eta \X }{d}\geq 1$ si $d\leq\Delta$ ce qui nous permet d'utiliser le lemme \ref{pas premiers} avec les choix $\widetilde{\eta}=\frac{\eta \X }{d}$ et $\X_i=
N_i+\frac{b_i(d)}{\eta\X}$ pour $i=1,2$. On peut ainsi majorer $\Sigma_{21}$ par 
\begin{align}
&\ll\sum_{\substack{\mathfrak{a},\mathfrak{b}\in \mathcal{M}(\z)\\D<N(\mathfrak{a}\mathfrak{b})\leq2D\\
d\leq\Delta,(d,q)=1\\N(\mathfrak{a})|d^3}}\tau_{\K}(\mathfrak{a}\mathfrak{b})^B\left|
\#\mathcal{S}\left(\frac{\eta \X }{d},N_1+\frac{b_1(d)}{\eta\X},N_2+\frac{b_2(d)}{\eta\X};a_1(d),a_2(d)
,q;\mathfrak{b}\right)
-
\frac{\eta^2  \X^2S_d\left(\mathfrak{a}\mathfrak{b}\right)}{d^2N\left(\mathfrak{b}\right)^2}\right|
\\
&\nonumber\ll\sum_{\substack{
d\leq\Delta}}\sum_{N(\mathfrak{a})|d^3}\tau_{\K}(\mathfrak{a})^B\left(\frac{D}{N(\mathfrak{a})}+\frac{\X}{d}\right)\z^{1+\varepsilon}
(\log  \X)^{c(B)}\\
&\ll \nonumber  \z^{1+\varepsilon}(\X+D)(\log \X)^{c(B)}\sum_{\substack{
d\leq\Delta}}\tau(d)^{c(B)}\nonumber\\
&\ll \Delta \z^{1+\varepsilon}(\X+D)(\log \X)^{c(B)}\label{Sigma22}.
\end{align}

À l'aide du lemme \ref{somme des diviseurs}, on dispose de l'estimation
\begin{align}
\Sigma_{22}&\ll\nonumber\sum_{d>\Delta}\tau(d)^{c(B)}
\sum_{\substack{1\leq n_1,n_2\leq q(\X+1)/d}}\tau_{\K}((n_1\omega_1+n_2\omega_2))^{c(B)}\\
&\ll 
  \X^2\Delta^{-1}(\log \X)^{c(B)}\label{Sigma23}.
\end{align}

Enfin, en reprenant la décomposition $\I=\mathfrak{a}\mathfrak{b}$ introduite 
pour $\Sigma_{21}$ et en écrivant $d=N(\mathfrak{a}_{1\text{-r}})e$
, on remarque que $d \asymp N(\mathfrak{a})e$ et $S_d(\I)\leq N(\mathfrak{b})\z$. Il s'ensuit que 
\begin{align}
\Sigma_{23}
&\nonumber\ll
\z   \X^2\sum_{\substack{\mathfrak{a}\mathfrak{b}\in \mathcal{M}(\z)\\D<N(\mathfrak{a}\mathfrak{b})\leq 2D}}\frac{
\tau_{\K}(\mathfrak{a}\mathfrak{b})^{B}}{N(\mathfrak{a})^{2}N(\mathfrak{b})}
\sum_{e\gg\Delta N(\mathfrak{a})^{-1}}\frac{1}{e^2}\\
&\nonumber\ll
\Delta^{-1}\z  \X^2\sum_{\substack{\mathfrak{a}\mathfrak{b}\in \mathcal{M}(\z)\\D<N(\mathfrak{a}\mathfrak{b})\leq 2D}}\frac{
\tau_{\K}(\mathfrak{a}\mathfrak{b})^{B}}{N(\mathfrak{a}\mathfrak{b})}\\
&\ll \Delta^{-1}\z  \X^2(\log \X)^{c(B)}\label{Sigma24}.
\end{align}

En combinant les estimations (\ref{Sigma22}), (\ref{Sigma23}) et (\ref{Sigma24}), il suit
\begin{align*}
\Sigma_{2}\ll \left(\Delta^{-1}\z  \X^2+\Delta(\X+D)\z^{1+\varepsilon}\right)(\log \X)^{c(B)}.
\end{align*}
Le choix $\Delta=\min(  \X^{1/2},\X D^{-1/2})$ conduit au résultat annoncé.
\end{proof}

Au vu de la définition de $\mathcal{M}(\z)$, on peut montrer que la contribution des éléments de $\mathcal{A}_{\I}$ où $\I\notin
\mathcal{M}(\z)$ est négligeable lorsque $\z$ est convenablement choisi.
\begin{lemme}\label{mathcalAI}
Soit $B\geq 0$. Il existe $c(B)\geq0$  tel que, uniformément en  $\X\geq 2$,
 $(N_1,N_2)\in\mathcal{N}(\eta)$ et $1\leq D\leq   \X^2$,
\begin{align*}
\Sigma_3:=\sum_{\substack{ D<N(\I)\leq 2D}}\tau_{\K}(\I)^B
\left|r(\mathcal{A},\I)\right|
\ll 
\left(  \X^{15/8}+  \X^{7/4}D^{1/8}\right)(\log \X)^{c(B)},
\end{align*}
où les $r(\mathcal{A},\I)$ sont définis par (\ref{définition rAI}).
\end{lemme}
\begin{proof}
Au regard du lemme \ref{admissible} et de la définition de $\alpha_q$ pour les idéaux non admissibles, 
on observe que la somme $\Sigma_3$ ne porte que sur les idéaux admissibles. 

On remarque dans un premier temps que les relations (\ref{tau nombre de norme}) et (\ref{définition rdX}) entraînent
\begin{align}
\nonumber\sum_{\substack{ \I \not\in \mathcal{M}(\z)\\D<N(\I)\leq 2D}}\tau_{\K}(\I)^B\# \mathcal{A}_{\I}
&\ll\sum_{\substack{D\leq d\leq 2D\\d_{q\text{-s}}>\z}}\tau(d)^{c(B)}\#\left\{1\leq n_1,n_2\leq q(\X 
+1):d|F(n_1,n_2)\right\}\\&\ll
q^2  \X^2\sum_{\substack{D<d\leq 2D\\d_{q\text{-s}}>\z}}\tau(d)^{c(B)}\frac{\gamma_F(d)}{d^2}
+\sum_{\substack{D<d\leq 2D\\d_{q\text{-s}}>\z}}\tau(d)^{c(B)}
\left|r_d(q(\X(1+\eta)+1))
\right|.\label{vers corollaire AI}
\end{align}
La première somme du membre de droite peut être majorée à l'aide de l'inégalité  
\begin{equation}
\sum_{\substack{D<d\leq 2D\\d_{q\text{-s}}>\z}}\tau(d)^{c(B)}\frac{\gamma_F(d)}{d^2}
 \leq \sum_{\substack{\z<d_1\leq 2D\\d_1\text{ }q\text{-singulier}}}
\tau(d_1)^{c(B)}\frac{\gamma_F(d_1)}{d_1^2}\sum_{d_2\leq \frac{2D}{d_1}}\tau(d_2)^{c(B)}\frac{\gamma_F(d_2)}{d_2^2}
\label{premier moyenne}.
\end{equation}
En utilisant les estimations (\ref{majoration de gammaf}), on obtient pour la somme sur $d_2$ la borne supérieure, uniforme en $D\geq 2$,
\begin{align}
\nonumber \sum_{d\leq D}\tau(d)^{c(B)}\frac{\gamma_F(d)}{d^2}
&\ll\prod_{p\leq D}\left(1+2^{c(B)}\frac{\gamma_F(p)}{p^2}\right)\\
&\nonumber\ll\prod_{p\leq D}\left(1+\frac{1}{p}\right)^{c(B)}\\
&\ll (\log 2D)^{c(B)}\label{diviseurs en moyenne AI}.
\end{align}
En combinant (\ref{majoration de gammaf})  à la méthode de Rankin, on observe  pour la somme en $d_1$ que
\begin{align}
 \sum_{\substack{\z<d\leq 2D\\d\text{ }q\text{-singulier}}}\tau(d)^{c(B)}\frac{\gamma_F(d)}{d^2}\nonumber
&\ll \sum_{\substack{\z<d\leq 2D\\d\text{ }q\text{-singulier}}}\frac{\tau(d)^{c(B)}}{d^{2/3}}\left(\frac{d}{\z}\right)^{1/2}\\
&\ll \z^{-1/2}(\log \X)^{c(B)}\label{fIRSsum AI}
\end{align}
où l'on a utilisé le fait que $\omega(q)\ll\log_2\X$ dans la dernière estimation.
On déduit  de l'inégalité de Cauchy-Schwarz, de (\ref{Type I Daniel}), du lemme \ref{somme des diviseurs} et 
de (\ref{diviseurs en moyenne AI}) que l'on  peut estimer le terme de reste de 
(\ref{vers corollaire AI}) par \begin{align}
\nonumber\sum_{\substack{D<d\leq 2D\\d_{q\text{-s}}>\z}}\tau(d)^{c(B)}
\left|r_d(q\X)
\right|&\leq\left(\sum_{\substack{D<d\leq 2D\\d_{q\text{-s}}>\z}}\tau(d)^{c(B)}
\left|r_d(q\X)
\right|\right)^{1/2}\left(\sum_{\substack{D<d\leq 2D\\d_{q\text{-s}}>\z}}
\left|r_d(q\X)
\right|\right)^{1/2}\\
&\ll\label{second sum AI}  \X\left(  \X^{1/2}D^{1/4}+D^{1/2}\right)(\log \X)^{c(B)}.
\end{align}
En combinant les  majorations (\ref{vers corollaire AI}), (\ref{premier moyenne}), (\ref{fIRSsum AI}) et (\ref{second sum AI}), on obtient finalement l'estimation
\begin{equation}\label{premier corollaire AI}
\sum_{\substack{ \I \not\in \mathcal{M}(\z)\\D<N(\I)\leq 2D}}\tau_K(\I)^B\# \mathcal{A}_\I\ll
\left(\frac{  \X^2}{\z^{1/2}}+  \X^{3/2}D^{1/4}+\X D^{1/2}
\right)(\log \X)^{c(B)}.
\end{equation}

En utilisant la multiplicativité de $\alpha_q$ et les estimations (\ref{écriture rho2 non singulier})
, (\ref{inégalité alpha singulier}) et (\ref{fIRSsum AI}), 
on a d'autre part que
\begin{align}
\nonumber\sum_{\substack{\I \not\in \mathcal{M}(\z)\\D<N(\I)\leq 2D}}\tau_{\K}(\I)^B\frac{\alpha_q(\I)}{N(\I)}&\ll
\sum_{\substack{ D<d\leq 2D\\d_{q\text{-s}}>\z}}\tau(d)^{c(B)}\sum_{N(\I)=d}\frac{\alpha_q(\I)}{N(\I)}\\
&\nonumber\ll
\sum_{\substack{ \z<d\leq 2D\\d\text{ }q\text{-singulier}}}\tau(d)^{c(B)}\sum_{N(\I)=d}
\frac{\alpha_q(\I)}{N(\I)}(\log \X)^{c(B)}\\
&\nonumber\ll  \z^{-1/2}(\log \X)^{c(B)}\sum_{\substack{ \z<d\leq 2D\\d\text{ }q\text{-singulier}}}
\tau(d)^{c(B)}d^{-1/6}\\
&\ll \z^{-1/2}(\log \X)^{c(B)}
\label{2 vers corollaire AI}.
\end{align}  

On obtient finalement le résultat escompté en utilisant le lemme \ref{AI admissible} avec  $\varepsilon=\frac{1}{2}$, les estimations 
(\ref{premier corollaire AI}) 
et (\ref{2 vers corollaire AI}) et en choisissant  $\z=\min\left(  \X^{1/4},  \X^{1/2}D^{-1/4}\right)$. 
\end{proof}

En conclusion de notre étude relative à la distribution multiplicative de  $\mathcal{A}$, nous considérons le cardinal de
\begin{align*}
\mathcal{A}_{\I ,d}:=\left\{\J\in\mathcal{A}_{\I}:d|N(\J\I^{-1})\right\}
\end{align*}
où $d$ est sans facteur carré.
Suivant l'argument développé à la page 265 de \cite{HM02}, le principe d'inclusion-exclusion permet d'écrire la formule
\begin{equation}\label{une forme de mobius}
 \mu(d)\sum_{\substack{\J_1|(\J_2,d)\\ d|N(\J_1)}}\mu_{\K}(\J_1)=
 \left\{\begin{array}{ll}
 1&\text{ si }d|N(\J_2),\\
 0&\text{ sinon}. 
 \end{array}
 \right.
\end{equation}
Compte tenu du lemme  \ref{mathcalAI}, la relation (\ref{une forme de mobius}) suggère d'écrire
\begin{align}
\# \mathcal{A}_{\I ,d}\label{définition rAId}
&=\mu(d)\sum_{\substack{\J|d\\d|N(\J)}}\mu_{\K}(\J)\#\mathcal{A}_{\I \J}
=\frac{\eta^2  \X^2}{\zeta_q(2)}\frac{\alpha_q(\I)\gamma_q(\I,d)}{N(\I)}+r(\mathcal{A},\I,d)
\end{align}
où
\begin{equation}\label{définition de gamma}
\gamma_{q}(\I,d)=\left\{\begin{array}{ll}
\mu(d)\sum_{\substack{\J|d\\d|N(\J)}}\frac{\mu_{\K}(\J)\alpha_q(\I\J)}{\alpha_q(\I)N(\J)}&\text{ si }\alpha_q(\I)\neq0,\\
0&\text{ sinon},\end{array}\right.
\end{equation}
et
\begin{align}
|r(\mathcal{A},\I,d)|\leq\sum_{\substack{\J|d\\d|N(\J)}}\mu^2(d)
\left|\#\mathcal{A}_{\I \J}-\frac{\eta^2  \X^2}{\zeta_q(2)}\frac{\alpha_q(\I\J)}{N(\I\J)}\right|.\label{I,q}
\end{align}
Il s'ensuit que, si $\alpha_q(\I)\neq0$, la fonction arithmétique $\gamma_q(\I,\cdot)$ est multiplicative 
et à support sur les entiers sans facteur carré. Pour $p$ un nombre premier $q$-régulier, on a notamment par (\ref{écriture rho2 non singulier}) que
\begin{equation}\label{écriture de gamma}
\gamma_{q}(\I,p)=\left\{\begin{array}{ll}
% \text{ and }n=1,\\
\frac{1}{p}&\text{ si } p|N(\I),\\
\frac{\nu_p}{1+p}&\text{ sinon},
\end{array}\right.
\end{equation}
où $\nu_p$ est le nombre d'idéaux $\mathfrak{p}$ tels que $N(\mathfrak{p})=p$.
De plus, si $p|q$ est $1$-régulier, alors 
\begin{equation}\label{inégalité gammap}
\gamma_q(\I,p)\in
\left\{0,\frac{1}{p},1\right\}.
\end{equation}
Dans le cas où $\I=\OK$, on pose $\gamma_q(d)=\gamma_q(\OK,d)$ pour tout entier $d\geq1$. 

                                                                                                                      \begin{lemme}\label{mathcalAIq}
Soit $B\geq 0$. Il existe $c(B)\geq0$  tel que, uniformément en  $\X\geq 2$
, $(N_1,N_2)\in\mathcal{N}(\eta)$, et $1\leq D\leq   \X^2$,
\begin{align*}
\Sigma_3:=\sum_{\substack{\I,d\\ D<N(\I)d\leq 2D}}
\mu_{\K}(d)^2\tau(N(\I)d)^B\left|r(\mathcal{A},\I,d)\right|\ll
\left(  \X^{15/8}+  \X^{7/4}D^{1/8}\right)(\log \X)^{c(B)}.
\end{align*}
\end{lemme}
\begin{proof}
Puisque la sommation (\ref{I,q}) ne fait intervenir que des idéaux $\J$ admissibles et des entiers
$d$ sans facteur carré, on observe à l'aide du lemme \ref{admissible}  
que $N(\J)\asymp d$. Par suite, on déduit du lemme \ref{mathcalAI} que
\begin{align*}
\Sigma_3&\ll\sum_{\substack{
N(\J)\ll D}}
\tau_{\K}(\J)^{c(B)}\left|\# \mathcal{A}_{\J}
-\frac{\eta^2  \X^2}{\zeta_q(2)}\frac{\alpha_q(\J)}{N(\J)}\right|\\
&\ll
\left(  \X^{15/8}+  \X^{7/4}D^{1/8}\right)(\log \X)^{c(B)}.
\end{align*}\end{proof}

\subsection{Niveau de distribution de $\mathcal{B}$}

Par analogie à $\mathcal{A}_{\I}$, on considère dans ce paragraphe le cardinal de l'ensemble 
\begin{align*}\mathcal{B}_{\I}:=\left\{\J\in\mathcal{B}:\I|\J\right\}.\end{align*}
L'étude de la distribution des idéaux de $\OK$ se base sur l'ordre moyen suivant, dû à Landau. 
\begin{theo}[\cite{La18}]\label{theo Weber}
 On a, uniformément pour $\X\geq 1$,
\begin{align*}
\#\left\{\J\in\JK:N(\J)\leq  \X\right\}=\lambda_{\K}\X+O\left(  \X^{2/3}\right).
\end{align*}
\end{theo}

On définit \begin{align*}
\mathcal{B}_{\I,d}:=\left\{\J\in\mathcal{B}_{\I}:d|N(\J\I^{-1})\right\}.
\end{align*}
 L'estimation de Type I relative à $\mathcal{B}_{\I,d}$ contenue dans le lemme ci-dessous est une généralisation
 du lemme 3.6 de \cite{HM02}.

\begin{lemme}
\label{mathcalBIq}
Soit $B\geq0$. Il existe $c(B)\geq 0$ tel que, uniformément en $\X\geq2$, $(N_1,N_2)\in\mathcal{N}(\eta)$, $D\geq 1$ 
et $\I$ un idéal de $\JK$, on ait
\begin{align*}
\sum_{D<d\leq2D}\tau(d)^B\mu(d)^2\left|\#\mathcal{B}_{\I,d}-\lambda_{\K}\frac{c(N_1,N_2)\eta q^3  \X^3}{N(\I)}\beta(d)\right|
\ll \frac{  \X^2D^{1/3}}{N(\I)^{2/3}}(\log 2D)^{c(B)}
\end{align*}
où $\beta$ est la fonction multiplicative à support sur les entiers sans facteur carré définie par
\begin{equation}
\beta(p)=\left(1-\prod_{\mathfrak{p}|p}\left(1-\frac{1}{N(\mathfrak{p})}\right)\right).\label{définition de beta}
\end{equation}
\end{lemme}
\begin{proof}
 Au vu de la formule (\ref{une forme de mobius}), on a la formule
 \begin{align*}
 \#\mathcal{B}_{\I,d}&=\mu(d)\sum_{\substack{\J|d\\d|N(\J)}}\mu_{\K}(\J) \#\left\{\mathfrak{b}\in\JK:\frac{c(N_1,N_2)q^3\X^3}{N(\I\J)}<N(\mathfrak{b})\leq  (1+\eta) \frac{c(N_1,N_2)q^3\X^3}{N(\I\J)}\right\}
 \end{align*} Au vu du théorème \ref{theo Weber}, le résultat se déduit de la suite d'estimations
\begin{align*}
 \sum_{D<d\leq2D}
 \tau(d)^B\mu(d)^2\sum_{\substack{\J|d\\d|N(\J)}}\frac{q^2  \X^2}{N(\I\J)^{2/3}}
 &\ll \frac{q^2  \X^2}{N(\I)^{2/3}}\sum_{D<d\leq2D}
 \frac{\tau(d)^{c(B)}}{d^{2/3}}\\&\ll\frac{  \X^2D^{1/3}}{N(\I)^{2/3}}(\log 2D)^{c(B)}.
\end{align*}
\end{proof}

\section{La fonction de densité $\sigma_q$}\label{Fonction de densité sigmaq} 
\subsection{Définition et premières propriétés}
\label{Premières propriétés}
Étant donné un sous-ensemble $C$ de $\N$ et un idéal $\I$, on considère le cardinal
\begin{equation}\label{notation SAIC}
S(\mathcal{A},\I  ,C):=\#\left\{\Qid:\I\Qid\in\mathcal{A}\text{ et }N(\Qid)\in C\right\}.
\end{equation}

En général, 
on ne peut espérer évaluer $S(\mathcal{A},\I  ,C)$ en utilisant uniquement les majorations de sommes de Type I 
obtenues dans la partie précédente%\ref{paragraphe Type I}
. La seule connaissance de ces estimations est par exemple insuffisante pour approcher $S(\mathcal{A},\I  ,C)$
lorsque les éléments de l'ensemble $C$ ont un nombre fixe de facteurs premiers : 
c'est là l'incidence du phénomène de parité, mentionné en introduction. 

 À la suite de Heath-Brown et Moroz, on introduit le cardinal
\begin{equation}\label{notation SBIC}
S(\mathcal{B},\I  ,C):=\#\left\{\Qid:\I\Qid\in\mathcal{B}\text{ et }N( \Qid)\in C\right\}
\end{equation} et l'on cherche à montrer l'existence d'une constante $\sigma_q(F)>0$ et d'une fonction  
$\sigma_q:\JK\rightarrow\R$
 telles que l'approximation
\begin{equation}\label{différence S A B}
S(\mathcal{A},\I  ,C)\sim \frac{\sigma_q(F)\eta}{c(N_1,N_2)q^3\X} \sigma_q(\I)S(\mathcal{B},\I  ,C)
\end{equation} 
soit vraie, en moyenne sur les idéaux $\I$ et pour une certaine classe  d'ensembles d'entiers $C$ inclus dans l'ensemble criblé
\begin{equation}\label{définition C crible}
C^-(  \X^{\tau}):=\left\{d:P^-(d)>   \X^{\tau}\right\}
\end{equation} 
où $\tau=o(1)$ est un paramètre qui sera précisé ultérieurement.

Dans ce paragraphe, on obtient  des valeurs heuristiques pour $\sigma_q(F)$ et $\sigma_q(\I)$ en étudiant 
$S(\mathcal{A},\I  ,C^-(  \X^{\tau}))$  et $S(\mathcal{B},\I  ,C^-(  \X^{\tau}))$. 
Au vu des estimations de sommes de Type I contenues dans les lemmes \ref{mathcalAIq} et  \ref{mathcalBIq}, 
l'application d'un lemme fondamental de crible - démarche qui sera effectuée en détail au paragraphe \ref{lemmes de cribles} -
suggère que l'on a 
\begin{equation}\label{asymptotique crible A}
 S(\mathcal{A},\I  ,C^-(  \X^{\tau}))\sim \frac{\eta^2  \X^2}{\zeta_q(2)}\frac{\alpha_q(\I)}{N(\I)}
\prod_{\substack{p\leq   \X^{\tau}}}\left(1-\gamma_q(\I,p)\right)
\end{equation}
et \begin{align}\label{asymptotique crible B}
S(\mathcal{B},\I  ,C^-(  \X^{\tau}))\sim \lambda_{\K}c(N_1,N_2)\frac{\eta q^3  \X^3}{N(\I)}\prod_{p\leq   \X^{\tau}}
\left(1-\beta(p)\right).
\end{align}

L'étude de la singularité de $\zeta_K(s)$ en $s=1$ permet d'obtenir l'analogue suivant de la formule (6.8) de \cite{HB01}, uniforme en $\Y\geq 2$, 
\begin{equation}\label{convergence produit beta}
 \prod_{p\leq  \Y}\left(1-\frac{\beta(p)}{p}\right)\left(1-\frac{1}{p}\right)^{-1}=\lambda_{\K}^{-1}\left(1+O\left((\log \Y)^{-2}\right)\right).
\end{equation}

Comme  observé dans le lemme 3.4 de \cite{HM02}, le théorème des  idéaux premiers implique de même que, 
uniformément en $\Y>P(q)%\max\{p:p\text{ est } q\text{-singulier}\}
 $, on a
\begin{equation}\label{convergence sigmaf}
\prod_{\substack{ p\geq \Y
}}(1-\gamma_q(p))\left(1+\frac{1}{p}\right)=\left(1+O\left((\log \Y)^{-2}\right)\right),
\end{equation}
où  $\gamma_q$ est défini par (\ref{définition de gamma}),
d'où l'on déduit la convergence du produit eulérien \begin{align}
\sigma_q(F):=\left(\prod_{\substack{\gamma_q(p)\neq1}}(1-\gamma_q(p))\left(1+\frac{1}{p}\right)\prod_{\substack{\gamma_q(p)=1}}\left(1+\frac{1}{p}\right)\right)\prod_{p|q}\left(1-\frac{1}{p^2}\right)^{-1}.
\label{définition sigmaqF}\end{align}
Ce produit eulérien, strictement positif d'après le lemme 3.4 de \cite{HM02}, satisfait, au vu de (\ref{inégalité gammap}), l'estimation uniforme  en $q\geq1$,
\begin{equation}\label{estimation sigmaqF}
 \sigma_q(F),\sigma_q(F)^{-1}\ll
\frac{q}{\varphi(q)}
\ll \log_2 (3q).
\end{equation}

On observe que la formule (\ref{écriture de gamma}) implique  les estimations suivantes,  uniformes en $N(\J)\ll q^3  \X^3$ et $\X^{\tau}>q$,
  \begin{align}\prod_{\substack{p>   \X^{\tau}\\p|N(\J)}}\left(1-\gamma_q(\J,p)\right)= 1+O\left(\frac{1}{\tau   \X^{\tau}}\right)\qquad\text{et}\qquad
       \prod_{\substack{p>   \X^{\tau}\\p|N(\J)}}\left(1-\gamma_q(p)\right)= 1+O\left(\frac{1}{\tau   \X^{\tau}}\right).\label{estimations distribution tronquée}
\end{align}
Sous réserve d'établir les formules
 (\ref{asymptotique crible A}) et (\ref{asymptotique crible B}), il s'ensuit la formule  uniforme en $N(\I)\ll q^3  \X^3$
\begin{align}\label{crible A}
S(\mathcal{A},\I  ,C^-(  \X^{\tau}))
&\sim \eta^2  \X^2\sigma_q(F)
\frac{\sigma_q(\I)}{N(\I)}
\prod_{p\leq   \X^{\tau}}\left(1-\frac{1}{p}\right)
\end{align}
avec
\begin{equation}\label{définition sigmaq}
\sigma_q(\I):=\widetilde{\sigma}_q(\I)
\prod_{\substack{\gamma_q(p)=1
}}\left(1-\gamma_q(\I,p)\right)\text{ et }
\widetilde{\sigma}_q(\I):=\alpha_q(\I)
\prod_{\substack{p|N(\I)\\\gamma_q(p)\neq1}}
\left(1-\gamma_q(\I,p)\right)
\left(1-\gamma_q(p)\right)^{-1}
.\end{equation} 
On a en particulier, pour  $p$ un premier $q$-régulier et $k\geq1$, la formule
\begin{equation}\label{définition kappa q}
\widetilde{\sigma}^{\Z}_q(p^k)=\nu_p
\left(\frac{p-1}{p+1}\right)\left(1-\frac{\nu_p}{1+p}\right)^{-1}=\nu_p+O\left(\frac{1}{p}\right).
\end{equation}
De plus, on peut déduire de l'estimation (\ref{inégalité alpha singulier}) 
la majoration uniforme en $p$ premier 
\begin{equation}\label{majoration sigma}
\widetilde{\sigma}_q(\I_p)\ll\alpha_q(\I_p)\ll N(\I_p)^{1/3}\min\left(p^{2v_p(q)}N(\I_p)^{2/3}\right).
\end{equation}

Au vu de la formule asymptotique (\ref{convergence produit beta}), on a également
\begin{align}\label{crible B}
S(\mathcal{B},\I  ,C^-(  \X^{\tau}))\sim c(N_1,N_2)\frac{\eta q^3  \X^3}{N(\I)}\prod_{p\leq   \X^{\tau}}\left(1-\frac{1}{p}\right) 
\end{align}
ce qui justifie la pertinence de (\ref{différence S A B}) lorsque $C=C^-(  \X^{\tau})$.

\subsection{Ordre moyen de $\sigma_q$ 
}\label{section ordre moyen}

On étudie dans ce paragraphe l'ordre moyen $S(\mathcal{B};\sigma_qh)$ défini par (\ref{définition SBh}) pour $h\in\mathcal{M}(z)$, c'est-à-dire $h$ multiplicative, à valeur dans le disque unité 
et satisfaisant $h(p)=z$ pour tout premier $p$.
Comme souligné dans l'introduction, la preuve du théorème \ref{ordre moyen fonction multiplicative} repose sur une estimation asymptotique de $S(\mathcal{B};\sigma_q h)$. De plus, les majorations de $S(\mathcal{B};\sigma_q)$ que nous développons ici seront utilisées à de nombreuses reprises dans les parties \ref{lemmes de cribles} et \ref{paragraphe Type II}.

Compte tenu de la formule  (\ref{définition kappa q}) relative à la valeur 
de $\widetilde{\sigma}_q$ aux entiers $q$-réguliers, on observe que, uniformément pour $p$  premier $q$-régulier 
 et $0\leq\Re(s)\leq 1$, on a
\begin{equation}\label{theo idéaux premiers sigma}
\frac{\widetilde{\sigma}^{\Z}_q(p)}{p^s}=\sum_{\mathfrak{p}|p}\frac{1}{N(\mathfrak{p})^s}+O\left(
%\sum_{p\leq \Y }
\frac{1}{p^{2\Re(s)}}\right).
\end{equation} 
Il suit de (\ref{définition kappa q}) que la série $\mathcal{H}(s)\zeta_{\K}(s)^{-z}$ 
où \begin{align}\mathcal{H}(s):=\sum_{\substack{n\geq1\\n\text{ }q\text{-régulier}}}
\frac{\widetilde{\sigma}_q^{\Z}(n)h(n)}{n^s}   
\label{définition Hs}    \end{align}
est développable en un produit eulérien absolument convergent pour $\Re(s)>1/2$.

De plus, la relation (\ref{theo idéaux premiers sigma}) 
entraîne que, uniformément en $q\geq1$ et  $\Re(s)>
1/2$,
\begin{align}
\mathcal{H}(s)\zeta_{\K}(s)^{-z}
\ll_{\Re(s)}
\prod_{p|q
}
\left(1+O\left(\frac{1}{p^{\Re(s)}}\right)\right)
 \ll_{\Re(s)
 }2^{\omega(q)}\label{estimation G(s) général}\end{align} et 
 \begin{align}
 \left.\mathcal{H}(s)\zeta_{\K}(s)^{-z}\right|_{s=1}
 \nonumber\ll
 \prod_{p|q}\left(1+O\left(\frac{1}{p}\right)\right)\ll\left(\frac{q}{\varphi(q)}\right)^c\ll (\log_2(q+2))^c
\label{estimation G(1)}
\end{align}
La méthode de Selberg-Delange permet d'établir le résultat central de ce paragraphe, à savoir la formule asymptotique
 de $S(\mathcal{B};\sigma_qh)$ contenue dans la proposition suivante.
 
\begin{prop}\label{Selberg-Delange h sigma}
Soient $z$ un nombre complexe tel que $|z|\leq1$ et $h\in\mathcal{M}(z)$. 
 Il existe $C>0$  tel que l'on ait, uniformément en 
$\X\geq2$ 
et $(N_1,N_2)\in\mathcal{N}(\eta)$,
\begin{equation}\label{Msigmaq}
S(\mathcal{B};\sigma_qh)=
c(N_1,N_2)\eta q^3  \X^3\log\left(c(N_1,N_2)q^3  \X^3\right)^{z-1}
\left(\frac{\sigma_q(F,h)}{\Gamma(z)}+
O\left(\frac{C^{\omega(q)}(\log(q+1))^4}{\log \X}\right)\right)
\end{equation}
où 
\begin{align}
 \sigma_q(F,h):=&\nonumber\prod_{\substack{p\\
 \gamma_q(p)\neq1}}
 \left(1-\frac{1}{p}\right)^z
 \left(1+
 \left(1-\gamma_q(p)\right)^{-1}\left(\sum_{k\geq1}\frac{h(p^k)}{p^k}
 \sum_{N(\I)=p^k}\alpha_q(\I)\left(1-\gamma_q(\I,p)\right)\right)\right)\\&\times
 \prod_{\substack{p\\
 \gamma_q(p)=1}}
 \left(1-\frac{1}{p}\right)^z
 \left(\sum_{k\geq1}\frac{h(p^k)}{p^k}
 \sum_{N(\I)=p^k}\alpha_q(\I)\left(1-\gamma_q(\I,p)\right)\right)
 \label{définition sigma qh}
\end{align}
et $\alpha_q(\cdot)$ et $\gamma_q(\cdot,\cdot)$ sont définis respectivement par (\ref{définition rho2}) et (\ref{définition de gamma}).

On a en particulier 
$\sigma_q(F,\mathbf{1})=\sigma_q(F)^{-1}\zeta_q(2)^{-1}$ et, uniformément en $\X\geq3$ et $q\leq(\log \X)^A$,
\begin{equation}\label{log sigmaq}
 \sum_{N(\I)\leq \X}\frac{\sigma_q(\I)}{N(\I)}= \frac{1}{\zeta_q(2)\sigma_q(F)}\log \X+O\left(C^{\omega(q)}(\log_2\X)^4\right).
\end{equation}
\end{prop}
\begin{proof}
On étudie dans un premier temps $\sigma_q h$ sur les entiers $q$-réguliers. Au vu de ce qui précède, 
on peut adapter  \textit{mutatis mutandis} la preuve du théorème II.5.2 de
\cite{Te08} en remplaçant la fonction zêta de Riemann $\zeta(s)$ par 
la fonction zêta de Dedekind $\zeta_{\K}(s)$, pour en déduire que l'on a,
 pour tout entier $N\geq0$ et uniformément en $\X\geq2$ et $q\geq1$, 
\begin{align}
 \sum_{\substack{n\leq \X\\n\text{ }q\text{-régulier}}}\widetilde{\sigma}^{\Z}_q(n)h(n)&=\X(\log \X)^{z-1}
 \left(\sum_{k=0}^N\frac{\lambda_k(h)}{\Gamma(z-k)(\log \X)^k}
 +O\left(\frac{2^{\omega(q)}}{(\log \X)^{N+1}}\right)\right)\label{Selberg-Delange kappa}
 \end{align}
 où les $\lambda_k(h)$ sont définis à l'aide du développement de Taylor en $s=1$
 \begin{align}\label{dev Taylor lambda}
 \frac{(s-1)^{z}\mathcal{H}(s)}{s}:=\sum_{k\geq0}\lambda_k(h)(s-1)^k.
  \end{align}
En particulier, on a 
\begin{equation}\label{valeur lambda1}
 \lambda_0(h)=\lambda_{\K}^{z}\left.\mathcal{H}(s)\zeta_{\K}(s)^{-z}\right|_{s=1}\end{equation}  tandis que 
 l'estimation (\ref{estimation G(s) général}) et  l'inégalité de Cauchy entraînent 
 les majorations
 \begin{equation}\label{estimation lambdah}
| \lambda_k(h)|\ll \sup_{|s-1|\leq \frac{1}{4}}\left\{|\mathcal{H}(s)\zeta_{\K}(s)^{-z}|
  \right\}\ll 2^{\omega(q)}.
 \end{equation}

 En choisissant $h(n)=1$ pour tout entier $n\geq1$, une sommation par parties entraîne directement la formule  uniforme en $q\geq1$ et $\X\geq3$ suivante 
 \begin{equation}\label{ordre moyen inverse}
  \sum_{\substack{n\leq \X\\n\text{ }q\text{-régulier}}}\frac{\widetilde{\sigma}^{\Z}_q(n)}{n}=
  \lambda_0(\mathbf{1})(\log \X)+O\left(2^{\omega(q)}\log_2\X\right)
  \end{equation}
avec
 \begin{displaymath}
 \lambda_0(\mathbf{1})=\prod_{\substack{p\text{ }q\text{-régulier}}}
 \left(1-\frac{1}{p}\right)
 \left(1+\left(\frac{p+1}{\nu_p}-1\right)^{-1}\right)  \end{displaymath}

Attardons-nous à présent sur la contribution des entiers $q$-singuliers. On peut montrer que pour tout $\sigma>1/3$, il existe $C(\sigma)>0$ tel que
\begin{equation}\label{estimation puissance}
\sum_{d\text{ }q\text{-singulier}} \frac{\sigma^{\Z}_q(d)\log d}{d^{\sigma}}\leq
q^{3(1-\sigma)}(\log (q+1))^4C(\sigma)^{\omega(q)+1}.
\end{equation}
En effet,  d'une part, on peut écrire à l'aide de 
(\ref{tau nombre de norme}) et (\ref{inégalité alpha singulier})les estimations suivantes, valides pour $p$ fixé et $1/3<\sigma<1$,
\begin{align*}
\sum_{k\geq0}\sum_{N(\I)=p^k}\frac{\alpha_q(\I)}{N(\I)^{\sigma}}
&\ll\sum_{k\leq 3v_p(q)}\tau(p^k)^2p^{k(1-\sigma)}+p^{2v_p(q)}
\sum_{k>3v_p(q)}\tau(p^k)^2p^{k(1/3-\sigma)}\\
&\ll_{\sigma} (\log (q+1))^{3}p^{3v_p(q)(1-\sigma)}.
\end{align*}
 En outre,
on observe que 
(\ref{écriture rho2 non singulier}) entraîne l'estimation uniforme  pour $\mathfrak{p}$ un idéal premier $1$-régulier et $1/3<\sigma<1$, 
\begin{align*}
 \sum_{k\geq0}\frac{\alpha_q(\mathfrak{p}^k)(1-\gamma_q(\mathfrak{p}^k,p))}{p^{k\sigma}}&
 \ll
p^{v_p(q)(1-\sigma)}
.
\end{align*}
On en déduit alors en utilisant la définition (\ref{définition sigmaq}) de $\sigma_q$ que, pour $1/3<\sigma<1$, on a
\begin{align*}
 \sum_{d\text{ }q\text{-singulier}} \frac{\sigma_q^{\Z}(d)}{d^{\sigma}}
 &\ll_{\sigma} C(\sigma)^{\omega(q)}\prod_{p\text{ }q\text{-singulier}}\left( 
 \sum_{k\geq0}\sum_{N(\I)=p^k}\frac{\alpha_q(\I)(1-\gamma_q(\I,p))}{N(\I)^{\sigma}}\right)\\
 &\ll_{\sigma}
q^{3(1-\sigma)}(\log (q+1))^3C(\sigma)^{\omega(q)}.
\end{align*}
Dans la mesure où le membre de gauche de (\ref{estimation puissance}) est la dérivée de la série  $\sum_{d\text{ }q\text{-singulier}} \frac{\sigma_q^{\Z}(d)}{d^{s}}$ en $\sigma$, on obtient l'estimation (\ref{estimation puissance})  en appliquant l'inégalité de Cauchy dans le cercle de centre $\sigma$ et de rayon $\frac{c(\sigma)}{\log q}$ 
avec $c(\sigma)>0$ suffisamment petit.
Ceci permet de déduire la convergence absolue de $\sigma_q(F,h)$ ainsi que l'estimation suivante, conséquence de la méthode de Rankin
,\begin{align}\nonumber
\sum_{\substack{d>\Y\\d\text{ }q\text{-singulier}}}
 \frac{\sigma_q^{\Z}(d)}{d}&\leq \Y ^{\sigma-1}\sum_{\substack{d>\Y\\d\text{ }q\text{-singulier}}}
 \frac{\sigma_q^{\Z}(d)}{d^{\sigma}}\\
 &\label{Rankin sigmaq}\ll_{\sigma} \Y^{\sigma-1}
q^{3(1-\sigma)}(\log (q+1))^3C(\sigma)^{\omega(q)}
\end{align}
où $1/3<\sigma<1$.

En utilisant (\ref{ordre moyen inverse}) et (\ref{Rankin sigmaq}), 
il s'ensuit alors que, uniformément en $q\leq(\log \X)^A$,
\begin{align*}
 \sum_{n\leq \X}\frac{\sigma_q^{\Z}(n)}{n}
 &=\sum_{\substack{d\leq  \X^{1/2}\\
 d\text{ }q\text{-singulier}}}\frac{\sigma_q^{\Z}(d)}{d}
 \sum_{\substack{n\leq \X/d\\n\text{ }q\text{-régulier}}}
 \frac{\widetilde{\sigma}^{\Z}_q(n)}{n}+O\left(
 \sum_{\substack{d>   \X^{1/2}\\
 d\text{ }q\text{-singulier}}}\frac{\sigma_q^{\Z}(d)}{d}
 \sum_{\substack{n\leq \X\\n\text{ }q\text{-régulier}}}
 \frac{\widetilde{\sigma}^{\Z}_q(n)}{n}\right)\\
 &=  \sum_{\substack{d\leq   \X^{1/2}\\
 d\text{ }q\text{-singulier}}}
 \frac{\sigma_q^{\Z}(d)}{d}\lambda_0(\mathbf{1})\log \left(\frac{\X}{d}\right)+O\left(C^{\omega(q)}(\log_2\X)^4\right)
\\&=  \sigma_q(F,\mathbf{1})\log \X+O\left(C^{\omega(q)}(\log_2\X)^4\right),
  \end{align*}
  en remarquant que 
  \begin{displaymath}
    \sum_{\substack{
 d\text{ }q\text{-singulier}}}
 \frac{\sigma_q^{\Z}(d)}{d}\lambda_0(\mathbf{1})=  \sigma_q(F,\mathbf{1}).
  \end{displaymath}
  
 Pour montrer que $\sigma_q(F,\mathbf{1})=\sigma_q(F)^{-1}\zeta_q(2)^{-1}$ et ainsi établir (\ref{log sigmaq}), 
on commence par remarquer que, pour tout premier $p$, la définition (\ref{définition de gamma}) entraîne la formule\begin{align}
\sum_{k\geq 1}\sum_{N(\I)=p^k}\frac{\alpha_q(\I)}{N(\I)}(1-\gamma_q(\I,p))
&=\nonumber
\sum_{k\geq 1}\sum_{N(\I)=p^k}\left(\frac{\alpha_q(\I)}{N(\I)}
+\sum_{\substack{J|p\\p|N(\J)}}\mu_{\K}(\J)\frac{\alpha_q(\I\J)}{N(\I\J)}\right)\\
&\label{vers calcul C}
=\sum_{k\geq1}\sum_{N(\I)=p^k}
\frac{\alpha_q(\I)}{N(\I)}\left(1+\sum_{\substack{J|(\I,p)\\p|N(\J)\\I\neq J}}\mu_{\K}(\J)\right).
\end{align}
En injectant la formule d'inclusion-exclusion 
 (\ref{une forme de mobius}) 
dans (\ref{vers calcul C}), on obtient ainsi 
 l'égalité
\begin{align}\sum_{k\geq 1}\sum_{N(\I)=p^k}\frac{\alpha_q(\I)}{N(\I)}(1-\gamma_q(\I,p))
=-
\sum_{k\geq 1}\sum_{N(\I)=p^k}\mu_{\K}(\I)\frac{\alpha_q(\I)}{N(\I)}
&=\gamma_q(p)\label{vers calcul C bis}
.\end{align}
Compte tenu des définitions (\ref{définition sigma qh}) et  (\ref{définition sigmaqF}), il s'ensuit que 
$\sigma_q(F,\mathbf{1})=\sigma_q(F)^{-1}\zeta_q(2)^{-1}$.

On conclut la preuve de la proposition en montrant la formule (\ref{Msigmaq}).  Compte tenu des définitions (\ref{définition SBh}) et (\ref{définition sigmaq}), on peut écrire
\begin{equation}\label{convolution sigm q régulier singulier}
S(\mathcal{B};\sigma_qh)
 =\sum_{\substack{
 d\text{ }q\text{-singulier}}}\sigma_q^{\Z}(d)h(d)
 \sum_{\substack{n\text{ }q\text{-régulier}\\c(N_1,N_2)q^3  \X^3<dn\leq c(N_1,N_2)q^3  \X^3(1+\eta)}}
 \widetilde{\sigma}^{\Z}_q(n)h(n).
\end{equation} Pour estimer la contribution des entiers $d\leq\X$, on utilise (\ref{Selberg-Delange kappa})  sous la forme
\begin{displaymath}
 \sum_{\substack{n\text{ }q\text{-régulier}\\\X<n\leq   \X(1+(\log\X)^{-B})}}
 \widetilde{\sigma}^{\Z}_q(n)h(n)=\frac{\X(\log\X)^{z-1}}{(\log\X)^B}\left(\frac{\lambda_0(h)}{\Gamma(z)}+O\left(\frac{2^{\omega(q)}}{\log\X}\right)\right)
\end{displaymath} valide dès que $B\geq1$. En majorant  trivialement la contribution des entiers $q$-singuliers $d> \X$, on vérifie que l'on a, uniformément en $\X\geq2$, $q\leq (\log \X)^A$ 
et $(N_1,N_2)\in\mathcal{N}(\eta)$,
\begin{align}
 \nonumber S(\mathcal{B};\sigma_qh)
  & =  c(N_1,N_2)q^3\eta \X^3\left(\log \left(c(N_1,N_2)q^3  \X^3\right)\right)^{z-1}
  \frac{\lambda_0(h)}{\Gamma(z)}\sum_{\substack{d\text{ }q\text{-singulier}}}\frac{\sigma_q^{\Z}(d)}{d}+R_1+R_2
+R_3
\end{align}
où 
\begin{align*}
 \left|R_1\right|\ll   c(N_1,N_2)q^3\eta \X^3\left(\log \left(c(N_1,N_2)q^3  \X^3\right)\right)^{\Re(z)-1}\frac{C^{\omega(q)}}{\log \X}\sum_{\substack{d\leq \X\\
 d\text{ }q\text{-singulier}}}
 \frac{\sigma_q^{\Z}(d)\log d}{d},
\end{align*}
\begin{align*} 
\left|R_2\right|\ll   c(N_1,N_2)q^3\eta \X^3\left(\log \left(c(N_1,N_2)q^3  \X^3\right)\right)^{\Re(z)-1}C^{\omega(q)}\sum_{\substack{d> \X\\d\text{ }q\text{-singulier}}}
 \frac{\sigma_q^{\Z}(d)}{d},
\end{align*}
et
\begin{align*} 
\left|R_3\right|\leq\sum_{\substack{n\leq c(N_1,N_2)q^3\X^3(1+\eta)
\\n_{q\text{-s}}>\X}}\sigma_q^{\Z}(n).
\end{align*}
En utilisant les estimations (\ref{estimation puissance}) et (\ref{Rankin sigmaq}), on obtient
\begin{align*}
 |R_1|,|R_2|\ll  c(N_1,N_2)q^3\eta \X^3\left(\log \left(c(N_1,N_2)q^3 \X^3\right)\right)^{\Re(z)-1}\frac{(\log (q+1))^4C^{\omega(q)}}{\log \X} .                                             
\end{align*}
On remarque d'autre part en utilisant la méthode de Rankin et les estimations  (\ref{log sigmaq}), (\ref{tau nombre de norme}) et (\ref{majoration sigma}) que, 
si $\omega_{1\text{-s}}$ désigne le nombre de premiers $1$-singuliers, alors
\begin{align*}
  \left|R_3\right|&\ll\sum_{\substack{n\ll q^3  \X^3}}
\sigma_q^{\Z}(n)\sum_{\substack{p\text{ }q\text{-singulier}\\
  \X^{1/(\omega(q)+\omega_{1\text{-s}})}< p^k\ll q^3  \X^{3}/n}}\widetilde{\sigma}_q^{\Z}(p^k)\\
&\ll q^3  \X^3
\sum_{n\ll q^3  \X^3}\frac{\sigma_q^{\Z}(n)}{n}
\sum_{\substack{p\text{ }q\text{-singulier}\\
  \X^{1/(\omega(q)+\omega_{1\text{-s}})}< p^k}}\frac{\widetilde{\sigma}_q^{\Z}(p^k)}{p^k}\\
&\ll (\log \X)^c   \X^{3-2/3(\omega(q)+\omega_{1\text{-s}})},
\end{align*}
ce qui permet de déduire (\ref{Msigmaq}) dans la mesure où $\omega(q)\ll\log_2(3\X)$.
\end{proof}

\section{Description de la méthode}
\label{description de la méthode}

Dans toute la suite de l'article, on fixe $
\varpi_0\in]0,1[$ et on introduit le paramètre $\tau:=(\log_2\X)^{-\varpi_0}$. Dans la série d'articles \cite{HB01,HM02,HM04}, Heath-Brown et Moroz réalisent le tour de force de montrer la validité de la formule (avec les notations (\ref{notation SAIC}) et (\ref{notation SBIC}))
\begin{align}\label{Type 2 HM}
 S\left(\mathcal{A},\OK,C
 \right)\sim 
 \frac{\sigma_q(F)\eta}{c(N_1,N_2)q^3\X}S\left(\mathcal{B},\OK, C
 \right),
\end{align}  pour certains ensembles $C$ inclus dans 
\begin{equation}
    C(m,n):=\left\{rs: \Omega(r)=\omega(r)=m\text{ et } \Omega(s)=\omega(s)=n,P^-(rs)>  \X^{\tau},   \X^{1+\tau}\leq s\leq   \X^{\frac{3}{2}-\tau}\right\}.
   \label{définition Cnm}\end{equation}

Pour tout entier $k\geq1$, on considère $\mathcal{E}(k)$ l'ensemble des parties de $\N$ qui s'écrivent  \begin{equation}\label{définition inégalités}
E(k)=\bigcap_{i=1}^{\nbin}E_{i}(k)\end{equation}où $1\leq \nbin\leq \tau^{-1}$ et les $E_{j}(k)$ sont des ensembles de la forme
\begin{equation}
 E_{i}(k)=E_{i}(k;\Y,\overrightarrow{\alpha},\overrightarrow{\beta},\prec):=\left\{m:\Omega(m)=\omega(m)=k,\Y P^{(\overrightarrow{\alpha})}(m)\prec P^{(\overrightarrow{\beta})}(m)
\right\}\label{définition Eij}
\end{equation}
où $\Y>0$, $\prec$ désigne $<$ ou $\leq$, $\overrightarrow{\alpha}:=(\alpha_1,\ldots,\alpha_{j})$ 
avec $1\leq\alpha_1<\cdots<\alpha_{j}$, $\overrightarrow{\beta}:=(\beta_1,\ldots,\beta_l)$ avec $1\leq\beta_1<\cdots<\beta_l$, 
\begin{align*}
 P^{(\overrightarrow{\alpha})}(m):=P^{(\alpha_1)}(m)\cdots P^{(\alpha_j)}(m)\qquad\text{et}
 \qquad 
 P^{(\overrightarrow{\beta})}(m):=
P^{(\beta_1)}(m)\cdots P^{(\beta_l)}(m),
\end{align*}  
 les $P^{(i)}(m)$ étant définis par la décomposition de $m$ en facteurs premiers
\begin{displaymath}
m=P^{(1)}(m)\cdots P^{(\Omega(m))}(m) \qquad\text{ avec }\qquad P^{(i)}(m)\leq P^{(i+1)}(m)
\end{displaymath}
avec la convention $ P^{(\overrightarrow{\alpha})}(m)=1$ (resp. $ P^{(\overrightarrow{\beta})}(m)=1$) si $j=0$ (resp. $l=0$).
En adaptant les arguments de Heath-Brown et Moroz, on montrera au paragraphe \ref{paragraphe Type II} que la formule (\ref{Type 2 HM}) se généralise sous la forme
\begin{equation}\label{asymptotique crible II}
 S\left(\mathcal{A},\I,E\cap C(m,n)\right)\sim 
 \frac{\sigma_q(F)\eta}{c(N_1,N_2)q^3\X}\sigma_q(\I)S\left(\mathcal{B},\I,E\cap C(m,n)\right)
\end{equation}
pour les ensembles $E\in\mathcal{E}(m+n)$\footnote{Dans la mesure où toute réunion peut s'écrire comme intersection d'ensembles, on peut en fait considérer des ensembles de la forme plus générale $E(k)=\bigcup_{i}
\bigcap_{j}E_{ij}(k)$}.

En vue d'exploiter la relation (\ref{asymptotique crible II}), on introduit l'ensemble $\mathcal{F}$ des fonctions $h$ pour lesquelles on peut écrire, pour tout entier $m$  satisfaisant
$\Omega(m^+(\X^{\tau}))=\omega(m^+(\X^{\tau}))=k$, la décomposition\begin{equation}h(m)=h_1(m^-(\X^{\tau}))
h_2(k)1_{E(k)}(m^+(\X^{\tau}))\label{décomposition 1 f}\end{equation}
où $E(k)\in\ens(k)$ et $h_1,h_2:\N\rightarrow\C$.
Il convient de noter que les fonctions définies dans l'introduction par (\ref{première décomposition h}) sont des éléments de $\mathcal{F}$.

En isolant les idéaux de $\mathcal{A}$ tels que $\J^+(  \X^{\tau})$ possède au moins un facteur carré, 
la décomposition (\ref{décomposition 1 f}) 
permet de réécrire l'ordre moyen $S(\mathcal{A};h)$ défini par (\ref{définition SAh}) sous la forme
\begin{equation}\label{identité combinatoire générale}
S(\mathcal{A};h)
=\sum_{P^+(N(\I))\leq   \X^{\tau}%\in\mathcal{I }
}h_1(\I)\sum_{k\geq0}h_2(k)S(\mathcal{A},\I  ,E(k)\cap C^-(\X^{\tau}))+O\left(\sum_{\substack{\J\in\mathcal{A}\\ N(\J)\in\Upsilon(  \X^{2\tau})
  }}|h(\J)|\right)
\end{equation}
où \begin{align} \label{définition Upsilon}
\Upsilon(\z):=\bigcup_{\substack{p}}\bigcup_{\substack{k\geq2\\p^k> \z}}p^k\Z.
         \end{align}
         
Sous réserve de la validité de (\ref{différence S A B}) et en remarquant que 
 (\ref{estimations distribution tronquée}) implique l'estimation 
\begin{equation}\label{partie tronquée gamma}
\sigma_q(\J)=\sigma_q(\J^-(  \X^{\tau}))\left(1+O\left(\frac{1}{\tau   \X^{\tau}} \right)\right),
 \end{equation} on peut espérer approcher $S(\mathcal{A};h)$ par
\begin{align*}  &\frac{\sigma_q(F)\eta}{c(N_1,N_2)q^3\X}\sum_{P^+(N(\I))\leq   \X^{\tau}
}h_1(\I)\sigma_q(\I)
\sum_{k\geq0}h_2(k)S(\mathcal{B},\I  ,E(k)\cap C^-(\X^{\tau}))\\=  &\frac{\sigma_q(F)\eta}{c(N_1,N_2)q^3\X}\left(S(\mathcal{B};\sigma_q h)\left(1+O\left(\frac{1}{\tau   \X^{\tau}}\right)\right)+O\left(\sum_{\substack{\J\in\mathcal{B}\\ N(\J)\in\Upsilon(  \X^{2\tau})
  }}\sigma_q(\J)|h(\J)|\right)\right)
\end{align*}
où $ S(\mathcal{B};\sigma_q h)$ est défini par (\ref{définition SBh}).

Sous couvert d'une certaine régularité pour $h$ -- par exemple si la série de Dirichlet $\mathcal{H}(s)$ définie par (\ref{définition Hs}) possède un prolongement analytique à gauche de la droite $\Re(s)=1$ -- , on pourra estimer $S(\mathcal{B};\sigma_qh)$ en utilisant des méthodes d'analyse asymptotique (voir les propositions \ref{Selberg-Delange h sigma} et \ref{moyenne friable sigma}).\\

 Dans la suite de ce paragraphe, on établit un raisonnement combinatoire qui conduira au théorème \ref{terme d'erreur combinatoire}. Celui-ci ramène le problème de l'estimation de 
 $S(\mathcal{A};h)$ (resp. $S(\mathcal{B},\sigma_qh)$) à l'estimation  des cardinaux 
$S\left(\mathcal{A},\I,E\cap C(m,n)\right)$ (resp. 
$S\left(\mathcal{B},\I,E\cap C(m,n)\right)$).

Soit $\varpi_1\in]0,\varpi_0[$. On introduit un paramètre $\tau_1\geq (\log_2\X)^{-\varpi_1}$ suffisamment petit qui sera rendu explicite dans les applications du paragraphe \ref{Partie Applications}, en vue de contrôler la contribution des idéaux $\J$ satisfaisant $N(\J^+(\X^{\tau}))\geq\X^{\tau_1}$.

Étant donné un idéal admissible $\J$ de $\OK$, on considère la décomposition de sa partie $  \X^{\tau}$-criblée
en produits d'idéaux premiers
\begin{align}\label{décomposition idéaux premiers}
\J^{+}(  \X^{\tau})=\mathfrak{p}_1\cdots\mathfrak{p}_k\text{ avec } N(\mathfrak{p}_{i})\leq N(\mathfrak{p}_{i+1}).
\end{align}

On effectue une partition de  $\OK$ sous la forme $\OK=\OK^{(1)}\cup\cdots\cup\OK^{(5)}$ où les $\OK^{(i)}$ sont les ensembles d'idéaux
définis 
par des conditions sur la norme de $\mathfrak{p}_k$ et $\mathfrak{p}_{k-1}$, à savoir \begin{align*}
\OK^{(1)}:=\left\{\J:  \X^2<N(\mathfrak{p}_k)\right\},\quad
\OK^{(2)}:=\left\{\J:  \X^{3/2}<N(\mathfrak{p}_k)\leq   \X^2\right\},\end{align*}\begin{align*}
\OK^{(3)}:=\left\{\J:\X<N(\mathfrak{p}_k)\leq   \X^{3/2}\right\},
\qquad \OK^{(4)}:=\left\{\J:N(\mathfrak{p}_k)\leq \X\text{ et }  \X^{3/2}<N(\mathfrak{p}_{k-1}\mathfrak{p}_k)\right\}\end{align*}
\begin{align*}\text{ et }\OK^{(5)}:=\left\{\J:N(\mathfrak{p}_k)\leq \X\text{ et }
N(\mathfrak{p}_{k-1}\mathfrak{p}_{k})\leq   \X^{3/2}\right\}.
   \end{align*}
Il s'ensuit que 
\begin{displaymath}
S(\mathcal{A};h)=\sum_{i=1}^5S(\mathcal{A}\cap\OK^{(i)};h)\qquad\text{ avec }\quad
S(\mathcal{A}\cap\OK^{(i)};h):=\sum_{\J\in\mathcal{A}\cap\OK^{(i)}}h(\J).
\end{displaymath}
Pour $i\in\{2,\ldots,5\}$, on peut remplacer le terme $S(\mathcal{A}\cap\OK^{(i)};h)$ par une formule où interviennent des cardinaux  de la forme $S(\mathcal{A},\I,C^{(i)}(\m,n))$ où $C^{(i)}(m,n)\subset C(m,n)$.
\begin{lemme}\label{lemme identité Type II}
Uniformément en $\X\geq2$ et $(N_1,N_2)\in\mathcal{N}(\eta)$, on a
\begin{multline}
 S(\mathcal{A};h)=S(\mathcal{A}\cap \OK^{(1)};h)+\sum_{i=2}^5\sum_{N(\I)\in\mathcal{I }}
h_1(\I)\sum_{m,n
}h_2(m+n)S(\mathcal{A},\I  ,C^{(i)}(m,n))\\+O\left(\sum_{\substack{\J\in\mathcal{A}\\N(\J)\in\Upsilon(  \X^{2\tau})}}|h(\J)|
+\sum_{\substack{\J\in\mathcal{A}\\N(\J^-(  \X^{\tau}))>  \X^{\tau_1}}}|h(\J)|+\Delta_1\left(\mathcal{A};|h|%
\right)+\Delta_2\left(\mathcal{A};|h|
\right)\right)
\label{identité Type II A}
\end{multline}
où
\begin{align}\label{définition mathcalI}
\mathcal{I }:=\left\{m\leq   \X^{\tau_1}\text{ et } P^+(m)\leq   \X^{\tau}\right\},
\end{align}
\begin{align}\label{définition C2}
 C^{(2)}(m,n
):=\left\{\begin{array}{ll}
\left\{rs\in C(m,n):rs\in E(m+n)\text{ et }P^+(s)\leq P^{-}(r)\right\}&\text{si }m=1,\\
\emptyset&\text{sinon},
\end{array}\right.
\end{align}
\begin{align}\label{définition C3}
 C^{(3)}(m,n
):=\left\{\begin{array}{ll}
\left\{rs\in C(m,n):rs\in E(m+n),
P^+(r)\leq P^{-}(s)
\right\} &\text{si }n=1,\\
\emptyset&\text{sinon},
\end{array}\right.
\end{align}
\begin{align}\label{définition C4}C^{(4)}(m,n):=\left\{\begin{array}{ll}
\left\{rs\in C(m,n):rs\in E(m+n),
P^+(s)\leq P^-(r),P^+(r)\leq   \X^{1-\tau_1}\right\}&\text{si }m=2,\\
\emptyset&\text{sinon},
\end{array}\right.\end{align}
\begin{multline}C^{(5)}(m,n):=
\left\{rs\in C(m,n):\vphantom{\X^{1+\tau}}rs\in E(m+n),P^+(r)\leq P^{-}(s), P^+(s)\leq   \X^{1-\tau_1},\right.\\\left.\frac{s}{P^{-}(s)}<   \X^{1+\tau}\right\} 
 \label{définition C5}\end{multline}
\begin{align*}
\Delta_1\left(\mathcal{A};|h|
\right):= 
\max_{\substack{\Y\geq   \X^{\frac{3}{2}-\tau_1}}}
\sum_{\substack{  \X^{1/2-\tau_1}\leq N(\mathfrak{p}_1)\leq N(\mathfrak{p}_2)\\
\Y\leq N(\mathfrak{p_1}\mathfrak{p}_2)\leq \Y   \X^{\tau_1}}}\sum_{\substack{\J\in\mathcal{A}_{\mathfrak{p}_1\mathfrak{p}_2}}}|h(\J)|\end{align*} et \begin{align*}
\Delta_2\left(\mathcal{A};|h|\right)
:= 
\max_{\Y\geq   \X^{\frac{1}{2}-\tau_1}}\sum_{\Y\leq N(\mathfrak{p})\leq \Y   \X^{\tau_1}}
\sum_{\substack{\J\in\mathcal{A}_{\mathfrak{p}}}}|h(\J)|.\end{align*}
\end{lemme}
\begin{proof}À l'image de (\ref{identité combinatoire générale}), on peut écrire, pour tout $i\in\{2,\ldots,5\}$,
 \begin{align*}
 S(\mathcal{A}\cap\OK^{(i)};h)
=\sum_{P^+(N(\I))\leq   \X^{\tau}
}h_1(\I)\sum_{k\geq0}h_2(k)S(\mathcal{A}\cap\OK^{(i)},\I,E(k))+O\left(
\sum_{\substack{\J\in\mathcal{A}\cap\OK^{(i)}\\ N(\J)\in\Upsilon(  \X^{2\tau})
  }}|h(\J)|\right).
 \end{align*}
 
Soit $\J$ un idéal de $\OK^{(2)}$. En utilisant la décomposition
(\ref{décomposition idéaux premiers}) et sous l'hypothèse supplémentaire que $N(\J^{-}(  \X^{\tau}))\leq   \X^{\tau_1}$, 
l'encadrement 
\begin{align*}  \X^{3/2}<N(\mathfrak{p}_k)\leq   \X^{2}\end{align*} entraîne largement que
\begin{align*}
   \X^{1-2\tau_1}\ll N(\mathfrak{p}_1\cdots\mathfrak{p}_{k-1})\ll   \X^{3/2+\tau_1}.
\end{align*}
En posant $r=N(\mathfrak{p}_k)$ et 
$s=N(\mathfrak{p}_1\cdots \mathfrak{p}_{k-1})$,
on déduit ainsi des  définitions (\ref{définition mathcalI}) et (\ref{définition C2}) que
\begin{multline*}
S(\mathcal{A}\cap\OK^{(2)};h)=\sum_{N(\I)\in\mathcal{I }}h_1(\I)\sum_{n\geq0}h_2(n+1)S(\mathcal{A}\cap\OK^{(2)},\I, C^{(2)}(1,n))
 \\+O\left(
\sum_{\substack{\J\in\mathcal{A}\cap\OK^{(2)}\\ N(\J)\in\Upsilon( \X^{2\tau}) } } |h(\J)|\right.\left.+\sum_{\substack{\J\in\mathcal{A}\cap\OK^{(2)}\\N(\J^-(  \X^{\tau}))>  \X^{\tau_1}}}|h(\J)|+\sum_{\substack{  \X^{3/2}\leq  N(\mathfrak{p}_k)\leq   \X^{3/2+\tau_1}\\
  \text{ou }  \X^{2-2\tau_1}\leq N(\mathfrak{p}_k)\leq   \X^{2}}}
  \sum_{\J\in\mathcal{A}_{\mathfrak{p}_k}}|h(\J)|
\vphantom{\sum_{\substack{\J\in\mathcal{A}\cap\OK^{(2)}\\ N(\J)\in\Upsilon(  \X^{2\tau}) } } |h(\J)|}\right).               
\end{multline*}
De même, on peut écrire, en posant maintenant $s=N(\mathfrak{p}_k)$, la formule
\begin{multline*}
S(\mathcal{A}\cap\OK^{(3)};h)=\sum_{N(\I)\in\mathcal{I }
}h_1(\I)\sum_{m\geq0}h_2(m+1)S(\mathcal{A}\cap\OK^{(3)},\I, C^{(3)}(m,1))\\+O\left(
\sum_{\substack{\J\in\mathcal{A}\cap\OK^{(3)}\\ N(\J)\in\Upsilon(  \X^{2\tau}) } } |h(\J)|\right.
+\sum_{\substack{\J\in\mathcal{A}\cap\OK^{(3)}\\N(\J^-(  \X^{\tau}))>  \X^{\tau_1}}}|h(\J)|
  \left.+\sum_{\substack{X\leq N(\mathfrak{p}_k)\leq   \X^{1+\tau}\\
  \text{ou }  \X^{3/2-\tau}\leq N(\mathfrak{p}_k)\leq   \X^{3/2}}}
  \sum_{\J\in\mathcal{A}_{\mathfrak{p}_k}}|h(\J)|\vphantom{\sum_{\substack{\J\in\mathcal{A}\cap\OK^{(3)}\\ N(\J)\in\Upsilon(  \X^{2\tau}) } } |h(\J)|}
\right)               
\end{multline*}
et,  avec $s=N(\mathfrak{p}_1\cdots \mathfrak{p}{k-2}$, 
\begin{multline*}
S(\mathcal{A}\cap\OK^{(4)};h)=\sum_{N(\I)\in\mathcal{I }
}h_1(\I)\sum_{n\geq0}h_2(n+2)S(\mathcal{A}\cap\OK^{(4)},\I, C^{(4)}(2,n))+O\left(
\sum_{\substack{\J\in\mathcal{A}\cap\OK^{(4)}\\ N(\J)\in\Upsilon(  \X^{2\tau}) } } |h(\J)|\right.\\
+\sum_{\substack{\J\in\mathcal{A}\cap\OK^{(4)}\\N(\J^-(  \X^{\tau}))>  \X^{\tau_1}}}|h(\J)|
  +\sum_{\substack{  \X^{1-\tau_1}\leq N(\mathfrak{p}_k)\leq \X}}
  \sum_{\J\in\mathcal{A}_{\mathfrak{p}_k}}|h(\J)|\\\left.
  +\sum_{\substack{  \X^{3/2}\leq N(\mathfrak{p}_{k-1}\mathfrak{p}_k)\leq   \X^{3/2+\tau_1}}}
  \sum_{\J\in\mathcal{A}_{\mathfrak{p}_{k-1}\mathfrak{p}_k}}|h(\J)|\vphantom{
\sum_{\substack{\J\in\mathcal{A}\cap\OK^{(4)}\\ N(\J)\in\Upsilon(  \X^{2\tau}) } } |h(\J)|}
\right).
\end{multline*}
Considérons enfin $\J\in\OK^{(5)}$. Si  $N(\mathfrak{p}_{k-1}\mathfrak{p}_{k})<  \X^{1+\tau}$ et 
$N(\mathfrak{p}_{k-2})>   \X^{1/2-2\tau}$, on observe les inégalités\begin{align*}
   \X^{1/2-2\tau}<N(\mathfrak{p}_{k-2})\leq N(\mathfrak{p}_{k-1})\leq N(\mathfrak{p}_k)<  \X^{1/2+3\tau}.
\end{align*}
D'autre part, si $N(\mathfrak{p}_{k-1}\mathfrak{p}_{k})<  \X^{1+\tau}$ et 
$N(\mathfrak{p}_{k-2})\leq   \X^{1/2-2\tau}$, alors l'hypothèse $N(\J^-(  \X^{\tau}))\leq   \X^{\tau_1}$
implique l'existence d'un indice  $j\geq2$ tel que
\begin{align*}
 N(\mathfrak{p}_{k-j+1}\cdots \mathfrak{p}_{k})<  \X^{1+\tau}
 \leq N(\mathfrak{p}_{k-j}\cdots \mathfrak{p}_{k})<   \X^{3/2-\tau}.
\end{align*}
Enfin, si $N(\mathfrak{p}_{k-1}\mathfrak{p}_{k})\geq   \X^{1+\tau}$, alors on a trivialement
\begin{align*}
   \X^{1+\tau}\leq N(\mathfrak{p}_{k-1}\mathfrak{p}_{k})<  \X^{3/2}.
\end{align*}
On en déduit que
\begin{multline*}
 S(\mathcal{A}\cap\OK^{(5)};h)=\sum_{N(\I)\in\mathcal{I }}h_1(\I)\sum_{m,n\geq0}h_2(n+m)S(\mathcal{A}\cap\OK^{(5)},\I, C^{(5)}(m,n))\\+O\left(
\sum_{\substack{\J\in\mathcal{A}\cap\OK^{(5)}\\ N(\J)\in\Upsilon(  \X^{2\tau}) } } |h(\J)|
+\sum_{\substack{\J\in\mathcal{A}\cap\OK^{(5)}\\N(\J^-(  \X^{\tau}))>  \X^{\tau_1}}}|h(\J)|\right.\\
   +\sum_{\substack{  \X^{1-\tau_1}\leq N(\mathfrak{p}_k)\leq \X\\
   \text{ou }  \X^{1/2-2\tau}\leq N(\mathfrak{p}_k)\leq   \X^{1/2+3\tau}}}
  \sum_{\J\in\mathcal{A}_{\mathfrak{p}_k}}|h(\J)|
  \\\left.+\sum_{\substack{  \X^{3/2-\tau}\leq N(\mathfrak{p}_{k-1}\mathfrak{p}_k)\leq   \X^{3/2}}}
  \sum_{\J\in\mathcal{A}_{\mathfrak{p}_{k-1}\mathfrak{p}_k}}|h(\J)|\vphantom{
\sum_{\substack{\J\in\mathcal{A}\cap\OK^{(5)}\\ N(\J)\in\Upsilon(  \X^{2\tau}) } } |h(\J)|}
\right)
\end{multline*}
ce qui achève la preuve.\end{proof}
Le traitement de $S(\mathcal{A}\cap\OK^{(1)};h)$ 
nécessite davantage de travail pour faire apparaître des parties de $C(m,n)$. 
 Au vu du lemme \ref{non associes}, on a
\begin{align*}
N(\mathfrak{p}_k)=
 \frac{c(N_1,N_2)q^3  \X^3}{N(\J^{-}(  \X^{\tau})\mathfrak{p}_1\cdots \mathfrak{p}_{k-1})}
 \left(1+O\left(\eta^{1/4}\right)\right)
 \end{align*}
 ce qui suggère d'associer à chaque ensemble $E_{i}(k)$ défini par (\ref{définition Eij}) 
 les ensembles $\widetilde{E}_{i}(k,\I)$
  des entiers $m$  satisfaisant $\Omega(m)=\omega(m)=k-1$ et 
\begin{align*}\left\{\begin{array}{ll}
 \frac{c(N_1,N_2)q^3  \X^3\Y}{N(\I)}
 P^{(\alpha_1)}(m)\cdots  P^{(\alpha_{j-1})}(m)
 \leq  mP^{(\overrightarrow{\beta})}(m) 
&\text{ si }\alpha_j=k,\\
mP^{(\overrightarrow{\alpha})}
(m)\leq  P^{(\beta_1)}(m)\cdots  P^{(\beta_{l-1})}(m)\frac{c(N_1,N_2)q^3  \X^3\Y}{N(\I)}
&\text{ si }\beta_l=k,\\
\Y P^{(\overrightarrow{\alpha})}%
(m)\leq  P^{(\overrightarrow{\beta})}(m)
&\text{ sinon}.
\end{array}\right.
\end{align*}
On définit encore 
 \begin{align*}
  \widetilde{E}(k,\I)=\bigcap_{i=1}^{\nbin}\widetilde{E}_{i}(k,\I).
 \end{align*}
Dans la mesure où les éléments de $\mathcal{A}\cap\OK^{(1)}$ satisfont
$N(\mathfrak{p}_k)\geq   \X^2$,  on peut écrire, en posant $\Qid=\mathfrak{p}_1\cdots \mathfrak{p}_{k-1}$,
 \begin{align*}
 S(\mathcal{A}\cap\OK^{(1)};h)=&
  \sum_{\substack{N(\I)\in\mathcal{I }
}}h_1(\I)\sum_{k}h_2(k)\sum_{\substack{N(\Qid)\leq   \X^{1-4\tau_1}\\P^-(N(\Qid))>   \X^{\tau}
}}1_{\widetilde{E}(k,\I)}(\Qid)
\pi(\mathcal{A},\I\Qid)\\
&+O\left(
\sum_{\substack{\J\in\mathcal{A}\cap\OK^{(1)}\\ N(\J)\in\Upsilon(  \X^{2\tau}) } } |h(\J)|
+\sum_{\substack{\J\in\mathcal{A}\cap\OK^{(1)}\\N(\J^-(  \X^{\tau}))>  \X^{\tau_1}}}|h(\J)|
  +\Delta_2(\mathcal{A};|h|)
 +\tau^{-1}\Delta_3(\mathcal{A};|h|)
\right) \end{align*}
 où 
 \begin{align*}
  \pi(\mathcal{A},\I):=\#\left\{\mathfrak{p}\text{ premier}:\I\mathfrak{p}\in\mathcal{A}
  \right\}
 \end{align*}
 et
 \begin{align*}
 \Delta_3(\mathcal{A};|h|)
:=\sum_{N(\I)\in\mathcal{I }}\sum_k
  \sum_{  \X^{\tau}<N(\mathfrak{p}_1)<\cdots<N(\mathfrak{p}_{k-1})}
  \max_{\Y\geq   \X^2}
  \sum_{\substack{\Y\leq N(\mathfrak{p}_k)\leq \Y (1+O(\eta^{1/4}))\\\I\mathfrak{p}_1\cdots \mathfrak{p}_k\in\mathcal{A}}}
|h(\I\mathfrak{p}_1\cdots \mathfrak{p}_k)|.
 \end{align*}
 Pour estimer $\pi(\mathcal{A},\I)$, 
on reproduit le raisonnement combinatoire à l'origine du lemme 4.1 de \cite{HM02}. 
En utilisant la notation (\ref{définition C crible}), l'identité de Buchstab et le lemme \ref{admissible} assurent que, 
si \begin{align*}\max\left(\{1\}\cup\{p^3: p\text{ premier } 1\text{-singulier}\}\right)<z_1\leq z_2,\end{align*} alors on a
\begin{multline}
S(\mathcal{A},\I  ,C^-(z_1))-S(\mathcal{A},\I  ,C^-(z_2))=\sum_{z_1<N(\mathfrak{p})\leq z_2}
S(\mathcal{A},\I  \mathfrak{p},C^-(N(\mathfrak{p}))\\+O\left(
\sum_{z_1<N(\mathfrak{p})\leq z_2}\sum_{k\geq2}S(\mathcal{A},\I\mathfrak{p}^k, C^-(N(\mathfrak{p})))\right).\label{Buchstab A}
\end{multline}
Ceci permet d'établir par récurrence la formule combinatoire suivante, analogue des formules (4.3) et (4.4) de \cite{HM02},  formule
\begin{multline*}
\pi(\mathcal{A},\I)=S(\mathcal{A},\I,C^-(\X^{3/2+\tau}))
=\sum_{n\geq0}(-1)^n
T^{(n)}\left(\mathcal{A},\I\right)
+\sum_{n\geq1}(-1)^{n}U^{(n)}(\mathcal{A},\I)\\-S_1(\mathcal{A},\I)  -S_2(\mathcal{A},\I)
+O\left(\#\left(\mathcal{A}_{\I}\cap\Upsilon(  \X^{2\tau})\right)\right) 
\end{multline*} où 
\begin{align*}S_1(\mathcal{A},\I):=
\sum_{  \X^{1-\tau}\leq N(\mathfrak{p})<  \X^{1+\tau}}S
(\mathcal{A},\I\mathfrak{p},C^-(N(\mathfrak{p}))),\end{align*}\begin{align*}
S_2(\mathcal{A},\I):=\sum_{  \X^{3/2-\tau}< N(\mathfrak{p})\leq   \X^{3/2+\tau}}
S(\mathcal{A},\I  \mathfrak{p},C^-(N(\mathfrak{p}))),\end{align*}
\begin{align}\label{définition Tn}
T^{(0)}(\mathcal{A},\I):= S(\mathcal{A},\I  ,C^-(  \X^{\tau})),\quad
U^{(1)}(\mathcal{A},\I):=\sum_{  \X^{1+\tau}\leq N(\mathfrak{p})\leq   \X^{\frac{3}{2}-\tau}}S
(\mathcal{A},\I\mathfrak{p},C^-(N(\mathfrak{p}))),
\end{align}
\begin{align*}
T^{(n)}(\mathcal{A},\I)
:=\sum_{\substack{
  \X^{\tau}< N(\mathfrak{p}_1)<\cdots<N(\mathfrak{p}_{n})<  \X^{1-\tau}\\
 N(\mathfrak{p}_1\cdots \mathfrak{p}_n)<  \X^{1+\tau}}}
 S(\mathcal{A},\I  \mathfrak{p}_1\cdots \mathfrak{p}_n,C^-(  \X^{\tau}))\quad\text{ si }n\geq1
\end{align*}et
\begin{align*}
U^{(n)}(\mathcal{A},\I):=\sum_{\substack{
  \X^{\tau}< N(\mathfrak{p}_1)<\cdots<N(\mathfrak{p}_{n})<  \X^{1-\tau}\\
N(\mathfrak{p}_2\cdots \mathfrak{p}_{n})<   \X^{1+\tau}\leq
N(\mathfrak{p}_1\cdots \mathfrak{p}_{n})}}S(\mathcal{A},\I  \mathfrak{p}_1\cdots \mathfrak{p}_n,C^-(
N(\mathfrak{p}_{1})))\quad\text{ si }n\geq2.
\end{align*}

Le niveau de crible $  \X^{\tau}$ impliqué dans la définition de $T^{(n)}(\mathcal{A},\I)$ étant suffisamment petit,
on établira au lemme \ref{crible S}, en utilisant le crible de Selberg, la formule asymptotique 
(\ref{asymptotique crible A})
en moyenne sur les idéaux $\I$ tels que $N(\I)\leq   \X^{1-3\tau_1}$.

Si $n\geq4$, les conditions \begin{align*}
  \X^{\tau}<N(\mathfrak{p}_1)<\cdots<N(\mathfrak{p}_n)<  \X^{1-\tau}\end{align*} et \begin{align*} 
N(\mathfrak{p}_2\cdots  \mathfrak{p}_{n})\leq   \X^{1+\tau}<N(\mathfrak{p}_1\cdots  \mathfrak{p}_{n})
\end{align*}
entraînent que
\begin{align*}
N(\mathfrak{p}_1\cdots \mathfrak{p}_{n})\leq   \X^{\frac{3}{2}-\tau}.\end{align*}
Par suite,  on peut espérer estimer la contribution de $U^{(n)}(\mathcal{A},\I\Qid)$  lorsque $n\geq4$ ou $n=1$ avec $N(\I)\in\mathcal{I }$,
 en utilisant la formule
(\ref{asymptotique crible II}) puisque l'on peut écrire, en posant $m_1=k$,\begin{align*}\sum_{k}h_2(k)\sum_{\substack{N(\Qid)\leq   \X^{1-4\tau_1}
}}1_{\widetilde{E}(k,\I)}(\Qid) U^{(n)}(\mathcal{A},\I\Qid)=
\sum_{\m}h_2(m_1)
S(\mathcal{A},\I, C^{(1)}(\m,n,\I))
\end{align*}
où $\m:=(m_1,m_2)$,
\begin{multline*}
 C(\m,n):=\left\{r_1r_2s:\omega(r_i)=\Omega(r_i)=m_i\text{ pour }i=1,2,
 \omega(s)=\Omega(s)=n\vphantom{ \X^{1+\tau}}\right.,\\\left. P^-(r_1r_2s)>  \X^{\tau},   \X^{1+\tau}\leq s\leq   \X^{3/2-\tau}\right\},
\end{multline*} 
\begin{align}\label{définition C11}
 C^{(1)}(\m,1,\I):=\left\{r_1r_2s\in C((m_1-1,m_2),1): 
 P^-(r_2)> s, 
 r_1\leq   \X^{1-4\tau_1}\text{ et }
r_1\in \widetilde{E}(m_1,\I)\right\}
\end{align}
 et, pour $n\geq 2$,
\begin{multline}
 C^{(1)}(\m,n,\I):=
 \left\{r_1r_2s\in C((m_1-1,m_2),n): \vphantom{\frac{s}{P^-(s)}}
 P^-(r_2)> P^{-}(s), P^+(s)<  \X^{1-\tau},\right.\\\left.\frac{s}{P^-(s)}<  \X^{1+\tau}, r_1\leq   \X^{1-4\tau_1}\text{ et }
r_1\in \widetilde{E}(m_1,\I)\right\},\label{définition C1}
\end{multline}
où les triplets $(r_1,r_2,s)$ sont comptés avec multiplicité, en posant $r_1=N(\Qid)$ et $s=N(\mathfrak{p}_1\cdots \mathfrak{p}_n)$ avec les notations de la définition de $U^{(n)}(\mathcal{A},\I)$.

Si $n=2\text{ ou }3$, l'inégalité
\begin{align*}
 N(\mathfrak{p}_1\cdots \mathfrak{p}_n)\leq   \X^{3/2-\tau}
\end{align*}
peut faire défaut. On s'inspire alors des identités (4.5) et (4.6) de \cite{HM02} en écrivant
\begin{align*}
U^{(2)}(\mathcal{A},\I):= U^{(2,1)}(\mathcal{A},\I)+U^{(2,2)}(\mathcal{A},\I)\end{align*} et \begin{align*}
U^{(3)}(\mathcal{A},\I):= U^{(3,1)}(\mathcal{A},\I)+S_3(\mathcal{A},\I)
\end{align*}
où 
\begin{align*}
U^{(2,1)}(\mathcal{A},\I):=\sum_{\substack{  \X^{\tau }< N(\mathfrak{p}_1)<N(\mathfrak{p}_2)<   \X^{1-\tau}\\
  \X^{1+\tau}\leq N(\mathfrak{p}_1\mathfrak{p}_2)\leq   \X^{\frac{3}{2}-\tau}}}
S(\mathcal{A},\I  \mathfrak{p}_1\mathfrak{p}_2,C^-(N(\mathfrak{p}_1)))
,\end{align*}
\begin{align*}
U^{(2,2)}(\mathcal{A},\I):=
\sum_{\substack{  \X^{\tau}< N(\mathfrak{p}_1)<N(\mathfrak{p}_2)<  \X^{1-\tau}\\   \X^{\frac{3}{2}-\tau}< 
N(\mathfrak{p}_1\mathfrak{p}_2)}}S
(\mathcal{A},\I\mathfrak{p}_1\mathfrak{p}_2,C^-(N(\mathfrak{p}_{1})))
,\end{align*}\begin{align*}
U^{(3,1)}(\mathcal{A},\I):
=\sum_{\substack{  \X^{\tau}< N(\mathfrak{p}_1)<N(\mathfrak{p}_2)<N(\mathfrak{p}_3)<  \X^{1-\tau}
\\
N(\mathfrak{p}_2\mathfrak{p}_3)<   \X^{1+\tau}\leq
N(\mathfrak{p}_1\mathfrak{p}_2\mathfrak{p}_3)\leq   \X^{\frac{3}{2}-\tau}}}S
(\mathcal{A},\I\mathfrak{p}_1\mathfrak{p}_2\mathfrak{p}_3,C^-(N(\mathfrak{p}_1)))
\end{align*}  
et \begin{align*}
S_3(\mathcal{A},\I):=\sum_{\substack{  \X^{\tau}< N(\mathfrak{p}_1)<N(\mathfrak{p}_2)<N(\mathfrak{p}_3)<  \X^{1-\tau}\\
N(\mathfrak{p}_2\mathfrak{p}_3)<  \X^{1+\tau}\\ N(\mathfrak{p}_1\mathfrak{p}_2\mathfrak{p}_3)>  \X^{\frac{3}{2}-\tau} }}
S(\mathcal{A},\I  \mathfrak{p}_1\mathfrak{p}_2\mathfrak{p}_3,C^-(N(\mathfrak{p}_{1}))).
\end{align*}

On peut observer que
\begin{align*}
\sum_{k}h_2(k)\sum_{\substack{N(\Qid)\leq   \X^{1-4\tau_1}\\
P^-(N(\Qid))>  \X^{\tau}}}1_{\widetilde{E}(k,\I)}(\Qid) U^{(2,1)}(\mathcal{A},\I\Qid)=\sum_{\m}h_2(m_1)
S(\mathcal{A},\I , C^{(1)}(\m,2,\I))
\end{align*}
et
\begin{align*}
\sum_{k}h_2(k)\sum_{\substack{N(\Qid)\leq   \X^{1-4\tau_1}\\
P^-(N(\Qid))>  \X^{\tau}}}1_{\widetilde{E}(k,\I )}(\Qid) U^{(3,1)}(\mathcal{A},\I\Qid)=\sum_{\m}h_2(m_1)
S(\mathcal{A},\I , C^{(1)}(\m,3,\I ))
\end{align*}
où $C^{(1)}(\m,2,\I )$ et $C^{(1)}(\m,3,\I )$ sont définis par (\ref{définition C1}).

De plus, on remarque que, si un idéal $\J$  intervient dans $U^{(2,2)}(\mathcal{A},\I\Qid)$, alors $\J$ est de la forme
 $\J=\I\Qid\Kid\mathfrak{p}_1\mathfrak{p}_2$ avec 
 \begin{align*}  \X^{3/2-\tau}< 
 N(\mathfrak{p}_1\mathfrak{p}_2)\leq   \X^{2-2\tau}\text{ et }P^-(N(\Kid))>N(\mathfrak{p}_1).\end{align*}
 Il s'ensuit que
 \begin{align*}
    \X^{1-\tau_1}\leq N(\Qid\Kid)\leq   \X^{3/2+2\tau}
 \end{align*}
ce qui suggère d'écrire, en posant $r=N(\mathfrak{p_1}\mathfrak{p}_2)$, $s_1=N(\Qid)$, $s_2=N(\Kid)$ et $n_1=k$,
\begin{multline*}
\sum_{k}h_2(k)\sum_{\substack{N(\Qid)\leq   \X^{1-4\tau_1}\\
P^-(N(\Qid))>  \X^{\tau}}}1_{\widetilde{E}(k,\I  )}(\Qid) 
U^{(2,2)}(\mathcal{A},\I\Qid)
=\sum_{\n}h_2(n_1)S(\mathcal{A},\I  ,C^{(1)}(2,\n,\I  ))\\
+O\left(
\sum_k|h_2(k)|
\sum_{\substack{N(\Qid)\leq   \X^{1-4\tau_1}\\P^-(N(\Qid))>  \X^{\tau}
}}1_{\widetilde{E}(k,\I)}(\Qid)
\sum_{\substack{  \X^{1/2}<N(\mathfrak{p}_1)<N(\mathfrak{p}_2)\\  \X^{3/2-\tau}\leq N(\mathfrak{p}_1\mathfrak{p}_2)\leq   \X^{3/2+\tau_1}\\
\text{ou }  \X^{2-2\tau_1}\leq N(\mathfrak{p}_1\mathfrak{p}_2)\leq   \X^{2-2\tau}}}
S(\mathcal{A}, \I\Qid\mathfrak{p}_1\mathfrak{p}_2, C^-(N(\mathfrak{p}_1)))
\right)
\end{multline*}
où 
\begin{multline*}
C(m,\n):=\left\{rs_1s_2:\omega(r)=\Omega(r)=m, \omega(s_i)=\Omega(s_i)=n_i\vphantom{\X^{3/2-\tau}}\text{ pour }i=1,2,
 \right.\\\left. P^-(rs_1s_2)>  \X^{\tau},   \X^{1+\tau}\leq s_1s_2\leq   \X^{3/2-\tau}\right\}.
\end{multline*} et
\begin{align}\label{définition C22}
C^{(1)}(2,\n,\I):=\left\{rs_1s_2\in C(2,(n_1-1,n_2)), 
s_1\in \widetilde{E}(n_1,\I), P^+(r)<  \X^{1-\tau}, P^-(r)< P^-(s_2)\right\}\end{align}
où les triplets $(r,s_1,s_2)$.

On observe que l'on peut majorer la somme sur $\I$ du terme d'erreur précédent par 
\begin{multline*}
\widetilde{\Delta}_{1}(\mathcal{A};|h|)
 :=
\max_{\substack{\Y\geq   \X^{\frac{3}{2}-\tau_1}}}
\sum_{\substack{  \X^{1/2-\tau_1}\leq N(\mathfrak{p}_1)\leq N(\mathfrak{p}_2)\\
\Y\leq N(\mathfrak{p_1}\mathfrak{p}_2)\leq \Y   \X^{\tau_1}}}
\sum_{\substack{N(\I)\in\mathcal{I }
}}|h_1(\I)|\sum_k|h_2(k)|\\\times
\sum_{\substack{N(\Qid)\leq   \X^{1-4\tau_1}\\P^-(N(\Qid))>  \X^{\tau}
}}1_{\widetilde{E}(k,\I)}(\Qid)
S(\mathcal{A}, \I\Qid\mathfrak{p}_1\mathfrak{p}_2, C^-(N(\mathfrak{p}_1))).
\end{multline*}
De même, en remarquant que les conditions de sommations définissant $S_3(\mathcal{A},\I)$ impliquent  
\begin{align*}  \X^{1/2-2\tau}\leq N(\mathfrak{p}_1)\leq   \X^{1/2+\tau/2},\end{align*}
on a, pour $i\in\{1,2,3\}$ 
\begin{align*}
\sum_{N(\I)\in\mathcal{I }}h_1(\I)\sum_kh_2(k)\sum_{\substack{N(\Qid)\leq   \X^{1-4\tau_1}\\
P^-(N(\I\Qid))>  \X^{\tau}}}1_{\widetilde{E}(k,\I)}(\Qid)S_i(\mathcal{A},\I\Qid)&\ll\widetilde{\Delta}_2(\mathcal{A};|h|).
\end{align*}
où
\begin{multline*}
 \widetilde{\Delta}_2(\mathcal{A};|h|
):=\max_{\Y\geq   \X^{\frac{1}{2}-\tau_1}}\sum_{\Y\leq N(\mathfrak{p})\leq \Y   \X^{\tau_1}}
\sum_{\substack{N(\I)\in\mathcal{I }
}}|h_1(\I)|\sigma_q(\I)\sum_k|h_2(k)|\\\times\sum_{\substack{N(\Qid)\leq   \X^{1-4\tau_1}\\P^-(N(\Qid))>  \X^{\tau}
}}1_{\widetilde{E}(k,\I)}(\Qid)S(\mathcal{A},\I\Qid\mathfrak{p},C^-(N(\mathfrak{p}))).
\end{multline*}
On déduit finalement de ce qui précède la formule
\begin{align}
\nonumber S(\mathcal{A}\cap\OK^{(1)};h)=&\sum_{N(\I)\in\mathcal{I}}h_1(\I)
\left(\sum_{k}h_2(k)
\sum_{\substack{N(\Qid)\leq \X^{1-4\tau_1}\\P^-(N(\Qid))>\X^{\tau}}}1_{\widetilde{E}(k,\I)}(\Qid)\sum_{n\geq0}(-1)^n
T^{(n)}\left(\mathcal{A},\I\Qid\right)
\right.\\
&\left.\vphantom{\sum_{\substack{N(\Qid)\leq \X^{1-4\tau_1}\\P^-(N(\Qid))>\X^{\tau}}}}\nonumber
+\sum_{n\geq1}(-1)^{n}\sum_{\m}h_2(m_1)S(\mathcal{A},\I,C^{(1)}(\m,n,\I))+\sum_{\n}h_2(n_1)S(\mathcal{A},\I,C^{(1)}(2,\n,\I))
\right)\\&\nonumber
+O\left(
\vphantom{\sum_{\substack{\J\in\mathcal{A}\\N(\J^-(  \X^{\tau}))>  \X^{\tau_1}}}}\Delta_2(\mathcal{A};|h|)+\tau^{-1}\Delta_3(\mathcal{A};|h|)+ \widetilde{\Delta}_1(\mathcal{A};|h|)+\widetilde{\Delta}_2(\mathcal{A};|h|)
\right.\\&
\left.+\sum_{\substack{\J\in\mathcal{A}\cap\Upsilon(  \X^{2\tau})}}|h(\J)|
+\sum_{\substack{\J\in\mathcal{A}\\N(\J^-(  \X^{\tau}))>  \X^{\tau_1}}}|h(\J)|
\right).
\label{combinatoire SAHOK1}\end{align}

L'argument développé ci-dessus s'adapte pour réécrire $S(\mathcal{B};\sigma_qh)$ à l'aide de
 cardinaux de la forme $S\left(\mathcal{B},\I,E\cap C(m,n)\right)$. En l'absence d'un analogue du lemme \ref{admissible} adapté aux idéaux de $\mathcal{B}$, il convient
toutefois d'user de précaution en considérant également la contribution des idéaux non admissibles dans l'utilisation de (\ref{Buchstab A}) au cours du traitement de $S(\mathcal{B}\cap\OK^{(1)};\sigma_qh)$. Cette nouveauté dans la transcription à 
$\mathcal{B}$ de
 l'argument combinatoire précédent 
 était déjà implicitement contenue dans les travaux de Heath-Brown et Moroz, à l'image de
 la formule (6.4) de \cite{HB01} ou encore lors de l'introduction de l'ordre total sur les idéaux premiers de $\OK$ dans le paragraphe
 4 de \cite{HM02}.  
 
 Avant d'énoncer la proposition qui résume la discussion de cette partie, 
 on harmonise les notations précédentes en introduisant, pour $\m:=(m_1,m_2)$, $\n:=(n_1,n_2)$ et $\I$ un idéal, \begin{align}\label{définition CImn}
C^{(i)}(\m,\n,\I):=\left\{\begin{array}{ll}
\left\{(r,1,s,1):rs\in C^{(i)}(m_1,n_1)\right\}&\text{si }2\leq i\leq 5
\text{ et }m_2=n_2=0,\\
\left\{(r_1,r_2,s,1):r_1r_2s\in C^{(1)}((m_1,m_2),n_1,\I)\right\}&\text{si }i=1\text{ et }n_2=0,\\
\left\{(r,1,s_1,s_2):rs_1s_2\in C^{(1)}((m_1,(n_1,n_2),\I)\right\}&\text{si }i=1\text{ et }m_2=0,\\  
\emptyset&\text{sinon},
                     \end{array}
 \right.
\end{align}
\begin{multline*}
S\left(\mathcal{D},\I,C^{(i)}(\m,\n,\I)\right)
:=\#\left\{(\Rid_1,\Rid_2,\Sid_1,\Sid_2):(N(\Rid_1),N(\Rid_2),N(\Sid_1),N(\Sid_2))\in C^{(i)}(\m,\n,\I)\right.\\\left. \text{ et }\vphantom{C^{(i)}}\I\Rid_1\Rid_2\Sid_1\Sid_2\in\mathcal{D}\right\}
\end{multline*}
et \begin{align*}
h_2(\m,\n):=\left\{\begin{array}{ll}
h_2(m+n)&\text{ si }i\in\{2,3,4,5\},\\
h_2(m_1)&\text{ si }i=1\text{ et }n_2=0,\\
h_2(n_1)&\text{ si }i=1\text{ et }m_2=0.
                     \end{array}
 \right.
\end{align*}

\begin{theo}\label{terme d'erreur combinatoire}
On a, uniformément en $\X\geq2$ et $(N_1,N_2)\in\mathcal{N}(\eta)$, 
 \begin{multline*}
S(\mathcal{A};h)-\frac{\eta\sigma_q(F)}{c(N_1,N_2)q^3\X}S(\mathcal{B};\sigma_qh)
\ll T(|h|)+ S(|h|)+\Theta(|h|,  \X^{\tau},  \X^{\tau_1})+ \Delta_0\left(  \X^{2\tau/3}\right)\\
%\Upsilon(
%\left(\#\mathcal{A}\cap\Upsilon(  \X^{\tau})+\frac{\eta\sigma_q(F)}{c(N_1,N_2)q^3\X}
%\sum_{\substack{\J\in\mathcal{B}\\N(\J)\in\Upsilon (  \X^{2\tau/3})}}\sigma_q(\J)\right)
+\Delta_1(|h|)+\Delta_2(|h|)+\tau^{-1}\Delta_3(|h|)+ \widetilde{\Delta}_1(|h|%,  \X^{\tau_1}
)+ \widetilde{\Delta}_2(|h|%,  \X^{\tau_1}
)
\end{multline*}
où 
\begin{align*}%\label{définition Th}
 T(|h|):=\sum_{N(\I)\in\mathcal{I }}|h_1(\I)|%\sum_{m}|h_2(m)|
 \sum_{\substack{
N(\Qid)\leq   \X^{1-4\tau_1}\\P^-(N(\Qid))>  \X^{\tau}}%I\in 
}%1_{E'(m,\I  _1)}(\I_2)%\omega(\I)=n_1}}
%|h^-(\I^-(  \X^{\tau}))
%|h_1(n_1)%1_{\widetilde{E}^{(n_1+1)}}
\sum_{n\geq0}\left|
T^{(n)}(\mathcal{A},\I\Qid)%S^{ K}(\mathcal{A},IQ,  \X^{\tau})
-\frac{\eta\sigma_q(F)}{c(N_1,N_2)q^3\X}
\sigma_q(\I)T^{(n)}(\mathcal{B},\I\Qid)%S^{ K}(\mathcal{B},IQ,   \X^{\tau})
\right|,
\end{align*}
$T^{(n)}(\mathcal{A},\I)$ et $T^{(n)}(\mathcal{B},\I)$ étant définis par (\ref{définition Tn}),
\begin{multline}
S(|h|):=\sum_{N(\I)\in\mathcal{I }}|h_1(\I)|\sum_{i=1}^5\sum_{m,n}|h_2(\m,\n)|
\left|S(\mathcal{A},\I  ,C^{(i)}(\m,\n,\I))\right.\\\left.-\frac{\eta\sigma_q(F)}{c(N_1,N_2)q^3\X}
\sigma_q(\I)S(\mathcal{B},\I  ,C^{(i)}(\m,\n,\I))\right|,\label{définition Sh}
\end{multline}
les $C^{(i)}(\m,\n,\I)$ étant définis par (\ref{définition CImn}) et (\ref{définition C2}), (\ref{définition C3}), (\ref{définition C4}), (\ref{définition C11}), (\ref{définition C1}) et (\ref{définition C22}),
\begin{align}\label{définition Theta}
\Theta(|h|,\Y,\z):=\sum_{\substack{\J\in\mathcal{A}\\N(\J^-(\Y))>\z}}|h(\J)|+
\frac{\eta\sigma_q(F)}{c(N_1,N_2)q^3\X}\sum_{\substack{\J\in\mathcal{B}\\
N(\J^-(\Y))>\z}}\sigma_q(\J)|h(\J)|,
\end{align}
\begin{multline}\label{définition Delta0}
 \Delta_0(|h|,\Y):=
\sum_{\substack{\J\in\mathcal{A}\\N(\J)\in\Upsilon (\Y)}}|h(\J)|+\frac{\eta\sigma_q(F)}{c(N_1,N_2)q^3\X}
\left(\sum_{\substack{\J\in\mathcal{B}\\N(\J)\in\Upsilon (\Y)}}\sigma_q(\J)|h(\J)|\right.
+\sum_{N(\I)\in\mathcal{I}}|h_1(\I)|\sigma_q(\I)\\\left.\times\sum_k|h_2(k)|\sum_{\substack{N(\Qid)\leq\X^{1-4\tau_1}\\P^-(N(\Qid))>\X^{\tau}}}1_{\widetilde{E}(k,\I)}(\Qid)
\#\left\{\Kid:\I\Kid\Qid\in\mathcal{B}, N(\Kid)\in\Upsilon(\Y)\right\}\right)
,
\end{multline}
\begin{align}\label{définition Delta1}
\Delta_1\left(|h|%,  \X^{\tau_1}
\right):= 
\max_{\substack{\Y\geq   \X^{\frac{3}{2}-\tau_1}}}
\sum_{\substack{  \X^{1/2-\tau_1}\leq N(\mathfrak{p}_1), %\leq
 N(\mathfrak{p}_2)\\
\Y\leq N(\mathfrak{p_1}\mathfrak{p}_2)\leq \Y   \X^{\tau_1}}}\left(\sum_{\substack{\J\in\mathcal{A}_{\mathfrak{p}_1\mathfrak{p}_2}}}|h(\J)|
+\frac{\eta\sigma_q(F)}{c(N_1,N_2)q^3\X}\sum_{\substack{\J\in\mathcal{B}_{\mathfrak{p}_1\mathfrak{p}_2}}}|h(\J)|\sigma_q(\J)\right),
\end{align}\begin{align}\label{définition Delta2}
\Delta_2(|h|):=
\max_{\Y\geq   \X^{\frac{1}{2}-\tau_1}}\sum_{\Y\leq N(\mathfrak{p})\leq \Y   \X^{\tau_1}}\left(
\sum_{\substack{\J\in\mathcal{A}_{\mathfrak{p}}}}|h(\J)|
+\frac{\eta\sigma_q(F)}{c(N_1,N_2)q^3\X}\sum_{\substack{\J\in\mathcal{B}_{\mathfrak{p}}}}|h(\J)|\sigma_q(\J)\right),
\end{align}\begin{multline}\label{définition Delta3}
\Delta_3\left(|h|%,  \X^{\tau_1}
\right):= \sum_{N(\I)\in\mathcal{I }}\sum_k
  \sum_{  \X^{\tau}<N(\mathfrak{p}_1)<\cdots<N(\mathfrak{p}_{k-1})} \max_{\Y\geq   \X^2}\left(
  \sum_{\substack{\Y\leq N(\mathfrak{p}_k)\leq \Y (\log \X)^c\\\I\mathfrak{p}_1\cdots \mathfrak{p}_k\in\mathcal{A}}}
|h(\I\mathfrak{p}_1\cdots \mathfrak{p}_k)|\right.\\
    \left.
 + \frac{\eta\sigma_q(F)}{c(N_1,N_2)q^3\X}\sigma_q(\I)\sum_{\substack{\Y\leq N(\mathfrak{p}_k)\leq \Y (\log \X)^c\\\I\mathfrak{p}_1\cdots \mathfrak{p}_k\in\mathcal{B}}}
|h(\I\mathfrak{p}_1\cdots \mathfrak{p}_k)|\right),
\end{multline}
\begin{multline}\label{définition Deltatilde1}
 \widetilde{\Delta}_1(|h|)
 :=
\max_{\substack{\Y\geq   \X^{\frac{3}{2}-\tau_1}}}
\sum_{\substack{  \X^{1/2-\tau_1}\leq N(\mathfrak{p}_1)\leq N(\mathfrak{p}_2)\\
\Y\leq N(\mathfrak{p_1}\mathfrak{p}_2)\leq \Y   \X^{\tau_1}}}
\sum_{\substack{N(\I)\in\mathcal{I }%N(\I^-)\leq   \X^{\tau_1}\\P^+(N(\I^-))\leq   \X^{\tau}
}}|h_1(\I)|\sum_k|h_2(k)|
\sum_{\substack{N(\I)\leq   \X^{1-4\tau_1}\\P^-(N(\Qid))>  \X^{\tau}%\\\omega(\I)=n
}}1_{\widetilde{E}(k,\I)}(\Qid)\\\times\left(
S(\mathcal{A}, \I\Qid\mathfrak{p}_1\mathfrak{p}_2, C^-(N(\mathfrak{p}_1)))+
\sigma_q(\I)\frac{\eta\sigma_q(F)}{c(N_1,N_2)q^3\X}S(\mathcal{B},\I  \Qid\mathfrak{p}_1\mathfrak{p}_2,C^-(N(\mathfrak{p}_1))\right)
\end{multline} et
\begin{multline}\label{définition Deltatilde2}
 \widetilde{\Delta}_2(|h|%,  \X^{\tau_1}
):=\max_{\Y\geq   \X^{\frac{1}{2}-\tau_1}}\sum_{\Y\leq N(\mathfrak{p})\leq \Y   \X^{\tau_1}}
\sum_{\substack{N(\I)\in\mathcal{I }%\leq   \X^{\tau_1}\\P^+(N(\I^-))\leq   \X^{\tau}
}}|h_1(\I)|\sum_k|h_2(k)|\sum_{\substack{N(\Qid)\leq   \X^{1-4\tau_1}\\P^-(N(\I_2))>  \X^{\tau}%\\\omega(\I)=n
}}1_{\widetilde{E}(k,\I)}(\Qid)\\\times\left(S(\mathcal{A},\I  \Qid\mathfrak{p},C^-(N(\mathfrak{p})))+\sigma_q(\I)
\frac{\eta\sigma_q(F)}{c(N_1,N_2)q^3\X}S(\mathcal{B},\I  \Qid\mathfrak{p},C^-(N(\mathfrak{p})))\right). 
\end{multline}
\end{theo}

On suppose dans toute la suite que $h$ est à valeurs dans le disque unité. Cette condition est réalisée par les fonctions  des théorèmes \ref{ordre moyen fonction multiplicative} et \ref{théorème principal friable}.
 Au cours des paragraphe \ref{lemmes de cribles} et \ref{paragraphe Type II}), on établira des bornes supérieures des différents termes d'erreur impliqués dans le théorème  \ref{terme d'erreur combinatoire}. Il en résultera le  corollaire suivant, qui est une version effective du théorème \ref{terme d'erreur combinatoire} avec $h$ borné.
\begin{corollaire}\label{corollaire des applications}
Supposons que $|h|\leq1$. Uniformément en $(N_1,N_2)\in\mathcal{N}(\eta)$ et $\X\geq2$, on a
\begin{equation}
S(\mathcal{A};h)-\frac{\eta\sigma_q(F)}{c(N_1,N_2)q^3\X}S(\mathcal{B};\sigma_qh)
\ll \tau_1\eta^2\X^2+\Delta(|h|)
\label{equation générale corollaire}\end{equation}
où \begin{align}\label{définition Delta}
 \Delta(|h|):=
 \sum_{N(\I)\in\mathcal{I }}|h_1(\I)|\sum_{m,n}\sum_{j\leq\nbin}\left(
 \Delta_{j}(m+n,\mathcal{A},\I,|h|)+
 \frac{\sigma_q(F)\eta}{c(N_1,N_2)q^3\X}\sigma_q(\I)\Delta_{j}(m+n,\mathcal{B},\I,|h|)\right)\end{align}
 et, pour $\mathcal{D}=\mathcal{A}$ ou $\mathcal{B}$, $\Delta_j(k,\mathcal{D},\I,|h|)$,
\begin{align*}
 \Delta_{j}(k,\mathcal{D},\I,|h|)
 :=\#\left\{\I\J\in\mathcal{D}\vphantom{P^{(\overrightarrow{\alpha})}}\right.:&
   \X^{\tau}<P^-(N(\J)),P^+(N(\J))\leq   \X^{1-\tau}, \\
 &  \X^{1-\tau_1}\leq P^{(\overrightarrow{\alpha})}(N(\J)), P^{(\overrightarrow{\beta})}(N(\J)) \leq   \X^{1+\tau_1}, \\
 &\left.\Y P^{(\overrightarrow{\alpha})}(N(\J))\leq P^{(\overrightarrow{\beta})}(N(\J))\leq
 \Y  \X^{O(\tau^6)}P^{(\overrightarrow{\alpha})}(N(\J))
 \right\}\end{align*}
 pour les $\Y$, $\overrightarrow{\alpha}$ et $\overrightarrow{\beta}$ introduits dans la définition (\ref{définition Eij}) de $E_j(k)$. 

De plus, si $h$ est à support sur les entiers $\z$ criblés avec $\z>\X^{\tau_1}$, alors on a\begin{equation}
S(\mathcal{A};h)-\frac{\eta\sigma_q(F)}{c(N_1,N_2)q^3\X}S(\mathcal{B};\sigma_qh)
\ll  \tau_1\frac{\log\X}{(\log\z)^2}\sigma_q(F)\eta^2\X^2+\Delta(|h|).
\label{equation criblée corollaire}\end{equation}
\end{corollaire}

\begin{proof}[Démonstration du corollaire \ref{corollaire des applications} en admettant provisoirement les résultats des paragraphes \ref{lemmes de cribles} et  \ref{paragraphe Type II}]
Les lemmes \ref{contribution entiers carrus Greaves}, \ref{grande partie friable} et \ref{crible S} fournissent respectivement une borne supérieure de $\Delta_0(|h|,\X^{2\tau/3})$, $\Theta(|h|,\X^{\tau},\X^{\tau_1})$ et $T(|h|)$ qui est négligeable, compte tenu des hypothèses sur $\tau$ et $\tau_1$.
En utilisant la proposition \ref{conséquence fin type II} avec le choix $\xi=\tau^{7}$  ainsi que le lemme \ref{Approximation S par Se}, on obtient  majoration
\begin{displaymath}
S(|h|)\ll \eta^2  \X^2(\log \X)^{-B}+\Delta(|h|)
\end{displaymath}
valide pour tout $B>0$.
En utilisant le lemme \ref{estimation petits intervalles} pour estimer $\Delta_1(|h|)$, $\Delta_2(|h|)$, $\Delta_3(|h|)$, $\widetilde{\Delta}_1(|h|)$ et $\widetilde{\Delta}_2(|h|)$, on en déduit finalement le résultat.
\end{proof}
Il convient de souligner qu'un résultat similaire pour des fonctions $h$ non bornées et à croissance suffisamment lente semble accessible à l'aide de la méthode développée dans cet article. Toutefois, obtenir les analogues des estimations du paragraphe \ref{lemmes de cribles} pour de telles fonctions $h$ nécessiterait davantage de travail, ce qui aurait rendu la lecture particulièrement fastidieuse.

\section{Premières applications des estimations de sommes de Type I}\label{lemmes de cribles}

Dans cette partie, on utilise les estimations de sommes de Type I du paragraphe \ref{paragraphe Type I} pour établir 
des bornes des différents
termes d'erreur impliqués dans le lemme \ref{terme d'erreur combinatoire}, à l'exception de $S(|h|)$, qui fera l'objet du paragraphe \ref{paragraphe Type II} sur les estimations de sommes de Type II. À partir de maintenant, on suppose que la fonction arithmétique $h$ est bornée en faisant l'hypothèse $|h(n)|\leq1$ pour tout entier $n$.

Le lemme ci-dessous donne une majoration de la contribution des idéaux dont la norme appartient à $\Upsilon\left((\log\X)^c\right)$, c'est-à-dire  est divisible par un $p^k>(\log \X)^c$ avec $k\geq2$. Il en découle une majoration du terme $\Delta_0( |h|,\X^{2\tau/3})$ défini par (\ref{définition Delta0}).
\begin{lemme}\label{contribution entiers carrus Greaves}
Soient $B\geq0$. Il existe $c(B)>0$ tel que, uniformément en $\X\geq2$, on ait
\begin{equation}
\#\left\{1\leq n_1,n_2\leq \X:
 F(n_1,n_2)\in\Upsilon((\log \X)^{c(B)})\right\}\ll   \X^2
 (\log \X)^{-B}\label{carrus mathcalA}
\end{equation}
et 
\begin{equation}
\sum_{\substack{N(\J)\leq \X\\
N(\J)\in\Upsilon((\log \X)^{c(B)})}}\sigma_q(\J)\ll \X
 (\log \X)^{-B}.\label{carrus mathcalB}
\end{equation}
En particulier, si $|h|\leq1$, il existe $c(B)>0$ tel que
\begin{align}\label{majoration Delta0}
\Delta_0( |h|,\X^{2\tau/3})\ll\X^2(\log\X)^{-B}.
\end{align}
\end{lemme}
Ce résultat, qui s'inspire largement du lemme 2 de \cite{Gr92},
 peut être considéré comme une extension du lemme 5 de \cite{De71}
aux formes binaires. 
\begin{proof}Soit $C>0$. Pour tout nombre premier $p$, on pose 
\begin{align*}k(p):=\max\left(2,\left\lfloor\frac{\log(C\log \X)}{\log p}\right\rfloor+1\right).
 \end{align*}
Étant donné 
des entiers $n_1$ et $n_2$  tels que $(p,n_1n_2)=1$ et 
 $p^{k(p)}|F(n_1,n_2)$, il existe $\omega$ tel que 
 \begin{align*}F(\omega,1)\equiv 0\pmod{p^{k(p)}}\quad\text{et}\quad 
n_1\equiv \omega n_2\pmod{p^{k(p)}}.
   \end{align*} 
Par suite,  on peut écrire d'après le lemme 1 de \cite{Gr92} les inégalités
\begin{align*}
&\#\left\{1\leq n_1,n_2\leq \X:(p,n_1n_2)=1,p^{k(p)}|F(n_1,n_2)\right\}\\\leq&
\sum_{\substack{1\leq\omega\leq p^{k(p)}\\
F(\omega,1)\equiv 0\pmod {p^{k(p)}}}}\#\left\{0\leq n_1,n_2\leq \X:
n_1\equiv\omega n_2\pmod{p^{k(p)}}\right\}\\
\leq&\sum_{\substack{1\leq\omega\leq p^{k(p)}\\
F(\omega,1)\equiv 0\pmod {p^{k(p)}}}}\left(\frac{  \X^2}{p^{k(p)}}+O\left(\frac{\X}{M_0(\omega,p^{k(p)})}\right)\right)
\end{align*}
où
\begin{align*}
 M_0(\omega,p^k):=\min_{\substack{(n_1,n_2)\neq(0,0)\\n_1=\omega n_2\pmod {p^k}}}\max(|n_1|,|n_2|).
\end{align*}
Dans la mesure où le nombre de racines modulo $p^k$ du polynôme $F(X_1,1)$ 
est borné
uniformément en $p$ et $k\geq1$, la contribution du terme principal peut-être estimée par
\begin{align*}
\sum_{p\ll   \X^{3/2}}\sum_{\substack{1\leq\omega\leq p^{k(p)}\\
F(\omega,1)\equiv 0\pmod {p^{k(p)}}}}\frac{  \X^2}{p^{k(p)}}
&\ll \sum_{p\leq (\log \X)^{C/2}}\frac{  \X^2}{(\log \X)^{C}}
+\sum_{(\log \X)^{C/2}<p\ll   \X^{3/2}}\frac{  \X^2}{p^{2}}&
\\&\ll 
  \X^{2}(\log \X)^{-C/2}.
\end{align*}
On décompose le terme de reste en écrivant
\begin{align*}
\sum_{p\ll   \X^{3/2}}\sum_{\substack{1\leq\omega\leq p^{k(p)}\\
F(\omega,1)\equiv 0\pmod {p^{k(p)}}}}\frac{\X}{M_0(\omega,p^{k(p)})}=S_1+S_2+S_3
\end{align*}
où
\begin{align*}
 S_1=\sum_{\substack{p\leq   \X^{1/2}}}
 \sum_{\substack{1\leq\omega\leq p^{k(p)}\\
F(\omega,1)\equiv 0\pmod {p^{k(p)}}}}\frac{\X}{M_0(\omega,p^{k(p)})},\end{align*}
                                                                    \begin{align*}
 S_2=\sum_{\substack{   \X^{1/2}<p\ll   \X^{3/2}}}\sum_{\substack{1\leq\omega\leq p^{k(p)}\\
F(\omega,1)\equiv 0\pmod {p^{k(p)}}\\M_0(\omega,p^{k(p)})>   \X^{3/4} }}\frac{\X}{M_0(\omega,p^{k(p)})}
\end{align*} et \begin{align*}
 S_3=\sum_{\substack{  \X^{1/2}<p\ll   \X^{3/2}}}\sum_{\substack{1\leq\omega\leq p^{k(p)}\\
F(\omega,1)\equiv 0\pmod {p^{k(p)}}\\M_0(\omega,p^{k(p)})\leq   \X^{3/4}}}\frac{\X}{M_0(\omega,p^{k(p)})}.
\end{align*}
En utilisant l'estimation triviale $M_0(\omega,p^k)\geq1$ pour $S_1$ et l'hypothèse 
$M_0(\omega,p^{k(p)})>  \X^{3/4}$ pour $S_2$, on en déduit que
\begin{align*}
 S_1\ll   \frac{\X^{3/2}}{\log\X}\qquad\text{et}\qquad S_2\ll   \frac{\X^{7/4}}{\log\X}.
\end{align*}
De l'estimation
\begin{align*}
 S_3\ll\sum_{1\leq n_1\leq   \X^{3/4}}\frac{\X}{n_1}
 \sum_{0\leq n_2\leq n_1}\sum_{\substack{  \X^{\frac{1}{2}}<p\\ p^2|F(n_1,n_2)}}1
\ll   \X^{7/4}, 
\end{align*}
on déduit finalement que \begin{align*}
\sum_{p\ll   \X^{3/2}}
\#\left\{1\leq n_1,n_2\leq \X:(p,n_1n_2)=1,p^{k(p)}|F(n_1,n_2)\right\}&\ll
  \X^2(\log \X)^{-C/2}.\end{align*}
On regarde à présent la contribution des $p$, $n_1$ et $n_2$ tels que $p|n_1n_2$. 
Si $p\nmid F(1,0)F(0,1)$, alors $p|(n_1,n_2)$ et l'on a 
\begin{align*}
 \sum_{\substack{p\nmid F(1,0)F(0,1)}}
\#\left\{1\leq n_1,n_2\leq \X:p| n_1n_2,p^{k(p)}|F(n_1,n_2) \right\}&\leq
 \sum_{\substack{p\nmid F(1,0)F(0,1)\\p\ll   \X^{3/2}}}\frac{  \X^2}{p^{2k(p)}}\\&\ll   \X^2(\log \X)^{-C/2}.
\end{align*}
Inversement, si $p|F(1,0)F(0,1)$, alors les estimations (\ref{Type I Daniel}) et (\ref{majoration de gammaf}) 
entraînent l'estimation
\begin{align*}
 \#\left\{1\leq n_1,n_2\leq \X:p^k|F(n_1,n_2)\right\}&\ll   \X^2\frac{\gamma_F(p^k)}{p^{2k}}+X(\log \X)^{C/2}
 (\log_2\X)^{7203}\\&\ll   \X^2(\log \X)^{-2C/3}.
\end{align*}
En choisissant $C$ suffisamment grand, il s'ensuit l'estimation  (\ref{carrus mathcalA}). 

Pour établir (\ref{carrus mathcalB}), on utilise la méthode de Rankin pour écrire 
\begin{align*}
 \sum_{\substack{N(\J)\leq \X\\
N(\J)\in\Upsilon((\log \X)^{C})}}\sigma_q(\J)&\ll 
\sum_{p\leq   \X^{1/2}}\sum_{\substack{k(p)\leq k
}}\sigma_q^{\Z}(p^{k})\sum_{N(\J)
\leq \X/p^k}\sigma_q(\J)\\
&\ll \X\sum_{p\leq   \X^{1/2}}\sum_{k(p)\leq k}\frac{\sigma_q^{\Z}(p^{k})}{p^k}\sum_{N(\J)\leq \X}\frac{\sigma_q(\J)}{N(\J)}.
\end{align*}
Au vu de (\ref{log sigmaq}) et (\ref{majoration sigma}), il suit
 \begin{align*}
 \sum_{\substack{N(\J)\leq \X\\
N(\J)\in\Upsilon((\log \X)^{C})}}\sigma_q(\J)&\ll \X(\log \X)^{c(A)}
\sum_{p\leq   \X^{1/2}}\sum_{k(p)\leq k}\frac{\sigma_q^{\Z}(p^{k})}{p^{k}}\\&\ll \X(\log \X)^{c(A)}\left(\sum_{p\leq(\log \X)^{C/2}}\frac{1}{(\log \X)^{2C/3}}+\sum_{(\log \X)^{C/2}<p\leq   \X^{1/2}}\frac{1}{p^{4/3}}\right)\\
&\ll \X(\log \X)^{c(A)-C/6},
\end{align*}
ce qui implique  (\ref{carrus mathcalB}) quitte à choisir $C$ suffisamment grand.

Sous l'hypothèse où $h$ est borné, on observe que 
\begin{align*}
&\sum_{N(\I)\in\mathcal{I}}|h_1(\I)|\sigma_q(\I)\sum_k|h_2(k)|\sum_{\substack{N(\Qid)\leq\X^{1-4\tau_1}\\P^-(N(\Qid))>\X^{\tau}}}1_{\widetilde{E}(k,\I)}(\Qid)
\#\left\{\Kid:\I\Kid\Qid\in\mathcal{B}, N(\Kid)\in\Upsilon((\log\X)^{C})\right\}\\\ll& e^{O(\tau^{-1})}\sum_{\substack{N(\J)\ll q^3\X^3\\
N(\J)\in\Upsilon((\log \X)^{C})}}\sigma_q(\J).
\end{align*}
L'estimation (\ref{majoration Delta0}) découle alors de (\ref{carrus mathcalA}) et (\ref{carrus mathcalB}).
\end{proof}

Les estimations de Type I établies dans la section \ref{paragraphe Type I} permettent, 
au moyen d'un lemme de crible, de
donner des bornes supérieures du bon ordre de grandeur des cardinaux $S(\mathcal{A},\I  ,C^-(\z))$ et $S(\mathcal{B},\I  ,C^-(\z))$ définies au paragraphe \ref{Premières propriétés}. De telles estimations seront centrales dans les  majorations de $\Theta(|h|,\Y,\z)$, $\Delta_1(|h|)$, $\Delta_2(|h|)$, $\Delta_3(|h|)$, $\widetilde{\Delta}_1(|h|)$
et $\widetilde{\Delta}_2(|h|)$.
\begin{lemme}\label{lemme préliminaire de crible}
Soient $B_1$ et $ B_2\geq0$. Il existe $c(B_1,B_2)>0$ tel que 
uniformément en $\z>\log \X$, on ait
\begin{align}\label{formule de crible pour A}
S(\mathcal{A},\I  ,C^-(\z))
\ll \sigma_q(F)\frac{\eta^2  \X^2}{\log \z}\frac{\sigma_q(\I^-(\z))}{N(\I)}+R_{\mathcal{A}}(\I,\z)\end{align}
et
\begin{align}S(\mathcal{B},\I  ,C^-(\z))
\ll \frac{\eta c(N_1,N_2)q^3  \X^3
}{N(\I)\log\z }
+R_{\mathcal{B}}(\I,\z)\label{formule de crible pour B}\end{align}
 avec
\begin{align*}
 \sum_{N(\I)\z^2\leq   \X^{2}(\log \X)^{-c(B_1,B_2)}}\tau_{\K}(\I)^{B_1}R_{\mathcal{A}}(\I,\z)\ll   \X^2(\log \X)^{-B_2}
\end{align*}
et
\begin{align*}
 \sum_{N(\I)\z^2\leq   \X^{3}(\log \X)^{-c(B_1,B_2)}}\tau_{\K}(\I)^{B_1}\sigma_q(\I)R_{\mathcal{B}}(\I,\z)\ll   \X^3(\log \X)^{-B_2}.\end{align*}\end{lemme}
\begin{proof}
 Puisque la fonction de densité $\gamma_q(\I,\cdot)$
 définie par (\ref{définition de gamma}) 
 satisfait  (\ref{écriture de gamma}) et  (\ref{inégalité gammap}),
  elle vérifie les hypothèses du crible linéaire 
$(\Omega_0)$ et $(\Omega_1)$ de \cite{HR74}, 
à savoir l'existence d'un réel $A>0$ tel que, uniformément en $q$, $\J$ et $p$, on ait
\begin{align*}
\gamma_q(\I,p) \leq \frac{A}{p+A}, 
\end{align*}
excepté pour les premiers $p$ tels que $\gamma_q(\I,p)=1$, auquel cas $S(\mathcal{A},\I  ,z)=0$ dès que $\z\geq p$

On peut donc appliquer  le théorème 6.2 de \cite{HR74}
ce qui entraîne \begin{align}\nonumber
S(\mathcal{A},\I  ,C^-(\z))
\ll \frac{\eta^2  \X^2}{\zeta_q(2)}\frac{\alpha_q(\I)}{N(\I)}
\prod_{\substack{p\leq \z}}
\left(1-\gamma_q(\I,p)\right)+R_{\mathcal{A}}(\I,\z)
\end{align}
où \begin{displaymath}
R_{\mathcal{A}}(\I,\z):=\sum_{d\leq\z^2}\tau^2(d)|r(\mathcal{A},I,d)|.
\end{displaymath}En estimant le terme de reste à l'aide  du lemme \ref{mathcalAIq}, on en déduit l'existence d'une constante $c(B_1,B_2)>0$ telle que
\begin{align*}
 \sum_{N(\I)\z^2\leq   \X^2(\log \X)^{-c(B_1,B_2)}}\tau_{\K}(\I)^{B_1}R_{\mathcal{A}}(\I,z)\ll   \X^{2}(\log \X)^{-B_2}.
\end{align*}
Sous l'hypothèse $\z>\log \X$, la formule de crible (\ref{formule de crible pour A}) 
est une conséquence de l'estimation uniforme
en $N(\I)\leq   \X^2$
\begin{align*}
 \frac{\alpha_q(\I)}{\zeta_q(2)} \prod_{p\leq \z}\left(1-\gamma_q(\I,p)\right)\ll 
 \frac{\sigma_q(F)\sigma_q(\I^-(\z))}{\log \z}\prod_{\substack{p|q\\p>\z}}\left(1+O\left(\frac{1}{p}\right)\right)
\end{align*}
qui suit  de (\ref{convergence sigmaf}),  (\ref{définition sigmaqF}) et (\ref{définition sigmaq}).

Pour établir (\ref{formule de crible pour B}), il suffit de reproduire le même argument en utilisant le lemme \ref{mathcalBIq} 
(resp. (\ref{convergence produit beta}))
en lieu et place du lemme \ref{mathcalAIq} (resp. (\ref{convergence sigmaf}), (\ref{définition sigmaqF}) et (\ref{définition sigmaq})).
Il vient ainsi que, pour $c(B_2)>0$ suffisamment grand, on a
\begin{align*}
\sum_{N(\I)\z^2\leq   \X^3(\log \X)^{-c(B_2)}}\tau_{\K}(\I)^{B_1}\sigma_q(\I)R_{\mathcal{B}}(\I,\z)
&\ll   \X^3(\log \X)^{-B_2}\sum_{N(\I)\leq   \X^3}\frac{\tau_{\K}(\I)^{B_1}\sigma_q(\I)}{N(\I)}.\end{align*}
En utilisant l'estimation $\widetilde{\sigma}_q^{\Z}(m)\ll\tau(m)$ valide pour tout entier 
 $q$-régulier $m$ ainsi que 
(\ref{estimation puissance}), il s'ensuit  que l'on a\begin{align*}
\sum_{N(\I)\leq   \X^3}\frac{\tau_{\K}(\I)^{B_1}\sigma_q(\I)}{N(\I)}&\ll   \sum_{\substack{\substack{d}\quad q\text{-singulier}
}}\frac{\sigma_q(d)}{d^{5/6}}\sum_{m\leq   \X^3}\frac{\tau(m)^{c(B_1)}}{m}\\
&\ll   (\log \X)^{c(B_1)}.
\end{align*}
\end{proof}

Une borne supérieure de la quantité $\Theta(|h|,\Y,\z)$  définie par (\ref{définition Theta})
peut être obtenue en s'inspirant de travaux antérieurs concernant les ordres moyens de fonctions 
arithmétiques sur les valeurs polynomiales. Par exemple, Tenenbaum~[\cite{Te90b}, lemme 3.7]  montre que, 
pour tout polynôme $F\in\Z[X_1]$, il existe $c(F)>0$ tel que, uniformément en $\z\geq \Y\geq 2$ et $\X\geq2$, on ait
\begin{align*}
 \#\left\{n\leq \X: \prod_{\substack{p\leq \Y \\p^{\nu}\parallel F(n)}}p^{\nu}>\z\right\}\ll\exp\left(-c(F)\frac{\log \z}{\log \Y }\right).
\end{align*}
En vue d'obtenir un analogue de ce résultat au cas des formes binaires et de remplacer le terme $\frac{\log \z}{\log \Y }$ par $\frac{\log \z}{\log \Y }\log\left(\frac{\log \z}{\log \Y }\right)$, 
on peut adapter 
la preuve du théorème 1 de \cite{Sh80}, démarche à l’origine du résultat suivant.

\begin{lemme}\label{grande partie friable}
Soit $\varepsilon>0$. Il existe des constantes $c_1,c_2>0$ tel que, uniformément en $\X\geq2$,
 $(N_1,N_2)\in\mathcal{N}(\eta)$, 
 et 
 $\z\geq \Y\geq \exp\left(\frac{\log \X}{(\log_2\X)^{1-\varepsilon}}\right)$, on ait
\begin{align}\label{A grande partie friable}
\#\left\{\J\in\mathcal{A}: N(\J^-(\Y))> \z\right\}
\ll \eta^2  \X^2(\log_2(q+2))^{c_1}\exp\left(-c_2\frac{\log \z}{\log \Y }\log\left(\frac{\log \z}{\log \Y }\right)\right)
\end{align}
et
\begin{align}\label{B grande partie friable}
 \sum_{\substack{\J\in\mathcal{B}\\ N(\J^-(\Y))> \z}}\sigma_q(\J)
 \ll \eta c(N_1,N_2)q^3   \X^3(\log_2( q+2))^{c_1}\exp\left(-
c_2\frac{\log \z}{\log \Y }
 \log\left(\frac{\log \z}{\log \Y }\right)\right).
\end{align}
En particulier, si $h\leq1$, on a
\begin{align*}
\Theta(|h|,\Y,\z)\ll \eta^2  \X^2(\log_2q)^{c_1}\exp\left(-c_2\frac{\log \z}{\log \Y }\log\left(\frac{\log \z}{\log \Y }\right)\right).
\end{align*}
\end{lemme}
\begin{proof}
Sans perte de généralité, on peut supposer que $\Y\leq \z^{1/4}$ et $\z\leq \X$, le résultat étant trivial autrement, au vu de (\ref{estimation sigmaqF}).
Étant donné un idéal $\J$  et la décomposition de la norme de sa partie $q$-régulière et $\Y$-friable
\begin{align*}
 N(\J_{q\text{-r}}^-(\Y))=p_1^{\alpha_1}\cdots  p_k^{\alpha_k}, p_i<p_{i+1},   
\end{align*}
 on considère
$j\geq0$ le plus grand entier  tel que 
\begin{align*}
p_1^{\alpha_1}\cdots  p_j^{\alpha_j}\leq \z^{1/2}
\end{align*}
et on définit les diviseurs $\J_1$ et $\J_2$ de $\J_{q\text{-r}}^-(\Y)$ par les conditions
\begin{align*}
N(\J_1)=p_1^{\alpha_1}\cdots  p_j^{\alpha_j}\quad\text{et}\quad N(\J_2)=p_{j+1}^{\alpha_{j+1}}\cdots  p_k^{\alpha_k}.
\end{align*}

On scinde l'ensemble des différents $\J$ tels que $N(\J^-(\Y))>Z$ en quatre classes
:
\begin{itemize}
 \item Classe I : $N(\J_{q\text{-s}})> \z^{1/2}$,
 \item Classe II : $N(\J_1)\leq \z^{1/4}$ et $N(\J_{q\text{-s}})\leq \z^{1/2}$,
 \item Classe III : $p_{j+1} \leq \log \X\log_2\X$, $N(\J_1)> \z^{1/4}$  et $N(\J_{q\text{-s}})\leq \z^{1/2}$,
 \item Classe IV : $\log \X\log_2\X\leq p_{j+1}<Y$, $N(\J_1)> \z^{1/4}$ et $N(\J_{q\text{-s}})\leq \z^{1/2}$,
\end{itemize}
et on estime la contribution de chacune de ces classes séparément.

\underline{Contribution de la classe I}. Compte tenu de l'hypothèse $N(\J_{q\text{-s}})> \z^{1/2}$,  il existe un premier $q$-singulier $p$ et
un entier $k\geq2$ tels que $p^k>\z^{1/(2\omega(q)+2\omega_{1\text{-s}})}$
où $\omega_{1\text{-s}}$ désigne le nombre de premiers
$1$-singuliers. Au regard des conditions $q\leq(\log \X)^A$ et 
$\z>\exp\left(\frac{\log \X}{(\log_2\X)^{1-\varepsilon}}\right)$%\exp\left((\log_2\X)^{2+\varepsilon}\right)$
, on en déduit que l'estimation $p^k\gg_c(\log \X)^c$ est valide pour tout $c>0$.
Par suite, le lemme \ref{contribution entiers carrus Greaves}
entraîne  que, pour tout $B>0$, il existe une constante $c(B)>0$ telle que
\begin{align*}
 \#\left\{\J\in\mathcal{A}:
 N(\J_{q\text{-s}})> \z^{1/2}\right\}
 &\ll\#\left\{1\leq n_1,n_2\leq q\X:
 F(n_1,n_2)\in\Upsilon\left((\log \X)^{c(B)}\right)\right\}\\&\ll   \X^2(\log \X)^{-B}
\end{align*}
et
\begin{align*}
 \sum_{\substack{\J\in\mathcal{B}\\ N(\J_{q\text{-s}})> \z^{1/2}}}\sigma_q(\J)\ll
 \sum_{\substack{N(\J)\leq q^3  \X^3\\ N(\J)\in\Upsilon((\log \X)^{c(B)})}}\sigma_q(\J)
 \ll   \X^3(\log \X)^{-B}.
\end{align*}

\underline{Contribution de la classe II}. Dans la mesure où les idéaux de cette classe vérifient
 $N(\J_1)\leq \z^{1/4}$ et $N(\J_{q\text{-r}}^-(\Y))>\z^{1/2}$, 
il s'ensuit que  $p_{j+1}^{\alpha_{j+1}}> \z^{1/4}$. Sous l'hypothèse  $\Y\leq \z^{1/4}$, on a également $\alpha_{j+1}\geq 2$ ce qui signifie que les idéaux de cette classe ont leur norme dans $\Upsilon(\z^{1/4})$.
En utilisant là encore le  lemme \ref{contribution entiers carrus Greaves}, on en déduit que, pour tout $B>0$, on a
\begin{align*}
\#\left\{\J\in\mathcal{A}:\J\text{ est dans la classe II}
\right\}
\ll   \X^2(\log \X)^{-B}
\end{align*}
et
\begin{align*}
 \sum_{\substack{\J\in\mathcal{B}\\\J\text{ dans la classe II}
 }}\sigma_q(\J)
 \ll   \X^3(\log \X)^{-B}.
\end{align*}

\underline{Contribution de la classe III}. En utilisant successivement le lemme \ref{AI admissible} puis 
la majoration $\alpha_q^{\Z}(m)\ll m^{1/2}(\log\X)^c$ consécutive à (\ref{inégalité alpha singulier}), on remarque que, sous l'hypothèse $\z\leq\X$, on a par la méthode de Rankin
\begin{align*}
&\#\left\{\J\in\mathcal{A}:\J\text{ est dans la classe III}\right\}\\ \ll&\frac{\eta^2\X^2}{\zeta_q(2)}
\sum_{\substack{\z^{1/4}<N(\I)\leq \z^{1/2}\\P^+(N(\I))\leq \log \X\log_2 \X}}
 \frac{\alpha_q(\I)}{N(I)}+  \X^{15/8}(\log \X)^{c}\\
 \ll& (\log\X)^c\X^2
\sum_{\substack{m\leq \X^{1/2}\\P^+(m)\leq \log \X\log_2 \X}}
 \frac{1}{m^{1/2}}\left(\frac{m}{\z^{1/4}}\right)^{1/2}+  \X^{15/8}(\log \X)^{c}\\
 \ll& (\log\X)^c  \X^{2}Z^{-1/8}\Psi(\X^{1/2},\log \X\log_2\X)+  \X^{15/8}(\log \X)^{c}.
 \end{align*}
L'estimation de De Bruijn de $\Psi(\X,t)$, uniforme en $\X\geq t> 2$, contenue dans le
 théorème 1 de \cite{Br66}
\begin{multline*}
\log \Psi(\X,t)=
\left(\frac{\log \X}{\log t}\log\left(1+\frac{t}{\log \X}\right)+\frac{t}{\log t}\log\left(1+\frac{\log \X}{t}\right)\right)
  \left(1+O\left(\frac{1}{\log t}+\frac{1}{\log_2(2X)}\right)\right)\end{multline*} 
entraîne la majoration $\Psi(\X^{1/2},\log\X\log_2\X)\ll\exp\left(O\left(\frac{\log\X}{\log_2\X}\right)\right)$. Par suite, il suit de  l'hypothèse $\z\geq\exp\left(\frac{\log \X}{(\log_2\X)^{1-\varepsilon}}\right)$, 
que l'on a, pour tout $B>0$,
\begin{align*}
 \#\left\{\J\in\mathcal{A}:\J\text{ est dans la classe III}
\right\}\ll   \X^2(\log \X)^{-B}.
\end{align*}

De façon similaire, le théorème \ref{theo Weber} et les estimations (\ref{majoration sigma}), (\ref{partie tronquée gamma}) et (\ref{log sigmaq}) 
 impliquent 
\begin{align*}
 \sum_{\substack{\J\in\mathcal{B}\\J\text{ dans la classe III} }}\sigma_q(\J)&\ll(\log\X)^c  \X^3 
 \sum_{\substack{\z^{1/4}<m\leq \z^{1/2}\\P^+(m)\leq \log \X\log_2 \X}} \frac{\sigma_q^{\Z}(m)}{m}+  \X^{7/3
 }(\log \X)^{c}\\
 &\ll (\log\X)^c  \X^3
\sum_{\substack{m\leq \X^{1/2}\\P^+(m)\leq \log \X\log_2 \X}}
 \frac{1}{m^{1/2}}\left(\frac{m}{\z^{1/4}}\right)^{1/2}
   \X^{7
 /3
 }(\log \X)^{c}\\
 &\ll   \X^3(\log \X)^{-B}.
\end{align*}

Au vu de l'hypothèse $\z\geq\Y\geq\exp\left(\frac{\log\X}{(\log_2\X)^{1-\varepsilon}}\right)$, on déduit de ce qui précède que les contributions des classes I, II et III sont négligeables dans (\ref{A grande partie friable}) et (\ref{B grande partie friable}). La contribution principale proviendra ainsi des idéaux de la classe
IV.

\underline{Contribution de la classe IV}.  On commence par découper l'ensemble de sommation de $p_{j+1}$ en sous-intervalles du type 
$\left]\z^{1/(s+1)},\z^{1/s}\right]$ avec $s_1\leq s\leq s_2$ où
\begin{align*}
 s_1:=\left\lfloor \frac{\log \z}{\log \Y }\right\rfloor\text{ et }s_2:=
 \left\lfloor\frac{\log \z}{\log\left(\log \X\log_2\X\right)}\right\rfloor.
\end{align*}
On observe alors que
\begin{align*}
 \#\left\{\J\in\mathcal{A}:\J\text{ est dans la classe IV}\right\}
\ll &
\sum_{ s_1\leq s\leq s_2}\sum_{\substack{\I\text{ }q\text{-singulier}\\N(\I)\leq \z^{1/2}}}
\sum_{\substack{\J_1\text{ }q\text{-régulier}\\\z^{1/4}\leq N(\J_1)< \z^{1/2}\\
  P^+(N(\J_1))< \z^{1/s}}}S\left(\mathcal{A},\I  \J_1, C^-(\z^{1/(s+1)})\right).
  \end{align*}
  Puisque $\z^{1/(s+1)}> \log \X$ si $s\leq s_2$, 
  on peut  utiliser la formule (\ref{formule de crible pour A}) du 
  lemme \ref{lemme préliminaire de crible} pour en déduire que, pour tout $B>0$, la contribution des idéaux dans la classe IV est majorée par 
  \begin{align*}
 \frac{\eta^2  \X^2\sigma_q(F)}{\log \z}\sum_{ s_1\leq s\leq s_2}(s+1)
  \sum_{d\text{ }q\text{-singulier}}\frac{\sigma^{\Z}_q(d^-(Z^{1/(s+1)}))}{d} \sum_{\substack{m\text{ }q\text{-régulier}\\\z^{1/4}\leq m< \z^{1/2}  \\P^+(m)\leq \z^{1/s}}}
   \frac{\widetilde{\sigma}_q^{\Z}(m)}{m}+   \X^2(\log \X)^{-B}.
  \end{align*}
  Dans la mesure où $\widetilde{\sigma}_q^{\Z}$ satisfait les hypothèses du lemme 4 de \cite{Sh80}, 
on dispose de l'estimation de la somme intérieure, 
uniforme en $s\leq \log \z/\log_2\z$,
\begin{align*}
     \sum_{\substack{m\text{ }q\text{-régulier}\\\z^{1/4}\leq m  \\P^+(m)\leq \z^{1/s}}}\frac{\widetilde{\sigma}_q^{\Z}(m)}{m}&\ll
     \exp\left(
     \sum_{\substack{p\text{ }q\text{-régulier}\\p\leq \z^{1/s}}}\frac{\widetilde{\sigma}_q^{\Z}(p)}{p}-
\frac{1}{20}s\log\left(s/2\right)\right)
     \\
     &\ll\frac{\log \z}{s}\exp\left(
-\frac{1}{20}s\log\left(s/2\right).     \right)
\end{align*}Par suite, si $s\leq \log \z/\log_2\X$, on remarque que
\begin{align*}
 \sum_{d\text{ }q\text{-singulier}}
 \frac{\sigma^{\Z}_q(d^-(\z^{1/(s+1)}))}{d}
 &=\prod_{\substack{p\text{ }q\text{-singulier}\\p\leq \z^{1/(s+1)}}}
 \left(\sum_{k\geq0}\frac{\sigma_q^{\Z}(p^k)}{p^k}\right)
 \prod_{\substack{p\text{ }q\text{-singulier}\\p> \z^{1/(s+1)}}}\left(1-\frac{1}{p}\right)^{-1}\\&
 \ll \sum_{d\text{ }q\text{-singulier}}\frac{\sigma^{\Z}_q(d)}{d}\ll\log_2(q+2)^c
\end{align*}
d'après (\ref{estimation sigmaqF})
ce qui implique l'estimation
\begin{align*}
   &\#\left\{\J\in\mathcal{A}:\J\text{ est dans la classe IV}\right\}\\
\ll &\frac{\eta^2  \X^2\sigma_q(F)}{\log \z}\log_2(q+2)^c
\sum_{ s_1\leq s\leq s_2}(s+1) \exp\left(
-\frac{1}{20}s\log\left(s/2\right)\right)
  +   \X^2(\log \X)^{-B}.\end{align*}
Au vu de l'hypothèse $\z\geq \Y\geq \exp\left(\frac{\log \X}{(\log_2\X)^{1-\varepsilon}}\right)$ et de l'estimation
(\ref{estimation sigmaqF}), on en déduit que
\begin{align*}
   \#\left\{\J\in\mathcal{A}:\J\text{ est dans la classe IV}\right\}
&\ll \eta^2  \X^2(\log_2q)^c\exp\left(
-\frac{1}{21}\frac{\log \z}{\log \Y }\log\left(\frac{\log \z}{\log \Y }\right)\right).
\end{align*}
En utilisant  (\ref{formule de crible pour B}) au lieu de
(\ref{formule de crible pour A}), on obtient également que
\begin{align*}
 \sum_{\substack{\J\in\mathcal{B}\\J\text{ dans la classe IV} }}\sigma_q(\J)&\ll \eta c(N_1,N_2)q^3  \X^3 
 (\log_2 q)^c\exp\left(
-\frac{1}{21}\frac{\log \z}{\log \Y }\log\left(\frac{\log \z}{\log \Y }\right)\right)
\end{align*}
ce qui permet de déduire (\ref{B grande partie friable}).
\end{proof}

À partir de maintenant, nous n'utiliserons le lemme \ref{lemme préliminaire de crible} que pour un niveau de crible $\z\geq   \X^{\tau}$. Le lemme suivant précise la borne supérieure disponible dans ce cadre là.
\begin{lemme}\label{petits intervalles}
Soit $B\geq0$. Il existe des constantes $c_1(B)$ et $c_2(B)>0$ tels que, uniformément en $\X\geq2$,
 $(N_1,N_2)\in\mathcal{N}(\eta)$,  
$\z\geq (\log \X)^{c_1(B)}$,  
et $C\subset \left(C^-\left((\log \X)^{c_1(B)}\right)\right)^2$ où $C^-(\cdot)$ est défini par (\ref{définition C crible}), on ait, pour tout idéal $\I$ tel que $N(\I)\in\mathcal{I }$,
\begin{align}
\nonumber&\sum_{\substack{\Qid_1,\Qid_2\in\JK\\(N(\Qid_1), N(\Qid_2))\in C\\N(\I\Qid_1\Qid_2)\z^2\leq   \X^{2}(\log \X)^{-c_2(B)}}}
S(\mathcal{A},\I\Qid_1\Qid_2,C^-(\z))
\\\ll&\frac{\eta^2  \X^2\sigma_q(F)}{\log \z}\frac{\sigma_q(\I)}{N(\I)}\sum_{\substack{(d_1,d_2)\in C}}\mu^2(d_1d_2)
\frac{\nu_{d_1d_2}}{d_1d_2}
+R_{\mathcal{A}}(\I) .\label{petit crible A}
\end{align}
et
\begin{align}\nonumber 
&\sum_{\substack{\Qid_1,\Qid_2\in\JK\\(N(\Qid_1), N(\Qid_2))\in C\\N(\I\Qid_1 \Qid_2)\z^2\leq   \X^{3}(\log \X)^{-c_2(B)}}}
S(\mathcal{B},\I  \Qid_1\Qid_2,C^-(\z))
\\\ll&
\frac{c(N_1,N_2)\eta q^3  \X^3}{N(\I)\log\z }
\sum_{\substack{(d_1,d_2)\in C}}\mu^2(d_1d_2)
\frac{\nu_{d_1d_2}}{d_1d_2}
+R_{\mathcal{B}}(\I)\label{petit crible chi}
\end{align} 
où
$\nu_{d}$   est le nombre d'idéaux $\J$ tels que $N(\J)=d$
et \begin{align*}
 \sum_{N(\I)\in\mathcal{I }}\left|R_{\mathcal{A}}(\I)\right|\ll   \X^{2}(\log \X)^{-B}\qquad\text{et}
 \qquad \sum_{N(\I)\in\mathcal{I }}\sigma_q(\I)\left|R_{\mathcal{B}}(\I)\right|
 \ll   \X^{3}(\log \X)^{-B}.
\end{align*}
\end{lemme}
\begin{proof}
Une application directe du lemme \ref{lemme préliminaire de crible} donne
\begin{align*}
 S(\mathcal{A},\I  \Qid_1\Qid_2,C^-(\z))
 \ll \frac{\eta^2  \X^2\sigma_q(F)}{\log \z}
 \frac{\sigma_q((\I\Qid_1\Qid_2)^-(\z))}{N(\I\Qid_1\Qid_2)}+R_{\mathcal{A}}(\I\Qid_1\Qid_2,\z)
\end{align*}
avec
\begin{align*}
\sum_{\substack{\I,\Qid_1,\Qid_2\in\JK\\N(\I\Qid_1 \Qid_2)\z^2\leq   \X^{2}(\log \X)^{-c_2(B)}}}
R_{\mathcal{A}}(\I\Qid_1\Qid_2,\z)\ll   \X^2(\log \X)^{-B}.
 \end{align*}

Dans la mesure où $\z>(\log \X)^{\max(A,1)}$, on peut remarquer que, uniformément en $\Qid_1$ et $\Qid_2$ tels que $N(\I\Qid_1\Qid_2)\leq   \X^2(\log \X)^{-c(B)}$,
on a
\begin{align*}
\sigma_q((\I\Qid_1\Qid_2)^-(\z))\ll\sigma_q(\I).
\end{align*}
Quitte à choisir $c_1(B)$ suffisamment grand, on peut alors majorer
la contribution des $\Qid_1\Qid_2$ avec facteur carré par
\begin{align*}
  \sum_{\substack{\I,\Qid_1,\Qid_2\\P^-(N(\Qid_1\Qid_2)\geq (\log\X)^{c(B)}\\\mu(N(\Qid_1\Qid_2))=0\\
  N(\I\Qid_1\Qid_2)\ll   \X^2}}
 \frac{\sigma_q(\I)}{N(\I\Qid_1\Qid_2)}
 \ll   
 \sum_{\substack{N(\J)\ll   \X^2}} \frac{\tau(\J)^2\sigma_q(\J)}{N(\J)(\log\X)^{c(B)}}
  \ll(\log \X)^{-B}
\end{align*}
ce qui achève la preuve de (\ref{petit crible A}).
Un raisonnement similaire permet d'établir (\ref{petit crible chi}).
\end{proof}
Le lemme suivant, qui rappelle essentiellement les estimations (7.1) et (7.2) de
\cite{HB01}, sera utilisé pour estimer les membres de droite de (\ref{petit crible A}) et (\ref{petit crible chi}).
\begin{lemme}\label{lemme petits écarts}
 On a, uniformément en $\Y\geq \z\geq2$
\begin{equation} \sum_{\Y\leq p\leq \Y \z}\frac{\nu_p}{p}
 \ll\frac{\log \z}{\log \Y }
\end{equation}
 et \begin{align*}
  \sum_{\z\leq p_1<\cdots<p_k\leq 
 \Y}\frac{\nu_{p_1\cdots  p_k}}{p_1\cdots  p_k}\ll
 \frac{\left(\log\frac{\log \Y }{\log \z}+O(1)\right)^k}{k!}.
\end{align*}
\end{lemme}
\begin{proof}
La première majoration, 
analogue de la formule (7.1) de 
\cite{HB01}, est une conséquence du théorème des idéaux premiers.
Au vu de la multiplicativité de $\nu_d$, la seconde estimation provient de l'inégalité
\begin{align*}
  \sum_{\z\leq p_1<\cdots<p_k\leq 
 \Y}\frac{\nu_{p_1}}{p_1}\cdots \frac{\nu_{p_k}}{p_k}
 &\leq
 \frac{1}{k!}\left(\sum_{\z\leq p\leq 
 \Y}\frac{\nu_{p}}{p}\right)^k\ll
 \frac{\left(\log\frac{\log \Y }{\log \z}+O(1)\right)^k}{k!}.
\end{align*}
\end{proof}

On établit dans le lemme suivant  des estimations de $\Delta_1(|h|)$, $\Delta_2(|h|)$, $\Delta_3(|h|)$, $\widetilde{\Delta}_1(|h|)$
et $\widetilde{\Delta}_2(|h|)$, définies par les formules (\ref{définition Delta1}) à (\ref{définition Deltatilde2}), suffisantes pour nos applications mais qui pourraient être précisées pour des fonctions $h$ spécifiques. 
 \begin{lemme}\label{estimation petits intervalles}
Supposons que $|h|\leq1
$. On a, uniformément en $\X\geq2$ 
 et $(N_1,N_2)\in\mathcal{N}(\eta)$,
\begin{align}\label{estimation Delta général}
 \Delta_1(|h|), \Delta_2(|h|),\Delta_3(|h|),\widetilde{\Delta}_1(|h|),\widetilde{\Delta}_2(|h|)
 &
\ll \tau_1\eta^2  \X^2.\end{align}

De plus, si $h$ est à support sur les entiers $\z$-criblés avec $\z\geq   \X^{\tau_1}$, alors on a
\begin{align}\label{estimation Delta crible}
 \Delta_1(|h|), \Delta_2(|h|),\Delta_3(|h|),\widetilde{\Delta}_1(|h|),\widetilde{\Delta}_2(|h|)
 &
\ll \sigma_q(F)\eta^2  \X^2\frac{\tau_1\log \X}{(\log\z )^2}.\end{align}
\end{lemme}
\begin{proof}Observant qu'un idéal ne contribue qu'un nombre fini de fois pour chaque  $\Delta_i(|h|)$ et $\widetilde{\Delta}_i(|h|)$, il vient
\begin{multline*}
 \Delta_1(|h|) ,\widetilde{\Delta}_1(|h|)\ll\max_{\substack{\Y\geq   \X^{\frac{3}{2}-\tau_1}}}
\sum_{\substack{  \X^{1/2-\tau_1}\leq N(\mathfrak{p}_1)\leq N(\mathfrak{p}_2)\\
\Y\leq N(\mathfrak{p_1}\mathfrak{p}_2)\leq \Y   \X^{\tau_1}}}
\sum_{N(\I)\in\mathcal{I }}
\left(S\left(\mathcal{A},\I\mathfrak{p}_1\mathfrak{p}_2,C^-(  \X^{\tau})\right)\vphantom{\frac{\eta\sigma_q(F)}{c(N_1,N_2)q^3\X}}\right.\\\left.+\frac{\eta\sigma_q(F)}{c(N_1,N_2)q^3\X}\sigma_q(\I)S\left(\mathcal{B},\I\mathfrak{p}_1\mathfrak{p}_2,C^-(  \X^{\tau})\right)
\right)
\end{multline*} et
\begin{multline*}
\Delta_2(|h|),\widetilde{\Delta}_2(|h|)\ll 
\max_{\Y\geq   \X^{\frac{1}{2}-\tau_1}} \sum_{\Y\leq N(\mathfrak{p})\leq \Y   \X^{\tau_1}}
\sum_{N(\I)\in\mathcal{I }}
\left(S\left(\mathcal{A},\I\mathfrak{p},C^-(  \X^{\tau})\right)\vphantom{\frac{\eta\sigma_q(F)}{c(N_1,N_2)q^3\X}}\right.\\\left.+\frac{\eta\sigma_q(F)}{c(N_1,N_2)q^3\X}\sigma_q(\I)S\left(\mathcal{B},\I\mathfrak{p},C^-(  \X^{\tau})\right)
\right).\end{multline*}

Pour majorer $\Delta_2(|h|)$ et $\widetilde{\Delta}_2(|h|)$, on peut faire appel aux lemmes \ref{petits intervalles} et \ref{lemme petits écarts}. 
Dans l'intervalle $ \X^{1/2-\tau_1}\leq \Y \leq   \X^{\frac{3}{2}-\tau_1}$, on a la majoration suivante, valide pour tout $B>0$, \begin{align*}
& \sum_{\Y\leq N(\mathfrak{p})\leq \Y   \X^{\tau_1}}
\sum_{N(\I)\in\mathcal{I }}\left(S\left(\mathcal{A},\I\mathfrak{p},
C^-(  \X^{\tau})\right)
+\frac{\eta\sigma_q(F)}{c(N_1,N_2)q^3\X}\sigma_q(\I)S\left(\mathcal{B},\I\mathfrak{p},
C^-(  \X^{\tau})\right)\right)\\
&\ll\sigma_q(F)\frac{\eta^2  \X^2}{\tau\log \X}
\max_{  \X^{1/2-\tau_1}\leq \Y \leq   \X^{\frac{3}{2}-\tau_1}} \sum_{\Y\leq p\leq \Y   \X^{\tau_1}}\frac{\nu_p}{p}
\sum_{N(\I)\in\mathcal{I }}\frac{\sigma_q(\I)}{N(\I)}+\X^2(\log\X)^{-B}\\&\ll
\tau_1\eta^2\X^2,
\end{align*}
où l'on a utilisé la majoration
\begin{align*}
\sum_{N(\I)\in\mathcal{I }}\frac{\sigma_q(\I)}{N(\I)}
\leq\prod_{p\leq\X^{\tau}}\sum_k\frac{\sigma_q^{\Z}(p^k)}{p^k}\ll\frac{\tau\log\X}{\sigma_q(F)\zeta_q(2)}
\end{align*}
qui découle de (\ref{ordre moyen inverse}).
En utilisant (\ref{log sigmaq}), on peut borner la contribution des $\Y>X^{\frac{3}{2}-\tau_1}$  par 
\begin{align*}
&\max_{\Y\leq   \X^{\frac{3}{2}+2\tau_1}} \sum_{ \Y\leq N(\J)\leq \Y   \X^{\tau_1}}
\left(S\left(\mathcal{A},\J,C^-(  \X^{3/2-\tau_1})\right)
+\frac{\eta\sigma_q(F)}{c(N_1,N_2)q^3\X}\sigma_q(\J)S\left(\mathcal{B},\J,C^-(  \X^{3/2-\tau_1})
\right)\right)\\
 \ll&\sigma_q(F)\frac{\eta^2  \X^2}{\log \X}\max_{\Y\leq   \X^{\frac{3}{2}+2\tau_1}} \sum_{ \Y\leq N(\J)\leq \Y   \X^{\tau_1}}
\frac{\sigma_q(\J)}{N(\J)}+  \X^2(\log \X)^{-B}\\
\ll& \tau_1\eta^2  \X^2
\end{align*}
et on montre de même que 
$
 \Delta_1(|h|),\widetilde{\Delta}_1(|h|)\ll\tau_1\eta^2  \X^2$.

Pour achever la démonstration de (\ref{estimation Delta général}), on remarque également que la contribution de
 $\Delta_3(|h|)$ est inférieure à
\begin{align*}
\ll &
\sum_{N(\I)\in\mathcal{I }}\sum_k\sum_{\substack{
\X^{\tau}<N(\mathfrak{p}_2)\leq \cdots\leq N(\mathfrak{p}_{k-1})}}
\max_{\Y\leq   \X^{1+\tau_1}}\sum_{\substack{\X^{\tau}\leq N(\mathfrak{p}_1)\\
\Y\leq N(\mathfrak{p}_1)\leq \Y (\log \X)^c}}
\Big(
S\left(\mathcal{A},\I\mathfrak{p}_1\cdots \mathfrak{p}_{k-1},C^-(  \X^2)
\right)
\\&\pushright{+\frac{\eta\sigma_q(F)}{c(N_1,N_2)q^3\X}\sigma_q(\I)
S\left(\mathcal{B},\I\mathfrak{p}_1\cdots \mathfrak{p}_{k-1},C^-(  \X^2)
\right)\Big)}\\
\ll &
\sigma_q(F)\frac{\eta^2  \X^2}{\log \X}
\sum_{N(\I)\in\mathcal{I }}\frac{\sigma_q(\I)}{N(\I)}\sum_k\sum_{\substack{
  \X^{\tau}<p_2< \cdots<p_{k-1}<  \X^{1+\tau_1}}}
\max_{\Y\leq   \X^{1+\tau_1}}
\sum_{\substack{
\Y\leq p_1\leq \Y (\log \X)^c}}\frac{\nu_{p_1\cdots  p_{k-1}}}{p_1\cdots  p_{k-1}}\\&+  \X^2(\log \X)^{-B}\\
\ll &
\sigma_q(F)\eta^2  \X^2\frac{\log_2\X}{\log \X}+  \X^2(\log \X)^{-B}\end{align*}
pour tout $B>0$.

On suppose désormais que $h$ est à support sur les entiers $\z$-criblés. En procédant comme précédemment, on peut majorer $\Delta_2(|h|)$ et $\widetilde{\Delta}_2(|h|)$ par
\begin{align*}
\ll
&\max_{  \X^{1/2-\tau_1}\leq \Y \leq   \X^{\frac{3}{2}-\tau_1}} \sum_{\Y\leq N(\mathfrak{p})\leq \Y   \X^{\tau_1}}\left(S\left(\mathcal{A},\mathfrak{p},C^-(\z)\right)
+\frac{\eta\sigma_q(F)}{c(N_1,N_2)q^3\X}S\left(\mathcal{B},\mathfrak{p},C^-(\z)
\right)\right)\\
&+ \max_{\Y\leq   \X^{\frac{3}{2}+2\tau_1}} 
\sum_{ \substack{\Y\leq N(\J)\leq \Y   \X^{\tau_1}\\P^-(N(\J))>\z}}
\left(S\left(\mathcal{A},\J,C^-(  \X^{3/2-\tau_1})\right)
+\frac{\eta\sigma_q(F)}{c(N_1,N_2)q^3\X}S\left(\mathcal{B},\J,C^-(  \X^{3/2-\tau_1})\right)
\right)\\
\ll& \sigma_q(F)\frac{\eta^2  \X^2}{\log \z}
\max_{  \X^{1/2-\tau_1}\leq \Y \leq   \X^{\frac{3}{2}-\tau_1}} \sum_{\Y\leq p\leq \Y   \X^{\tau_1}}\frac{\nu_p}{p}\\&+
\sigma_q(F)\frac{\eta^2  \X^2}{\log \X}
\sum_k\sum_{\z<p_2<\cdots<p_k\leq   \X^{3/2}}\frac{\nu_{p_2\cdots  p_k}}{
p_2\cdots  p_k}\max_{\z<\Y}\sum_{\substack{ \Y\leq p_1\leq \Y   \X^{\tau_1}}}
\frac{\nu_{p_1}}{p_1}+  \X^2(\log \X)^{-B}\\\\
\ll& \sigma_q(F)\eta^2  \X^2\frac{\tau_1\log \X}{(\log\z )^2}
\end{align*}
et une estimation identique est valable pour
  $\Delta_1(|h|)$ et $\widetilde{\Delta}_1(|h|)$. 
  %\ll  \sigma_q(F)\frac{\eta^2  \X^2}{\log \z}\frac{\tau_1\log \X}{\log \z}+  \X^2(\log \X)^{-B}
%(\ref{estimation Delta crible})
\end{proof}

Avec le choix de paramètre de crible $\z=\X^{\tau}$, on peut reproduire l'argument du
paragraphe 6  de\cite{HB01}, basé sur le crible de Selberg, 
pour obtenir  un équivalent asymptotique de $S(\mathcal{D},\I,C^-(\X^{\tau}))$. Ceci valide la pertinence de (\ref{différence S A B}) pour le choix de l'ensemble $C^-(\X^{\tau})$, conduisant ainsi à une borne supérieure de 
\begin{align*}T(|h|):= \sum_{N(\I)\in\mathcal{I }}|h_1(\I)|
 \sum_{\substack{
N(\Qid)\leq   \X^{1-4\tau_1}\\P^-(N(\Qid))>  \X^{\tau}}
}
\sum_{n\geq0}\left|
T^{(n)}(\mathcal{A},\I\Qid)
-\frac{\eta\sigma_q(F)}{c(N_1,N_2)q^3\X}
\sigma_q(\I)T^{(n)}(\mathcal{B},\I\Qid)
\right|
\end{align*}
où 
$T^{(n)}(\mathcal{A},\cdot)$ et $T^{(n)}(\mathcal{B},\cdot)$ sont définis par (\ref{définition Tn}).

\begin{lemme}\label{crible S}
Supposons que $|h|\leq1$. 
 Pour tout $B>0$, on a  uniformément en  $\X\geq2$ et  $(N_1,N_2)\in\mathcal{N}(\eta)$,
\begin{align*}
T(|h|)
\ll \frac{\eta^2   \X^2}{\tau^3\log \X}\sum_{N(\I)\in \mathcal{I }}\frac{|h_1(\I)|\sigma_q(\I)}{N(\I)}
\exp\left(-\frac{\tau_1}{\tau}\right)+  \X^{2}(\log \X)^{-B}.\end{align*}
\end{lemme}
\begin{proof}
Ayant remarqué au cours de la preuve du lemme \ref{lemme préliminaire de crible} que $\gamma(\I,\cdot)$ et $\beta$ 
satisfont les hypothèses du crible linéaire, on peut appliquer
le théorème 7.1 de \cite{HR74} avec les choix de paramètres $'z'=  \X^{\tau}$ et $'\xi'=  \X^{\tau_1}$. On obtient ainsi\begin{multline*}
S(\mathcal{A},\I \Qid\mathfrak{p}_1\cdots \mathfrak{p}_n,C^-(  \X^{\tau}))
=
\eta^2  \X^2\sigma_q(F)\frac{\sigma_q(\I)\alpha_q(\Qid\mathfrak{p}_1\cdots \mathfrak{p}_n)}{N(\I\Qid\mathfrak{p}_1\cdots \mathfrak{p}_n)
}\prod_{p\leq   \X^{\tau}}
\left(1-\frac{1}{p}\right)\left(1+O\left(\exp\left(-\frac{\tau_1}{\tau}\right)
\right)\right)
\\+\sum_{d\leq\X^{2\tau_1}}\tau^2(d)|r(\mathcal{A},\I\Qid\mathfrak{p}_1\cdots \mathfrak{p}_n,d)|
\end{multline*}
et \begin{multline*}
S(\mathcal{B},\I\Qid\mathfrak{p}_1\cdots \mathfrak{p}_n,C^-(  \X^{\tau}))
=\frac{c(N_1,N_2)\eta q^3  \X^3}{N(\I\Qid\mathfrak{p}_1\cdots \mathfrak{p}_n)
}
\prod_{p\leq   \X^{\tau}}\left(1-\frac{1}{p}\right)\left(1+O\left(\exp\left(-\frac{\tau_1}{\tau}\right)\right)
\right)\\+\sum_{d\leq\X^{2\tau_1}}\tau^2(d)|r(\mathcal{B},\I\Qid\mathfrak{p}_1\cdots \mathfrak{p}_n,d)|.
\end{multline*}

Au vu de la définition (\ref{écriture rho2 non singulier}), on observe alors que, pour tout idéal $  \X^{\tau}$-criblé $\J$ admissible
tel que $N(\J)\ll   \X^2$, on a $
\alpha_q(\J)=1+O\left(\frac{1}{\tau   \X^{\tau}}\right)$. 

La contribution des idéaux $\Qid\mathfrak{p}_1\cdots \mathfrak{p}_n$ dont la norme possède des facteurs carrés peut être majorée par
\begin{align}
 \ll \sum_{\substack{N(\J)\ll   \X^2}}    \frac{\tau_{\K}(\J)}{N(\J)}\left(\sum_{p>   \X^{\tau/2}}\frac{1}{p^2}+
   \sum_{p>   \X^{\tau/3}}\frac{1}{p^3}\right)\ll      \X^{-\tau/2}\sum_{\substack{N(\J)\ll   \X^2}} 
\frac{\tau_{\K}(\J)}{N(\J)},                                                     \label{Th facteurs carrés} \end{align}
 ce qui permet de déduire que \begin{align*}
\sum_{\substack{N(\Qid)\leq   \X^{1-4\tau_1}\\
P^-(N(\Qid))>  \X^{\tau}}}\sum_{n}\sum_{\substack{  \X^{\tau}<N(\mathfrak{p}_1)<\cdots<N(\mathfrak{p}_n)\\N(\mathfrak{p}_1\cdots \mathfrak{p}_n)\leq
  \X^{1+\tau}}}
\frac{\left|\alpha_q(\Qid\mathfrak{p}_1\cdots \mathfrak{p}_n)-1\right|}{N(\Qid\mathfrak{p}_1\cdots \mathfrak{p}_n)
}
\ll&  \left(\tau^{-1}\X^{-\tau}+\X^{-\tau/2}\right)\sum_{\substack{N(\J)\ll   \X^2}}\frac{\tau_{\K}(\J)}{N(\J)}\\\ll&\X^{-\tau/2}(\log\X)^c.
\end{align*}

Dans la mesure où $N(\I\Qid\mathfrak{p}_1\cdots \mathfrak{p}_n)\leq   \X^{2-3\tau_1+\tau}$, la contribution des termes $r(\mathcal{D},\I\Qid\mathfrak{p}_1\cdots \mathfrak{p}_n,d)|$ est négligeable au vu des lemmes \ref{mathcalAIq} et \ref{mathcalBIq}. On en déduit alors que, pour tout $B>0$,
\begin{align*}
 T(|h|)\ll\frac{\eta^2   \X^2 \sigma_q(F)}{\tau\log \X}\sum_{N(\I)\in\mathcal{I }}\frac{|h_1(\I_1)|\sigma_q(\I)}{N(\I)}
 \sum_{\substack{N(\J)\ll   \X^{2}\\P^-(N(\J))>  \X^{\tau}}}\frac{\tau_{\K}(\J)}{N(\J)}\exp\left(-\frac{\tau_1}{\tau}\right)
 +  \X^2(\log \X)^{-B}.
\end{align*}

En utilisant le lemme \ref{lemme petits écarts} pour estimer la contribution des idéaux $\J$ dont la norme est sans facteur carré et la majoration (\ref{Th facteurs carrés}) pour  les autres idéaux, il suit 
\begin{align} \nonumber\sum_{\substack{N(\J)\leq   \X^{2}\\P^-(N(\J))>  \X^{\tau}}}\frac{\tau_{\K}(\J)}{N(\J)}\ll &
\sum_{\substack{N(\J)\ll   \X^2}}    \frac{\tau_{\K}(\J)}{N(\J)}\left(\sum_{p>   \X^{\tau/2}}\frac{1}{p^2}+
   \sum_{p>   \X^{\tau/3}}\frac{1}{p^3}\right)+
\sum_{\substack{N(\J)\ll   \X^2\\P^-(N(\J))>  \X^{\tau}}}\mu^2(N(\J))
\frac{\tau_{\K}(\J)}{N(\J)}\\\ll& \X^{-\tau/2}(\log\X)^c+
\sum_{n}2^n\sum_{  \X^{\tau}<p_1<\cdots<p_n\ll   \X^2}\frac{\nu_{p_1\cdots  p_n}}{
p_1\cdots  p_n}\nonumber
\\\ll&  \X^{-\tau/2}(\log\X)^c+\sum_{n}\frac{2^n}{n!}\left(\log\tau^{-1}+0(1)\right)^n\ll \tau^{-2}\label{théo idéaux premiers criblés},
\end{align}
ce qui achève la preuve du lemme.
\end{proof}

\section{Estimations de sommes de Type II}\label{paragraphe Type II}

Dans cette partie, on établit (cf. Proposition \ref{conséquence fin type II} \textit{infra}) une majoration  de 
 \begin{multline*}
S(|h|):=\sum_{N(\I)\in\mathcal{I }}|h_1(\I)|\sum_{i=1}^5\sum_{m,n}|h_2(\m,\n)|
\left|S(\mathcal{A},\I  ,C^{(i)}(\m,\n,\I))\right.\\\left.-\frac{\eta\sigma_q(F)}{c(N_1,N_2)q^3\X}
\sigma_q(\I)S(\mathcal{B},\I  ,C^{(i)}(\m,\n,\I))\right|
\end{multline*}
où
les $C^{(i)}(\m,\n,\I)$ sont définis par (\ref{définition CImn}) et (\ref{définition C2}), (\ref{définition C3}), (\ref{définition C4}), (\ref{définition C11}), (\ref{définition C1}) et (\ref{définition C22}).  On étend pour ce faire  
 l'approche développée dans les travaux de Heath-Brown et Moroz,
 à savoir la transformation du problème en des estimations de sommes de Type II adéquates. 

Au vu de (\ref{définition inégalités}) et (\ref{définition Eij}), on peut réécrire les différents ensembles   
$C^{(i)}(\m,\n,\I)$ sous la forme 
  \begin{align*}C^{(i)}(\m,\n,\I)=
\bigcup_{(s_1,s_2)\in\mathcal{S}^{(i)}(\n,\I)} \mathcal{R}^{(i)}(\m,(s_1,s_2),\I)\times\{(s_1,s_2)\}\end{align*}
où \begin{align*}
\mathcal{S}^{(i)}(\n,\I):
=\left\{(s_1,s_2):\Omega(s_i)=\omega(s_i)=n_i,  (s_1,s_2)\text{ satisfait }(E_j^{(i)}(\n,\I))\text{ pour } 1\leq j\ll \tau^{-1}
\right\},\end{align*} où
 $(E_j^{(i)}\n,\I))$ désigne une inégalité du type 
\begin{align*}
\Y N(\I)^{\varepsilon}P^{(\overrightarrow{\alpha}^{(1)})}(s_1)P^{(\overrightarrow{\alpha}^{(2)})}(s_2)\prec P^{(\overrightarrow{\beta}^{(1)})}(s_1)P^{(\overrightarrow{\beta}^{(2)})}(s_2)\end{align*} avec $\Y>0$, $\varepsilon\in\{-1,0,1\}$ et $\prec$ désigne l'ordre $\leq$ ou $<$, \begin{multline*}
\mathcal{R}^{(i)}(\m,(s_1,s_2),\I):=
\left\{(r_1,r_2):\Omega(r_i)=\omega(r_i)=m_i, P^-(r_1r_2)>\X^{\tau}
,\vphantom{\tau^{-1}}\right.\\\left.
(r_1,r_2)\text{ satisfait } (F_j^{(i)}(\m,(s_1,s_2),\I))\text{ pour }1\leq j\ll \tau^{-1}\right\}
\end{multline*}
où $(F_j^{(i)}(\m,(s_1,s_2),\I))$ désigne une inégalité du type 
\begin{multline*}
\Y N(\I)^{\varepsilon}P^{(\overrightarrow{\alpha}^{(1)})}(s_1)P^{(\overrightarrow{\alpha}^{(2)})}(s_2)P^{(\overrightarrow{\gamma}^{(1)})}(r_1)P^{(\overrightarrow{\gamma}^{(2)})}(r_2)\\\prec 
P^{(\overrightarrow{\beta}^{(1)})}(s_1)P^{(\overrightarrow{\beta}^{(2)})}(s_2)P^{(\overrightarrow{\delta}^{(1)})}(r_1)P^{(\overrightarrow{\delta}^{(2)})}(r_2).
\end{multline*}
Soulignons que les ($(E_j^{(i)}\n,\I))$) et $(F_j^{(i)}(\m,(s_1,s_2),\I))$ introduits ci-dessus correspondent non seulement aux ensembles $E_j(k)$ définis dans \ref{définition Eij} qui interviennent dans la définition de $h$, mais aussi aux contraintes relatives à la définition des $C^{(i)}(\m,\n,\I)$, en particulier la condition importante $X^{1+\tau}\leq s_1s_2\leq \X^{3/2-\tau}$.

La première étape dans l'estimation de $S(\mathcal{D},C^{(i)}(\m,\n,\I))$ consiste à rendre les 
ensembles $\mathcal{R}^{(i)}(\m,(s_1,s_2),\I)$ indépendants de $(s_1,s_2)$ et $\I$. 
 Pour ce faire, on introduit un  paramètre 
$\xi\leq\tau$ que l'on explicitera dans les applications du paragraphe \ref{Partie Applications}. Dans leurs travaux, 
Heath-Brown et Moroz effectuent le choix $\xi=(\log_2 \X)^{-\varpi_2}$ où $0<\varpi_2<1$. 
 À la lecture de \cite{HB01} et \cite{HM02}, il
apparaît en fait que leurs arguments demeurent valables pour des $\xi$ satisfaisant $(\log \X)^{-\varpi_2}
\leq\xi$ où $0<\varpi_2<1$, hypothèse  que l'on suppose à présent. 
Une idée importante dans l'estimation de $S(\mathcal{D},\I,C^{(i)}(\m,\n,\I))$ consiste à trier les facteurs $P^{(\overrightarrow{\alpha})}(s_i)$ dans des intervalles du type $\left[  \X^{v\xi},   \X^{(v+1)\xi}\right|$ puis à remplacer les occurrences de $P^{(\overrightarrow{\alpha})}(s_i)$ par $  \X^{v\xi}$  dans les  inégalités, procédure que l'on détaille ci-dessous. 
Posant  $v_0:=\left\lfloor\frac{\log N(\I)}{\xi\log \X}\right\rfloor$, on introduit les ensembles d'exposants \begin{multline}
\iota(\mathcal{S}^{(i)}(\n,v_0)):=\left\{\overrightarrow{\v}\in\Z^{n_1}\times\Z^{n_2}:
\overrightarrow{\v}\text{ satisfait }(\widehat{E_j}^{(i)}(\n,v_0))\text{ pour }1\leq j\ll \tau^{-1}
\right.\\\left.\vphantom{(\widehat{E_j}^{(i)}}
\text{ et }v_i\neq v_j\text{ si }i\neq j\right\}\label{définition iota}\end{multline}
où $\overrightarrow{\v}:=(\overrightarrow{v}^{(1)},\overrightarrow{v}^{(2)})$ et 
  $(\widehat{E_j}^{(i)}(\n,v_0))$ désigne l'inégalité associée à  
$(E_j^{(i)}(\n,\I))$
 définie par
\begin{align*}
\frac{\log \Y }{\xi\log \X}+(\varepsilon v_0+1)+\sum_{k=1}^{k_1}(v_{\alpha_{k}^{(1)}}^{(1)}+1)+
\sum_{k=1}^{k_2}(v_{\alpha_{k}^{(2)}}^{(2)}+1)
<
\sum_{l=1}^{l_1}v_{\beta_{l}^{(1)}}^{(1)}+\sum_{l=1}^{l_2}v_{\beta_{l}^{(2)}}^{(2)}
\end{align*} 
ainsi que
\begin{multline*}
\mathcal{R}^{(i)}(\m,\overrightarrow{\v},v_0):=
\left\{(r_1,r_2):\Omega(r_i)=\omega(r_i)=m_i
,\vphantom{\widehat{F_j}^{(i)}}\right.\\\left.
(r_1,r_2)\text{ satisfait }(\widehat{F_j}^{(i)}(\m,\overrightarrow{\v},v_0))\text{ pour }1\leq j\ll \tau^{-1}\right\}
\end{multline*}
où $(\widehat{F_j}^{(i)}(\m,\overrightarrow{\v},v_0))$ désigne l'inégalité associée à  $(F_j^{(i)}(\m,(s_1,s_2),\I))$ définie par %vi
\begin{multline*}
 \frac{\log \left(\Y P^{(\overrightarrow{\gamma}^{(1)})}(r_1)P^{(\overrightarrow{\gamma}^{(2)})}(r_2)\right)
}{\xi\log \X}+(\varepsilon v_0+1)+\sum_{k=1}^{k_1}(v_{\alpha_{k}^{(1)}}^{(1)}+1)+\sum_{k=1}^{k_2}(v_{\alpha_{k}^{(2)}}^{(2)}+1)
\\<\frac{\log \left(
P^{(\overrightarrow{\delta}^{(1)})}(r_1)P^{(\overrightarrow{\delta^{(1)}})}(r_2)\right)
}{\xi\log \X}+
\sum_{l=1}^{l_1}v_{\beta_{l}^{(1)}}^{(1)}+\sum_{l=1}^{l_2}v_{\beta_{l}^{(2)}}^{(2)}.
\end{multline*}

Dans ce qui suit, nous adaptons la méthode du paragraphe  3 de \cite{HB01} et du paragraphe 5 de \cite{HM02}. On introduit, pour
$\overrightarrow{\v}\in \N^{n_1}\times\N^{n_2}$, le poids $d_{\n}(\Sid,\overrightarrow{\v})$ en posant
\begin{align}\label{définition dSv}
d_{\n}(\Sid,\overrightarrow{\v})=
\frac{\log N(\mathfrak{p}_1^{(1)})\cdots  \log N(\mathfrak{p}_{n_1}^{(1)}) 
 \log N(\mathfrak{p}_{1}^{(2)})\cdots  \log N(\mathfrak{p}_{n_2}^{(2)})       }
 {v_{1}^{(1)}\cdots  v_{n_1}^{(1)}v_{1}^{(2)}\cdots  v_{n_2}^{(2)}(\xi\log \X)^{n_1+n_2}} 
\end{align}  si $\Sid=\Sid_1\Sid_2$ où
$\Omega(N(\Sid_i))=\Omega_{\K}(\Sid_i)=n_i$,
 $\Sid_i=\mathfrak{p}^{(\I)}_{1}\cdots \mathfrak{p}^{(\I)}_{n_i}$  
avec $  \X^{v_k^{(\I)}\xi}\leq N(\mathfrak{p}_k^{(\I)})<  \X^{(v_k^{(\I)}+1)\xi}%&\in \left[  \X^{v_k\xi},  \X^{(v_k+1)\xi}\right[
$ pour $1\leq k\leq n_i$ et $i=1$ et $2$, 
et $d_{\n}(\Sid,\overrightarrow{\v})=0$ sinon.
On pose également, pour  tout sous-ensemble $\mathcal{R}$ de $\N^2$,
\begin{equation}\label{définition bRv}
b(\Rid_1,\Rid_2,\mathcal{R})=\left\{\begin{array}{ll}
1&\text{ si }(N(\Rid_1),N(\Rid_2))\in \mathcal{R},\\
0&\text{ sinon}.\end{array}\right.
\end{equation} 

Dans l'esprit du lemme 3.7 de \cite{HB01} et du lemme 5.1 de \cite{HM02},
il est naturel d'approcher $S(\mathcal{D},C^{(i)}(\m,\n,\I))$ 
 par la quantité  
\begin{align*}
\widehat{S}(\mathcal{D},\I,C^{(i)}(\m,\n,\I)):
=\sum_{\overrightarrow{\v}\in\iota(\mathcal{S}^{(i)}(\n,v_0))}\widehat{S}
(\mathcal{D},\I,\mathcal{R}^{(i)}(\m,\overrightarrow{\v},v_0))
\end{align*}
où  
\begin{equation}\label{définition S chapeau}
\widehat{S}(\mathcal{D},\I,\mathcal{R}^{(i)}(\m,\overrightarrow{\v},v_0))
:=\sum_{\substack{\Rid_1,\Rid_2,\Sid\in\JK\\\I\Rid_1\Rid_2\Sid\in\mathcal{D}}}b(\Rid_1,\Rid_2,\mathcal{R}^{(i)}(\m,\overrightarrow{\v},v_0))
d_{\n}(\Sid,\overrightarrow{\v}).
\end{equation}

Il en résulte l'apparition des termes d'erreur \begin{align}\label{définition Rimn}
R^{(i)}(\m,\n):= \sum_{N(\I)\in\mathcal{I }}|h_1(\I)|\left(
R(\mathcal{A},\I,C^{(i)}(\m,\n,\I))+\frac{\sigma_q(F)\eta}{c(N_1,N_2)q^3\X}
\sigma_q(\I)R(\mathcal{B},\I,C^{(i)}(\m,\n,\I))\right)\end{align} où, pour $\mathcal{D}=\mathcal{A}$ ou $\mathcal{B}$,
\begin{align*}
R(\mathcal{D},\I,C^{(i)}(\m,\n,\I))= \left|S(\mathcal{D},\I,C^{(i)}(\m,\n,\I))-\widehat{S}(\mathcal{D},\I,C^{(i)}(\m,\n),\I)\right|.
\end{align*}
On peut obtenir une borne supérieure de telles quantités en reproduisant les différentes étapes de la démonstration du lemme 3.7 de \cite{HB01} qui utilisent essentiellement    les lemmes \ref{petits intervalles} et \ref{estimation petits intervalles}.
\begin{lemme}\label{Approximation S par Se}
Supposons que $|h|\leq1$. Pour tout $B>0$ et $\varepsilon>0$, on a  uniformément en  $\X\geq2$, 
 $(N_1,N_2)\in\mathcal{N}(\eta)$ et 
 $\xi\leq \tau^2$,
\begin{equation}
\sum_{\m,\n}R^{(1)}(\m,\n)\ll  
\left(\xi\tau^{-6}\frac{\sigma_q(F)\eta^2  \X^2}{\log \X}\sum_{N(\I)\in\mathcal{I }}\sigma_q(\I)\frac{|h_1(\I)|}{N(\I)}+   \X^{2}(\log \X)^{-B}\right)
\end{equation}
et, pour $i\in\{2,\ldots,5\}$, \begin{align}\label{R4}
\sum_{\m,\n}R^{(i)}(\m,\n)\ll 
\xi\tau^{-6}\frac{\sigma_q(F)\eta^2  \X^2}{\log \X}\sum_{N(\I)\in\mathcal{I }}\sigma_q(\I)\frac{|h_1(\I)|}{N(\I)}+   \X^{2}(\log \X)^{-B}
+\Delta_4(|h|)
\end{align}
où
\begin{align}\label{définition Delta4}
 \Delta_4(|h|):=
 \sum_{N(\I)\in\mathcal{I }}|h_1(\I)|\sum_{m,n}\sum_{j\leq\nbin}
 \Delta_{j}(m+n,\mathcal{A},\I,|h|)+
 \frac{\sigma_q(F)\eta}{c(N_1,N_2)q^3\X}\sigma_q(\I)\Delta_{j}(m+n,\mathcal{B},\I,|h|)\end{align}
 et, pour $\mathcal{D}=\mathcal{A}$ ou $\mathcal{B}$, $\Delta_j(k,\mathcal{D},\I,|h|)$,
\begin{align*}
 \Delta_{j}(k,\mathcal{D},\I,|h|)
 :=\#\left\{\I\J\in\mathcal{D}\vphantom{P^{(\overrightarrow{\alpha})}}\right.:&
   \X^{\tau}<P^-(N(\J)),P^+(N(\J))\leq   \X^{1-\tau}, \\
 &  \X^{1-\tau_1}\leq P^{(\overrightarrow{\alpha})}(N(\J)), P^{(\overrightarrow{\beta})}(N(\J)) \leq   \X^{1+\tau_1}, \\
 &\left.\Y P^{(\overrightarrow{\alpha})}(N(\J))\leq P^{(\overrightarrow{\beta})}(N(\J))\leq
 \Y  \X^{O(\xi\tau^{-1})}P^{(\overrightarrow{\alpha})}(N(\J))
 \right\}\end{align*}
 pour les $\Y$, $\overrightarrow{\alpha}$ et $\overrightarrow{\beta}$ introduits dans la définition (\ref{définition Eij}) de $E_j(k)$. 
\end{lemme}

\begin{proof} On ne décrit ci-dessous que la contribution des termes $R(\mathcal{A},\I,C^{(i)}(\m,\n,\I))$, le traitement de $R(\mathcal{B},\I,C^{(i)}(\m,\n,\I))$ étant en tout point similaire. Soient $\I$ un idéal tel que $N(\I)\in\mathcal{I }$ et $v_0=\left\lfloor \frac{\log N(\I)}{\xi\log \X}\right\rfloor$.
On peut décomposer le terme d'erreur $R(\mathcal{A},\I,C^{(i)}(\m,\n,\I))$ 
sous la forme
\begin{align*}
R(\mathcal{A},\I,C^{(i)}(\m,\n,\I))\leq R_1(\mathcal{A},\I,C^{(i)}(\m,\n,\I))+R_2(\mathcal{A},\I,C^{(i)}(\m,\n,\I))+R_3(\mathcal{A},\I,C^{(i)}(\m,\n,\I))
\end{align*}
où  le terme $R_1(\mathcal{A},\I,C^{(i)}(\m,\n,\I))$  résulte de l'introduction de $\iota(\mathcal{S}^{(i)}(\n,v_0))$, 
c'est-à-dire défini comme la différence
\begin{align*}\left|S(\mathcal{A},C^{(i)}(\m,\n,\I))-
\sum_{\overrightarrow{v}\in\iota(\mathcal{S}^{(i)}(\n,v_0))}\sum_{\substack{\Sid_1,\Sid_2,\Rid_1,\Rid_2\\\I\Rid_1\Rid_2\Sid_1\Sid_2\in\mathcal{A}\\
d_{\n}(\Sid_1\Sid_2,\overrightarrow{\v})\neq0}}b(\Rid_1,\Rid_2,\mathcal{R}^{(i)}(\m,(N(\Sid_1),N(\Sid_2)),\I))\right|\end{align*}
le terme $R_2(\mathcal{A},\I,C^{(i)}(\m,\n,\I))$ provient du remplacement de $1$ par $d_{\n}(\Sid_1,\Sid_2,\overrightarrow{v})$ et vaut ainsi
\begin{align*}\sum_{\overrightarrow{v}\in\iota(\mathcal{S}^{(i)}(\n,v_0))}
\sum_{\substack{\Sid_1,\Sid_2,\Rid_1,\Rid_2\\\I\Rid_1\Rid_2\Sid_1\Sid_2\in\mathcal{A}\\
d_{\n}(\Sid_1\Sid_2,\overrightarrow{\v})\neq0}}
\left|1-d_{\n}(\Sid_1\Sid_2,\overrightarrow{\v})\right|
b(\Rid_1,\Rid_2,\mathcal{R}^{(i)}(\m,(N(\Sid_1),N(\Sid_2)),\I))\end{align*}
et le terme $R_3(\mathcal{A},C^{(i)}(\m,\n,\I))$ apparaît en remplaçant 
$\mathcal{R}^{(i)}(\m,(N(\Sid_1),N(\Sid_2)),\I)$ par l'ensemble $
\mathcal{R}^{(i)}(\m,\overrightarrow{\v},v_0)$, c'est-à-dire
\begin{align*}
  R_3(\mathcal{A},C^{(i)}(\m,\n,\I))=\sum_{\overrightarrow{v}\in\iota(\mathcal{S}^{(i)}(\n,v_0))}
\sum_{\substack{\Sid_1,\Sid_2
}}d_{\n}(\Sid_1\Sid_2,\overrightarrow{\v})
\Delta(\mathcal{A},\m,\Sid_1,\Sid_2,\I)\end{align*}
où 
\begin{multline*}
\Delta(\mathcal{A},\m,\Sid_1,\Sid_2,\I):= 
\#\left\{(\Rid_1,\Rid_2):(N(\Rid_1),N(\Rid_2))\in\mathcal{R}^{(i)}(\m,(N(\Sid_1),N(\Sid_2)),\I)\Delta 
\mathcal{R}^{(i)}(\m,\overrightarrow{\v},v_0),\right.
\\\left.\vphantom{\mathcal{R}^{(i)}}\I\Rid_1\Rid_2\Sid_1\Sid_2\in
\mathcal{A}\right\}
\end{multline*}
où $\Delta $ désigne la différence symétrique.

Au vu de la définition des $(\widetilde{E}^{(i)}_j(\n,v_0)$, on observe que tous les idéaux $\Sid_1$ et $\Sid_2$ comptés dans $\widehat{S}(\mathcal{A},\I,C^{(i)}(\m,\n,\I)$ apparaissent dans $S(\mathcal{A},\I,C^{(i)}(\m,\n,\I)$. Par suite, dans la mesure où 
\begin{align}
 1-d_{\n}(\mathfrak{p}_1^{(1)}\cdots \mathfrak{p}_{n_1}^{(1)}\mathfrak{p}_1^{(2)}
 \cdots \mathfrak{p}_{n_2}^{(2)},\overrightarrow{\v})&=\nonumber1-
 \frac{\log N(\mathfrak{p}_1^{(1)})}{v_1^{(1)}\xi\log \X}\cdots 
 \frac{\log N(\mathfrak{p}_{n_1}^{(1)})}{v_{n_1}^{(1)}\xi\log \X}
 \frac{\log N(\mathfrak{p}_1^{(2)})}{v_1^{(2)}\xi\log \X}\cdots 
 \frac{\log N(\mathfrak{p}_{n_2}^{(2)})}{v_{n_2}^{(2)}\xi\log \X}\\&
 \ll\xi\tau^{-2}\label{estimation dn 1}
\end{align}
dès que $  \X^{v_k^{(\I)}\xi}\leq N(\mathfrak{p}_k^{(\I)})< 
  \X^{(v_k^{(\I)}+1)\xi}$ (cf [\cite{HB01} p.44]), il vient l'inégalité
\begin{align}\label{inétalité triviale R2}
 R_2(\mathcal{A},\I,C^{(i)}(\m,\n,\I))=
\xi\tau^{-2}S(\mathcal{A},\I,C^{(i)}(\m,\n,\I)).\end{align}

On sépare la suite de la démonstration en deux parties selon que $i=1$ ou non.

\underline{Cas $i\in\{2,\ldots,5\}$}. 
Au vu des définitions (\ref{définition C2}), (\ref{définition C3}), (\ref{définition C4}) et (\ref{définition C5}) de $C^{(i)}(m,n)$, il existe au plus un couple $(m,n)$ pour lequel un idéal $\Rid\Sid$ donné appartienne 
à $ C^{(i)}(m,n)$. 
 Par suite, on peut observer, en suivant l'argument de la preuve du lemme 3.7 de \cite{HB01} que l'on a, pour tout $N(\I)\in\mathcal{I }$ et $\mathcal{D}=\mathcal{A}$ ou $\mathcal{B}$, 
\begin{align*}
\sum_{\m,\n}R_1(\mathcal{A},\I,C^{(i)}(\m,\n,\I))\ll& \tau^{-1}
\sum_{n}\max_{\Delta\in\Delta(0,n)}\sum_{\substack{(1,N(\Sid))\in \Delta\\N(\Sid
)\leq   \X^{\frac{3}{2}-\tau}
}}S(\mathcal{A},\I\Sid,
C^-(  \X^{\tau}))\\
&+
\sum_n\sum_{j=1}^{n-1}
\sum_{\substack{  \X^{\tau}<N(\mathfrak{p}_1)<\cdots< N(\mathfrak{p}_n)
\\N(\mathfrak{p}_1\cdots \mathfrak{p}_n)\leq   \X^{3/2-\tau}\\
N(\mathfrak{p}_{j+1})\leq   \X^{\xi}N(\mathfrak{p}_j)}}S(\mathcal{A},\I\mathfrak{p}_1\cdots \mathfrak{p}_n,C^-(  \X^{\tau}))
\end{align*}où $\Delta(m,n)$ désigne l'ensemble des parties $\Delta$ de la forme
\begin{multline}
\Delta:=\left\{(r,s):\Omega(r)=\omega(r)=m,\Omega(s)=\omega(s)=n, P^-(rs)>  \X^{\tau}, \vphantom{P^{(\overrightarrow{\alpha})}}\right.\\\left.
 \Y P^{(\overrightarrow{\gamma})}(r)P^{(\overrightarrow{\alpha})}(s)\leq 
P^{(\overrightarrow{\delta})}(r)
P^{(\overrightarrow{\beta})}(s)\leq \Y   \X^{%\tau^{-1}
O(\xi\tau^{-1})}
P^{(\overrightarrow{\gamma})}(r)P^{(\overrightarrow{\alpha})}(s)\right\}. \label{définition Deltaij}
\end{multline}
En utilisant les lemmes \ref{petits intervalles} et \ref{lemme petits écarts}, il suit ainsi l'estimation
\begin{align*}
 &\sum_{N(\I)\in\mathcal{I }}|h_1(\I)|
 \sum_{\m,\n}R_1(\mathcal{A},\I,C^{(i)}(\m,\n,\I))
&\ll
\xi\tau^{-5}
\frac{\sigma_q(F)\eta^2   \X^2}{\log \X}
\sum_{N(\I)\in\mathcal{I }}\frac{|h_1(\I)|\sigma_q(\I)}{N(\I)}+  \X^2(\log \X)^{-B}
\end{align*}
valide pour tout $B\geq0$.

Dans la mesure où $
\sum_{m,n}S(\mathcal{A},\I,C^{(i)}(m,n))\leq
S(\mathcal{A},\I,C^{-}(  \X^{\tau}))$,  la majoration 
\begin{align*}&\sum_{N(\I)\in\mathcal{I }}|h_1(\I)|
 \sum_{\m,\n}R_2(\mathcal{A},\I,C^{(i)}(\m,\n,\I))+&\ll \xi\tau^{-3}\frac{\sigma_q(F)\eta^2  \X^2}{\log \X}
 \sum_{N(\I)\in\mathcal{I }}\frac{|h_1(\I)|\sigma_q(\I)}{N(\I)}+   \X^{2}(\log \X)^{-B}
\end{align*}
est aussi une conséquence immédiate du lemme \ref{petits intervalles} combiné à (\ref{inétalité triviale R2}).

En considérant $R_3(\mathcal{A},\I,C^{(i)}(\m,\n,\I))$, on remarque que des termes d'erreur proviennent 
de la contribution des idéaux
$\I\Rid\Sid$ satisfaisant les conditions $P^-(N(\Sid))  \X^{-\xi}\leq P^+(N(\Rid))\leq P^-(N(\Sid))$ ou
$P^+(N(\Sid))\leq P^-(N(\Rid))\leq P^+(N(\Sid))  \X^{\xi}$. D'autres termes d'erreur apparaissent également lors du remplacement de la condition $N(\Rid\Sid)\in E_{j}(m+n)$, réécrite sous la forme 
\begin{align}\label{vers définition Deltaij}
 \Y P^{(\overrightarrow{\alpha})}(N(\Sid))P^{(\overrightarrow{\gamma})}(N(\Rid))\prec
 P^{(\overrightarrow{\beta})}(N(\Sid))P^{(\overrightarrow{\delta})}(N(\Rid)),
\end{align}
par 
\begin{align*}
 \frac{\log \left(\Y P^{(\overrightarrow{\gamma})}(N(\Rid))\right)}{\xi\log \X}+
 \sum_{k=1}^{k_1}(v_{\alpha_{k}^{(1)}}^{(1)}+1)
<\frac{\log \left(
P^{(\overrightarrow{\delta})}(N(\Rid))\right)
}{\xi\log \X}+
\sum_{l=1}^{l_1}v_{\beta_{l}^{(1)}}^{(1)}.
\end{align*}
et $d_{\n}(\Sid,\overrightarrow{v})\neq0$. On peut ainsi écrire
\begin{align*}
 \sum_{\m,\n}R_3(\mathcal{A},\I,C^{(i)}(\m,\n,\I))
 &\ll\sum_{  \X^{1+\tau}<N(\mathfrak{p}_1)\leq N(\mathfrak{p}_2)\leq   \X^{\xi}N(\mathfrak{p}_1)}
 S(\mathcal{A},\I\mathfrak{p}_1\mathfrak{p}_2,C^-(  \X^{\tau}))\\
 &+  \sum_{ \substack{\X^{\tau}<N(\mathfrak{p}_1)\leq N(\mathfrak{p}_2)\leq   \X^{\xi}N(\mathfrak{p}_1)\\N(\mathfrak{p}_1\mathfrak{p}_2)\leq   \X^{2-2\tau_1})}}
 S(\mathcal{A},\I\mathfrak{p}_1\mathfrak{p}_2,C^-(  \X^{\tau}))\\
 &+ \sum_{j} \sum_{(N(\Rid),N(\Sid))\in \Delta(E_{j})}S(\mathcal{A},\Rid\Sid,C^-(  \X^{\tau}))
\end{align*}
où $\Delta(E_{j})$ désigne l'ensemble défini par (\ref{définition Deltaij}) où l'inégalité sur les facteurs premiers
est celle associée à (\ref{vers définition Deltaij}).
On décompose la sommation correspondant à $\Delta(E_{j})$ en quatre composantes en introduisant les conditions supplémentaires
\begin{itemize}
 \item 
$P^{(\overrightarrow{\alpha})}(N(\Sid))P^{(\overrightarrow{\gamma})}(N(\Rid))\leq   \X^{1-\tau_1}$ ou $
 P^{(\overrightarrow{\beta})}(N(\Sid))P^{(\overrightarrow{\delta})}(N(\Rid))\leq   \X^{1-\tau_1}$;\\
\item $P^{(\overrightarrow{\alpha})}(N(\Sid))P^{(\overrightarrow{\gamma})}(N(\Rid))>   \X^{1+\tau_1}$ ou $
 P^{(\overrightarrow{\beta})}(N(\Sid))P^{(\overrightarrow{\delta})}(N(\Rid))>   \X^{1+\tau_1}$;\\
\item $  \X^{1-\tau_1}< P^{(\overrightarrow{\alpha})}(N(\Sid))P^{(\overrightarrow{\gamma})}(N(\Rid))\leq   \X^{1+\tau_1}$  et $
   \X^{1-\tau_1}< P^{(\overrightarrow{\beta})}(N(\Sid))P^{(\overrightarrow{\delta})}(N(\Rid))\leq   \X^{1+\tau_1}$.
\end{itemize}
Ceci permet de borner
la contribution des différents $R_3(\mathcal{A},\I,C^{(i)}(\m,\n,\I))$  par 
\begin{multline}
 \ll\sum_{j\leq\nbin} \Delta_{j}(m+n,\mathcal{A},\I,|h|) + \tau^{-1}\sum_{\substack{  \X^{1+\tau}\leq N(\J_1)\leq N(\J_2)\leq N(\J_1)  \X^{O(\xi\tau^{-1})}\\P^-(N(\J_1\J_2))>  \X^{\tau}}}
 S(\mathcal{A},\I\J_1\J_2,C^-(  \X^{\tau}))\\
 +  \tau^{-1}\sum_{\substack{ N(\J_2)= N(\J_1)\X^{O(\xi\tau^{-1})} \\N(\J_1\J_2)\leq\X^{2-2\tau_1}\\P^-(N(\J_1\J_2))>  \X^{\tau}}}
 S(\mathcal{A},\I\J_1\J_2,C^-(  \X^{\tau}))
 ,\label{premiere majoration R3}
\end{multline}
où les $\Delta_{j}$ ont été définis dans l'énoncé du lemme \ref{Approximation S par Se}.

Observant que les lemmes \ref{non associes}, \ref{petits intervalles} et \ref{lemme petits écarts} permettent d'écrire les  bornes supérieures
\begin{align*}
&\sum_{N(\I)\in\mathcal{I}}|h_1(\I)|\sum_{\substack{  N(\J_2)= N(\J_1)\X^{O(\xi\tau^{-1})} \\N(\J_1\J_2)\leq\X^{2-2\tau_1}\\P^-(N(\J_1\J_2))>  \X^{\tau}}}
 S(\mathcal{A},\I\J_1\J_2,C^-(  \X^{\tau}))\\\ll& \frac{\eta^2\X^2\sigma_q(F)}{\tau\log\X}
\sum_{N(\I)\in\mathcal{I}}|h_1(\I)|\frac{\sigma_q(\I)}{N(\I)}\sum_{m,n} \sum_{\substack{\X{\tau}<p_1<\cdots<p_m\ll\X^2\\\X^{\tau}<p'_1<\cdots<p'n\ll\X^2
\\p'_1\cdots  p'_n=p_1\cdots  p_m\X^{O(\xi\tau^{-1})}}}\frac{\nu_{p_1\cdots  p_m}}{p_1\cdots  p_m}\frac{\nu_{p'_1\cdots  p'_n}}{p'_1\cdots  p'_n}
 +\X^2(\log\X)^{-B}\\
 \ll&
\xi\tau^{-5}\frac{\sigma_q(F)\eta^2   \X^2}{\log \X}
\sum_{N(\I)\in\mathcal{I }}\frac{|h_1(\I)|\sigma_q(\I)}{N(\I)}+  \X^2(\log \X)^{-B},
\end{align*}
et
\begin{align*}&\sum_{N(\I)\in\mathcal{I}}|h_1(\I)|\sum_{\substack{  \X^{1+\tau}\leq N(\J_1)\leq N(\J_2)\leq N(\J_1)  \X^{O(\xi\tau^{-1})}\\P^-(N(\J_1\J_2))>  \X^{\tau}}}
 S(\mathcal{A},\I\J_1\J_2,C^-(  \X^{\tau}))\\
 \ll& \sum_{N(\I)\in\mathcal{I}}|h_1(\I)|\sum_{\substack{  N(\J_3)= \frac{\X^{3+O(\xi\tau^{-1})}}{N(\I\J_1)} \\N(\I\J_1\J_3)\leq\X^{2-\tau}(\log\X)^c\\P^-(N(\J_1\J_3))>  \X^{\tau}}}
 S(\mathcal{A},\I\J_1\J_3,C^-(  \X^{\tau}))\\
 \ll&
\xi\tau^{-5}\frac{\sigma_q(F)\eta^2   \X^2}{\log \X}
\sum_{N(\I)\in\mathcal{I }}\frac{|h_1(\I)|\sigma_q(\I)}{N(\I)}+  \X^2(\log \X)^{-B}
,\end{align*} 
on déduit de (\ref{premiere majoration R3}) une estimation de $R_3(\mathcal{A},\I,C^{(i)}(\m,\n,\I))$ 
ce qui achève l'estimation de $R^{(i)}(\m,\n)$ lorsque $i\in\{2,\ldots,5\}$.

\underline{Cas $i=1$}. 
En observant que le nombre de couples $(\m,\n)$ dans lequel intervient chaque idéal $\Rid_1\Rid_2\Sid_1\Sid_2$ est borné si $N(\Rid_1\Rid_2\Sid_1\Sid_2)\ll\X^3(\log\X)^c$, les arguments précédents s'adaptent sans  difficulté nouvelle pour estimer la contribution issues de
de $C^{(1)}(\m,(1,0),\I)$ et $C^{(1)}((2,0),\n,\I)$, conduisant aux bornes supérieures
\begin{align*}
 \sum_{\m}R^{(1)}(\m ,(1,0))\ll 
\xi\tau^{-3}\frac{\sigma_q(F)\eta^2   \X^2}{\log \X}
\sum_{N(\I)\in\mathcal{I }}\frac{|h_1(\I)|\sigma_q(\I)}{N(\I)}+  \X^2(\log \X)^{-B}
\end{align*}
et
\begin{align*}
 \sum_{\n}R^{(1)}((2,0),\n )\ll
\xi\tau^{-4}\frac{\sigma_q(F)\eta^2   \X^2}{\log \X}
\sum_{N(\I)\in\mathcal{I }}\frac{|h_1(\I)|\sigma_q(\I)}{N(\I)}+  \X^2(\log \X)^{-B}
\end{align*}
et il suffit d'estimer $R(\mathcal{A},\I,C^{(1)}(\m,(n,0),\I))$ lorsque $n\geq2$.

Au vu de la définition (\ref{définition C1})  de $C^{(1)}(\m,(n,0),\I)$, on a
\begin{align*}
 &\sum_{\m}R_1(\mathcal{A},C^{(1)}(\m,(n,0),\I)\\\ll&\sum_{j\ll \tau^{-1}}
\sum_{\substack{N(\Rid_1)\leq   \X^{1-4\tau_1}\\P^-(N(\Rid_1))\geq   \X^{\tau}}}
\sum_{\substack{(N(\mathfrak{p}_1),\ldots,N(\mathfrak {p}_{n}))\in\Delta_j(n)}}S(\mathcal{A},\I\Rid_1\mathfrak{p}_1\cdots \mathfrak{p}_n,C^-(N(\mathfrak{p}_1)))
\end{align*}
où 
\begin{multline*}
\Delta_j(n):=
\left\{(p_1,\ldots, p_n): 
  \X^{\tau}<p_1< \cdots< p_n<  \X^{1-\tau},\right.\\\left.
\quad p_2\cdots  p_n\leq   \X^{1+\tau}<p_1\cdots  p_n\leq     \X^{3/2-\tau},\vphantom{X^{\tau}}(p_1,\ldots,p_n) \text{ satisfait }E(n)\right\}
\end{multline*}
et $E(n)$ désigne une inégalité d'une des formes suivantes :
\begin{align*}
   \X^{\tau}\leq p_1\leq   \X^{\tau+\xi},\quad   \X^{1+\tau-O(\xi\tau^{-1})}\leq p_2\cdots  p_n\leq   \X^{1+\tau},\quad
  \X^{1+\tau}\leq p_1\cdots  p_n\leq   \X^{1+\tau+O(\xi\tau^{-1})},
\end{align*}
\begin{align*}
   \X^{3/2-\tau-O(\xi\tau^{-1})}\leq p_1\cdots  p_n\leq   \X^{3/2-\tau}\quad\text{ ou }\quad p_j< p_{j+1}\leq   \X^{\xi}p_j.\end{align*}
On traite la contribution des inégalités
\begin{align*}  \X^{1+\tau}\leq p_1\cdots  p_n\leq   \X^{1+\tau+O(\xi\tau^{-1})}\quad\text{ et }\quad
  \X^{3/2-\tau-O(\xi\tau^{-1})}\leq p_1\cdots  p_n\leq   \X^{3/2-\tau}\end{align*}
en utilisant
les lemmes \ref{petits intervalles} et \ref{lemme petits écarts}. En effet,  le lemme \ref{non associes}  nous amène à estimer la contribution des idéaux $\I$, $\Rid_1$ et $\Rid_2$ satisfaisant
\begin{align*}
 N(\I\Rid_1\Rid_2)=\Y\X^{O\left(\tau^{-1}\xi\right)}.
\end{align*} 
avec $\Y=\X^{3/2-\tau}$ ou $\X^{2-\tau}$ et l'on observe que l'on a
 \begin{align*}
&\sum_{N(\I)\in\mathcal{I }}|h_1(\I)|\sum_{\substack{P^-(N(\Rid_1\Rid_2))\geq   \X^{\tau}\\
N(\I\Rid_1\Rid_2)=\Y  \X^{O\left(\xi\tau^{-1}\right)}
}}S(\mathcal{A},\I  \Rid_1\Rid_2,C^-(  \X^{\tau}))\\
\ll&
\xi\tau^{-5}\sigma_q(F)\frac{\eta^2  \X^2}{\log \X}\sum_{N(\I)\in\mathcal{I }}\frac{|h_1(\I)|\sigma_q(\I)}{N(\I)}+  \X^2(\log \X)^{-B}.
\end{align*}

Supposons $n\geq3$. Pour les ensembles $\Delta_j(n)$ restants,  au moins l'un des $p_i$ 
n'apparaît pas dans la définition de $E(n)$, disons $p_1$. Observant que
\begin{align*}&\sum_{\substack{N(\Rid_1)\leq   \X^{1-4\tau_1}\\P^-(N(\Rid_1))>   \X^{\tau}}}
\sum_{\substack{(N(\mathfrak{p}_1),\ldots,N(\mathfrak {p}_{n}))\in\Delta_j(n}}S(\mathcal{A},\I\Rid_1\mathfrak{p}_1\cdots \mathfrak{p}_n,C^-(N(\mathfrak{p}_1)))\\\ll&\tau^{-1}
%\sum_{N(\I)\in\mathcal{I }}
\sum_{\substack{N(\Rid_1)\leq   \X^{1-4\tau_1}\\P^-(N(\Rid_1))>   \X^{\tau}}}
\sum_{\substack{\X^{\tau}<N(\mathfrak{p}_2)<\cdots<N(\mathfrak{p}_n)\\(1,N(\mathfrak{p}_2),\ldots,N(\mathfrak {p}_{n}))\text{ satisfait }E(n)\\
N(\mathfrak{p}_2\cdots \mathfrak{p}_{n})\leq   \X^{1+\tau}}}S(\mathcal{A},\I\Rid_1\mathfrak{p}_1\cdots \mathfrak{p}_{n-1},C^-(  \X^{\tau})),
\end{align*}on déduit en utilisant les lemmes \ref{petits intervalles} et \ref{lemme petits écarts} l'estimation
\begin{align*}
& \sum_{N(\I)\in\mathcal{I }}|h_1(\I)|\sum_{\m}
 \sum_{n\geq3}R_1(\mathcal{A},C^{(1)}(\m,(n,0),\I))
 \\\ll& \xi\tau^{-6}\frac{\sigma_q(F)\eta^2  \X^2}{\log \X}
 \sum_{N(\I)\in\mathcal{I }}\frac{|h_1(\I)|\sigma_q(\I)}{N(\I)}+   \X^2(\log \X)^{-B}.
 \end{align*}
 
Si $n=2$,
l'argument précédent échoue uniquement pour  l'inégalité 
$p_1<p_2\leq p_1  \X^{\xi}$. Le cas échéant, il apparaît que $  \X^{1/2}<p_1<  \X^{3/4-\tau/2}$ ce qui permet de majorer cette contribution par
\begin{align*}
&\ll\sum_{N(\I)\in\mathcal{I }}|h_1(\I)|
\sum_{\substack{  \X^{1/2}<N(\mathfrak{p}_1)<N(\mathfrak{p}_2)\leq N(\mathfrak{p}_1)  \X^{\xi}\leq   \X^{\frac{3}{4}}}}S(\mathcal{A},\I\mathfrak{p}_1\mathfrak{p}_2,C^-(  \X^{\tau}))\\\ll&
 \xi\tau^{-1}\frac{\sigma_q(F)\eta^2  \X^2}{\log \X}\sum_{N(\I)\in\mathcal{I }}\frac{|h_1(\I)|\sigma_q(\I)}{N(\I)}+   \X^{2}(\log \X)^{-B}.
\end{align*}

Par des arguments similaires, les estimations (\ref{estimation dn 1}) et (\ref{inétalité triviale R2}) permettent de majorer la contribution des termes $R_2(\mathcal{A},C^{(1)}(\m,(n,0),\I)$ par \begin{align*}\ll&\xi\tau^{-3}\sum_{N(\I)\in\mathcal{I }}|h_1(\I)|
 \sum_{\substack{N(\Rid_1)\leq   \X^{1-4\tau_1}\\P^-(N(\Rid_1))>  \X^{\tau}}}\sum_{n\geq2}
 \sum_{\substack{  \X^{\tau}<N(\mathfrak{p}_2)<\cdots< N(\mathfrak{p}_{n})\\
 N(\mathfrak{p}_2\cdots \mathfrak{p}_{n})\leq   \X^{1+\tau}}}
 S(\mathcal{A},\I  \Rid_1\mathfrak{p}_2\cdots \mathfrak{p}_{n},C^-(  \X^{\tau}))\\
\ll
& \xi\tau^{-6}\frac{\sigma_q(F)\eta^2  \X^2}{\log \X}
 \sum_{N(\I)\in\mathcal{I }}\sigma_q(\I)\frac{|h_1(\I)|}{N(\I)}+   \X^{2}(\log \X)^{-B}.
\end{align*}

Achevons la démonstration du lemme en étudiant $R_3(\mathcal{A}, C^{(1)}(\m,(n,0),\I))$. Les arguments précédents s'adaptent sans difficulté
pour estimer les différentes contributions, à l'exception de l'inégalité 
\begin{align*}
P^-(s)\leq P^-(r_2)\leq P^-(s)  \X^{\xi}.
 \end{align*}
Si $n\geq3$, on majore ce terme d'erreur par\begin{align*}
  &\sum_{N(\I)\in\mathcal{I }}|h_1(\I)|\sum_{\substack{N(\Rid_1)\leq   \X^{1-4\tau_1}\\P^-(N(\Rid_1))>   \X^{\tau}}}
\sum_{\substack{  \X^{\tau}<N(\mathfrak{p}_1)\leq \cdots  \leq N(\mathfrak{p}_n)\\
N(\mathfrak{p}_1\cdots \mathfrak{p}_{n-1})\leq   \X^{1+\tau}}}\sum_{N(\mathfrak{p}_1)\leq N(\mathfrak{p})\leq 
N(\mathfrak{p}_1)  \X^{\xi}}S(\mathcal{A},\I\Rid_1\mathfrak{p}\mathfrak{p}_1\cdots \mathfrak{p}_n,C^-(N(\mathfrak{p})))\\\ll
&\tau^{-2}
 \sum_{N(\I)\in\mathcal{I }}|h_1(\I)|\sum_{\substack{N(\Rid_1)\leq   \X^{1-4\tau_1}\\P^-(N(\Rid_1))>   \X^{\tau}}}
\sum_{\substack{  \X^{\tau}<N(\mathfrak{p}_1)\leq \cdots  \leq N(\mathfrak{p}_{n-2})\\N(\mathfrak{p}_1)<N(\mathfrak{p})\leq N(\mathfrak{p}_1)  \X^{\xi}\\
N(\mathfrak{p}\mathfrak{p}_1\cdots \mathfrak{p}_{n-2})\leq   \X^{1+\tau}}}S(\mathcal{A},\I  \Rid_1\mathfrak{p}\mathfrak{p}_1\cdots \mathfrak{p}_{n-2},C^-(  \X^{\tau}))\\
\ll& \xi\tau^{-6}\frac{\sigma_q(F)\eta^2  \X^2}{\log \X}\sum_{N(\I)\in\mathcal{I }}\frac{|h_1(\I)|\sigma_q(\I)}{N(\I)}
+   \X^{2}(\log \X)^{-B}.\end{align*}
Lorsque $n=2$, il advient que $P^-(s)\leq\X^{3/4}$ et la contribution de tels idéaux est inférieure à
\begin{align*}\ll & \sum_{N(\I)\in\mathcal{I }}|h_1(\I)|\sum_{\substack{N(\Rid_1)\leq   \X^{1-4\tau_1}\\P^-(N(\Rid_1))>   \X^{\tau}}}
\sum_{\substack{  \X^{\tau}<N(\mathfrak{p}_1)\leq   \X^{1/2}}}
\sum_{N(\mathfrak{p}_1)\leq N(\mathfrak{p})\leq 
N(\mathfrak{p}_1)  \X^{\xi}}S(\mathcal{A},\I  \Rid_1\mathfrak{p}\mathfrak{p}_1,C^-(N(\mathfrak{p}_1)))\\
&+ \sum_{N(\I)\in\mathcal{I }}|h_1(\I)|\sum_{\substack{  \X^{1/2}<N(\mathfrak{p}_1)\leq   \X^{3/4}}}
\sum_{N(\mathfrak{p}_1)\leq N(\mathfrak{p})\leq 
N(\mathfrak{p}_1)  \X^{\xi}}S(\mathcal{A},\I\mathfrak{p}\mathfrak{p}_1,C^-(  \X^{\tau}))\\\ll&
 \xi\tau^{-4}\frac{\sigma_q(F)\eta^2  \X^2}{\log \X}\sum_{N(\I)\in\mathcal{I }}\frac{|h_1(\I)|\sigma_q(\I)}{N(\I)}
+   \X^{2}(\log \X)^{-B}.
\end{align*}
\end{proof}

Au vu du lemme précédent, il convient  de montrer que  $\frac{\eta\sigma_q(F)}{c(N_1,N_2)q^3\X^3}\sigma_q(\I)\widehat{S}(\mathcal{B},\I,C^{(i)}(\m,\n,\I))$ constitue une bonne approximation de $\widehat{S}(\mathcal{A},\I,C^{(i)}(\m,\n,\I))$.
Estimer $\widehat{S}(\mathcal{B},\I,C^{(i)}(\m,\n,\I))$ 
se réduit essentiellement à étudier les idéaux $\mathfrak{p}_1\cdots \mathfrak{p}_n$ tels que  $(N(\mathfrak{p}_1),\ldots,N(\mathfrak{p}_n))$ appartienne à l'ensemble
\begin{align*}
G(\overrightarrow{v},t):=\left\{\overrightarrow{\X}\in\R^n:x_i\in\left[  \X^{v_i\xi},  \X^{(v_i+1)\xi}\right[\text{ pour }1\leq i\leq n\text{ et }\prod_{i=1}^nx_i\leq t\right\}
\end{align*} 
Pour considérer de tels idéaux, Heath-Brown obtient dans \cite{HB01} la généralisation suivante  du théorème des idéaux premiers.

\begin{lemme}[\cite{HB01}, Lemme 4.10]
\label{forme multiplicative PIT}
Il existe $c>0$ tel que, uniformément en $t\geq 1$, $\X\geq2$, %$\frac{1}{\log   \X^{\tau}}\ll\xi$, %\frac{1}{\log \X}<\xi\leq\tau$,
 $1\leq n\ll \tau^{-1}$ %$  
et $\overrightarrow{v}\in \N^n$ satisfaisant %$\tau^{-1}\xi^{-1}\leq\log \X$ et
 $\tau\leq v_j\xi$ pour $ j\in\{1,\ldots, n\}$, 
on a % for $1\leq j\leq n$ and $t\geq 1$,
\begin{equation}
\sum_{(N(\mathfrak{p}_1),\ldots,N(\mathfrak{p}_n))\in G(\overrightarrow{v},t)}\prod_{i=1}^n\log N(\mathfrak{p}_i)=w(t)
+O\left(t%n(\xi\log \X)^{n-1}
\exp(-c\sqrt{\log   \X^{\tau}})\right)
\end{equation}
où $w(t)=w(\overrightarrow{v},t)$ désigne la mesure de $G(\overrightarrow{v},t)$. 
\end{lemme}

Le comportement analytique de la fonction $w$ définie dans le lemme \ref{forme multiplicative PIT} a été étudié
dans le paragraphe 8 de \cite{HB01}. 
\begin{lemme}[Formules (8.3) et (8.4) de \cite{HB01}]\label{propriétés de w}
Pour $t> 0$, $h\geq 0$ et $n\geq 2$, on a
\begin{align*}
\left|w'(t+h)-w'(t)\right|\leq \frac{h}{t}(\xi\log \X)^{n-2}
\end{align*}
et \begin{align*}
0\leq w'(t)\leq (\xi\log \X)^{n-1}.\end{align*}
De plus, si $n=1$, alors la dérivée à droite $w'(v_1,t)$ est la fonction caractéristique de l'intervalle $\left[  \X^{v_1\xi},  \X^{(v_1+1)\xi}\right[$.
\end{lemme}

En reproduisant l'argument développé dans le paragraphe 10 de \cite{HB01}, 
on peut utiliser les lemmes \ref{forme multiplicative PIT} et \ref{propriétés de w}  pour obtenir une estimation de 
$\widehat{S}(\mathcal{B},\I,C^{(i)}(\m,\n,\I))$. 

\begin{lemme}\label{Type II B} 
Soit $i\in\{1,\ldots,5\}$. On a, uniformément en  $\X\geq2$,
$(N_1,N_2)\in\mathcal{N}(\eta)$, %$\left(\tau\log \X\right)^{-1}\ll \xi$, 
$v_0\geq0$ et $I\in\mathcal{I }$ tel que $N(\I)\in \left[  \X^{v_0\xi},  \X^{(v_0+1)\xi)}\right[$,\begin{align}\nonumber 
&\sum_{\m,\n}\left|\widehat{S}(\mathcal{B},\I,C^{(i)}(\m,\n,\I)) %\mathcal{R}^{(\I)}(\m,\n,\overrightarrow{v}))
-\frac{\eta c(N_1,N_2)q^3  \X^3}{N(\I)}\sum_{\overrightarrow{\v}\in\iota(\mathcal{S}^{(i)}(\n,v_0))}\Sigma(\I,\mathcal{R}^{(\I)}(\m,
\overrightarrow{\v},v_0),\overrightarrow{\v}) \right|\\\ll& \frac{\eta^{2}c(N_1,N_2)q^3  \X^3}{\xi^2N(\I)}
\label{formule type II B}
\end{align}
où
\begin{equation}\label{définition Sigma}
\Sigma(\I,\mathcal{R},\overrightarrow{\v})=\sum_{\Rid_1, \Rid_2 
}\frac{b(\Rid_1,\Rid_2,\mathcal{R})}{v_1^{(1)}\cdots  v_{n_1}^{(1)}v_1^{(2)}\cdots  v_{n_2}^{(2)}(\xi\log \X)^{n_1+n_2}N(\Rid_1\Rid_2)}
w'\left(\frac{c(N_1,N_2)q^3  \X^3}{N(\I\Rid_1\Rid_2)}\right).
\end{equation}
\end{lemme}

\begin{proof} 
Compte tenu des définitions (\ref{définition iota}), (\ref{définition S chapeau}) et (\ref{définition bRv}),  
il suffit de montrer que, uniformément en $n\geq 1$, $\overrightarrow{v}\in \N^n$ tel que 
$v_i\neq v_j$ si $i\neq j$, $\tau\leq \xi v_i$ et tout ensemble $ \mathcal{R}(\m)$ tel que \begin{align*}
 \mathcal{R}(\m)\subset\left\{(r_1,r_2):\Omega(r_1)=\omega(r_1)=m_1,\Omega(r_2)=\omega(r_2)=m_2\text{ et }P^-(r_1r_2)>   \X^{\tau}\right\},
\end{align*} on a
\begin{equation}\label{réduction Type II B}
\sum_{\m}\left|\widehat{S}(\mathcal{B},\I,\mathcal{R}(\m),\overrightarrow{v})
-c(N_1,N_2)q^3  \X^3\eta\frac{\Sigma(\I,\mathcal{R}(\m),\overrightarrow{v})}{N(\I)} \right|
\ll \frac{\eta^{2}c(N_1,N_2)q^3  \X^3}{N(\I)v_1\cdots  v_n}
\end{equation} 
où
\begin{align*}
\widehat{S}(\mathcal{B},\I,\mathcal{R}(\m),\overrightarrow{v}) :=\sum_{\I\Rid_1\Rid_2\Sid\in\mathcal{D}}
b(\Rid_1,\Rid_2,\mathcal{R}(\m))d_{n}(\Sid,\overrightarrow{v}),
\end{align*}
\begin{align*}
\Sigma(\I,\mathcal{R}(\m),\overrightarrow{v})
=\sum_{\Rid_1, \Rid_2 
}\frac{b(\Rid_1,\Rid_2,\mathcal{R}(\m))}{v_1\cdots  v_{n}(\xi\log \X)^{n}N(\Rid_1\Rid_2)}
w'\left(\frac{c(N_1,N_2)q^3  \X^3}{N(\I\Rid_1\Rid_2)}\right)
 \end{align*}
et \begin{align*}
d_{n}(\Sid,\overrightarrow{v})=\left\{\begin{array}{ll} \frac{\log N(\mathfrak{p}_1)\cdots \log N(\mathfrak{p}_n)}{v_1\cdots  v_n}
 &\text{ si }\Sid=\mathfrak{p}_1\cdots \mathfrak{p}_n\text{ avec } N(\mathfrak{p}_i)\in\left[  \X^{v_i\xi},  \X^{(v_i+1)\xi}\right[,\\
 0&\text{ sinon}.\end{array}\right.
       \end{align*}
Une fois établie cette estimation, la majoration 
(\ref{formule type II B})  se déduira directement de l'inégalité  \begin{align}\label{estimation v1v2}
 \sum_{\n}\sum_{\overrightarrow{\v}\in\mathcal{S}^{(i)}(\n,v_0)}\frac{1}{v_1^{(1)}\cdots  v_{n_1}^{(1)}v^{(2)}_1\cdots 
 v^{(2)}_{n_2}}\leq\left(\sum_{n}\frac{1}{n!}\left(
 \sum_{v\ll\xi^{-1}}\frac{1}{v}\right)^n\right)^2\ll\xi^{-2}.
\end{align}

 Au vu de   (\ref{définition dSv}) et (\ref{définition S chapeau}), on peut écrire
\begin{align}
%\sum_{1\leq N(\I)\leq \Y }
&\widehat{S}\nonumber(\mathcal{B},\I,\mathcal{R}(\m),\overrightarrow{v})\\=&
%\sum_{1\leq N(\I)\leq \Y }
%\sum_{1\leq N(\I)\leq \Y }
\sum_{\Rid_1, \Rid_2}\frac{b(\Rid_1,\Rid_2,\mathcal{R}(\m))}{v_1\cdots  v_{n} (\xi\log \X)^{n}}\sum_{\substack{
  \X^{v_1\xi}\leq N(\mathfrak{p}_1)<  \X^{(v_1+1)\xi}\\\cdots \\   \X^{v_n\xi}\leq N(\mathfrak{p}_n)<  \X^{(v_n+1)\xi}\\\frac{
c(N_1,N_2)q^3  \X^3}{N(\I\Rid_1\Rid_2)}<N(\mathfrak{p}_1\cdots \mathfrak{p}_n))\leq \frac{c(N_1,N_2)q^3  \X^3}{N(\I\Rid_1\Rid_2)}(1+\eta)}}
\prod_{i=1}^{n}\log N(\mathfrak{p}_i)\label{convolution B}.
\end{align}

Estimant la somme intérieure à l'aide du lemme 
\ref{forme multiplicative PIT}, il vient 
\begin{align}
&\nonumber\sum_{\substack{
  \X^{v_1\xi}\leq N(\mathfrak{p}_1)<  \X^{(v_1+1)\xi}\\\cdots \\   \X^{v_n\xi}\leq N(\mathfrak{p}_n)<  \X^{(v_n+1)\xi}\\
\frac{c(N_1,N_2)q^3  \X^3}{N(\I\Rid_1\Rid_2)}<N(\mathfrak{p}_1\cdots \mathfrak{p}_n)\leq \frac{c(N_1,N_2)q^3  \X^3}{N(\I\Rid_1\Rid_2)}(1+\eta)}}
\prod_{i=1}^{n}\log N(\mathfrak{p}_i)\\\nonumber=&w\left(\frac{c(N_1,N_2)q^3  \X^3(1+\eta)}{N(\I\Rid_1\Rid_2)}\right)-w\left(
\frac{c(N_1,N_2)q^3  \X^3}{N(\I\Rid_1\Rid_2)}\right)\\
&+O\left(\frac{c(N_1,N_2)q^3   \X^3}{N(\I\Rid_1\Rid_2)}
\exp\left(-c\sqrt{\log   \X^{\tau}}\right)\right).\label{produit sur B}
\end{align}

En utilisant les hypothèses faites sur $\mathcal{R}(\m)$, les formules 
(\ref{convolution B}), (\ref{produit sur B}) et (\ref{théo idéaux premiers criblés}) permettent  d'écrire
\begin{align}
 &\nonumber\sum_{\m}\bigg|\sum_{\Rid_1, \Rid_2%\in\JK
}b(\Rid_1,\Rid_2,\mathcal{R}(\m))\left(w\left(\frac{c(N_1,N_2)q^3  \X^3(1+\eta)}{N(\I\Rid_1\Rid_2)}\right)-w\left(
\frac{c(N_1,N_2)q^3  \X^3}{N(\I\Rid_1\Rid_2)}\right)\right)\qquad\\\nonumber
&\pushright{-\widehat{S}(\mathcal{B},\I,\mathcal{R}(\m),\overrightarrow{v})
\bigg|}\\&\pushright{\ll
\frac{c(N_1,N_2)q^3  \X^3}{N(\I)v_1\cdots  v_n}\exp\left(-c\sqrt{\log   \X^{\tau}}\right)}.
\label{intermédiaire B}
\end{align}

Supposons que $n\neq1$. 
En utilisant le théorème des accroissements finis et le lemme  \ref{propriétés de w}, on remarque que
\begin{align}\nonumber
&w\left(\frac{c(N_1,N_2)q^3  \X^3(1+\eta)}{N(\I\Rid_1\Rid_2)}\right)-w\left(
\frac{c(N_1,N_2)q^3  \X^3}{N(\I\Rid_1\Rid_2)}\right)\\=&\frac{c(N_1,N_2)\eta q^3  \X^3}{N(\I\Rid_1\Rid_2)}\left(w'\left(
\frac{c(N_1,N_2)q^3  \X^3}{N(\I\Rid_1\Rid_2)}\right)+O\left(\eta(\xi \log \X)^{n-2}\right)\right).\label{w - w}
\end{align}
Au vu de (\ref{définition Sigma}), (\ref{intermédiaire B}), (\ref{w - w}) et (\ref{théo idéaux premiers criblés}), il s'ensuit que
\begin{align}
\sum_{\m}\left|\widehat{S}(\mathcal{B},\I,\mathcal{R}(\m),\overrightarrow{v})\nonumber
-\eta c(N_1,N_2)q^3  \X^3\frac{\Sigma(\I,\mathcal{R}(\m),\overrightarrow{v})}{N(\I)}\right|\ll&
\frac{\eta^2c(N_1,N_2)q^3  \X^3}{N(\I)v_1\cdots  v_n}.\label{Type II B n}
\end{align}

Supposons à présent que $n=1$. La dérivée à droite $w'$ étant la fonction caractéristique 
de $\left[  \X^{v_1\xi},  \X^{(v_1+1)\xi}\right[$, on en déduit que l'estimation
\begin{align*}
w\left(\frac{c(N_1,N_2)q^3  \X^3}{N(\I\Rid_1\Rid_2)}(1+\eta)\right)-w\left(\frac{c(N_1,N_2)q^3  \X^{3}}{N(\I\Rid_1\Rid_2)}\right)=\eta\frac{c(N_1,N_2)q^3  \X^3}{N(\I\Rid_1\Rid_2)}
w'\left(\frac{c(N_1,N_2)q^3  \X^3}{N(\I\Rid_1\Rid_2)}\right)
\end{align*}
est valide sauf éventuellement
pour les idéaux vérifiant $\frac{c(N_1,N_2)q^3  \X^3}{N(\I\Rid_1\Rid_2)}<  \X^{v\xi}
\leq\frac{c(N_1,N_2)q^3  \X^3(1+\eta)}{N(\I\Rid_1\Rid_2)}$ avec $v=v_1\text{ ou }v_1+1$, auquel cas 
on dispose de la borne triviale
\begin{align*}
w\left(\frac{c(N_1,N_2)q^3  \X^3}{N(\I\Rid_1\Rid_2)}(1+\eta)\right)-w\left(\frac{c(N_1,N_2)q^3  \X^3}{N(\I\Rid_1\Rid_2)}\right)\ll\eta\frac{c(N_1,N_2)q^3  \X^3}{N(\I\Rid_1\Rid_2)}.
\end{align*}

Par suite, on a
\begin{align*}
&\sum_{\m}\left|\widehat{S}(\mathcal{B},\I,\mathcal{R}(\m),v_1)
-\eta c(N_1,N_2) q^3  \X^3\frac{\Sigma(\I,\mathcal{R}(\m),v_1)}{N(\I)}\right|\\
\ll&%+O\left(
\frac{\eta c(N_1,N_2)q^3  \X^3}{N(\I)v_1\xi\log \X}\sideset{}{^{(\I)}}\sum_{\Rid_1, \Rid_2
}\frac{b(\Rid_1,\Rid_2,\mathcal{R}(\m))}{N(\Rid_1\Rid_2)}+\frac{c(N_1,N_2)q^3  \X^3}{N(\I)v_1}\exp(-c\sqrt{\log   \X^{\tau}})
\end{align*}
où la sommation $\sideset{}{^{(\I)}}\sum$ porte sur les idéaux $\Rid_1$ et $\Rid_2$ satisfaisant  
\begin{displaymath}\frac{c(N_1,N_2)q^3  \X^{3-v\xi}}{N(\I)}\leq N(\Rid_1\Rid_2)\leq \frac{c(N_1,N_2)q^3  \X^{3-v\xi}}{N(\I)}(1+\eta)\end{displaymath} avec  $v=v_1\text{ ou }v_1+1$. 
En utilisant le fait que $\tau_{\K}(\J^-(X^{\tau}))\ll e^{O(\tau^{-1})}$ dès que  $N(\J)\ll\X^4$ ainsi que l'estimation suivante, valide pour $\Y>\X$ et conséquence du théorème \ref{theo Weber}, \begin{displaymath}
\sum_{\substack{\Y\leq N(\J)\leq \Y(1+\eta)}}\frac{1}{N(\J)}\ll\eta,\end{displaymath}
on en déduit finalement (\ref{réduction Type II B})
.\end{proof}

L'argument de la preuve précédente 
 ne permet pas d'estimer l'analogue de (\ref{convolution B}) lorsque $\mathcal{D}=\mathcal{A}$, à savoir
\begin{align*}
 \sum_{\substack{
  \X^{v_1\xi}\leq N(\mathfrak{p}_1)<  \X^{(v_1+1)\xi}\\\cdots \\   \X^{v_n\xi}\leq N(\mathfrak{p}_n)<  \X^{(v_n+1)\xi}\\
\I\mathfrak{p}_1\cdots \mathfrak{p}_n\Rid_1\Rid_2\in\mathcal{A}}}
\prod_{i=1}^{n}\log N(\mathfrak{p}_i).
\end{align*}
Le rôle manifeste du phénomène de parité dans une telle somme rend caduque la possibilité de l'étudier directement 
à partir des estimations de Type I  obtenues au paragraphe \ref{paragraphe Type I}. 
En vue de lever cette obstruction et de contourner partiellement la dépendance en $\n$ dans la somme précédente, 
on introduit, en s'inspirant des notations (5.10), (5.11) et (5.12) de \cite{HM02}, les quantités \begin{align*}
\widehat{S}_{e}(\mathcal{A},\I,C^{(i)}(\m,\n,\I)):=
\sum_{\overrightarrow{v}\in\iota(\mathcal{S}^{(i)}(\n,v_0))}
\widehat{S}_{e}(\mathcal{A},\I,\mathcal{R}^{(i)}(\m,\overrightarrow{\v},v_0))
\end{align*}
où
\begin{align*}
\widehat{S}_{e}(\mathcal{A},\I,\mathcal{R}^{(i)}(\m,\overrightarrow{\v},v_0)):=
\sum_{\I\Rid_1\Rid_2\Sid\in\mathcal{A}}b(\Rid_1,\Rid_2,\mathcal{R}^{(i)}(\m,\overrightarrow{\v},v_0))e_{\n}(\Sid,\overrightarrow{\v}),
\end{align*}
\begin{equation}
\label{définition eSv}e_{\n}(\Sid,\overrightarrow{\v}):=\frac{w'(N(\Sid))}{v_1^{(1)}\cdots  v_{n_1}^{(1)}v_1^{(2)}\cdots  v_{n_2}^{(2)}
(\xi\log \X)^{n_1+n_2}}
\sum_{\substack{\J|\Sid \\N(\J)< \X^{\tau/2}}}\mu_{\K}(\J)\log\frac{\X^{\tau/2}}{N(\J)}.
\end{equation}

Compte tenu de la définition de $w$  et du lemme \ref{propriétés de w}, on peut observer
que $e_{\n}(\Sid,\overrightarrow{\v})$ et $d_{\n}(\Sid,\overrightarrow{\v})$ 
sont
\begin{equation}\label{estimation en dn}
 \left\{\begin{array}{ll}
 \ll  \tau_{\K}(\Sid)\log \X&\text{ si }0\leq \frac{\log N(\Sid)}{\xi\log \X}-
 \left(v_1^{(1)}+\cdots +v_{n_1}^{(1)}+v_1^{(2)}+\cdots  +v_{n_2}^{(2)}\right)
 \leq n_1+n_2,\\
=0 &\text{ sinon}.
                                                               \end{array}
\right.
\end{equation}

Une étape importante consiste à montrer que l'erreur introduite en remplaçant $d_{\n}(\Sid,\overrightarrow{\v})$ par $e_{\n}(\Sid,\overrightarrow{\v})$ est suffisamment petite, ce qui  fera l'objet du lemme \ref{Estimations de Type II}. Avant cela, on s'attache à estimer asymptotiquement 
$\widehat{S}_e(\mathcal{A},\I,C^{(i)}(\m,\n,\I))$.
Les estimations de Type I combinées à la formule de Perron permettent permettent en particulier de le relier à $\widehat{S}(\mathcal{B},\I,C^{(i)}(\m,\n,\I))$  en suivant essentiellement l'argument développé au paragraphe 8 de \cite{HB01}. %ous sommes à présent en mesure d'estimer Introducing, $e(S,\overrightarrow{v})$ we are able to estimate  $
%\widehat{S}(\mathcal{A},\I,C))$ using Perron's formula and Type I estimate.

\begin{lemme}\label{estimation de Se(A,I,C)} Soit $i\in\{1,\ldots,5\}$. On a, uniformément en $\X\geq2$, 
 $(N_1,N_2)\in\mathcal{N}(\eta)$ et $v_0\geq1$,
\begin{multline*}
\sum_{\substack{N(\I)\in\mathcal{I }\\  \X^{v_0}\leq N(\I)<  \X^{(v_0+1)\xi}}}\sum_{\m,\n}\left|\vphantom{\sum_{\overrightarrow{\v}\in\iota(\mathcal{S}^{(i)}(\n,v_0))}}\widehat{S}_{e}(\mathcal{A},C^{(i)}(\m,\n,\I))\right.\\\left.-\sigma_q(F)\eta^2  \X^2\frac{\sigma_q(\I)}{N(\I)}\sum_{\overrightarrow{\v}\in\iota(\mathcal{S}^{(i)}(\n,v_0))}
\Sigma(\I,\mathcal{R}^{(\I)}(\m,\overrightarrow{\v},v_0),\overrightarrow{\v})\right|\ll 
\eta^{\frac{11}{5}}  \X^2(\log \X)^{c}
\end{multline*}
où $\Sigma(\I,\mathcal{R},\overrightarrow{v})$ est défini par (\ref{définition Sigma}).
\end{lemme}
\begin{proof}
Dans l'esprit de la preuve du lemme  3.9 de
\cite{HB01} et compte tenu du lemme \ref{non associes}, 
on commence par remplacer $w'(N(\Sid_1\Sid_2))$ par $w'\left(\frac{c(N_1,N_2)q^3  \X^3}{N(\I\Rid_1\Rid_2)}\right)$. 
Si $n_1+n_2\geq2$, on utilise successivement le lemme \ref{non associes}, 
 le lemme \ref{propriétés de w}, l'inégalité de Hölder et le lemme  \ref{somme des diviseurs} de sorte que 
\begin{align}
&
\nonumber\sum_{\substack{\I,\Rid_1,\Rid_2,\Sid_1,\Sid_2\\\I\Rid_1\Rid_2\Sid_1\Sid_2\in\mathcal{A}}}
\frac{b(\Rid_1,\Rid_2,\mathcal{R}^{(i)}(\m,\overrightarrow{\v},v_0))
}{(\xi\log \X)^{n_1+n_2}}
\left|
w'(N(\Sid_1\Sid_2))-
w'\left(\frac{c(N_1,N_2)q^3  \X^3}{N(\I\Rid_1\Rid_2)}\right)\right|
\left|\sum_{\substack{\J|\Sid _1\Sid_2\\N(\J)<\X^{\tau/2}}}\mu_{\K}(\J)\log\frac{\X^{\tau/2}}{N(\J)}\right|
\nonumber\\&\ll
\frac{\eta^{\frac{1}{4}}}{\xi^2\log \X}
\sum_{\substack{\J\in\mathcal{A}}}
\tau_{\K}(\J)^5\nonumber\\
&\ll \eta^{\frac{11}{5}}  \X^2(\log \X)^{c}.\label{approximation 1 widetilde Se}%\
\end{align}
Dans la cas où $n_1+n_2=1$, 
le lemme \ref{non associes} entraîne que la formule $w'(N(\Sid_1\Sid_2))= w'\left(
\frac{c(N_1,N_2)  \X^3}{N(\I\Rid_1\Rid_2)}\right)$ est vraie excepté lorsque l'on a 
 $N(\I\Rid_1\Rid_2)=c(N_1,N_2)q^3  \X^{3-v\xi}(1+O(\eta^{\frac{1}{4}}))$ pour $v=v_1\text{ ou }v_1+1$. 
Par suite, on peut 
reproduire l'argument développé au cours de la preuve du lemme \ref{Type II B} conduisant à l'estimation
de $\sideset{}{^{(\I)}}\sum \frac{b(\Rid_1,\Rid_2,\mathcal{R}(\m))}{N(\Rid_1\Rid_2)}$
pour en déduire que (\ref{approximation 1 widetilde Se}) est valable aussi dans ce cas. 

Définissons 
\begin{align*}
\Sigma_1(\I,\mathcal{R}^{(i)}(\m,\overrightarrow{\v},v_0)):=&
\sum_{\substack{\Rid_1,\Rid_2%(N(\Rid),N(\I))=1
}}\frac{b(\Rid_1,\Rid_2,\mathcal{R}^{(i)}(\m,\overrightarrow{\v},v_0))}
{v_1^{(1)}\cdots  v_{n_1}^{(1)}v_1^{(2)}\cdots  v_{n_2}^{(2)}(\xi\log \X)^{n_1+n_2}}w'\left(\frac{c(N_1,N_2)
  \X^3}{N(\I\Rid_1\Rid_2)}\right)\\&\times\sum_{\substack{N(\J)<
  \X^{\tau/2}}}\mu_{\K}(\J)\#\mathcal{A}_{\Rid_1\Rid_2\I\J}\log\frac{\X^{\tau/2}}{N(\J)}.
\end{align*}
Dans la mesure où $N(\I\Rid_1\Rid_2\J)\ll\X^{2-\tau/2}$, le lemme \ref{mathcalAI} et l'estimation $\alpha_q(\Rid_1\Rid_2)=1+O\left(\frac{1}{\tau   \X^{\tau}}\right)$
 entraînent l'inégalité 
\begin{align}\nonumber&\sum_{N(\I)\in\mathcal{I }} \sum_{\m}\left|\Sigma_1(\I,\mathcal{R}^{(i)}(\m,\overrightarrow{\v},v_0))
-\frac{\eta^2  \X^2}{\zeta_q(2)}
\Sigma(\I,\mathcal{R}^{(i)}(\m,\overrightarrow{\v},v_0),\overrightarrow{\v})
\sum_{\substack{N(\J)<\X^{\tau/2}}}
\frac{\alpha_q(\I\J)}{N(\I\J)}\mu_{\K}(\J)\log\frac{\X^{\tau/2}}{N(\J)}\label{définition Sigma 1}\right|\\\ll&
\frac{  \X^{2}}{v_1^{(1)}\cdots  v_{n_1}^{(1)}v_1^{(2)}\cdots  v_{n_2}^{(2)}}(\log \X)^{-B}.\end{align}
valide pour tout $B>0$.

On montre à présent que, pour tout $B>0$, on a 
\begin{align}\nonumber
&\sum_{\m}\sum_{N(\I)\in\mathcal{I }}
\frac{\Sigma(\I,\mathcal{R}^{(i)}(\m,\overrightarrow{\v},v_0),\overrightarrow{v})}{N(\I)}\left|\zeta_q(2)^{-1}
\sum_{\substack{N(\J)<\X^{\tau/2}}}
\frac{\alpha_q(\I\J)}{N(\J)}\mu_{\K}(\J)\log\left(\frac{\X^{\tau/2}}{N(\J)}\right)%\Sigma_2(\I)
-\sigma_q(F)\sigma_q(\I)\right|\\\ll&\frac{1}{v_1^{(1)}\cdots  v_{n_1}^{(1)}v_1^{(2)}\cdots  v_{n_2}^{(2)}}(\log \X)^{-B}
\label{réduction étude gammatilde}
\end{align}
ce qui conclura la preuve du lemme au vu de (\ref{estimation v1v2}).

Dans cette direction, on commence par écrire 
\begin{align}
\sum_{\substack{N(\J)<\X^{\tau/2}}}\frac{\alpha_q(\I\J)}{N(\J)}\mu_{\K}(\J)\log\left(
\frac{\X^{\tau/2}}{N(\J)}\right)=\alpha_q(\I)\sum_{m<\X^{\tau/2}}
\frac{\widehat{\gamma}_q(\I,m)}{m}\log\left(\frac{\X^{\tau/2}}{m}\right)
\end{align}
où $\widehat{\gamma}_q(\I,m)=0$ si $\alpha_q(\I)=0$ et $\widehat{\gamma}_q(\I,\cdot)$ est la fonction multiplicative définie 
 par 
\begin{align}
\widehat{\gamma}_q(\I,p^{ k })       \sum_{\substack{N(\J)=p^{ k }}}\mu_{\K}(\J)\frac{\alpha_q(\I\J)}{\alpha_q(\I)}
\label{defwidetildegamma}
\end{align}sinon. 
Compte tenu de la décomposition des idéaux premiers dans $\JK$, on observe que, pour tout premier $p$ et $k\geq4$,
\begin{equation}\widehat{\gamma}_q(\I,p^k)=0\label{gamma 0}.\end{equation}
De plus, si $p$ est un premier $q$-régulier et $k\geq1$, la formule
(\ref{écriture rho2 non singulier}) entraîne que
\begin{equation}\widehat{\gamma}_q(\I,p^{ k })=\left\{\begin{array}{ll}
-\frac{\nu_p}{1+p^{-1}}&\text{ si }k=1\text{ et }p\nmid N(\I),\\
-1&\text{ si }k=1\text{ et }p|N(\I),\\
0&\text{ sinon}\end{array}\right.\label{widehatgamma sing}\end{equation}

Soit $\I$ un idéal tel que $\alpha_q(\I)\neq0$. La formule de Perron (voir le théorème II.2.5 de  \cite{Te08}) 
permet d'établir (\ref{réduction étude gammatilde}). 
En effet, au vu de (\ref{gamma 0}) et (\ref{widehatgamma sing}), la série de Dirichlet $\zeta_{\widehat{\gamma}_q(\I,\cdot)}(s)$
définie par \begin{align*}
\zeta_{\widehat{\gamma}_q(\I,\cdot)}(s):=\sum_{n\geq 1}\frac{\widehat{\gamma}_q(\I,n)}{n^{s}},\end{align*}
est telle que $\zeta_{\widehat{\gamma}_q(\I,\cdot)}\zeta_{\K}(s)$ est absolument convergente si $\sigma>\frac{1}{2}$. Par suite, on a
\begin{align}
\sum_{\substack{N(\J)<\X^{\tau/2}}}\frac{\alpha_q(\I\J)}{\alpha_q(\I)N(\J)}\mu_{\K}(\J)\log\left(
\frac{\X^{\tau/2}}{N(\J)}\right)%\Sigma_2(\I)
&=\frac{1}{2i\pi}\int^{\frac{1}{\log \X}+i\infty}_{\frac{1}{\log \X}-i\infty}
\zeta_{\widehat{\gamma}_q(\I,\cdot)}(s+1)\X^{s\tau/2}\frac{\text{d}s}{s^2}.\label{formule de Perron}
\end{align}
Dans la mesure où $\zeta_{\K}(s)$ possède la région sans zéro standard 
\begin{equation}\label{région sans zero sigma}
\Re(s)\geq 1-\frac{c}{\log (2+|\Im(s)|)} 
\end{equation}
pour un certain $c>0$, 
dans laquelle on a
\begin{equation}\label{estimation zêta}
\zeta_{\K}(s)^{-1}\ll\log(1+|s|), 
\end{equation}
on peut 
appliquer le théorème des résidus, ce qui conduit à la formule
\begin{align*}
\sum_{\substack{N(\J)<\X^{\tau/2}}}\frac{\alpha_q(\I\J)}{\alpha_q(\I)N(\J)}\mu_{\K}(\J)\log\left(
\frac{\X^{\tau/2}}{N(\J)}\right)&=
\text{Res}\left(\zeta_{\widehat{\gamma}_q(\I,\cdot)}(s+1)\X^{s\tau/2}s^{-2}\right)_{s=0}+R_2(\I)
\\
&=\lambda_{\K}^{-1}(\zeta_{\widehat{\gamma}_q(\I,\cdot)}\zeta_{\K})(1)+R(\I)
\end{align*}
où
\begin{align*}
R(\I):=\frac{1}{2i\pi}\int_{\mathcal{V}}\zeta_{\widehat{\gamma}_q(\I,\cdot)}(s+1)\X^{s\tau/2}\frac{\text{d}s}{s^2}
\end{align*}
et $\mathcal{V}$ est le chemin de sommets
$\frac{1}{\log \X}-i\infty$, $\frac{1}{\log \X}-iT$, 
 $-\frac{c}{\log T}-iT$, $-\frac{c}{\log T}+iT$, 
  $\frac{1}{\log \X}+iT$, $\frac{1}{\log \X}+i\infty$. 
où $T=\exp(\sqrt{\log \X^{\tau/2}})$.
En calculant directement le produit eulérien, on observe que 
\begin{align*}
\frac{\alpha_q(\I)}{\zeta_q(2)\lambda_{\K}}(\zeta_{\widehat{\gamma}_q(\I,\cdot)}\zeta_{\K})(1)=\sigma_q(F)\sigma_q(\I).
\end{align*}
Pour traiter $R(\I)$, on remarque à l'aide de (\ref{defwidetildegamma}), (\ref{inégalité alpha singulier}) 
et (\ref{widehatgamma sing}) que l'on a, dans la région (\ref{région sans zero sigma}), 
\begin{align*}
 \alpha_q(\I)(\zeta_{\widehat{\gamma}_q(\I,\cdot)}\zeta_{\K})(s)&\ll\alpha_q(\I)
 \prod_{p|qN(\I)}\left(1+O\left(\frac{1}{p^{\Re(s)}}\right)\right)
 \prod_{p\quad q\text{-singulier}}\left(\sum_{k=0}^3\frac{ \gamma_q(\I_p,p^k)}{p^{k\Re(s)}}\right)
 \\&\ll(\log \X)^c\tau_{\K}(\I)^c
 N(\I_{q\text{-s}})^{1/3}.
 \end{align*}
Il suit alors de (\ref{estimation zêta}) que \begin{align*}
\alpha_q(\I)R(\I)&\ll\int_{\mathcal{V}}|\zeta_{\widetilde{\gamma}_q(\I,\cdot)}\zeta_{\K}(s+1)||\zeta_{\K}^{-1}(s+1)|
L^s\frac{\text{d}s}{s^2}\\
&\ll
\tau_{\K}(\I)^c
N(\I_{q\text{-s}})^{1/3}\exp\left(-c\sqrt{\log \X^{\tau/2}}\right).
\end{align*}
En sommant sur $\I$, on obtient alors que, pour tout $\varepsilon>0$, on a 
\begin{align}
\nonumber\sum_{\substack{ N(\I)\ll   \X^2}}\frac{\alpha_q(\I)}{N(\I)}R(\I)
&\ll \exp\left(-c\sqrt{\log \X^{\tau/2}}\right)
\sum_{\substack{N(\I_1)\ll   \X^2\\\I_1\text{ } q\text{-régulier}}}
\frac{\tau_{\K}(\I_1)^c}{N(\I_1)}
\sum_{\substack{\I_2 \text{ }q\text{-singulier}}}
\frac{\tau_{\K}(\I_2)^c}{N(\I_2)^{2/3}}\\
\label{R2}
&\ll \exp\left(-c\sqrt{\log \X^{\tau/2}}\right)
\end{align}
ce qui établit (\ref{réduction étude gammatilde}). 
\end{proof}

L'estimation du terme d'erreur qui apparaît avec le remplacement de $d_{\n}(\Sid,\overrightarrow{\v})$ par $e_{\n}(\Sid,\overrightarrow{\v})$ 
a constitué le tour de force majeur des travaux de Heath-Brown et Moroz  (voir notamment les paragraphes 8 et  10 à 12 de \cite{HB01}, les paragraphes 6 et 7 de \cite{HM02} et les paragraphes  4 et 5 de \cite{HM04}). 
 De telles estimations, dites de Type II, constituent l'ingrédient nécessaire pour s'affranchir du phénomène
 de parité.

\begin{lemme}\label{Estimations de Type II}
Soient $B\geq0$, $C$ une classe d'idéaux,  
$h_1:\JK\rightarrow\C$ et $b:\JK^2\rightarrow\C$ deux fonctions à valeurs dans le disque unité.
On a, uniformément en $\X\geq2$, $  \X^{1+\tau}\leq V\leq   \X^{\frac{3}{2}-\tau}$ et $%(\log \X)^{-c_2}\leq
\xi \leq \tau$,
 \begin{align*}
S_V:=\sum_{\substack{\Sid\in C\cap \JK\\\I,\Rid_1,\Rid_2\in\JK\\\I\Rid_1\Rid_2\Sid\in\mathcal{A}\\V<N(\Sid)\leq 2V
}}h_1(\I)
b(\Rid_1,\Rid_2)(d_{\n}(\Sid,\overrightarrow{\v})-e_{\n}(\Sid,\overrightarrow{\v}))\ll   \X^2(\log \X)^{-B}
 \end{align*}
où la constante implicite dépend du zéro de Siegel et de la classe $C$.
\end{lemme}
\begin{proof} 
D'après la discussion algébrique du paragraphe \ref{Algèbre} et le théorème des facteurs invariants, il existe des entiers
$z_1,z_2\geq1$ et une base entière $(w_1,w_2,w_3)$ de $\Cl (\delta^{-1})$ tels que $z_1|z_2$ et
$(z_1w_1\delta,z_2w_2\delta)$ soit
une base entière de $\Lambda(\omega_1,\omega_2)$. Au vu de la formule (\ref{def f}), on peut supposer sans perte de généralité que $z_1=1$.
En considérant le changement de base ainsi induit, le problème général se réduit à considérer les valeurs
\begin{align*}
 \frac{N_{\K/\Q}\left((a'_1+n'_1q)w_1\delta+(a'_2+n'_2q)z_2w_2\delta)\right)}{N(\mathfrak{d})}
\end{align*}
avec $(a'_1,a'_2,q)=1$ sous les conditions $(n'_1,n'_2)\in \comp'.X$ où $\comp'$ est de même nature que $\comp$ 
et
$(a'_1+n'_1q,a'_2+n'_2q)=1$.
En scindant la région en $n_1'$ en classes de congruences modulo $z_2$, on se ramène ainsi à étudier 
\begin{align*}
 \frac{N_{\K/\Q}\left((a''_1+n'_1qz_2)w_1\delta+(a'_2z_2+n'_2qz_2)w_2\delta)\right)}{N(\mathfrak{d})}.
\end{align*}
Par un argument de convolution, on peut finalement supposer que 
$(a''_1+n'_1qz_2,a'_2z_2+n'_2qz_2)=1$. En résumé, on peut supposer sans perte de généralité que $F$ est défini par (\ref{def f}) et qu'il 
existe une base entière $(w_1,w_2,w_3)$ de $\Cl (\delta^{-1})$ telle que $w_i\delta=\omega_i$ pour $i=1,2$.

Grâce à ces réductions et en utilisant la définition de $(w_1,w_2,w_3)$,
on remarque que l'existence d'un premier $p$ tel que 
$p\OK|(n_1\omega_1+n_2\omega_2)\mathfrak{d}^{-1}$ entraînerait que $p|(n_1,n_2)$. Par suite, on a $\mathcal{A}\subset\prim$ 
où $\prim$ désigne l'ensemble des idéaux primitifs, à savoir
\begin{align*}
\prim:=\left\{\J\in\JK:p\OK\nmid \J\text{ pour tout }p\text{ premier}\right\}. 
\end{align*}

En vue de s'affranchir de la condition de coprimalité dans la définition de 
$\mathcal{A}$, on peut écrire par le principe d'inclusion-exclusion 
\begin{align*}
S_V&
=\sum_{\substack{\I,\Rid_1,\Rid_2\in\JK\\\I\Rid_1\Rid_2\in\prim}}h_1(\I)b(\Rid_1,\Rid_2)
\sum_{\substack{\Sid\in\prim\cap C\\V<N(\Sid)\leq 2V \\\I\Rid_1\Rid_2\Sid\in\mathcal{A}}}
f_{\n}(\Sid,\overrightarrow{\v})\\&
=\sum_{m\geq1}\mu(m)S_V(m)
\end{align*}
où \begin{align}\label{définition fnsV}f_{\n}(\Sid,\overrightarrow{\v}):=d_{\n}(\Sid,\overrightarrow{\v})-e_{\n}(\Sid,\overrightarrow{\v}),\end{align}
\begin{align*}
S_V(m):=\sum_{\substack{\I,\Rid_1,\Rid_2\in\JK\\\I\Rid_1\Rid_2\in\prim}}h_1(\I)b(\Rid_1,\Rid_2)
\sum_{\substack{\Sid\in\prim\cap C\\V<N(\Sid)\leq 2V \\\I\Rid_1\Rid_2\Sid\in\mathcal{A}(m)}}
f_{\n}(\Sid,\overrightarrow{\v})
   \end{align*}
et
\begin{align*}
\mathcal{A}(m)=
\left\{((a_1+n_1q)\omega_1+(a_2+n_2q)\omega_2)\mathfrak{d}^{-1}:
(n_1,n_2)\in \mathcal{C}(N_1,N_2):m|(a_1+n_1q,a_2+n_2q)\right\}.
\end{align*}

Soit $M\geq1$ un paramètre qui sera explicité plus tard.
Au vu de (\ref{estimation en dn}), du lemme \ref{somme des diviseurs} et de la borne sur $b$ et $h$,
on peut estimer la contribution des entiers $m>M$ en s'inspirant de la majoration  (11.2) de \cite{HB01} par \begin{align}
 \sum_{m> M}\left|\sum_{\Rid_1\Rid_2\I\in\prim}h_1(\I)b(\Rid_1,\Rid_2)
 \sum_{\substack{\Sid\in\prim\cap C\\V<N(\Sid)\leq 2V \\\Rid_1\Rid_2\I\Sid\in\mathcal{A}(m)}}f_{\n}(\Sid,\overrightarrow{\v})
\right|\nonumber
&\leq \log \X\sum_{m>M}\sum_{\J\in\mathcal{A}(m)}\tau_{\K}(\J)^4\\&
\ll\frac{  \X^2}{M%^{1/2}
}(\log \X)^{c}.
\end{align}

Pour traiter les entiers $m\leq M$, on suit essentiellement la démarche développée dans 
la preuve du lemme 3.10 de \cite{HB01}, de la proposition 6.1 de \cite{HM02} ainsi que de la proposition 4.2
de \cite{HM04}.
En appliquant l'inégalité de Cauchy-Schwarz, on observe tout d'abord que
\begin{align*}
S_V(m)&\leq \left(\sum_{N(\J)\ll \frac{q^3  \X^3}{V}}\tau_{\K}(\J)^4 \right)^{\frac{1}{2}}
S_V^{(0)}(m)^{\frac{1}{2}}\\
&\ll\left(\frac{  \X^3}{V}\right)^{\frac{1}{2}}S^{(0)}_V(m)^{\frac{1}{2}}(\log \X)^{c}
\end{align*} où 
\begin{align*}
S^{(0)}_V(m):=\sum_{\J\in\prim}
\left|\sum_{\substack{\Sid\in\prim\cap C\\V<N(\Sid)\leq 2V\\ \J\Sid\in\mathcal{A}(m)}}
f_{\n}(\Sid,\overrightarrow{\v}
)\right|^2.
\end{align*}
En développant le carré et en isolant les termes diagonaux, on peut écrire la décomposition 
$S^{(0)}_V(m)=S^{(1)}_V(m)+S_V^{(2)}(m)$ avec
\begin{align*}
&S_V^{(1)}(m):= \sum_{\J\in\prim}
\sum_{\substack{\Sid\in\prim\cap C\\V<N(\Sid)\leq 2V\\\J\Sid\in\mathcal{A}(m)}}f_{\n}(\Sid,\overrightarrow{\v}
)^2,\\\qquad
 &S_V^{(2)}(m):=\sum_{\substack{\Sid_1,\Sid_2\in\prim\cap C}}f_{\n}(\Sid_1,\overrightarrow{\v})f_{\n}(\Sid_2,\overrightarrow{\v})\delta_{V,m}(\Sid_1,\Sid_2)
\end{align*}
et 
\begin{align*}
\delta_{V,m}(\Sid_1,\Sid_2):=\left\{\begin{array}{ll}\#
\left\{\J%\in\JK
:\J\Sid_1,\J\Sid_2\in\mathcal{A}(m)\right\}
&\text{si }V<N(\Sid_1),N(\Sid_2)\leq 2V\text{ et }\Sid_1\neq \Sid_2,\\
0&\text{sinon}.
                          \end{array}\right.
\end{align*}
En utilisant
(\ref{estimation en dn}) et le lemme \ref{somme des diviseurs}, on a la majoration de $S^{(1)}_V(m)$ suivante 
\begin{align}
S_V^{(1)}(m)\ll(\log \X)^2\tau(m)^{c}\sum_{n_1,n_2\ll \frac{qX}{m}}\tau_{\K}(n_1\omega_1+n_2\omega_2)^3
\ll \frac{\tau(m)^c}{m^2}  \X^2(\log \X)^{c}.\label{estimation SV termes diagonaux} 
\end{align}

Une difficulté essentielle qui apparaît dans l'estimation d'une telle quantité réside dans le caractère éparse de $\mathcal{A}(m)$.
Une idée originale développée par Heath-Brown et Moroz consiste à  réécrire $\delta_{V,m}(\Sid_1,\Sid_2)$ 
en utilisant le fait que $\mathcal{A}(m)$ est inclus dans un sous-module de $\Cl (\delta^{-1})$ de rang $2$.  

Pour chaque idéal $\Sid$, on considère le représentant de $\Sid$ introduit dans \cite{HM02}, 
 à savoir l'unique entier algébrique 
 $\beta(\Sid)$ de $\IK$ associé à $\Sid$ par l'isomorphisme  $\IK/\OK^*\simeq G(K)$  étudié dans le paragraphe \ref{Algèbre} 
 et satisfaisant en outre, si $\rK=1$,
\begin{equation}\label{choix générateur 1}
|\beta(\Sid)|=
N(\beta(\Sid))^{\frac{1}{3}}\varepsilon_1^z \text{ où } z\in\left]-\frac{1}{2},\frac{1}{2}\right]
\end{equation} et, si $\rK=3$,
\begin{equation}\label{choix générateur 2}
|\sigma_j(\beta(\Sid))|=N(\beta(\Sid))^{\frac{1}{3}}|\sigma_j(\varepsilon_1)|^{z_1}
|\sigma_j(\varepsilon_2)|^{z_2}\text{ où }
z_1,z_2\in\left]-\frac{1}{2},\frac{1}{2}\right]\text{ et }\J\in\{1,2,3\},
\end{equation} 
$\varepsilon_1$ et $\varepsilon_2$ désignant les unités fondamentales introduites dans le paragraphe \ref{Algèbre}.

L'application trace $\text{Tr}$  facilite la détection de la condition 
$\J\Sid\in\mathcal{A}(m)$. 
 En effet, étant donné la définition (\ref{base duale}) de la base duale $(\widetilde{w}_1,\widetilde{w}_2,\widetilde{w}_3)$, 
 l'idéal
 $\J\Sid$ est de la forme $(n_1w_1+n_2w_2)$ si et seulement s'il existe $\alpha\in\IK$ tel que $(\alpha)=\J$ et
\begin{equation}\label{trace nulle}
\text{Tr}(\alpha\beta(\Sid)\widetilde{w}_3)=0.
\end{equation}
On fixe une base $(v_1,v_2,v_3)$  de $(C\Cl (\delta))^{-1}$. Si  
 $\alpha=\alpha_1v_1+\alpha_2v_2+\alpha_3v_3$ pour des entiers $\alpha_1$, $\alpha_2$ et $\alpha_3$, 
 alors l'identité (\ref{trace nulle}) est équivalente à la relation linéaire
\begin{equation}\label{produit vectoriel nul}
\sum_{j=1}^3\alpha_j\text{Tr}(v_j\beta(\Sid)\widetilde{w}_3)=0.
\end{equation}

Étant donné deux idéaux $\Sid_1$ et $\Sid_2$, on pose $\beta_i=\beta(\Sid_i)$ pour $i=1,2$ et on introduit
 le système de coordonnées 
 \begin{align*}\widehat{\beta}_i:=
 (\text{Tr}(v_1\beta_i\widetilde{w}_3),
 \text{Tr}(v_2\beta_i\widetilde{w}_3),
 \text{Tr}(v_3\beta_i\widetilde{w}_3))\quad
 \text{ et }\quad\widehat{\alpha}:=(\alpha_1,\alpha_2,\alpha_3).\end{align*}
Supposons que $\J\Sid_i=((a_1+n_1^{(\I)}q)\omega_1+(a_2+n_2^{(\I)}q)\omega_2)\mathfrak{d}^{-1}$ pour $i=1,2$.
En écrivant les équations (\ref{produit vectoriel nul}) associées à $\Sid_1$ et $\Sid_2$, on obtient
un système linéaire de rang 2 (puisque $\Sid_1\neq \Sid_2$) dont la résolution permet d'exprimer $\alpha$
en fonction de $\beta_1$ et $\beta_2$,  compte tenu de la primitivité de
$\alpha$. On a ainsi 
\begin{equation}\label{écriture alpha}
\widehat{\alpha}= \pm \frac{\widehat{\beta}_1\wedge\widehat{\beta}_2}{D(\widehat{\beta}_1\wedge\widehat{\beta}_2)}
\end{equation}
où $D(b_1,b_2,b_3)=\text{p.g.c.d}(b_1,b_2,b_3)$. 
On peut alors retrouver la condition $m|(a_1+n_1^{(\I)}q,a_2+n_2^{(\I)}q)$ en remarquant que (\ref{écriture alpha})
implique que
\begin{align*}
a_j+n_j^{(\I)}q=\left| D(\widehat{\beta}_1\wedge\widehat{\beta}_2)^{-1} h_{ij}(\widehat{\beta}_1,\widehat{\beta}_2)\right|
\end{align*}
où $h_{ij}$ désigne la forme cubique à coefficients rationnels définie par la formule (6.4) 
de \cite{HM02}, à savoir
\begin{equation}\label{définition coordonnées}
h_{ij}(\widehat{\beta}_1,\widehat{\beta}_2):=\text{Tr}\left(\beta_i\sum_{k=1}^3(\widehat{\beta}_1\wedge\widehat{\beta}_2)_kv_k\widetilde{w}_j\right).
\end{equation}

Au vu de ces préliminaires, 
on dispose des analogues des formules (11.5), (11,6) et (11.7) de \cite{HB01}, à savoir
\begin{equation}\label{estimation beta chapeau}
|\widehat{\beta}_1|,|\widehat{\beta}_2|\asymp V^{\frac{1}{3}},
\qquad  |\widehat{\alpha}|\asymp q\X V^{-\frac{1}{3}}\text{ et }D(\widehat{\beta}_1\wedge\widehat{\beta}_2)\ll V  \X^{-1}.
\end{equation}

On étend finalement la définition de $f_{\n}(\cdot,\overrightarrow{\v})$ et $\delta_{V,m}$  
aux entiers algébriques de $\IK$  en posant
\begin{align*}
f_{\n}(\beta,\overrightarrow{\v}):=\left\{\begin{array}{ll}
f_{\n}(\Sid,\overrightarrow{\v})&\text{ si }S\text{ est primitif et } \beta(\Sid)=\beta,\\
0        &\text{ sinon},
\end{array}
\right.
\end{align*}

\begin{align*}
\delta_{V,m}(\beta_1,\beta_2):=\left\{\begin{array}{ll}
\delta_{V,m}(\Sid_1,\Sid_2)&\text{ si }\beta_i=\beta(\Sid_i)\text{ pour }i=1,2,\\
0&\text{ sinon},
                          \end{array}\right.
\end{align*}
et on définit la norme $N$ sur $\R^3$ comme le polynôme 
 dont les valeurs en $\overrightarrow{b}\in\Q^3$ satisfont $
 N(\overrightarrow{b})=N(\beta(\overrightarrow{b}))$
où $\beta(\overrightarrow{b})$ désigne l'entier algébrique de $C$ satisfaisant 
 $\widehat{\beta(\overrightarrow{b})}=\overrightarrow{b}$.

Soit $\Y$ un paramètre qui sera explicité plus tard tel que 
$q^{\frac{1}{2}}\leq \Y \ll   \X^{\frac{\tau}{3}}$. La contribution des vecteurs satisfaisant $D(\widehat{\beta}_1\wedge\widehat{\beta}_2)\leq 
V  \X^{-1}\Y^{-1}$ peut être traitée par un argument  de découpage d'une région de $\R^3$ en cubes 
suffisamment petits et en classes de congruences (voir le lemme 11.1 de \cite{HB01}). 
En reproduisant la preuve du lemme 11.2 de \cite{HB01}, l'inégalité 
$f_{\beta_1}f_{\beta_2}\leq\left(f_{\beta_1}^2+f_{\beta_2}^2\right)$ implique l'estimation suivante, 
analogue de la majoration de $S_3$ de \cite{HB01}, pp. 70--71, 
\begin{align*}
S_V^{(3 )}(m)&:=\sum_{\substack{(\beta_1),(\beta_2)\in 
\prim\cap C\\D(\widehat{\beta}_1\wedge\widehat{\beta}_2)\leq V  \X^{-1}\Y^{-1}}}
f_{\n}(\beta_1,\overrightarrow{\v})f_{\n}(\beta_2,\overrightarrow{\v})\delta_{V,m}(\beta_1,\beta_2)
\\
&\ll(\log \X)^2\sum_{\substack{\beta_1,\beta_2\\D(\widehat{\beta}_1\wedge\widehat{\beta}_2)
\leq V  \X^{-1}\Y^{-1}}}\tau_{\K}((\beta_1))^2\delta_{V,m}(\beta_1,\beta_2)\ll  V\X\Y^{-1}(\log \X)^{c}.\end{align*}

L'estimation de la contribution des termes $\beta_1$ et $\beta_2$
tels que  $D(\widehat{\beta}_1\wedge\widehat{\beta}_2)> V  \X^{-1}\Y^{-1}$ consiste principalement à 
reprendre les arguments de la preuve du lemme 4.3 et du paragraphe 5 de \cite{HM04} 
en tenant compte de la dépendance en $m$ et $q$, pour obtenir un résultat uniforme en $q\leq (\log \X)^A$.
On est donc conduit à estimer 
 la quantité 
\begin{align*}
 S_V^{(4 )}(m):=\sum_{\substack{(\beta_1),(\beta_2)\in \prim\cap C \\D(\widehat{\beta}_1\wedge
 \widehat{\beta}_2)> V  \X^{-1}\Y^{-1}}}f_{\n}(\beta_1,\overrightarrow{\v})
 f_{\n}(\beta_2,\overrightarrow{\v})\delta_{V,m}(\beta_1,\beta_2).
\end{align*} 
En  découpant la région de sommation de $D(\widehat{\beta_1}\wedge\widehat{\beta_2})$ 
en intervalles $]\Delta,2\Delta]$ où 
 $\Delta$ parcourt des puissances de $2$, puis en sous-intervalles $\left](\z-1)Y^{-2}
 \Delta,\z\Y^{-2}\Delta\right]$, on observe l'existence de 
 $V  \X^{-1}\Y^{-1}<\Delta\ll V  \X^{-1}$ et $\Y^2<\z\leq 2\Y^2$ tels que l'on ait 
\begin{align*}
 S_V^{(4 )}(m)\ll S_V^{(5 )}(m,\z,\Delta)\Y^2\log \X 
\end{align*}
où
\begin{align*}
 S_V^{(5 )}(m,\z,\Delta):=\sum_{D\in \left]\frac{\z-1}{\Y^2}
 \Delta,\frac{\z}{\Y^2}\Delta\right]}
 \left|\sum_{\substack{(\beta_1),(\beta_2)\in\prim\cap C \\
 D(\widehat{\beta}_1\wedge\widehat{\beta}_2)=D}}
 f_{\n}(\beta_1,\overrightarrow{\v}) f_{\n}(\beta_2,\overrightarrow{\v})
 \delta_{V,m}(\beta_1,\beta_2)\right|.
\end{align*}
Au vu de ce qui précède, on peut écrire $\delta_{V,m}(\beta_1,\beta_2)=1_{\mathcal{R}_{D,V}}(\widehat{\beta}_1,\widehat{\beta}_2)\delta_{q,m}(\widehat{\beta}_1,\widehat{\beta}_2,D)$ où %peut être considéré comme l'indicatrice de l'ensemble 
$\mathcal{R}_{D,V}$ est l'ensemble  
%essentiellement défini par des inégalités polynomiales
 des $(\widehat{\beta}_1,\widehat{\beta}_2)$ satisfaisant les contraintes (\ref{choix générateur 1}) (resp. 
 (\ref{choix générateur 2})) si $\rK=1$ (resp. $\rK=3$) ainsi que les conditions
\begin{align*}
V<N(\widehat{\beta}_1),N(\widehat{\beta}_2)\leq 2V\text{ et } 0< \left|\frac{h_{ij}(\widehat{\beta}_1,\widehat{\beta}_2)}{D}\right|-(a_j+q\X(1-N_j
\eta))\leq q\eta \X \text{ pour }i,j=1,2,
\end{align*}
tandis que $\delta_{q,m}$ est l'indicatrice définie par
\begin{align*}
 \delta_{q,m}(\widehat{\beta}_1,\widehat{\beta}_2,D)=\left\{\begin{array}{ll}
                                             1&\text{ si }mD|h_{ij}(\widehat{\beta}_1,\widehat{\beta}_2)\text{ et }D^{-1}| h_{ij}(\widehat{\beta}_1,\widehat{\beta}_2)|
=a_j\mod q,\\
0&\text{ sinon}.
                                            \end{array}
\right.
\end{align*}

Au regard de (\ref{estimation beta chapeau}), on recouvre $\mathcal{R}_{D,V}$ en 
domaines du type  $\mathfrak{C}:=\mathfrak{C}_1\times\mathfrak{C}_2$ où 
$\mathfrak{C}_1$ et $\mathfrak{C}_2$ sont des cubes 
de la forme 
\begin{equation}\label{cube hypersurface}\prod_{j=1}^3\left]V^{1/3}\frac{n_j-1}{\Y^2},
V^{1/3}\frac{n_j}{\Y^2}\right]\quad\text{avec }n_1,n_2,n_3\ll \Y^2.\end{equation}

 Considérons les $\mathfrak{C}$ intersectant la frontière de la région $\mathcal{R}_{D,V}$ 
 (dits de classe II dans les travaux de Heath-Brown et Moroz) 
 ou dans lesquels $h_{ij}$ change de signe. 
 
Lorsque $\rK=1$, on réitère l'argument de comptage de points sur une 
 hypersurface développé dans le paragraphe 12 de \cite{HB01} et le paragraphe 6 de \cite{HM02}.
Pour éclaircir l'exposé, on s'attarde 
aux cubes $\mathfrak{C}_1(n_1,n_2,n_3)\times\mathfrak{C}_2(n_4,n_5,n_6)$ sus-cités qui possèdent un point ne vérifiant pas l'inégalité
\begin{align*}0\leq 
\left|\frac{h_{ij}(\widehat{\beta}_1,\widehat{\beta}_2)}{D}\right|-(a_j+q\X(1-N_j
\eta)).%\leq q\eta \X .
\end{align*}
pour un $D\in\left](\z-1)\Y^{-2}
 \Delta,\z\Y^{-2}\Delta\right]$.
Puisque $D=\z\Y^{-2}\Delta+O\left(\Y^{-2}\Delta\right)$, il s'ensuit que, pour les cubes considérés, on a
\begin{align*}
 h_{ij}(n_1,\ldots, n_6)=\Delta \Y^6V^{-1}(a_j+q\X(1-N_j\eta))+O\left(q\Y^4\right)
\end{align*}
et $\nabla h(n_1,\ldots,n_6)\gg \Y^4 $. En appliquant  alors
 le lemme 4.9 de \cite{HB01} avec les choix $S_0=q$ et $R=R_0=\Y^2$ et en utilisant l'hypothèse $q\leq \Y ^2$,
 on en déduit qu'il y a $O\left(\Y^{10}q^{-5}\right)$ hypercubes de côté $q$ contenant de tels
 $(n_1,\ldots,n_6)$.
 Par suite, il y a  $O\left(q\Y^{10}\right)$ cubes $\mathfrak{C}_1\times\mathfrak{C}_2$ intersectant la frontière de $\mathcal{R}_{D,V}$ ou sur lesquels $h_{ij}$ 
 changent de signe. 
 
 Le cas $\rK=3$ est traité de la même manière, à l'exception de  l'inégalité (\ref{choix générateur 2}). 
 Comme il est indiqué p. 282 de \cite{HM02}, le problème peut alors se ramener à compter les cubes $\mathfrak{C}$
 contenant des couples $(\widehat{\beta}_1,\widehat{\beta}_2)$ satisfaisant une équation du type
 \begin{align*}
  A\log\left|\sigma_1(\beta_i)\right|+ B\log\left|\sigma_2(\beta_i)\right|+ C\log\left|\sigma_3(\beta_i)\right|=0 \qquad i=1\text{ ou }2
 \end{align*}
pour des réels $A$, $B$ et $C$. 

 En utilisant (\ref{estimation en dn}) et le lemme 11.1 de \cite{HB01} pour majorer la contribution des cubes précédemment étudiés, 
on en déduit qu'il existe deux cubes $\mathfrak{C}_1$ et 
 $\mathfrak{C}_2$ de la forme (\ref{cube hypersurface}) tels que 
 $\mathfrak{C}_1\times\mathfrak{C}_2$ soit inclus dans
 $\mathcal{R}_{D,V}$ pour tout $D\in \left](\z-1)\Y^{-2} \Delta,\z\Y^{-2}\Delta\right]$, $h_{ij}$
 soit de signe constant (disons positif) sur $\mathfrak{C}_1\times\mathfrak{C}_2$ et
\begin{align*}
 S_V^{(5 )}(m,\z,\Delta)\ll V\X\Y^{-3}(\log \X)^{c} + \Y^{12}
S_V^{(6 )}(m,\z,\Delta,\mathfrak{C}_1\times\mathfrak{C}_2)
\end{align*}
où 
\begin{align*}
 S_V^{(6 )}(m,\z,\Delta,\mathfrak{C}_1\times\mathfrak{C}_2):=
 \sum_{D\in \left]\frac{\z-1}{\Y^2}\Delta,\frac{\z}{\Y^2}\Delta\right]}
 \left|\sum_{\substack{(\beta_1),(\beta_2)\in \prim\cap C \\\widehat{\beta}_1,\widehat{\beta}_2\in \mathfrak{C}_1
 \times\mathfrak{C}_2 \\D(\widehat{\beta}_1\wedge\widehat{\beta}_2)=D}}
 f_{\n}(\beta_1,\overrightarrow{\v}) f_{\n}(\beta_2,\overrightarrow{\v})
 \delta_{q,m}(\widehat{\beta}_1,\widehat{\beta}_2,D)\right|.
\end{align*}
Par convolution, on peut réécrire $S_V^{(6 )}(m,\z,\Delta,\mathfrak{C}_1\times\mathfrak{C}_2)$ sous la forme
\begin{align*} 
& S_V^{(6 )}(m,\z,\Delta,\mathfrak{C}_1\times\mathfrak{C}_2))\\=&
 \sum_{D\in \left]\frac{\z-1}{\Y^2}\Delta,\frac{\z}{\Y^2}\Delta\right]}
 \left|\sum_{d\geq1}\mu(d)\sum_{\substack{(\beta_1),(\beta_2)\in  \prim\cap C
 \\\widehat{\beta}_1,
 \widehat{\beta}_2\in \mathfrak{C}_1\times\mathfrak{C}_2 \\dD|
 \widehat{\beta}_1\wedge\widehat{\beta}_2}}
 f_{\n}(\beta_1,\overrightarrow{\v}) f_{\n}(\beta_2,\overrightarrow{\v})\delta_{q,m}(\widehat{\beta}_1,\widehat{\beta}_2,D)\right|.
\end{align*}
Dans le paragraphe 12 de \cite{HB01} et p. 282 de \cite{HM02}, Heath-Brown et Moroz développent
un argument de réseau qui permet d'estimer la contribution des entiers satisfaisant 
$dD>d_0$ où $d_0=\Y^{15}V  \X^{-1}+V^{\frac{1}{6}}$. Ils obtiennent en particulier l'estimation
\begin{align*}
 \sum_{D\in \left]\frac{\z-1}{\Y^2}\Delta,\frac{\z}{\Y^2}\Delta\right]}
 \sum_{dD\geq d_0}\sum_{\substack{(\beta_1),(\beta_2)\in  \prim
 \\\widehat{\beta}_1,
 \widehat{\beta}_2\in \mathfrak{C}_1\times\mathfrak{C}_2 \\dD|
 \widehat{\beta}_1\wedge\widehat{\beta}_2}}
 \left|f_{\n}(\beta_1,\overrightarrow{\v}) f_{\n}(\beta_2,\overrightarrow{\v})
\right|\ll V\X\Y^{-15}(\log \X)^c.
\end{align*}
On en déduit donc que
\begin{align*} 
 S_V^{(6 )}(m,\z,\Delta,\mathfrak{C}_1\times\mathfrak{C}_2)\ll V\X\Y^{-15}(\log \X)^{c}+S_V^{(7 )}
 (m,\z,\Delta,\mathfrak{C}_1\times\mathfrak{C}_2)
\end{align*}
où  $S_V^{(7 )}(m,\z,\Delta,\mathfrak{C}_1\times\mathfrak{C}_2)$ est la somme définie par 
\begin{align*} \sum_{D\in \left]\frac{\z-1}{\Y^2}\Delta,\frac{\z}{\Y^2}\Delta\right]}\left|\sum_{dD\leq d_0}
 \mu(d)\sum_{\substack{(\beta_1),(\beta_2)\in  \prim\cap  C\\\widehat{\beta}_1,\widehat{\beta}_2\in
 \mathfrak{C}_1
 \times\mathfrak{C}_2 \\dD|\widehat{\beta}_1\wedge\widehat{\beta}_2}}
f_{\n}(\beta_1,\overrightarrow{\v}) f_{\n}(\beta_2,\overrightarrow{\v})\delta_{q,m}
(\widehat{\beta}_1,\widehat{\beta}_2,D)\right|.
\end{align*}

Soit $d^*$ le dénominateur commun des $h_{ij}$. 
On peut écrire $h_{ij}=(d^*)^{-1}%\nu_{ij}
H_{ij}$ où 
$H_{ij}$ est 
une forme cubique à coefficients entiers 
 qui satisfait $H_{ij}(\widehat{\beta},\lambda\widehat{\beta})=0$ au vu de la définition 
 (\ref{définition coordonnées}).
En introduisant des caractères additifs, on peut reformuler, en vue de l'utilisation  d'une inégalité de grand crible,
les congruences
définissant $\delta_{q,m}(\cdot,\cdot)$ ainsi que la condition  $dD|\widehat{\beta}_1\wedge\widehat{\beta}_2$.
 On obtient ainsi l'analogue de la formule (5.2) de \cite{HM04} suivante 
\begin{align*}
  S_V^{(7 )}(m,\z,\Delta,\mathfrak{C}_1\times\mathfrak{C}_2)=
  \sum_{D\in \left]\frac{\z-1}{\Y^2}\Delta,\frac{\z}{\Y^2}\Delta\right]}
  \left|\sum_{dD\leq d_0}\mu(d)S^{(8 )}(m,d,D,\mathfrak{C}_1\times\mathfrak{C}_2)\right|,
\end{align*}
où
\begin{align*}
S^{(8 )}(m,d,D,\mathfrak{C}_1\times\mathfrak{C}_2):=&
\frac{1}{(d^*qmdD)^6}
\sideset{}{^{*}}\sum_{\substack{\overrightarrow{b}_1,\overrightarrow{b_2},\overrightarrow{\eta}_1,\overrightarrow{\eta_2}
\in (\Z/(d^*qmdD)\Z)^3}}
\e\left(-\frac{\overrightarrow{b}_1.\widehat{\eta}_1+
\overrightarrow{b}_2.\widehat{\eta}_2}{d^*qmdD}\right)\\&\times S\left(d^*qmdD,\overrightarrow{b}_1,\mathfrak{C}_1\right)
S\left(d^*qmdD,\overrightarrow{b}_2,\mathfrak{C}_2\right)
\end{align*}
où $\sideset{}{^{*}}\sum$ indique que la sommation  porte sur les vecteurs $\widehat{\eta}_1$, $\widehat{\eta}_2$ primitifs 
satisfaisant
\begin{align*}
dD|\widehat{\eta}_1\wedge\widehat{\eta}_2,\quad 
H_{ij}\left(\widehat{\eta}_1,\widehat{\eta}_2\right)\equiv Dd^*a_i\pmod{qDd^*}\text{ et }
H_{ij}\left(\widehat{\eta}_1,\widehat{\eta}_2\right)\equiv 0\pmod{mDd^*},
\end{align*}
et\begin{align*}
 S\left(l,\overrightarrow{b},\mathfrak{C}\right)=
 \sum_{\substack{\widehat{\beta}\in\mathfrak{C}\cap\Z^3\\(\beta)\in \prim\cap C}}
 \e\left(\frac{\overrightarrow{b}.\widehat{\beta}}{l}\right)f_{\n}(\beta,\overrightarrow{\v}).
\end{align*} On remarque alors que la condition $dD|\widehat{\eta}_1\wedge\widehat{\eta}_2$ est équivalente
à l'existence de $1\leq\lambda\leq dD$ tel que   $\widehat{\eta}_2\equiv\lambda\widehat{\eta}_1\mod dD$ et 
 $(\lambda,dD)=1$ en raison de la primitivité de $\widehat{\eta_1}$ et $\widehat{\eta}_2$.
 En écrivant $\widehat{\eta}_2=\lambda\widehat{\eta}_1+dD\widehat{\eta}_3$ et 
 $\widehat{\eta}_1=\widehat{\eta}_4+d^*qm\widehat{\eta_5}$, on remarque  que 
$\widehat{\eta}_3$ et $\widehat{\eta}_4$ sont solutions des congruences \begin{align*}
  \left\{\begin{array}{l}
H_{ij}\left(\widehat{\eta}_4,\lambda\widehat{\eta}_4+dD\widehat{\eta}_3
 \right)\equiv Dd^*a_i\pmod{qDd^*},\\
H_{ij}\left(\widehat{\eta}_4,\lambda\widehat{\eta}_4+dD\widehat{\eta}_3
 \right)\equiv 0\pmod{mDd^*}.
\end{array}
 \right.\end{align*}
 En effet, d'après la formule de Taylor, il existe un polynôme $\widetilde{H}_{ij}$ tel que
 \begin{align*}
     H_{ij}(\widehat{\eta}_1,\lambda\widehat{\eta}_1+dD\widehat{\eta}_3)= 
     H_{ij}(\widehat{\eta}_1,\lambda\widehat{\eta}_1)+dD\widetilde{H}_{ij}(\widehat{\eta}_1,\widehat{\eta}_3
     ,\lambda,d,D).
                                                                                    \end{align*}
Au vu de la définition (\ref{définition coordonnées}), $    H_{ij}(\widehat{\eta}_1,\lambda\widehat{\eta}_1)=0$ ce qui entraîne que  
 \begin{align*}
  \left\{\begin{array}{l}
H_{ij}\left(\widehat{\eta}_1,\widehat{\eta}_2\right)\equiv Dd^*a_i\pmod{qDd^*},\\
H_{ij}\left(\widehat{\eta}_1,\widehat{\eta}_2\right)\equiv 0\pmod{mDd^*}
\end{array}
 \right.\end{align*}
 si et seulement si 
 \begin{align*}
  \left\{\begin{array}{l}
d\widetilde{H}_{ij}(\widehat{\eta}_1,\widehat{\eta}_3
     ,\lambda,d,D)\equiv d^*a_i\pmod{qd^*},\\
d\widetilde{H}_{ij}(\widehat{\eta}_1,\widehat{\eta}_3
     ,\lambda,d,D)\equiv 0\pmod{md^*}.
\end{array}
 \right.\end{align*}
 Ce dernier système ne dépendant que des classes de $\widehat{\eta}_1$ et $\widehat{\eta}_3$ modulo $mqd^*$, 
 on en déduit qu'il est équivalent à
  \begin{align*}
  \left\{\begin{array}{l}
d\widetilde{H}_{ij}(\widehat{\eta}_4,\widehat{\eta}_3
     ,\lambda,d,D)\equiv d^*a_i\pmod{qd^*},\\
d\widetilde{H}_{ij}(\widehat{\eta}_4,\widehat{\eta}_3
     ,\lambda,d,D)\equiv 0\pmod{md^*}.
\end{array}
 \right.\end{align*}
 En effectuant le raisonnement en sens inverse, il s'ensuit finalement que le système est équivalent à
  \begin{align*}
  \left\{\begin{array}{l}
H_{ij}\left(\widehat{\eta}_4,\lambda\widehat{\eta}_4+dD\widehat{\eta}_3
 \right)\equiv Dd^*a_i\pmod{qDd^*},\\
H_{ij}\left(\widehat{\eta}_4,\lambda\widehat{\eta}_4+dD\widehat{\eta}_3
 \right)\equiv 0\pmod{mDd^*}.
\end{array}
 \right.\end{align*}
Dans la mesure où la somme définissant $S(l,\overrightarrow{b},\mathfrak{C})$ porte sur les vecteurs
$\widehat{\beta}$ primitifs, la propriété d'orthogonalité des caractères additifs entraîne que, 
si $\widehat{\eta}$ n'est pas primitif, alors
\begin{align*}\sum_{\substack{\overrightarrow{b}
\in (\Z/(d^*qmdD)\Z)^3}}
\e\left(-\frac{\overrightarrow{b}.\widehat{\eta}}{d^*qmdD}\right) S\left(d^*qmdD,\overrightarrow{b},\mathfrak{C}\right)=0.
\end{align*}
Au vu de ce qui précède, on peut donc écrire
\begin{align*}
S^{(8 )}(m,d,D,\mathfrak{C}_1\times\mathfrak{C}_2):=&
\frac{1}{(d^*qmdD)^6}
\sum_{\substack{\lambda\in (\Z/(dD)\Z)^*\\
\overrightarrow{b}_1,\overrightarrow{b_2}\in (\Z/(d^*qmdD)\Z)^3}}
T\left(\lambda, \overrightarrow{b}_1,\overrightarrow{b}_2\right)\\&\times
S\left(d^*qmdD,\overrightarrow{b}_1,\mathfrak{C}_1\right)
S\left(d^*qmdD,\overrightarrow{b}_2,\mathfrak{C}_2\right)
\end{align*}
où 
\begin{align*}
T\left(\lambda, \overrightarrow{b}_1,\overrightarrow{b}_2\right)
:=\sum_{\widehat{\eta}_3,\widehat{\eta}_4,\widehat{\eta}_5}\e\left(
-\frac{\left(\overrightarrow{b}_1+\lambda\overrightarrow{b}_2\right).\widehat{\eta}_5}{dD}\right)
\e\left(-\frac{\left(\overrightarrow{b}_1+\lambda\overrightarrow{b}_2\right).\widehat{\eta}_4+dD\overrightarrow{b}_2
.\widehat{\eta}_3}{d^*qmdD}\right)
\end{align*}
où $\widehat{\eta}_3,\widehat{\eta}_4\in \Z/(d^*qm)\Z$ et $\widehat{\eta}_5\in\Z/(dD)\Z$  satisfont 
\begin{align*}
Dd^*a_i\equiv H_{ij}\left(\widehat{\eta}_4,\lambda\widehat{\eta}_4+dD\widehat{\eta}_3
 \right)\mod qDd^*\text{ et }
H_{ij}\left(\widehat{\eta}_4,\lambda\widehat{\eta}_4+dD\widehat{\eta}_3\right)\equiv 0\mod mDd^*.
\end{align*}
La propriété d'orthogonalité des caractères additifs
 et la majoration triviale entraînent  que 
\begin{align*}
 \left|T\left(\lambda,\overrightarrow{b}_1,\overrightarrow{b}_2\right)\right|\leq\left\{\begin{array}{ll}
(dD)^3(d^*qm)^6&\text{ si }\overrightarrow{b}_1+\lambda\overrightarrow{b}_2\equiv0\mod dD,\\
0&\text{ sinon}.
                                                                 \end{array}
\right.
\end{align*}
De plus, si $T\left(\lambda,\overrightarrow{b}_1,\overrightarrow{b}_2\right)\neq0$, alors $\overrightarrow{b}_2$ est déterminé de manière unique modulo $dD$ par la donnée de $\lambda$ et $\overrightarrow{b}_1$. 
 L'inégalité de Cauchy-Schwarz  entraîne donc l'estimation
\begin{align*} 
 \left|S^{(8 )}(m,d,D,\mathfrak{C}_1\times\mathfrak{C}_2)\right|\ll \frac{(qm)^3}{(dD)^2}
\max_{i=1,2}\sum_{\overrightarrow{b}\mod d^*qmdD}\left|S\left(d^*qmdD,\overrightarrow{b},\mathfrak{C}_i\right)\right|^2.
\end{align*}
Afin d'appliquer une inégalité de grand crible, on réduit les phases  $\overrightarrow{b}(d^*qmdD)^{-1}$ 
à des phases irréductibles $\overrightarrow{b}l^{-1}$, c'est-à-dire telle que $(l,b_1,b_2,b_3)=1$.
On  remarque pour cela qu'une phase irréductible $\overrightarrow{b}l^{-1}$ apparaît avec un poids qui peut être majoré par
\begin{align*}
 \sum_{D\geq V  \X^{-1}\Y^{-1}}\sum_{\substack{d\\l|d^*qmdD}}\frac{(qm)^3}{(dD)^2}\ll 
(qm)^5\sum_{\substack{v\geq V  \X^{-1}\Y^{-1}\\l|v}}\frac{\tau(v)}{v^2}\ll(qm)^5\frac{\tau(l)}{l}\X\Y\frac{\log V}{V}.
\end{align*}
On en déduit que\begin{align*}
S_V^{(7 )}(m,\z,\Delta,\mathfrak{C}_1\times\mathfrak{C}_2)
\ll \frac{m^5XY(\log \X)^c}{V}
\max_{i=1,2}
\sum_{l\leq mqd^*d_0}\frac{\tau(l)}{l}\sideset{}{^{(l)}}\sum_{\overrightarrow{b} \mod l}
\left|S\left(l,\overrightarrow{b},\mathfrak{C}_i\right)\right|^2
\end{align*} où la sommation $\sideset{}{^{(l)}}\sum$ porte sur les vecteurs $\overrightarrow{b}$ tels que $(l,b_1,b_2,b_3)=1$.

On peut alors reproduire étapes par étapes les arguments du paragraphe 13 de \cite{HB01}. 
Il s'ensuit  l'estimation
\begin{align*}
S_V^{(7 )}(m,\z,\Delta,\mathfrak{C}_1\times\mathfrak{C}_2)\ll m^5
\X V\left(\Y Q_1^4\exp\left(-c\sqrt{\log L}\right)+\Y Q_1^{-1/4}+m^3\Y^{46}  \X^{-\tau/2}\right)(\log \X)^{c}
\end{align*}
où $Q_1:=(\log\X)^{c_1}$ avec $c_1>0$ une constante qui peut-être choisie arbitrairement grande.
Par suite, en collectant les différents termes d'erreur, il vient que
\begin{align*}
 S_V(m)\ll 
   \X^2\left(\Y^{-\frac{1}{2}} +m^{5/2}\Y^{15/2}
 \left(Q_1^2\exp\left(-c\sqrt{\log L}\right)+Q_1^{-1/8}\right)+m^4\Y^{30}  \X^{-\tau/4}\right)(\log \X)^c
\end{align*}
et donc
\begin{align*}
& S_V\ll   \X^2\Big(M^{-1}+M\Y^{-\frac{1}{2}} +M^{7/2}\Y^{15/2}
 \left(Q_1^2\exp\left(-c\sqrt{\log L}\right)+Q_1^{-1/8}\right)\qquad\qquad\qquad\\&\pushright{
 +M^5\Y^{30}  \X^{-\tau/4}\Big)(\log \X)^c.}
\end{align*}
Quitte à supposer la constante $c_1$ introduite dans la définition de $Q_1$ suffisamment grande, les choix $\Y=M^4$ et $M=Q_1^{1/276}$
 assurent que $q^{\frac{1}{2}}\leq \Y $ et impliquent finalement que
\begin{align*}
 S_V\ll \X^2(\log\X)^{-B}.
\end{align*}
  L'uniformité du résultat en $(\log \X)^{-\varpi_2}\leq \xi\leq \tau$ ainsi que la non-effectivité
 de la constante implicite 
 sont des conséquences directes de la preuve du lemme 3.8 de \cite{HB01}.
\end{proof}

En conclusion de cette partie, on peut énoncer une borne supérieure du terme d'erreur $S(|h|)$ définit par (\ref{définition Sh}) sous l'hypothèse $|h|\leq1$. 
\begin{prop}\label{conséquence fin type II}
Soit $B>0$. Supposons que $|h|\leq1$. Il existe une constante $c(B)>0$ tel que, uniformément 
en $\X\geq 2$, $(\log \X)^{-\tau}\leq\xi\leq \tau$,  $\eta=(\log \X)^{-c(B)}$ et $(N_1,N_2)\in\mathcal{N}(\eta)$, on ait
\begin{align*}
S(|h|)\ll \eta^2  \X^2(\log \X)^{-B}+\sum_{i=1}^5\sum_{\m,n}R^{(i)}(\m,\n)
\end{align*}
où les $R^{(i)}(\m,\n)$ sont définis par (\ref{définition Rimn}).
\end{prop}
\begin{proof}
Soit $c_1>0$. Puisqu'il y a au plus $O\left(\xi^{-1}\right)$ choix possibles pour chaque $v_i$,
les lemmes
\ref{Type II B}, \ref{estimation de Se(A,I,C)}  et \ref{Estimations de Type II}  
entraînent l'existence d'une constante $c>0$ indépendante de $c_1$ et $\eta$ telle que, pour $ i\in\{1,\ldots,5\}$,
\begin{align*}
 \sum_{\m,\n}\sum_{N(\I)\in\mathcal{I }}\left|\widehat{S}(\mathcal{A},C^{(i)}(\m,\n,\I)-\frac{\eta\sigma_q(F)}{c(N_1,N_2)q^3\X}
\sigma_q(\I)\widehat{S}(\mathcal{B},C^{(i)}(\m,\n,\I))\right|\\\ll_{c_1}   \X^2\left(\eta^{\frac{11}{5}}+\xi^{-c\tau^{-1}}(\log\X)^{-c_1}\right)(\log \X)^{c}.
\end{align*}
En remarquant que la condition $(\log \X)^{-\tau}\leq\xi\leq\tau$ entraîne que $\xi^{-c\tau^{-1}}\leq(\log \X)^c$, on en déduit l'existence de $c_2>0$ tel que
\begin{align*}
 \sum_{\m,\n}\sum_{N(\I)\in\mathcal{I }}\left|\widehat{S}(\mathcal{A},C^{(i)}(\m,\n,\I)-\frac{\eta\sigma_q(F)}{c(N_1,N_2)q^3\X}\sigma_q(\I)
\widehat{S}(\mathcal{B},C^{(i)}(\m,\n,\I))\right|\\\ll_{c_1}   \X^2\left(\eta^{\frac{17}{8}}+(\log\X)^{-c_1}\right)(\log \X)^{c_2}.
\end{align*}
En choisissant $c(B)=8(c_2+B)$ et $c_1=17(c_2+B)$ et 
en utilisant la formule
\begin{multline*}S(|h|)\leq \sum_{i=1}^5\sum_{\m,\n}\sum_{N(\I)\in\mathcal{I }}\left|\widehat{S}(\mathcal{A},C^{(i)}(\m,\n,\I)-\frac{\eta\sigma_q(F)}{c(N_1,N_2)q^3\X}\sigma_q(\I)
\widehat{S}(\mathcal{B},C^{(i)}(\m,\n,\I))\right|\\+\sum_{i=1}^5\sum_{\m,\n}R^{(i)}(\m,\n),
\end{multline*}
on en déduit le résultat.
\end{proof}

\section{Applications}\label{Partie Applications} 
\subsection{Application aux fonctions de $\mathcal{M}(z)$ : preuve du théorème \ref{ordre moyen fonction multiplicative}}\label{Application 1}
\begin{proof}On pose $\tau=(\log_2\X)^{\varepsilon/2-1}$ et $\tau_1=(\log_2\X)^{\varepsilon-1}$. La multiplicativité de $h$ implique qu'une décomposition de la forme (\ref{décomposition 1 f})  existe puisque l'on peut écrire $ h(m)=h(m^-(  \X^{\tau}))z^k$ pour les entiers $m$ tels que $m^+(  \X^{\tau})$ soit sans facteur carré et $\omega(m^+(  \X^{\tau}))=k$.
On peut donc utiliser le corollaire \ref{corollaire des applications}. Dans la mesure où la somme définissant $\Delta(|h|)$ est vide (il n'y a pas de $E_j(k)$), on  en déduit que l'on a, uniformément en $(N_1,N_2)\in\mathcal{N}(\eta)$, la formule
\begin{align*}
 S(\mathcal{A};h)=\frac{\eta\sigma_q(F)}{c(N_1,N_2)q^3\X} S(\mathcal{B};\sigma_q h))
+O\left( \tau_1 \eta^2  \X^2\right).
\end{align*}
La proposition \ref{Selberg-Delange h sigma} permet  d'approcher $S(\mathcal{B};\sigma_qh)$. Dans la mesure où $\omega(q)\ll \frac{\log_2\X}{\log_3\X}$, il vient
 \begin{align*}  S(\mathcal{B};\sigma_qh)= 
c(N_1,N_2)\eta q^3\X^3\sigma_q(F)
\left(3\log \X
\right)^{z-1}
\left(
\frac{\sigma_q(F,h)}{\Gamma(z)}
+O\left(\frac{1}{
(\log\X)^{1-\varepsilon}}
\right)
\right)
\end{align*}
Au vu de (\ref{décomposition Jordaniesque}), on en déduit le résultat en sommant sur les différents $(N_1,N_2)\in\mathcal{N}(\eta)$ 
et en estimant trivialement la contribution des $(N_1,N_2)\notin\mathcal{N}(\eta)$ à l'aide de 
(\ref{contribution éléments hors mathcalN eta}) et (\ref{contribution éléments hors mathcalN eta}).
\end{proof}
\subsection{Valeurs friables de formes binaires cubiques : preuve du théorème \ref{théorème principal friable}}

Considérant le cas $h=1_{S(\Y)}$ où $S(\Y)$ est l'ensemble des entiers $\Y$-friables, on peut écrire  la décomposition  (\ref{décomposition 1 f}) pour les entiers $m$ tels que $\omega(m^+(  \X^{\tau}))=\Omega(m^+(  \X^{\tau}))=k$
sous la forme 
\begin{align*}
1_{S(\Y)}(m)=1_{E(k)}(m^+(  \X^{\tau})) \qquad\text{avec}\qquad
 E(k):=\left\{m:\omega(m)=\Omega(m)=k\text{ et }P^{(k)}(m)\leq \Y 
 \right\}
\end{align*}
dès que $\Y>\X^{\tau}$. 

Concentrons nous dans un premier temps sur la démonstration de (\ref{formule Psi1 irr}).
En appliquant le corollaire \ref{corollaire des applications} avec les choix $\tau:=(\log_2\X)^{\varepsilon/2-1}$ et $\tau_1:=(\log_2\X)^{\varepsilon-1}$, il suit la formule
\begin{displaymath}
S(\mathcal{A};h)-\frac{\eta\sigma_q(F)}{c(N_1,N_2)q^3\X}S(\mathcal{B};\sigma_q1_{S(\Y)})
\ll  \tau_1 \eta^2  \X^2+\Delta(1_{S(\Y)}).
\end{displaymath} 
Compte tenu de la définition de $E(k)$ ci-dessus et de la définition \ref{définition Delta} de $\Delta(1_{S(\Y)})$, on remarque qu'il n'y a pas de terme  $\Delta(k,\mathcal{D},\I,1_{S(\Y)})$ qui intervient, dans la mesure
où l'inégalité $P^+(m)\leq\Y$ est de la forme
\begin{displaymath}
\Y P^{(\overrightarrow{\alpha})}(m)\prec P^{(\overrightarrow{\beta})}(m)
\end{displaymath}
avec $\beta=\emptyset$, d'où
\begin{equation}\label{avant Selberg-Delange grande friabilité}
S(\mathcal{A};h)-\frac{\eta\sigma_q(F)}{c(N_1,N_2)q^3\X}S(\mathcal{B};\sigma_q1_{S(\Y)})
\ll  \tau_1 \eta^2  \X^2.
\end{equation} 

Dans le domaine (\ref{deuxième domaine irréductible}) défini par \begin{equation*}
\X\geq3,\qquad \exp\left(\frac{\log \X}{(\log_2 \X)^{1/2-\varepsilon}}\right)\leq \Y \leq \X^{1/2-\varepsilon},\end{equation*}
on observe que seuls interviennent les ensembles  $C^{(5)}(m,n)$ défini par (\ref{définition C5}). 
Il suit alors de la discussion du paragraphe \ref{description de la méthode} que l'on  a
\begin{align*}
S(\mathcal{A};1_{S(\Y)})-\frac{\eta\sigma_q(F)}{c(N_1,N_2)q^3\X}S(\mathcal{B};\sigma_q1_{S(\Y)})\ll   e^{c\tau^{-1}}\Delta_0\left(  \X^{2\tau/3}\right)+\Theta(|h|,  \X^{\tau},  \X^{\tau_1})+S(|h|).
\end{align*} En utilisant les lemmes \ref{contribution entiers carrus Greaves}, \ref{grande partie friable} et \ref{Approximation S par Se}, il vient respectivement que, pour tout $B>0$, on a
\begin{displaymath}
\Delta_0\left(  \X^{2\tau/3}\right)\ll\X^2(\log\X)^{-B}
\end{displaymath}
\begin{displaymath}
\Theta(|h|,  \X^{\tau},  \X^{\tau_1})\ll
\eta^2\X^2(\log_3\X)^{c_1}\exp\left(-c_2\frac{\tau_1}{\tau}\log\frac{\tau_1}{\tau}\right)
\end{displaymath}
et
\begin{displaymath}
S(|h|)\ll\xi\tau^{-5}\eta^2  \X^2.
\end{displaymath}
On pose alors  $\tau_1:=\varepsilon/2$, $\tau:=(\log_2\X)^{\varepsilon/2-1/2}$ et 
 $\xi:=(\log \X)^{-\tau}$ conduisant ainsi à la borne supérieure
 \begin{equation}\label{avant Selberg-Delange petite friabilité}
S(\mathcal{A};h)-\frac{\eta\sigma_q(F)}{c(N_1,N_2)q^3\X}S(\mathcal{B};\sigma_q1_{S(\Y)})
\ll  \exp\left(-(\log_2\X)^{1/2-\varepsilon}\right)
\eta^2  \X^2.
 \end{equation} 

Pour achever la preuve du théorème, il convient  de déterminer une formule asymptotique de $S(\mathcal{B};\sigma_q1_{S(\Y)})$. Une telle estimation peut-être obtenue en utilisant les résultats de \cite{HTW08} où Hanrot, Tenenbaum et Wu développent une version friable de la méthode de Selberg-Delange. \begin{prop}\label{moyenne friable sigma}
  Il existe $C>0$ tel que, uniformément en $(N_1,N_2)\in\mathcal{N}(\eta)$, 
$\X\geq3$ %, $q\leq (\log \X)^A$
et $
\exp\left((\log_2\X)^{2+\varepsilon}\right)\leq \Y \ll q^3  \X^3%(\X,\Y)\in H_{\varepsilon}
$, on ait\begin{multline*}
S(\mathcal{B};\sigma_q1_{S(\Y)})=\frac{1}{\zeta_q(2)\sigma_q(F)}\eta c(N_1,N_2)q^3  \X^3\rho\left(3u
\right)\\
+O\left(
 \left(C^{\omega(q)}\log(q+1)^4+\log_2\X\right)\eta c(N_1,N_2)q^3\X^3\rho(u)
 \frac{\log(u+1)}{\log \Y }+\eta^2\X^3\right)
\end{multline*} 
où $u=\frac{\log \X}{\log \Y }$.
\end{prop}
\begin{proof}
Au vu de l'estimation (\ref{estimation G(s) général}), on peut appliquer le théorème 1.2 de \cite{HTW08}.
Il s'ensuit que, uniformément en %  $q\leq(\log\X)^A $, 
$\X\geq3$ et 
$\exp\left((\log_2 \X)^{5/3+\varepsilon)}\right)\leq \Y \leq \X$, \begin{equation}\label{formule sigma moyenne terme erreur}
\sum_{\substack{n\leq \X\\P^+(n)\leq \Y \\n\text{ }q\text{-régulier}}}\widetilde{\sigma}_q^{\Z}(n)
=\Lambda_{\widetilde{\sigma}_q^{\Z}}(\X,\Y)+O\left(\frac{\X\rho\left(u\right)}
{L_{\varepsilon/3}(\Y)}\right)
\end{equation}
où \begin{align*}
L_{\varepsilon/3}(\Y):=
\exp\left((\log\Y)^{3/5-\varepsilon/3}\right),\qquad
\Lambda_{\widetilde{\sigma}_q^{\Z}}
\left(\X,\Y \right):=
\X\int_{0}^u\rho(u-v)\text{d}\left(\frac{M_{\widetilde{\sigma}_q^{\Z}}(\Y^v)}{\Y^v}\right)\end{align*}
et
\begin{align*} M_{\widetilde{\sigma}_q^{\Z}}(\Y^v):=
\sum_{n\leq \Y^v}\widetilde{\sigma}_q^{\Z}(n).\end{align*}
Au vu de la formule (1.19) de \cite{HTW08}, on dispose de l'évaluation uniforme en $q\geq1$ et $t\geq1$,
\begin{align}
M_{\widetilde{\sigma}_q^{\Z}}(t)=\lambda_0(\mathbf{1})t\left(1+O\left(\frac{2^{\omega(q)}}{t^{1/6-\varepsilon}}\right)\right)\label{ordre moyen sigma pour friable}
\end{align}
 où $\lambda_0(\mathbf{1})$ est défini par (\ref{estimation lambdah}).
En utilisant une intégration par parties, il vient la formule 
\begin{align}
 \Lambda_{\widetilde{\sigma}_q^{\Z}}
\left(\X,\Y \right)
&\label{IPP rho}=\lambda_0(\mathbf{1})\X\rho\left(u
\right)+\X\int_0^{u}\rho'\left(u-v\right)\frac{M_{\widetilde{\sigma}_q^{\Z}}(\Y^v)-\lambda_0(\mathbf{1})\Y^v}{\Y^{v}}\text{d}v+
O\left( 2^{\omega(q)}\X^{5/6+\varepsilon}\right).
\end{align}
Compte tenu de l'estimation uniforme en $0\leq v\leq u$ suivante
\begin{align*} 
 \rho'(u-v)\ll \log(u+1)\rho(u)\exp\left(O(v\log(u+1))\right)
\end{align*}
établie dans [\cite{Te08}, p.507], on obtient alors la formule 
\begin{align}
\int_0^{u}\rho'\left(u-v\right)\frac{M_{\widetilde{\sigma}_q^{\Z}}(\Y^v)-\lambda_0(\mathbf{1})\Y^v}{\Y^{v}}\text{d}v\ll2^{\omega(q)}\frac{\rho(u)\log(u+1)}{\log\Y}\label{Psi moyenne gros intervalles}.
\end{align} 
Le corollaire III.5.14 de \cite{Te08} impliquant pour tout $j\geq0$ la majoration valide
  \begin{equation}\label{Taylor rho}
\rho^{(j)}(u+v)-\rho^{(j)}(u)\ll\left\{\begin{array}{ll}
&\text{ si }u\leq j\text{ et }\lfloor u\rfloor\neq\lfloor u+v\rfloor,\\
v\rho(u)\left(\log (u+1)\right)^{j+1}&\text{ sinon }                               
                                \end{array}\right.
 \end{equation}  pour $ u,v\geq0$,
 il s'ensuit que l'on a, dans le domaine  $\exp\left((\log_2 \X)^{5/3+\varepsilon}\right)\leq \Y \leq \X$,\begin{align}
 \nonumber&\int_0^{\infty}
\left(\rho'\left( u
-v+\frac{\log(1+\eta)}{\log \Y }\right)-\rho'\left(\nonumber u-v
\right)\right)\frac{M_{\widetilde{\sigma}_q^{\Z}}(\Y^v)-\lambda_0(\mathbf{1})\Y^v
}{\Y^{v}}\text{d}v
\\\ll&\nonumber2^{\omega(q)}\frac{\eta }{\log \Y }
(\log(u+1))^2\int_0^u\frac{\rho(u-v)}{\Y^{(1/6-\varepsilon)v}}\diff v
+2^{\omega(q)}
\int_{u-1}^{u-1+\frac{\log(1+\eta)}{\log \Y }}\frac{1}{\Y^{(1/6-\varepsilon)v}}\diff v\\
\ll& 2^{\omega(q)}\eta\rho(u)\left(\frac{\log(u+1)}{\log \Y }\right)^2.
\end{align}
Il s'ensuit finalement que la formule 
\begin{align}\label{estimation friable sigmatilde}
\sum_{\substack{\X<n\leq (1+\eta) \X \\P^+(n)\leq \Y \\n\text{ }q\text{-régulier}}}\widetilde{\sigma}_q^{\Z}(n)
&=
\lambda_0(\mathbf{1})\eta \X \rho\left(u\right)\left(1+O
\left(\frac{2^{\omega(q)}\log(u+1)}{\log \Y }
\right)\right)
\end{align}
est valide dans le domaine $\exp\left((\log_2 \X)^{5/3+\varepsilon)}\right)\leq \Y \leq \X$, puis dans le domaine  $\X\leq \Y \leq   \X^2$ au vu de (\ref{ordre moyen sigma pour friable}).

Afin d'évaluer la quantité $S(\mathcal{B};\sigma_q1_{S(\Y)})$, on écrit la convolution suivante, qui est l'analogue de (\ref{convolution sigm q régulier singulier}),
\begin{align}
\sum_{\substack{\X<n\leq (1+\eta) \X \\P^+(n)\leq \Y }}\sigma_q^{\Z}(n)&=
\sum_{\substack{d\leq \X(1+\eta) \\ d\text{ }q\text{-singulier}}}\sigma_q^{\Z}(d)
\sum_{\substack{\X/d<n\leq (1+\eta) \X /d\\P^+(n)\leq \Y \\n\text{ }q\text{-régulier}}}\widetilde{\sigma}_q^{\Z}(n)\label{convolution kappaq 2}.
\end{align}
On traite la contributions des entiers $q$-singuliers $d\leq (\log \X)^B$  en utilisant 
 (\ref{ordre moyen sigma pour friable}). Dans le domaine $\exp\left((\log_2 \X)^{2+\varepsilon)}\right)\leq \Y \ll   q^3\X^3$,  l'estimation
 \begin{displaymath}
 \rho\left(u-\frac{\log d}{\log\Y}\right)
 =\rho(u)\left(1+O\left(\frac{\log d\log(u+1)}{\log\Y}\right)\right)
 \end{displaymath}
 entraîne,  en utilisant (\ref{estimation puissance}), que l'on a, 
\begin{multline*}
\sum_{\substack{d\leq (\log \X)^B \\ d\text{ }q\text{-singulier}}}\sigma_q^{\Z}(d)
\sum_{\substack{\X/d<n\leq (1+\eta) \X /d\\P^+(n)\leq \Y \\n\text{ }q\text{-régulier}}}\widetilde{\sigma}_q^{\Z}(n)
=\sum_{\substack{d\leq (\log \X)^B \\ d\text{ }q\text{-singulier}}}\frac{\sigma_q^{\Z}(d)}{d}\lambda_0(\mathbf{1})
\eta \X \rho(u)
\\+O\left( \eta \X\rho(u) C^{\omega(q)}\log(q+1)^4\frac{\log(u+1)}{\log \Y }\right).\end{multline*}
On complète la somme sur $d$ ci-dessus en estimant l'erreur ainsi produite à l'aide de la méthode de Rankin et de (\ref{Rankin sigmaq}). Quitte à choisir $B$ suffisamment grand, il vient ainsi 
\begin{multline*}
 \sum_{\substack{d\leq (\log \X)^c \\ d\text{ }q\text{-singulier}}}\frac{\sigma_q^{\Z}(d)}{d}
\sum_{\substack{\X/d<n\leq (1+\eta) \X /d\\P^+(n)\leq \Y \\n\text{ }q\text{-régulier}}}\widetilde{\sigma}_q^{\Z}(n)=\frac{  \eta \X \rho(u)}{\zeta_q(2)\sigma_q(F)}
\\+O\left( \eta \X\rho(u) C^{\omega(q)}\log(q+1)^4\frac{\log(u+1)}{\log \Y }\right). 
\end{multline*}
Au vu de   la majoration triviale
\begin{align*}
 \sum_{\substack{\X/d<n\leq (1+\eta) \X /d\\P^+(n)\leq \Y \\n\text{ }q\text{-régulier}}}\widetilde{\sigma}_q^{\Z}(n)\leq M_{\widetilde{\sigma}_q^{\Z}}((1+\eta) \X /d)
\end{align*}
on estime la contribution des entiers $d>(\log\X)^B$ en utilisant (\ref{Selberg-Delange kappa}) et (\ref{ordre moyen sigma pour friable}) et l'on obtient finalement 
\begin{align*}\sum_{\substack{\X<n\leq (1+\eta) \X \\P^+(n)\leq \Y }}\sigma_q^{\Z}(n)=\frac{
  \eta \X \rho(u)}{\zeta_q(2)\sigma_q(F)}+O\left( \eta \X\left(\rho(u) C^{\omega(q)}\log(q+1)^4\frac{\log(u+1)}{\log \Y }+\eta\right)\right).
\end{align*}
En observant que, uniformément en $(N_1,N_2)\in\mathcal{N}(\eta)$, on a l'estimation
\begin{align*}
 \rho\left(3u+\frac{\log(c(N_1,N_2)q^3)}{\log \Y }\right)=\rho(3u)\left(1+O\left(\frac{\log_2\X\log(u+1)}{\log \Y }\right)\right),
\end{align*}
on en déduit le résultat.
\end{proof}

La proposition précédente 
 entraîne  que, uniformément en $(N_1,N_2)\in\mathcal{N}(\eta)$ et dans le domaine (\ref{domaine cubique irréductible}) défini par\begin{align*}
\X\geq3,\qquad \exp\left(\frac{\log \X}{(\log_2 \X)^{1-\varepsilon}}\right)\leq \Y, 
\end{align*}
on a
 \begin{align*}
  \frac{\eta\sigma_q(F)}{c(N_1,N_2)q^3\X}S(\mathcal{B};\sigma_q1_{S(\Y)})=\frac{\eta^2  \X^2}{\zeta_q(2)}\rho\left(3u  \right)
+O\left(
\frac{\eta^2  \X^2}{(\log_2\X)^{1-\varepsilon}}\right).
 \end{align*}
Raisonnant comme dans la preuve du théorème \ref{ordre moyen fonction multiplicative}  en utilisant (\ref{décomposition Jordaniesque}), on  déduit le théorème \ref{théorème principal friable} de (\ref{avant Selberg-Delange grande friabilité}) et 
(\ref{avant Selberg-Delange petite friabilité}).

\end{document}